\documentclass[11pt,a4paper,twoside]{article}
\usepackage{titling}
\usepackage{titletoc}
\usepackage{url}

\usepackage[utf8]{inputenc}
\usepackage[T1]{fontenc}
\usepackage{amsmath}
\usepackage{amsfonts}
\usepackage{amssymb}
\usepackage{amsthm}

\usepackage{mathrsfs}
\usepackage{mathtools}
\usepackage{stmaryrd}
\usepackage[all]{xy}
\usepackage{makeidx}
\usepackage{natbib}
\usepackage[english,french]{babel}\frenchsetup{GlobalLayoutFrench=false}
\usepackage{xspace}
\usepackage{etoolbox}

\makeatletter\def\th@plain{\slshape}\makeatother
\makeatletter\patchcmd{\th@remark}{\itshape}{\slshape}{}{}\makeatother

\newcounter{bidon}
\newcommand{\rdb}{\refstepcounter{bidon}}

\newcommand\Subsubsection[1]{
\rdb\addcontentsline{toc}{subsubsection}{#1} 
\subsubsection*{#1}}

\renewcommand\paragraph[1]{
\rdb
\addcontentsline{toc}{paragraph}{#1} 
\medskip \noindent $\bullet$ \textbf{#1}}

\newcommand{\DT}{\fD_\gT}
\newcommand{\VT}{\fV_\gT}

\newcommand{\di}{\,{\vert}\,}

\newcommand{\Di}{\mathrm{Di}}
\newcommand{\Vr}{\mathrm{Vr}}
\newcommand{\Rn}{\mathrm{Rn}}
\newcommand{\vr}{\mathrm{vr}}

\input{MathMacrosSpecVal.tex}

\renewcommand{\SCO}[6]{
\xymatrix @R = .2em @C =4em{
                                  &1 \ar@{-}[dl] \ar@{-}[dr] \\
#3 \ar@{-}[ddr]                   &   & #6 \ar@{-}[ddl] \\
                                  &\bullet\ar@{-}[d] \\
                                  &\bullet   \\
#2 \ar@{-}[ddr] \ar@{-}[uur]      &   & #5 \ar@{-}[ddl] \ar@{-}[uul] \\
                                  &\bullet \ar@{-}[d] \\
                                  &\bullet  \\
#1 \ar@{-}[uur]                   &   & #4 \ar@{-}[uul] \\
                                  & 0 \ar@{-}[ul] \ar@{-}[ur] \\
}
}

\newcommand{\SCOR}[8]{
\xymatrix @R =.2em @C =4em {
                                  &\bullet \ar@{-}[ddl] \ar@{-}[ddr] \\
                                  &#7\ar@{-}[u] \\
#3 \ar@{-}[ddr]                   &   & #6 \ar@{-}[ddl] \\
                                  &\bullet\ar@{-}[d] \\
                                  &\bullet   \\
#2 \ar@{-}[ddr] \ar@{-}[uur]      &   & #5 \ar@{-}[ddl] \ar@{-}[uul] \\
                                  &\bullet \ar@{-}[d] \\
                                  &\bullet  \\
#1 \ar@{-}[uur]                   &   & #4 \ar@{-}[uul] \\
                                  &#8\ar@{-}[d] \\
                                  &\bullet \ar@{-}[uul] \ar@{-}[uur] \\
}
}

\renewcommand \DeuxRegles[2]{%
\vspace{-1.5em}\DeuxCols
{\begin{enumerate}  #1
\end{enumerate}
}
{\begin{enumerate}  #2
\end{enumerate}
}
\vspace{-.5em}
}

\renewcommand \Regles[1]{%
\vspace{-1.5em}\UneCol{
\begin{enumerate}
{#1} 
\end{enumerate}
}
\vspace{-.5em}
}

\DeclareMathAlphabet{\mathpzc}{OT1}{pzc}{m}{it}

\usepackage{stixgras}
\RequirePackage{amsbsy}

\newcommand\Exists{\boldsymbol{\stixexists}}

\newcommand\VDash{\boldsymbol{\stixvdash}}
\newcommand\Land{\boldsymbol{\stixwedge}}
\newcommand\Lor{\boldsymbol{\stixvee}}
\newcommand\Top{\boldsymbol{\stixtop}}
\newcommand\Bot{\boldsymbol{\stixbot}}
\newcommand\bigLand{\boldsymbol{\stixbigwedge}}

\newcommand \vdg{\VDash}
\newcommand \Vd {\,\vdg}
\newcommand \vd {\,\,\vdg}
\newcommand \vii{\Land}
\newcommand \vuu{\Lor}

\newcommand \Vdi[1] {\mathrel{\Vd_{#1}}}
\newcommand \Vii{\bigLand}

\newcommand \sibrouillon[1]{}

\newcommand \hum[1] {\sibrouillon{\noindent {\sf hum: #1}}}

\usepackage[bookmarksopen=false,breaklinks=true,%
      backref=page,pagebackref=true,plainpages=false,%
      hyperindex=true,pdfstartview=FitH,colorlinks=true,%
      pdfpagelabels=true,linkcolor=blue,%
      citecolor=red,urlcolor=red,
      ]%
   {hyperref}

\begin{document}
\selectlanguage{english}

\thispagestyle{empty}
~ 
\vspace{3cm}

\noindent In this file you find the English version starting on the page  numbered \pageref{beginenglish}.

\medskip \noindent  {\Large \bf Valuative lattices and spectra}

\bigskip \noindent  
Then the French version begins on the page numbered \pageref{beginfrench}.

\medskip\noindent   {\Large \bf Treillis et spectres valuatifs}

\smallskip \noindent Le lecteur ou la lectrice sera sans doute surprise de l'alternance des sexes ainsi que de l'orthographe du mot 'corolaire', avec d'autres innovations auxquelles elle n'est pas habituée. En fait, nous avons essayé de suivre au plus près les préconisations de l'orthographe nouvelle recommandée, telle qu'elle est enseignée aujourd'hui dans les écoles en France.  

\bigskip\noindent   {\large \bf Authors}  

\smallskip \noindent Henri Lombardi, Université de Franche-Comté, F-25030 Besançon Cedex, France, \\
email: {\tt Henri.lombardi@univ-fcomte.fr}

\smallskip \noindent  Assia Mahboubi,  Nantes Université, École Centrale Nantes, CNRS, INRIA, LS2N, UMR 6004, F-44000 Nantes, France. email: {\tt Assia.Mahboubi@inria.fr}

\bigskip  \noindent The English version appears in the book

\noindent \textsl{Algebraic, Number Theoretic, and Topological Aspects of Ring Theory}

\noindent Editors: Jean-Luc Chabert • Marco Fontana • Sophie Frisch • Sarah Glaz • Keith Johnson

\noindent Springer 2023

\noindent ISBN 978-3-031-28846-3

\noindent DOI  10.1007/978-3-031-28847-0

\newpage
\thispagestyle{empty}

~

\pagestyle{headings}
\patchcmd{\sectionmark}{\MakeUppercase}{}{}{}
\setcounter{page}{0}\renewcommand\thepage{E\arabic{page}}

\newcommand \sitout[1]{}

\begingroup



\theoremstyle{plain}
\newtheorem{theorem}{Theorem}[subsection]
\newtheorem{thdef}[theorem]{Theorem and definition}
\newtheorem{lemma}[theorem]{Lemma}
\newtheorem{corollary}[theorem]{Corollary}
\newtheorem{proposition}[theorem]{Proposition}
\newtheorem{propdef}[theorem]{Proposition and definition}
\newtheorem{plcc}[theorem]{Concrete local-global principle}
\newtheorem{fact}[theorem]{Fact}
\newtheorem{valsatz}[theorem]{\vst}

\newtheorem{theoremc}[theorem]{Theorem\etoz}
\newtheorem{lemmac}{Lemma\etoz}
\newtheorem{corollaryc}{Corollairy\etoz}
\newtheorem{propositionc}{Proposition\etoz}
\newtheorem{factc}{Fait\etoz}
\newtheorem*{Principleofcoveringbyquotients}{Principle of covering by quotients}

\theoremstyle{definition}
\newtheorem{conjecture}[theorem]{Conjecture}
\newtheorem{definition}[theorem]{Definition}
\newtheorem{definitions}[theorem]{Definitions}
\newtheorem{notation}[theorem]{Notation}
\newtheorem{definota}[theorem]{Definition and notation} 
\newtheorem{convention}[theorem]{Convention}
\newtheorem{problem}[theorem]{Problem}
\newtheorem{question}[theorem]{Question}

\theoremstyle{remark}
\newtheorem{remark}[theorem]{Remark}
\newtheorem{remarks}[theorem]{Remarks}
\newtheorem{comment}[theorem]{Comment}
\newtheorem{comments}[theorem]{Comments}
\newtheorem{example}[theorem]{Example}
\newtheorem{examples}[theorem]{Examples}

\newcommand \Glio {\MA{\mathsf{Liog}}}

\newcommand{\vou}{\MA{\tsbf{ or }}}
\newcommand{\Vou}{\MA{\tsbf{OR}}}
\newcommand \EXists[1] {\tsbf{Introduce }{#1}\tsbf{ such that }\,}
\newcommand \vet {{\tsbf{,}}\,}
\newcommand \Atcl {\mathrm{Clat}}
\newcommand \Atclv {\mathrm{VClat}}
\newcommand \tcl {\mathrm{Clt}}
\newcommand \Tcl {\tcl}

\newcommand\comm{\rdb
\noi{\it Comment. }}

\newcommand\COM[1]{\rdb
\noi{\it Comment #1. }}

\newcommand\comms{\rdb
\noi{\it Comments. }}

\newcommand\Pb{\rdb
\noi{\bf Problem. }}

\newcommand \rem{\rdb
\noi{\sl Remark. }}

\newcommand \REM[1]{\rdb
\noi{\sl Remark#1. }}

\newcommand \rems{\rdb
\noi{\sl Remarks. }}

\newcommand \exl{\rdb
\noi{\bf Example. }}

\newcommand \EXL[1]{\rdb
\noi{\bf Example: #1. }}

\newcommand \exls{\rdb
\noi{\bf Examples. }}

\newcommand\gui[1]{``{#1}''}

\newcommand \thref[1] {Theorem~\ref{#1}}
\newcommand \paref[1] {page~\pageref{#1}}
\newcommand \pstfref[1] {Positivstellensatz formel~\ref{#1}}
\newcommand \pstref[1] {Positivstellensatz~\ref{#1}}

\newcommand \num {{n$^{\mathrm{ o}}$}}

\newcommand\subsubsec[1] {\subsubsection*{#1}}

\newcommand \hdr {induction hypothesis\xspace}
\newcommand \ssi {if and only if\xspace}
\newcommand \cnes {necessary and sufficient condition\xspace}
\newcommand \spdg {without loss of generality\xspace}
\newcommand \Propeq {T.F.A.E.\xspace}
\newcommand \propeq {t.f.a.e.\xspace}
\newcommand \disept {17$^{th}$ Hilbert's problem\xspace}

\def \cad {i.e.\xspace}

\newcommand \Vrai {\mathsf{True}}
\newcommand \Faux {\mathsf{False}}


\newcommand \Amo {$\gA$-module\xspace}
\newcommand \Amos {$\gA$-modules\xspace}

\newcommand \Bmo {$\gB$-module\xspace}
\newcommand \Bmos {$\gB$-modules\xspace}

\newcommand \Zmo {$\gZ$-module\xspace}
\newcommand \Zmos {$\gZ$-modules\xspace}

\newcommand \ZZmo {$\ZZ$-module\xspace}
\newcommand \ZZmos {$\ZZ$-modules\xspace}

\newcommand \Ali {$\gA$-\ali}
\newcommand \Alis {$\gA$-\alis}

\newcommand \Alg {$\gA$-\alg}
\newcommand \Algs {$\gA$-\algs}

\newcommand \kev {$\gk$-vector space\xspace}
\newcommand \kevs {$\gk$-vector spaces\xspace}

\newcommand \Kev {$\gK$-vector space\xspace}
\newcommand \Kevs {$\gK$-vector spaces\xspace}

\newcommand \klg {$\gk$-\alg}
\newcommand \klgs {$\gk$-\algs}

\newcommand \Klg {$\gK$-\alg}
\newcommand \Klgs {$\gK$-\algs}

\newcommand \ac {algebraically closed\xspace}
\newcommand \alc {\agq closure\xspace}

\newcommand \adv {valuation domain\xspace}
\newcommand \advs {valuation domains\xspace}

\newcommand \arv {valuation ring\xspace}
\newcommand \arvs {valuation rings\xspace}

\newcommand \agq {algebraic\xspace}

\newcommand \alg {algebra\xspace}
\newcommand \algs {algebras\xspace}

\newcommand \agB {Boolean \alg}

\newcommand \algo{algorithm\xspace}
\newcommand \algos{algorithms\xspace}

\newcommand \algq{algorithmic\xspace}

\newcommand \ali {\lin map\xspace}
\newcommand \alis {\lin maps\xspace}

\newcommand \anar {\ari \ri}
\newcommand \anars {\ari \ris}
\newcommand \Anars {\Ari \ris}

\newcommand \ari{arith\-metic\xspace}

\newcommand \auto {automorphism\xspace}
\newcommand \autos {automorphisms\xspace}


\newcommand \cac {algebraically closed field\xspace}
\newcommand \cacs {algebraically closed fields\xspace}

\newcommand \cara{characteristic\xspace}  

\newcommand \carn{characterization\xspace}  
\newcommand \carns{characterizations\xspace}  

\newcommand \cdi{discrete field\xspace}  
\newcommand \cdis{discrete fields\xspace}  

\newcommand \cdf{fraction field\xspace}
\newcommand \cdfs{fraction fields\xspace}
 
\newcommand \cli{integral closure\xspace}  
\newcommand \clis{integral closures\xspace}  

\newcommand \codi {discrete ordered field\xspace}
\newcommand \codis {discrete ordered fields\xspace}

\newcommand \coe {coefficient\xspace}
\newcommand \coes {coefficients\xspace}

\newcommand \cof {\cov}

\newcommand \coh {coherent\xspace}

\newcommand \coli {linear combination\xspace}
\newcommand \colis {linear combinations\xspace}

\newcommand \com {comaximal\xspace}

\newcommand \coo {coordinate\xspace}
\newcommand \coos {coordinates\xspace}

\newcommand \cop {complementary\xspace}

\newcommand \cosv {conservative\xspace}

\newcommand \cvd {valued discrete field\xspace}
\newcommand \cvds {valued discrete fields\xspace}

\newcommand \cvdsc {separably closed valued discrete field\xspace}
\newcommand \cvdscs {separably closed valued discrete fields\xspace}

\newcommand \cvdac {algebraically closed valued discrete field\xspace}
\newcommand \cvdacs {algebraically closed valued discrete fields\xspace}


\newcommand \dcd {residually discrete\xspace}

\newcommand \ddp {Prüfer domain\xspace}
\newcommand \ddps {Prüufer domains\xspace}

\newcommand \ddk {Krull dimension\xspace}

\newcommand \demo {proof\xspace}
\newcommand \dems {proofs\xspace}
\newcommand \demos {\dems}

\newcommand \dfn{definition\xspace}  
\newcommand \Dfn{Definition\xspace}  
\newcommand \Dfns{Definitions\xspace}  
\newcommand \dfns{definitions\xspace}  

\newcommand \dij{disjunctive\xspace}
\newcommand \wdij{weakly \dij}

\newcommand \discri{discriminant\xspace}
\newcommand \discris{discriminants\xspace}

\newcommand \dok {Dedekind domain\xspace}
\newcommand \doks {Dedekind domains\xspace}

\newcommand \dve {divisibility\xspace}

\newcommand \dvz {zerodivisor\xspace}
\newcommand \dvzs {zerodivisors\xspace}

\newcommand \eco{\com \elts}  

\newcommand \egmt{also\xspace} 

\newcommand \egt{equality\xspace} 
\newcommand \egts{equalities\xspace} 

\newcommand \elr{elementary\xspace}  

\newcommand \elt{element\xspace}  
\newcommand \elts{elements\xspace}  

\def \endo {endomorphism\xspace}
\def \endos {endomorphisms\xspace}

\newcommand \entrel {entailment relation\xspace}
\newcommand \entrels {entailment relations\xspace}

\newcommand \eqn  {equation\xspace}
\newcommand \eqns  {equations\xspace}

\newcommand \eqv  {equivalent\xspace}

\newcommand \eqvc  {equivalence\xspace}
\newcommand \eqvcs  {equivalences\xspace}

\newcommand \eseq{essentially equivalent\xspace} 
\newcommand \Eseq{Essentially equivalent\xspace} 

\newcommand \esid{essentially identical\xspace} 
\newcommand \Esid{Essentially identical\xspace} 

\newcommand \evc{vector space\xspace} 
\newcommand \evcs{vector spaces\xspace} 


\newcommand \fab {bounded \fcn}
\newcommand \fabs {bounded \fcns}

\newcommand \fac {total \fcn}

\newcommand \facile{\begin{proof}
Left to the reader.
\end{proof}}

\newcommand \fap {partial \fcn}
\newcommand \faps {partial \fcns}

\newcommand \fcn {factorization\xspace}
\newcommand \fcns {factorizations\xspace}

\newcommand \fdi{strongly discrete\xspace} 


\newcommand\gmq{geometric\xspace}

\newcommand\gne{generalised\xspace}

\newcommand\gnl{general\xspace}

\newcommand\gnlt{generally\xspace}

\newcommand\gnn{generalization\xspace}
\newcommand\gnns{generalizations\xspace}

\newcommand\gnq{generic\xspace}

\newcommand\grl{$\ell$-group\xspace}
\newcommand\grls{$\ell$-groups\xspace}

\newcommand \gtr{generator\xspace}  
\newcommand \gtrs{generators\xspace}  


\newcommand \homo {homomorphism\xspace}
\newcommand \homos {homomorphisms\xspace}

\newcommand \id {ideal\xspace}
\newcommand \ids {ideals\xspace}

\newcommand \idd {de\-ter\-mi\-nantal \id}
\newcommand \idds {de\-ter\-mi\-nantal \ids}

\newcommand \idema {maximal \id}
\newcommand \idemas {maximal \ids}

\newcommand \idep {prime \id}
\newcommand \ideps {prime \ids}

\newcommand \idemi {minimal prime\xspace}
\newcommand \idemis {minimal primes\xspace}

\newcommand \idf {Fitting \id}
\newcommand \idfs {Fitting \ids}

\newcommand \idm {idempotent\xspace}
\newcommand \idms {idempotents\xspace}

\newcommand \idp {principal \id}
\newcommand \idps {principal \ids}

\newcommand \idtr {indeterminate\xspace}
\newcommand \idtrs {indeterminates\xspace}

\newcommand \ifr {fractional \id}
\newcommand \ifrs {fractional \ids}

\newcommand \inteq {intuitively \eqv}

\newcommand \itf {\tf \id}
\newcommand \itfs {\tf \ids}

\newcommand \iso {isomorphism\xspace}
\newcommand \isos {isomorphisms\xspace}

\newcommand \iv {invertible\xspace}

\newcommand \lec {reader\xspace}

\newcommand \lgb {local-global\xspace}

\newcommand \lin {linear\xspace}

\newcommand \lon {localization\xspace}
\newcommand \lons {localizations\xspace}

\newcommand \lop {\lot principal\xspace}

\newcommand \losd {\lot \sdz\xspace}

\def \lot {locally\xspace}

\newcommand \mlp {principal \lon matrix\xspace}
\newcommand \mlps {principal \lon matrices\xspace}

\newcommand \mnp {manipulation\xspace}
\newcommand \mnps {manipulations\xspace}
\newcommand \mnr {\elr \mnp}
\newcommand \mnrs {\elr \mnps}

\newcommand \mo {monoid\xspace}
\newcommand \mos {monoids\xspace}
\newcommand \moco {\com \mos}

\newcommand \mpf {\pf module\xspace}
\newcommand \mpfs {\pf modules\xspace}

\newcommand \mpn {\pn matrix\xspace}
\newcommand \mpns {\pn matrices\xspace}

\newcommand \mpr {\pro module\xspace}
\newcommand \mprs {\pro modules\xspace}

\newcommand \mprn {\prn matrix\xspace}
\newcommand \mprns {\prn matrices\xspace}

\newcommand \mptf {\ptf module\xspace}
\newcommand \mptfs {\ptf modules\xspace}

\newcommand \mrc {projective module of constant rank\xspace}
\newcommand \mrcs {projective modules of constant rank\xspace}


\newcommand \ncr{necessary\xspace}

\newcommand \ncrt{necessarily\xspace}

\newcommand \ndz {regular\xspace}

\newcommand \noe {Noetherian\xspace}
\newcommand \noco {\noe\coh}

\newcommand \nst {Nullstellensatz\xspace}
\newcommand \nsts {Nullstellens\"atze\xspace}

\newcommand \odz {Zariski open set\xspace}

\newcommand \oqc {\qc open set\xspace}
\newcommand \oqcs {\qc open sets\xspace}

\newcommand \ort {orthogonal\xspace}


\newcommand \pa {saturated pair\xspace}
\newcommand \pas {saturated pairs\xspace}

\newcommand \pb{problem\xspace}  
\newcommand \pbs{problems\xspace}

\newcommand \peq {purely equational\xspace}

\newcommand \pf {finitely presented\xspace}

\newcommand \plg {\lgb principle\xspace}
\newcommand \plgs {\lgb principles\xspace}

\newcommand \plga {abstract \plg}
\newcommand \plgas {abstract \plgs}

\newcommand \Plgc {Concrete \plg}
\newcommand \plgc {concrete \plg}
\newcommand \plgcs {concrete \plgs}

\newcommand \pn {presentation\xspace}
\newcommand \pns {presentations\xspace}

\newcommand \pol {polynomial\xspace}
\newcommand \pols {polynomials\xspace}

\newcommand \polcar {characteristic \pol}

\newcommand \prc {rank constant \pro}

\newcommand \prmt {precisely\xspace}
\newcommand \Prmt {Precisely\xspace}

\newcommand \prn {projection\xspace}
\newcommand \prns {projections\xspace}

\newcommand \pro {projective\xspace}

\newcommand \proi {potential prime\xspace}
\newcommand \prois {potential primes\xspace}

\newcommand \proc {potential chain\xspace}
\newcommand \procs {potential chains\xspace}

\newcommand \proel {elementary \proc}
\newcommand \proels {elementary \procs}
\newcommand \proelo {\proel of length }
\newcommand \proelos {\proels of length }

\newcommand \prolo {\proc of length }
\newcommand \prolos {\procs of length }

\newcommand \prt {property\xspace}
\newcommand \prts {properties\xspace}

\newcommand \pst {Positivstellensatz\xspace}
\newcommand \psts {Positivstellens\"atze\xspace}

\newcommand \ptf {\tf \pro}


\newcommand \qc {quasi-compact\xspace}

\newcommand \qiri {pp-ring\xspace}
\newcommand \qiris {pp-rings\xspace}

\newcommand \ralg {Horn rule\xspace}
\newcommand \ralgs {Horn rules\xspace}

\newcommand \rcf {real closed field\xspace}
\newcommand \rcfs {real closed fields\xspace}

\newcommand \rdl {linear dependance relation\xspace}
\newcommand \rdls {linear dependance relations\xspace}

\newcommand \rdi {integral dependance relation\xspace}
\newcommand \rdis {integral dependance relations\xspace}

\newcommand \rdij {\dij rule\xspace}
\newcommand \rdijs {\dij rules\xspace}

\newcommand \rdv {valuative divisibility relation\xspace}
\newcommand \rdvs {valuative divisibility relations\xspace}

\newcommand \rdy {dynamical rule\xspace}
\newcommand \rdys {dynamical rules\xspace}

\newcommand \red {direct rule\xspace}
\newcommand \reds {direct rules\xspace}

\newcommand \rex {existential rule\xspace}
\newcommand \rexs {existential rules\xspace}

\newcommand \ri {ring\xspace}
\newcommand \ris {rings\xspace}


\newcommand \sad {dynamical algebraic structure\xspace}
\newcommand \sads {dynamical algebraic structures\xspace}
\newcommand \SAD {Dynamical algebraic structure\xspace}
\newcommand \SADs {Dynamical algebraic structures\xspace}

\newcommand \salg {algebraic structure\xspace}
\newcommand \salgs {algebraic structures\xspace}

\newcommand \sdz {without \dvz}

\newcommand \sfio {fundamental system of orthogonal idempotents\xspace}

\newcommand \sgr {\gtr set\xspace}
\newcommand \sgrs {\gtr sets\xspace}

\newcommand \sli {\lin \sys}
\newcommand \slis {\lin \syss}

\newcommand \spb {separable\xspace}
\newcommand \spl {separable\xspace}

\newcommand \sps {spectral space\xspace}
\newcommand \spss {spectral spaces\xspace}

\newcommand \sys {system\xspace}
\newcommand \syss {systems\xspace}

\newcommand \talg {Horn theory\xspace}
\newcommand \talgs {Horn theories\xspace}

\newcommand \tco {coherent theory\xspace}
\newcommand \tcos {coherent theories\xspace}

\newcommand \twdij {\wdij theory\xspace}
\newcommand \twdijs {\wdij theories\xspace}

\newcommand \tdy {dynamical theory\xspace}
\newcommand \tdys {dynamical theories\xspace}

\newcommand \tel {regular theory\xspace}
\newcommand \tels {regular theories\xspace}

\newcommand \telri {cartesian theory\xspace}
\newcommand \telris {cartesian theories\xspace}

\newcommand \tf {finitely generated\xspace}

\newcommand \tfo {formal theory\xspace}
\newcommand \tfos {theory formelles\xspace}

\newcommand \tgm {\gmq theory\xspace}
\newcommand \tgms {\gmq theories\xspace}

\newcommand \Tho {Theorem\xspace}
\newcommand \tho {theorem\xspace}
\newcommand \thos {theorems\xspace}

\newcommand \tpe {purely equational theory\xspace}
\newcommand \tpes {purely equational theories\xspace}

\newcommand \uvl {universal\xspace}

\newcommand \trdi {distributive lattice\xspace}
\newcommand \trdis {distributive lattices\xspace}

\newcommand \vfn {verification\xspace}
\newcommand \vfns {verifications\xspace}

\newcommand \vst {Valuativstellensatz\xspace}
\newcommand \vsts {Valuativstellensätze\xspace}

\newcommand \zed {zero-dimensional\xspace}
\newcommand \zedr {zero-dimensional reduced\xspace}


\newcommand \cov {constructive\xspace}

\newcommand \coma {\cov \maths}
\newcommand \clama {classical \maths}

\renewcommand \cot {constructively\xspace}

\newcommand \mathe {mathematical\xspace}
\newcommand \maths {mathematics\xspace}

\newcommand \matn {mathematician\xspace}

\newcommand \pte {excluded middle principle\xspace}

\newcommand \prco {\cov proof\xspace}
\newcommand \prcos {constructive proofs\xspace}

\newcommand \tcg {compactness theorem\xspace}
\newcommand \Tcgi {The \tcg implies the following result. }

\title{Valuative lattices and spectra}

\newcommand\Subsection[1]{
\rdb\addcontentsline{toc}{subsection}{#1} \subsection*{#1}}

\newcommand\Subsubsectio[2]{
\rdb\addcontentsline{toc}{subsubsection}{#2} 
\subsubsection*{#1}}

\newcommand\Section[1]{\subsection{#1}}
\newcommand\siAqi[1]{}
\newcommand\siGRL[1]{}

\def\proofname{\textsl{Proof}}

\author{Henri Lombardi, Assia Mahboubi}
\maketitle

\startcontents[english]

\rdb
\label{beginenglish}

\begin{abstract}
The first part of the present article consists in a survey about the dynamical \cov method designed using \tdys and \sads. Dynamical methods uncovers a hidden computational content for numerous abstract objects of \clama, which seem a priori inaccessible \cot, e.g., the \agq closure of a (discrete) field. When a \demo in \clama uses these abstract objects and results in a concrete outcome, dynamical methods  \gnlt make possible to discover an \algo for this concrete outcome.
The second part of the article applies this dynamical method to the theory of \dve. We compare two notions of valuative spectra present in the literature and we introduce a third notion, which is implicit in an article devoted to the \tdy of \ac \cvds. The two first notions are, respectively, due to Huber \& Knebusch and to Coquand. We prove that the corresponding valuative lattices are essentially the same. We establish formal \vsts corresponding to these theories, and we compare the various resulting notions of valuative dimensions.
\end{abstract}

\medskip \noindent {\bf Keywords}. 
Coherent theory, Dynamical theory, Dynamical algebraic structure, Constructive mathematics, 
Distributive lattice, Spectral space, 
Zariski lattice, Zariski spectrum, Valuative lattice, Valuative spectrum, Valuative dimensions. 

\smallskip \noindent {\bf MSC}. 03Fxx 13xx 18Fxx

\setcounter{tocdepth}{4}
\markboth{Contents}{Contents}
\small

\printcontents[english]{}{1}{}
\normalsize

\section{Introduction}
This paper is written in the style of \coma à la Bishop, see \cite{Bi67,BB85,BR1987,CACM,MRR,Yen2015}. An updated French version of \cite{CACM} is \cite{ACMC}.

We use the terminology and notations of \tdys. See \cite{BC2005,BC2019,Coq2005,CLR01,Lom98,Lom02,Lom06}.

We continue the \cov approach to \agq curves and valuation domains initiated in \cite{Coq2009,CL2016,CLQ2010,CLR01} and \cite[section~XIII-8]{CACM}. 

We hope to open the way for a simple \cov approach to the following topics. 
\begin{itemize}\itemsep=.1em
\item Understand (\cot) the theory of divisors of Weil-Kronecker when applied to a geometric ring\footnote{A finitely presented algebra over a discrete field.} (continuing the study made in \cite{CL2016}).
\item Understand (\cot) the \tho stating that the theory of the ring of algebraic integers is complete \cite{vdD88,vDM1990,Pre82,PS1990}.
\item Understand (\cot) the Riemann-Roch's \tho 
in its general algebraic version (\cite{Fulton,Edw1990}). 
\item Understand (\cot) the Grothendieck version of Riemann-Roch's \tho 
in~\cite{BS1958}.

\end{itemize}

\medskip 

The three first sections of the present article provide an overview of the constructive dynamical method. This method relies on dynamical theories and on dynamical algebraic structures. Its purpose is to unveil a computational content hidden in a number of abstract objects in classical mathematics. Dynamical methods provide effective counterparts even for objects seemingly inaccessible to a constructive treatment, e.g., the algebraic closure of a (discrete) field. They are typically able to turn a classical proof of a concrete result into an effective algorithm.

The second part of the article uses this method to study the theory of divisibility. In particular, we compare two notions of valuative spectrum present in the literature. The first one arises in the theory of valuation domains, where \cite{HK1994} introduced the valuative spectrum of an arbitrary commutative ring, akin to the Zariski spectrum and to the real spectrum. The second one is the constructive version of the valuative spectrum of an integral domain, introduced by \cite{Coq2009}.

Sections~\ref{secdival0} to~\ref{subsectrdisad} give general facts about \trdis, \spss, \tgms, \sads and relations between these objects. The central notion of \sad intuitively corresponds to that of an incompletely specified algebraic structure. The existence of models for these structures is a pervasive tool in classical mathematics. Most of the time, the only constructive counterpart to the existence of such a purely idealistic existence is the fact that the corresponding \sad does not collapse. But in general, this is in fact sufficient to establish constructively results obtained in classical mathematics via the purely ideal existence of these models.

Section~\ref{secdival} develops several \tdys for \advs, in relation with
 \sa{val} and~\sa{Val}, which correspond, respectively, to the approaches of \cite{HK1994} and of \cite{Coq2009}. Theorem~\ref{thValval} establishes the isomorphism of distributive lattices (and hence that of spectral spaces) corresponding to the \sads defined in theories~\sa{val} and \sa{Val}, for an \alg $\gk\to\gK$ when~$\gk$ is a sub-ring of a \cdi $\gK$.

Section~\ref{subsecdival} addresses the valuative dimension of commutative rings. In particular, we prove that the valuative dimension of a commutative ring $\gA$ is the \ddk of the associated \trdi with the \sad $\sa{val}(\gA,\gA)$. In the case of an integral domain, the isomorphism of distributive lattices of Theorem~\ref{thValval} establishes the correspondence with the valuative dimension defined constructively by \cite{Coq2009}. We also show constructively the equivalence of several constructive approaches to the valuative dimension of a ring, or of an algebra. For this purpose, we introduce the minimal pp-closure of a ring ~$\gA$, which replaces the too hypothetical \gui{integral domain generated by~$\gA$.}

Section~\ref{secdival6} revisits the \tdy of \cvds given in \cite{CLR01}. The objective of this section is akin to that of Section~\ref{secdival}: showing that the \sads associated with the theories \sa{val}, \sa{Val} and \sa{Vdf} are essentially the same. We rely on the formal \vsts established for \sa{Vdf} and we develop analogue formal \vsts for  \sa{val} and \sa{Val}. This provides an interesting historical perspective on various results à la \vst. In particular, the formal \vst~\ref{thVstformelvalter} explains the identity between the algebraic certificates associated with the formal theories, respectively, developed (implicitly or explicitly) in \cite{HK1994}, \cite{CLR01} and \cite{Coq2009}.


\section{Distributive lattices and \spss}\label{secdival0}
References: \cite{Sto37,CC00,CL05} and \cite[Chapters XI and XIII]{CACM}.

\subsection{The seminal paper by Stone}\label{subsecStoneAntiequiv}
In modern language, the main result of the seminal paper \cite{Sto37} can be stated as follows:

\noindent \textsl{The category of \trdis is, in \clama, antiequivalent to the category of \spss.}

Let us explain this with some details.

\Subsubsection{Ideals and filters in a \trdi}\label{subsecTrdiIdeFi}

If $\varphi :\gT\rightarrow \gT'$ is a \trdi \homo, $\varphi^{-1}(0)$ is called an \emph{\id of $\gT$}. An \id $\fb $ of $\gT$ is a subset of $\gT$ subjected to the following constraints
\begin{equation}\label{eqIdeal}
\left.
\begin{array}{rcl}
 & & 0 \in \fb \\
x,y\in \fb & \Longrightarrow & x\vu y \in \fb \\
x\in \fb ,\; z\in \gT& \Longrightarrow & x\vi z \in \fb \\
\end{array}
\right\}
\end{equation}
(the last is rewritten as $(x\in \fb ,\;y\leq x)\Rightarrow y\in
\fb$).
A \emph{\idp} is an \id generated by a single \elt $a$,
it is equal to $\dar a:=\sotq{x\in \gT}{x\leq a}$.%

The \id $\dar a$, equipped with the laws $\vi$ and $\vu$ of $\gT$, is a \trdi in which the maximum \elt is $a$. The canonical injection $\dar a\rightarrow \gT$ is not a morphism of \trdis because the image of $a$ is not equal to $1_\gT$. However, the 
map $\gT\rightarrow \dar a,\;x\mapsto x\vi a$ is a surjective morphism, which therefore defines $\dar a$ as a quotient structure $\gT/(a=1)$.

\rdb \label{NOTAuara}
The opposite notion to that of an \id is the notion of a {\em filter}. The principal filter generated by $a$ is equal to $\uar a$.%

Let $\fa$ be an \id and $\ff$ be a filter of $\gT$, we say that $(\fa,\ff)$
is a \emph{saturated pair} in~$\gT$ if we have the following implications
\[(g\in \ff,\; x\vi g \in \fa) \Longrightarrow x\in \fa \;\hbox{ and }\; (a\in \fa,\; x\vu a \in \ff) \Longrightarrow x\in \ff.
\]
A saturated pair can also be defined as being $(\varphi^{-1}(0),\varphi^{-1}(1))$ for a morphism $\varphi:\gT\to\gT'$ of \trdis. 
 When $(\fa,\ff)$ is a saturated pair, we have the \eqvcs
\[ 
1\in \fa\; \;\Longleftrightarrow\;\; 0\in \ff
\;\; \Longleftrightarrow\;\; (\fa,\ff)=(\gT ,\gT ).
\]

If $A$ and $B$ are two subsets of $\gT$ we denote
\begin{equation}\label{eqvuvi}
A\vu B=\big\{ a\vu b \mid a\in A,\,b\in B\,\big\} \; \hbox{ and } \; A\vi
B=\big\{ a\vi b \mid a\in A,\,b\in B\,\big\}.
\end{equation}
Then the \id generated by two \ids $\fa$ and $\fb$ is equal to
\begin{equation}\label{eqSupId}
\cI_\gT(\fa\cup \fb) = \fa\vu\fb.
\end{equation}
The set of \ids of $\gT$ itself forms a \trdi\footnote{Actually we need to introduce a restriction to truly obtain a set, in order to have a well-defined procedure to construct the \ids under consideration.
For example we can consider the set of \ids obtained from \idps via certain predefined operations, 
such as countable unions and intersections.} w.r.t.\ the inclusion and for greatest lower bound of $\fa$ and $\fb$, the \id 
\begin{equation}\label{eqInfId}
\fa\cap \fb=\fa\vi\fb.
\end{equation}

Thus the operations $\vu$ and $\vi$ defined in (\ref{eqvuvi}) correspond to the supremum and the infimum in the lattice of \ids.

~
When we consider the lattice of filters, 
we must pay attention to what the reversing of the order relation produces: $\ff\cap\ffg=\ff\vu\ffg$ is the infimum of the filters~$\ff$ and $\ffg$, whereas their supremum is equal to $\cF_\gT(\ff\cup \ffg)=\ff\vi \ffg$.

\Subsubsection{The spectrum of a \trdi} 

In \clama, a \textsl{prime ideal} $\fp$ of a \trdi $\gT\neq \Un$ is an ideal whose complement $\ff$ is a filter ({\sl
a prime filter}). The quotient lattice $\gT/(\fp=0,\allowbreak\ff=1)$ is isomorphic to~$\Deux$. Giving a \idep of~$\gT$ is the same thing as giving a lattice morphism $\gT\rightarrow \Deux$.
We will write $\theta_\fp:\gT\to\Deux$ the morphism
corresponding to the \idep $\fp$.

If $S$ is a system of generators for a 
\trdi $\gT$, a \idep~$\fp$ of $\gT$ is characterized by its trace $\fp\cap S$ (see \cite{CC00}).

\begin{definition} \label{defiSpecTrdi}
The (Zariski) \textsl{spectrum of the \trdi $\gT$} is the set $\Spec
\gT$ whose elements are \ideps of $\gT$, with the following topology: an open basis is provided by the subsets $\DT(a)\eqdefi\sotq{\fp\in\Spec
\gT}{a\notin\fp}=\sotq{\fp}{\theta_\fp(a)=1}$.
\end{definition}

In \clama we have
\begin{equation} \label{eqDa}
\left.
\begin{array}{rclcrcl}
 \DT(a\vi b) & = & \DT(a)\cap \DT(b) ,&\quad & \DT(0) & = & 
\emptyset ,\\
 \DT(a\vu b) & = & \DT(a)\cup \DT(b) ,&& \DT(1) & = & 
\Spec\,\gT.
 \end{array}
\right\}
\end{equation}

The complement of $\DT(a)$ is a \textsl{basic closed set} denoted by $\VT(a)$.

This notation is extended to $I\subseteq\gT$: we let
$\VT(I)\eqdefi\bigcap_{x\in I}\VT(x)$. {If $\fII$} is the ideal generated by $I$, one
has $\VT(I)=\VT(\fII)$. The closed set $\VT(I)$ is also called \textsl{the subvariety of $\SpecT$ defined by $I$}.

The closure of a point $\fp\in\SpecT$ is provided by all $\fq\supseteq \fp$. Maximal ideals are the closed points of $\SpecT$. The spectrum $\SpecT$ is empty iff $0=_\gT1$.

\begin{definition} \label{defiEspaceSpectral}
A topological space homeomorphic to a space $\Spec(\gT)$
is called a \textsl{\sps}. 
\end{definition}

Spectral spaces come from \cite{Sto37}.
\cite{Joh1986} calls them \textsl{coherent spaces}. \cite{BW74} give them the name \textsl{Stone space}. \cite{Hoc1969} uses the name \textsl{\sps} in a famous paper  where he proves that all \spss can be obtained as Zariski spectra of commutative rings. 

With classical logic and choice axiom, the space $\Spec (\gT)$
has \gui{enough points}: the lattice~$\gT$ can be recovered from its spectrum.

\smallskip 
An element (a point) $x$ of a \sps $X$ is the
\textsl{generic point of the closed subset~$F$} \hbox{if $F=\ov{\so{x}}$}. This point
(when it exists) is \ncrt unique because \spss
are Kolmogoroff. In fact, closed subsets $\ov{\so{x}}$ are (in \clama)
all irreducible closed subsets of $X$. The order relation $y\in\ov{\so{x}}$ will be denoted as $x\leq_X y$, and we have equivalences
\begin{equation} \label {eqOrdreSpec}
x\leq_X y\;\Longleftrightarrow\; \ov{\so{y}}\subseteq \ov{\so{x}}\,.
\end{equation}
When $X=\Spec(\gT)$ the order relation $\fp\leq_X \fq$ is merely the usual inclusion relation $\fp\subseteq \fq$ between \ideps of $\gT$.

\Subsubsection{Stone's antiequivalence}

First we have Krull's \tho.


\medskip\noindent
{\bf Krull's \tho }\label{ThKrull} (in \clama)\\ 
{\it Let $\fa$ be an \id and $\fv$ a filter of a distributive lattice $\gT$. Suppose that 
$\fa\cap\fv=\emptyset$.
Then there exists a \idep $\fp$ such \hbox{that $\fa\subseteq\fp$} and
$\fp\cap\fv=\emptyset$.
 }

\medskip
One deduces the following.
\begin{itemize}\itemsep=.05em
\item The map $a\in\gT\,\mapsto\,\DT(a)\in\cP(\Spec\,\gT)$
is injective: it identifies $\gT$ with a lattice of sets 
(\textsl{Birkhoff representation \tho}).
\item If $\varphi : \gT\to\gT'$ is an injective \homo the dual map
$\varphi^\star:\Spec\,\gT'\to\Spec\,\gT$ is onto.
\item Any \id of $\gT$ is the intersection of \ideps above it.
\item The map $\fII\mapsto \VT(\fII)$, from \ids of $\gT$ to 
closed subsets of $\Spec\,\gT$ is an \iso of posets (w.r.t.\ to inclusion 
and reversed inclusion).
\end{itemize}

One proves also the following results.
\begin{itemize}\itemsep=.05em
\item The \oqcs of $\Spec \,\gT$ are exactly the
$\DT(a)$'s. 
\item The \oqcs of $\Spec \,\gT$
form a \trdi of subsets of~$\Spec \,\gT$, isomorphic to $\gT$ (\egts (\ref{eqDa})). 
\item If $X$ is a \sps, its \oqcs form a \trdi, denoted as
$\OQC(X)$. 
\item For a \trdi $\gT$,
$\OQC(\Spec(\gT))$ is canonically isomorphic to $\gT$ and for a
\sps $X$, $\Spec(\OQC(X))$ is canonically 
homeomorphic to~$X$. 
\end{itemize}
 
\begin{definition} \label{defiApplispectrale}
A morphism $\varphi :\gT\rightarrow \gT'$ of \trdis gives
\textsl{by duality} a continuous map 
\[\varphi^\star:\Spec\,\gT'\rightarrow \Spec \,\gT, \,\fp\mapsto \varphi^{-1}(\fp)
\]
which is called a \textsl{spectral map}. 
\end{definition}

A map between
\spss is spectral \ssi the preimage
of any \oqc is a \oqc. So it is \ncrt continuous.

\smallskip The seminal paper by Stone gives the following \carn of \spss. They are the topological spaces satisfying the following \prts (see \cite[{II-3.3}, coherent locales]{Joh1986}, \cite{Sto37}):
\begin{itemize}\itemsep=.05em
\item The space is \qc.\footnote{The nowadays standard terminology is \qc, as in Bourbaki and Stacks, rather than compact.}
\item Every open set is a union of \oqcs.
\item The intersection of two \oqcs is a \oqc.
\item For two distinct points, there is an open set containing one of them but not the other.
\item For any closed set $F$ and any set $S$ of \oqcs such that 
\[\textstyle F\cap
\bigcap_{U\in S'} U\neq \emptyset\,\hbox{ for any finite subset }\,S'
\,\hbox{ of }\,S
\]
we have also $F\cap \bigcap_{U\in S} U\neq \emptyset$.
\end{itemize}

\smallskip \cite{Hoc1969} shows that the last property can be replaced with: 
\begin{itemize}\itemsep=.1em
\item Every irreducible closed set is the closure of a point.
\end{itemize}

\smallskip 
The precise modern formulation of Stone's antiequivalence is Theorem~\ref{thStoneAntiequiv}.
\begin{theorem} \label{thStoneAntiequiv}
 The contravariant functors $\Spec$ and $\OQC$ between the categories of \trdis and \spss define an antiequivalence. 
\end{theorem}

In \coma, \spss may have no points, and we try to translate the discourse in \clama on \spss (very frequent in \alg) in a \cov discourse on corresponding \trdis.

A topological subspace $Y$ of a \sps $X$ is called a \textsl{sub\sps} if the inclusion morphism $Y\to X$ is a spectral map. The notion of sub\sps is translated by the notion of quotient \trdi (\thref{propSESP}). We have also good translations for the notions of Krull dimension, normal \sps, lying over morphisms, going up and going down (see Section~\ref{subsecAntiEquiv}). 

\smallskip 
When we replace $\gT$ with the opposite lattice $\gT\oop$, by reversing the order, points of $X$ remain the same ones in \clama and the topology is replaced with the \textsl{opposite topology}, where \oqcs are the subsets $\fV_\gT(a)$.

There is also the \textsl{constructible topology} (or patch topology) where \oqcs are Boolean combinations of $\fO_\gT(a)$ and $\fV_\gT(b)$. This \sps is the dual of the \agB $\Bo(\gT)$ generated by the \trdi $\gT$. This kind of \sps is a \textsl{Stone's space} in the today terminology. 

\Subsubsection{Finite \spss}
 
Finite \trdis correspond to finite \spss. These ones are merely finite posets since it suffices to know the closure of points for defining the topology. Basic opens are the $\dar a$'s. 
In \clama all open sets are quasi-compacts. These are the initial parts, and closed subsets are the final parts. Finally, a map between finite \spss is spectral \ssi it is non-decreasing.

So, we see that the general notion of \sps is a relevant \gnn of the notion of finite poset. See \cite[Theorem XI-5.6, duality between finite posets and finite \trdis]{CACM}.

In the finite case, if we identify the underlying sets of $\Spec\,\gT$ and $\Spec\,\gT\oop$, we get two opposite posets (reversing the order).

\subsection{Distributive lattices and \entrels}

A particularly important rule for distributive lattices, known as
 \textsl{cut}, is
\begin{equation}\label{coupure1}
 \bigl(\,x\vi a\; \leq\; b\,\bigr)\quad\&\quad \bigl(\,a\; \leq\; x\vu b\,\bigr)
\quad \Longrightarrow \quad a \leq\; b.
\end{equation}

For $A\in\Pfe(\gT)$ (finitely enumerated subsets of~$\gT$) we write
\[\ndsp \Vu A:=\Vu _{x\in A}x\qquad {\rm and}\qquad \Vi A:=\Vi _{\!x\in A}x.\]
We denote by $A \vda B$ or $A \vdash_\gT B$ the relation defined as follows over the set $\Pfe(\gT)$
\[A \vda B \; \; \equidef\; \; \Vi A\;\leq \;
\Vu B.
\]

This relation satisfies the following axioms, in which we write $x$ for $\{x\}$ and $A, B$ for $A\cup B$.
\[\arraycolsep3pt\begin{array}{rcrclll}
& & a &\vda& a &\; &(R)  \\[1mm]
A \vda B & \; \Longrightarrow \; & A,A' &\vda& B,B' &\; &(M)  \\[1mm]
(A,x \vda B)\;\;
\&
\;\;(A \vda B,x) & \; \Longrightarrow \; & A &\vda& B &\;
&(T).
\end{array}
\]
\rdb
We say that the relation is \textsl{reflexive}, \label{remotr} \textsl{monotone} and
\textsl{transitive}.
The third rule (transitivity) can be seen as a version of rule \pref{coupure1} and is also called the \textsl{cut} rule.

\begin{definition}
\label{defEntrel}
For an arbitrary set $S$, a relation over $\Pfe(S)$ which is reflexive, monotone and transitive is called an \textsl{\entrel}.
\end{definition}

The following \tho is fundamental. It says that the three \prts of \entrels are exactly what is needed for the interpretation in the form of a \trdi to be adequate.

\begin{theorem}[Fundamental \tho of \entrels] \label{thEntRel1} 
\emph{See \cite[Theorem 1]{CC00}, \cite[\hbox{XI-5.3}]{CACM}, \cite[Satz 7]{Lor1951}.} 
Let $S$ be a set with an \entrel $\vdash_S$ on $\Pfe(S)$. We consider the \trdi~$\gT$ defined by \gtrs and relations as follows: the \gtrs are the \elts of $S$ and the relations are the
\[A\; \vdash_\gT \; B
\]
each time that $A\; \vdash_S \; B$. Then, for all $A$,
 $B$ in $\Pfe(S)$, we have
\[A\; \vdash_\gT \; B
\; \Longrightarrow \; A\; \vdash_S \; B.
\]
\end{theorem}

\smallskip 
\rem The relation $x\vdash_S y$ is a priori a preorder, and not an order, on~$S$. Let us denote by~$\ov x$ the \elt $x$ seen in the ordered set~$\ov{S}$ defined by this preorder. For a subset~$A$ of~$S$ let us denote $\ov A=\sotq{\ov x}{x\in A}$.
In the \tho we consider a \trdi $\gT$ which gives on $S$ the same \entrel as $\vdash_S$.
Strictly speaking, we should have written $ \ov A\; \vdash_\gT \; \ov B$
instead of $A\; \vdash_\gT \; B$ since the \egt in $\gT$ is coarser than in $S$. In particular, it is $\ov{S}$, and not $S$, which can be identified with a subset of $\gT$.
\eoe

\subsection{Gluing \trdis and spectral subspaces}

\Subsubsection{Quotients, covers, gluing procedures}

A \emph{quotient lattice $\gT'$ of a lattice $\gT$} can \egmt be given by a binary relation $\preceq$ over $\gT$ satisfying the following \prts
\begin{equation}\label{eqPreceq}
\left.
\begin{array}{rcl}
a\leq b& \Longrightarrow & a\preceq b \\
a\preceq b,\,b\preceq c& \Longrightarrow & a\preceq c \\
a\preceq b,\,a\preceq c& \Longrightarrow & a\preceq b\vi c \\
b\preceq a,\,c\preceq a& \Longrightarrow & b\vu c\preceq a
\end{array}
\right\}
\end{equation}
The relation $\preceq$ then induces a lattice structure over the quotient set~$\gT'$ obtained with the new \egt%
\[(a,b\in\gT)\quad:\quad\quad a=_{\gT'}b \equidef (a\preceq b \;\mathrm{ and }\;
b\preceq a)
\]
Naturally if $\gT$ is distributive, the same goes for $\gT'$.

\begin{proposition}
\label{propIdealFiltre} Let $\gT$ be a \trdi and $(J,U)$ be a pair of subsets of $\gT$.
Consider the quotient $\gT'$ of $\gT$ defined by the relations $x=0$ for each $x\in J$, and $y=1$ for each $y\in U$. Then the inequality $a\leq_{\gT'}b$ is satisfied \ssi
there exist $J_0\in\Pfe( J)$ and $U_0\in\Pfe( U)$ such that
\begin{equation}\label{eqpropIdealFiltre}
a \vi \Vi U_0 \; \leq_\gT\; b \vu \Vu J_0.
\end{equation}
We will denote by $\gT/(J=0,U=1)$ this quotient lattice $\gT'.$
\end{proposition}
In particular, for an ideal $\fa$, the natural morphism 
$\varphi:\gT\to\gT'=\gT/(\fa=0)$
satisfies $\varphi^{-1}(0_{\gT'})=\fa$. In the case of the quotient by a principal \id $\dar a$ we obtain $\gT/(\dar a=0)\simeq\uar a$ with the morphism $y\mapsto y\vu a$ from $\gT$ to $\uar a$.


We see in the example of totally ordered sets that a quotient structure of a \trdi is not \gnlt
characterized by the equivalence classes of $0$ and $1$.

In commutative \alg, when $\fa$ and $\fb$ are two \ids of a ring 
$\gA$, there is an exact sequence of \Amos 
\[
0\to\gA/(\fa\cap\fb)\vers{j}(\gA/\fa) \times 
(\gA/\fb)\vers{\delta}\gA/(\fa+\fb)\to 0 \quad (\delta(x,y)=x-y \mod \fa+\fb)
\]
otherwise said: the congruences system
$x\equiv a \mod \fa$, $x\equiv b \mod \fb$ has a solution \ssi
$a\equiv b\ \mod \fa+\fb$, and in this case, the solution is unique modulo 
$\fa\cap\fb$.
It is remarkable that this \gui{Chinese remainder \tho} generalizes to an \textsl{arbitrary}  system of congruences \ssi the ring is
\textsl{arithmetic} \cite[Theorem XII-1.6]{CACM}, \cad, if the lattice of \ids is distributive. By contrast, the usual Chinese remainder \tho  concerns the special case of a family of  pairwise comaximal \ids, and it works for an arbitrary ring.

Localizations provide other epimorphisms of the category of commutative rings. They enjoy~a fruitful gluing principle, analogous to the Chinese remainder \tho, in this instance the basic local-global principle. This principle asserts that: \textsl{if $(\xn)$ is a system of \eco in a ring $\gA$,
the morphism $\gA\to\prod_{i\in\lrbn}\gA[1/x_i]$ identifies $\gA$ to a} subproduct\footnote{A subobject of the considered finite product, in the category of commutative rings.} \textsl{of its localized rings} (see the \plgc XV-4.2 in~\cite{CACM}, and also II-2.3, XV-2.1, XV-2.2, XV-2.3, XV-2.4 and XV-2.5). 
These principles are  \cov versions of abstract \plgs. They ensure that certain  \prts of an \Amo or an \Alg are satisfied \ssi they are satisfied after \ localization in any \idep (or, sometimes, in the neighborhood of any \idep).

\smallskip Similarly, it is possible to reconstruct a \trdi from a finite number of quotients when the latter carry ``enough'' information. 
This can be seen either as a gluing procedure, or as~a Chinese remainder \tho for  \trdis.

Covering a \trdi by quotient lattices is dual to covering a spectral space by spectral subspaces.
Note, however, that a set-theoretical cover of the total space by a family of  spectral subspaces does not suffice to reconstruct  the spectral topology of the total space from those of the subspaces.
Similarly, the three-by-three compatibility conditions on gluing isomorphisms do not suffice to glue spectral spaces: usually, gluing also requires  some additional properties on their spectral subspaces.

Therefore, Theorems~\ref{factRecolTD} and~\ref{propRecolTD} consider very specific quotient \trdis. These results provide analogues, for the category of \trdis, to the similar results available in the category of \grls~\cite[Covering principles XI-2.10 and XI-2.21]{CACM} and in the category of modules over a commutative ring~\cite[Covering principles XI-4.19 and XIII-3.3, Gluing principles XV-4.4, XV-4.4 bis and XV-4.6]{CACM}. We can now describe the situation in detail.



\begin{definition}
\label{defRecolTD}
Let $\gT$ be a \trdi and $(\fa_i)_{i\in\lrbn}$ (respectively 
$(\ff_i)_{i\in\lrbn}$)
a finite family of \ids (respectively of filters) of $\gT$. We say that the \ids
$\fa_i$ \textsl{cover $\gT$} if~~$\bigcap_i\fa_i=\so{0}$. Similarly we say that
the filters $\ff_i$ \textsl{cover $\gT$} if~~$\bigcap_i\ff_i=\so{1}$.
\end{definition}

Let $\fb$ be an ideal of $\gT$; we write $x\equiv y\mod\fb$ as meaning $x\equiv y\mod (\fb=0)$. Let us recall that for $s\in\gT$ the quotient $\gT/(s=0)$ is isomorphic to the principal filter $\uar s$ (one sees this filter as a \trdi with $s$ as $0$ element).

\begin{theorem}[Covering a \trdi by suitable quotients]
\label{factRecolTD}
Let $\gT$ be a \trdi, $(\fa_i)_{i\in\lrbn}$ a finite family 
of principal ideals ($\fa_i=\dar s_i$) and $\fa=\bigcap_i\fa_i$.
\begin{enumerate}
\item If $(x_i)$ is a family in $\gT$ s.t.\ for each $i,j$ one has $x_i\equiv x_j\,\mod\,\fa_i\vu\fa_j$, then there exists a unique $x$ 
modulo
$\fa$ satisfying: $x\equiv x_i\,\mod\,\fa_i\;(i\in\lrbn)$.
\item Let us write $\gT_i=\gT/(\fa_i=0)$, 
$\gT_{ij}=\gT_{ji}=\gT/(\fa_i\vu\fa_j=0)$,
$\pi_i:\gT\to\gT_i$ and $\pi_{ij}:\gT_i\to\gT_{ij}$ the canonical maps.
If the ideals $\fa_i$ cover $\gT$, the system $(\gT,(\pi_i)_{i\in\lrbn})$ is the inverse limit of the diagram 
\[((\gT_i)_{1\leq i\leq n},(\gT_{ij})_{1\leq 
i<j\leq
n};(\pi_{ij})_{1\leq i\neq j\leq n}).
\]
\item The analogous result works with quotients by principal filters.

 {\small\hspace*{10em}{
\xymatrix @R=2em @C=7em{
   & \gT \ar[rd]^{\pi _{k}}\ar[d]^{\pi _{j}}\ar[ld]_{\pi _{i}}\\
 \gT _i\ar[d]_{\pi _{ij}}\ar@/-0.75cm/[dr]^{\pi _{ik}} &
  \gT _j\ar@/-.8cm/[dl]_{\pi _{ji}}\ar@/-.8cm/[dr]^{\pi _{jk}} &
  \gT _k\ar@/-0.75cm/[dl]_{\pi _{ki}}\ar[d]^{\pi _{kj}} &
\\
 \gT _{ij} & 
 \gT _{ik} & 
  \gT _{jk} 
}
}}

\end{enumerate}
\end{theorem}

There is also a gluing procedure for quotient \trdis $\gT_i\to \gT_{ij}$ in certain particular cases.

\begin{definition}[Morphism of passage to quotient] \label{defimorquotlattice} 
Let $\gT$ be a \trdi and $u\in \gT$.
We identify $\uar u$ to the quotient $\gT/(\dar u=0)$ via the morphism $p_u:x\mapsto x\vu u$. 
More \gnlt, a morphism of \trdis $\alpha:\gT\to \gT'$ is called a \textsl{morphism of passage to quotient by the ideal $\dar u$} if there exists an \iso (\ncrt unique) $\lambda:\gT'\to \gT/(u=0)$ such that $\alpha=\lambda\circ p_u$.
\end{definition}

\begin{lemma}[In a \trdi, principal quotients are \gui{split}] \label{lemquoprinctrdi} ~ 
\\
Let $\pi:\gT \to \gT'$ be a morphism of \trdis and $s\in \gT$.
\Propeq
\begin{enumerate}
\item $\pi$ is a morphism of passage to quotient of $\gT$ by the principal 
ideal $\fa=\dar s$.
\item There exists a morphism $\varphi:\gT'\to\,\uar s$ such that
$\pi\circ \varphi=\Id_{\gT'}$.
\end{enumerate}
In this case $\varphi$ is uniquely determined by $\pi$ and $s$.\\
Naturally, the \gui{reversed} lemma is valid for a quotient by a principal filter. 
\end{lemma}
%

\begin{theorem}[Gluing \trdis]
\label{propRecolTD} 
Let~$I$ be a finite set, a diagram of \trdis 
\[{\big((\gT_i)_{i\in I},(\gT_{ij})_{i<j\in I},(\gT_{ijk})_{i<j<k\in I};
(\pi_{ij})_{i\neq j},(\pi_{ijk})_{i< j, j\neq k\neq i}\big)}
\]
and a family of \elts 
${(s_{ij})_{i\neq j\in I}\in \prod\nolimits_{i\neq j\in I}\gT_{i}}$
 satisfying the following \prts.
\begin{itemize}\itemsep=.1em
\item The diagram is commutative. 
\item If $i\neq j$, $\pi_{ij}$ is a quotient morphism w.r.t. the ideal $\dar s_{ij}$.
\item If $i$, $j$, $k$ are distinct, $\pi_{ij}(s_{ik})=\pi_{ji}(s_{jk})$ and $\pi_{ijk}$ is a quotient morphism w.r.t.\ the ideal $\dar\pi_{ij}(s_{ik})$. 
\end{itemize}

 {\small\hspace*{10em}
\xymatrix @R=2em @C=7em{
 \gT_i\ar[d]_{\pi _{ij}}\ar@/-0.75cm/[dr]^{\pi _{ik}} &
  \gT_j\ar@/-.8cm/[dl]_{\pi _{ji}}\ar@/-.8cm/[dr]^{\pi _{jk}} &
  \gT_k\ar@/-0.75cm/[dl]_{\pi _{ki}}\ar[d]^{\pi _{kj}} &
\\
 ~\gT_{ij}~ \ar[rd]_{\pi _{ijk}} & 
 ~\gT_{ik}~ \ar[d]^{\pi _{ikj}} & 
  ~\gT_{jk}~ \ar[ld]^{\pi _{jki}} 
\\
 & ~\gT_{ijk}~ 
\\
}
}

\smallskip \noindent Let $\big(\gT\,;\,(\pi_i)_{i\in I}\big)$ be the limit of the diagram. Then there exist $s_i$'s in $\gT$ such that the principal ideals $\dar s_i$ cover $\gT$ and the diagram is isomorphic to the one in Theorem~\ref{factRecolTD}.
More precisely each~$\pi_i$ is a quotient morphism w.r.t. the \id $\dar s_i$ and $\pi_i(s_j)=s_{ij}$ for all $i\neq j$.

\noindent The analogous result works with quotients by principal filters.
\end{theorem}

\Subsubsection{The dual viewpoint}

\begin{definition} \label{defisubespace-spectral}
A subset $X'$ of a \sps $X$ is called a \textsl{sub\sps} when the topology on $X'$ induced by $X$ is spectral and
 $\OQC(X')=\sotq{U\cap X'}{U\in\OQC(X)}$ (in other words the canonical injection is a spectral morphism). 
\end{definition}

The following theorem explains that the notion of \textsl{spectral subspace} is translated by the notion of \textsl{quotient \trdi}. Some details are added. See also \thref{th-dico-trdi-spec-mor1}.

\begin{theoremc}[Sub\spss] ~
\label{propSESP} 
\begin{enumerate}
\item Let $\gT'$ be a quotient lattice of $\gT$ and $\pi:\gT\to\gT'$ the quotient morphism. Let us write $X'=\Spec\gT'$, $X=\Spec\gT$ and
$\pi^\star:X'\to X$ the dual map of $\pi$. Then $\pi^\star$ identifies $X'$ with a sub\sps of $X$. 
\item A subset $X'$ of a \sps $X$ is a sub\sps 
\ssi it is closed for the patch topology.
\item If $Z$ is an arbitrary subset of 
$X=\Spec\gT$, its closure for the patch topology
is given by 
$X'=\Spec\gT'$, where
$\gT'$ is the quotient lattice of~$\gT$ defined by the following preorder~$\preceq$:
\begin{equation} \label{eqSSES}
a\preceq b\quad \Longleftrightarrow\quad (\DT(a)\cap Z)\subseteq 
(\DT(b)\cap Z)
\end{equation}
\end{enumerate}
\end{theoremc}

A gluing of \trdis as in  Theorem~\ref{propRecolTD} corresponds to a gluing of topological spaces along 2 by 2 intersections when they are compatible 3 by 3. In this way we glue \spss along suitable \oqcs\footnote{Variant: closed subsets complement of \oqcs. This corresponds to the opposite \trdis and the opposite topology on \spss.} when they are 3 by 3 compatible.\footnote{In Theorem~\ref{propRecolTD} the compatibility is described thanks to a \trdi $\gT_{ijk}$. When gluing topological spaces 
we have a priori three distinct version of $U_1\cap U_2\cap U_3$, respectively, subspaces of $U_1$, $U_2$ and $U_3$. So it is necessary to first identify the two versions of $U_i\cap U_j$ inside $U_i$ and $U_j$ through homeomorphisms. And the three identifications have to give the same $U_1\cap U_2\cap U_3$.} 

Here, the gluing of \spss corresponding to the figure is intuitively evident and easy to prove in \clama. In contrast, Theorem~\ref{propRecolTD} has a rather subtle proof (see \cite{CLQ2006}).

{
\smallskip 
\centerline{\includegraphics[width=11cm]{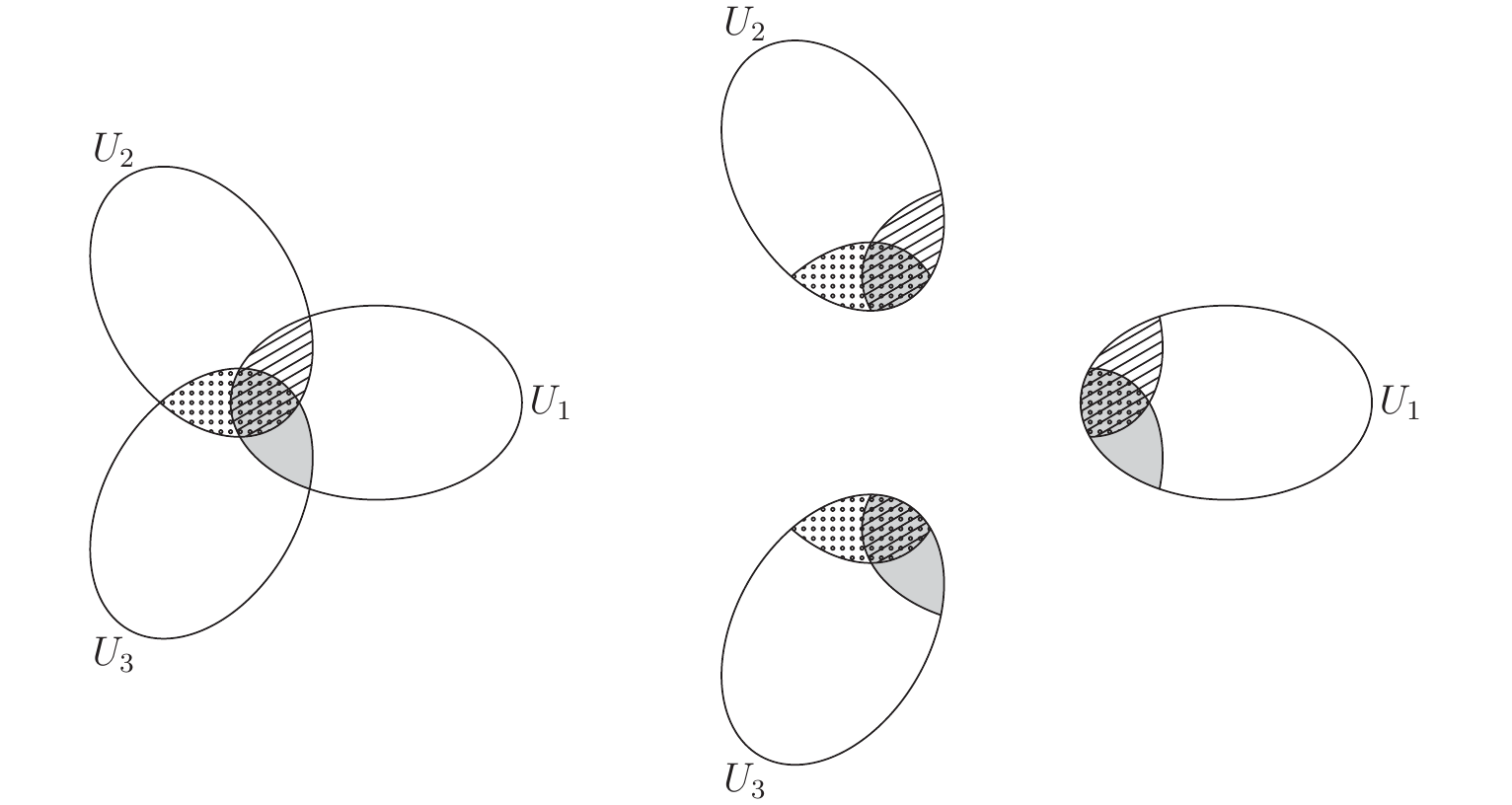}}
}

In the following subsection we give other relevant comparisons between \trdis and dual \spss.

\subsection{Short dictionary of Stone's antiequivalence}
\label{subsecAntiEquiv}

References: Krull's \tho page~\pageref{ThKrull}, \cite[Theorem~IV-2.6]{BW74}, \cite{CC00,CL2001-2018,Lom2020}.

We now recall a few results about Stone's antiequivalence of categories between distributive lattices and spectral spaces, and provide references to the proofs available in the literature. We also prove Theorem~\ref{lemrcfdimtrdi}, which finds applications in Section~\ref{subsecdival}.

\smallskip We are in the following context: $f:\gT\to\gT'$ is a morphism of \trdis and $\Spec(f)$, denoted as $f\sta$, is the dual morphism, from $X'=\Spec(\gT')$ to $X=\Spec(\gT)$.

\smallskip We first recall some usual \dfns in \clama.
\begin{itemize}\itemsep=.05em
\item The morphism $f$ is said to be \textsl{lying over} when~$f\sta$ is onto: any \idep of $\gT$ is the preimage of a \idep in~$\gT'$.
\item The morphism $f$ is said to be \textsl{going up} when one has: \textsl{if $\fq\in X'$, $f\sta(\fq)=\fp$, and $\fp\subseteq\fp_2$ in~$ X $, then there exists $\fq_2\in X'$ such that
$\fq\subseteq\fq_2$ and $f\sta(\fq_2)=\fp_2$}.
\item In a similar way $f$ is said to be \textsl{going down} when one has: \textsl{if $\fq\in X'$, $f\sta(\fq)=\fp$, and $\fp\supseteq\fp_2$ in~$ X $, then there exists $\fq_2\in X'$ such that $\fq\supseteq\fq_2$ and $f\sta(\fq_2)=\fp_2$}.
\item The morphism $f$ \textsl{has the incomparability property} when one has: if $\fq_1\subseteq \fq_2\in X$ and $f\sta(\fq_1)=f\sta(\fq_2)$ in~$X'$ then $\fq_1= \fq_2$.
\item The \sps $ X $ is said to be \textsl{normal} if for all $x$, the closure $\ov{\so z}$ contains a unique closed point.
\item The \sps $\SpecT$ is said to be \textsl{completely normal} if for all $x,y,z$ such that $x\in\ov{\so z}$ and $y\in\ov{\so z}$ one has $x\in\ov{\so y}$ or $y\in\ov{\so x}$. 
\end{itemize}

\Subsubsection{Properties of morphisms}

\begin{theorem} \label{th-dico-trdi-spec-mor1} \emph{\cite[Theorem~IV-2.6]{BW74}}
In \clama we have the following equivalences. 
\begin{enumerate}
\item $f\sta$ is onto ($f$ is lying over) $\Longleftrightarrow$ $f$ is injective $\Longleftrightarrow$ $f$ is a monomorphism $\Longleftrightarrow$ $f\sta$ is an epimorphism.
\item $f$ is an epimorphism $\Longleftrightarrow$ $f\sta$ is a monomorphism $\Longleftrightarrow$ $f\sta$ is injective.
\item $f$ is onto\footnote{In other words, $f$ is a quotient morphism.} $\Longleftrightarrow$ $f\sta$ is an \iso on its image, which is a sub\sps of $X$.
\end{enumerate} 
\end{theorem}

There are bijective morphisms of \spss that are not isomorphisms. For example, the morphism $\Spec(\Bo(\gT))\to\Spec\gT$ is rarely an \iso and the lattice morphism $\gT\to\Bo(\gT)$ is an injective epimorphism which is rarely onto.

\begin{theorem} \label{th-dico-trdi-spec-mor2} \emph{(\cite{CL2001-2018})}.
In \clama we have the following equivalences. 
\begin{enumerate}
\item $f$ is going up $\Longleftrightarrow$ for each $a,c\in\gT$ and $y\in\gT'$ we have
\[
f(a)\leq f(c)\vu y \;\Rightarrow\;\exists x\in\gT\; (a\leq c \vu x \hbox{ and } f(x)\leq y).
\]

\item $f$ is going down $\Longleftrightarrow$ for each $a,c\in\gT$ and $y\in\gT'$ we have
\[
f(a)\geq f(c)\vi y \;\Rightarrow\;\exists x\in\gT\; (a\geq c \vi x \hbox{ and } f(x)\geq y).\]

\item $f$ has the \prt of incomparability $\Longleftrightarrow$ $f$ is zero-dimensional.\footnote{See \thref{th-dico-trdi-spec-dim2}.} 
\end{enumerate} 
\end{theorem}


\begin{theorem}[Open spectral map] \label{th-dico-trdi-spec-mor3} \emph{(\cite{Lom2020})}
In \clama \propeq 
\begin{enumerate}
\item $\Spec(f)$ is an open map.
\item There exists a map $\wi f:\gT'\to \gT$
with the following \prts.
\begin{enumerate}
\item \label{i2a} For $c\in\gT$ and $b\in\gT'$, one has $b\leq f(c) \Leftrightarrow \wi f(b)\leq c$. \\
In particular, $b\leq f(\wi f(b))$ and $\wi f(b_1\vu b_2)=\wi f(b_1)\vu \wi f(b_2)$.
\item \label{i2b} For $a,c\in\gT$ and $b\in\gT'$, one has \hbox{$f(a)\vi b\leq f(c) \Leftrightarrow a\vi\wi f(b)\leq c $}.
\item \label{i2c} For $a\in\gT$ and $b\in\gT'$, one has $\wi f(f(a)\vi b)=a\vi \wi f(b)$.
\item \label{i2d} For $a\in\gT$, one has $\wi f(f(a))=\wi f(1)\vi a$. 
\end{enumerate}
\item There exists a map $\wi f:\gT'\to \gT$
satisfying \prt \ref{i2b}. 
\item For $b\in \gT$ the g.l.b.\ {$\Vi\limits_{b\leq f(c)} c$} exists, and if we write it $\wi f(b)$, 
the \prt \ref{i2b} holds. 
\end{enumerate} 
\end{theorem}

\Subsubsection{Dimension properties}

In \clama, the dimension of a \sps, based on chains of irreducible closed subsets, is called its \textsl{Krull dimension}. The dimension of the empty \sps is $-1$.
Theorem~\ref{th-dico-trdi-spec-dim1} explains why the definition of the Krull dimension of a distributive lattice in constructive mathematics, denoted $\Kdim\gT$, is rather based on its Item \textsl{2}. The base case is for the trivial, singleton lattice, which has dimension $-1$. Since the equivalence between Item \textsl{2} and Item \textsl{3} is constructive, Item \textsl{3} also provides an appropriate constructive definition of the Krull dimension of a \trdi in non-trivial cases.

\cite{CL2003} explain the relation between Item \textsl{1} and Item \textsl{3} of the \tho, as well as the connection with the pioneering approach of \cite{Joy76}. Item \textsl{2} in the \tho goes back to \cite{CLR05}.

\begin{theorem}[Dimension of spaces] \label{th-dico-trdi-spec-dim1} \emph{See \cite{CL2003,CLR05}.} Let $n\in\NN$, in \clama 
\propeq
\begin{enumerate}
\item The \sps $\Spec(\gT)$ has Krull dimension $\leq n$ (defined using chains of primes).
\item For any $x\in\gT$ the quotient lattice $\gT/(x=0,I_x=0)$, where $I_x=\sotq{y}{x\vi y=0}$, has dimension $\leq n-1$.
\item 
For each sequence $(x_0,\dots,x_n)$ in $\gT$ there exists a \emph{complementary sequence} $(y_0,\dots,y_n)$, which means
\begin{equation}\label{eqC2G}
\left.\arraycolsep3pt
\begin{array}{rcl}
1& \vda & y_n, x_n\\
 y_n, x_n & \vda & y_{n -1}, x_{n -1} \\
\vdots~~~~& \vdots &~~~~ \vdots \\
 y_1, x_1& \vda & y_0, x_0 \\
y_0, x_0& \vda & 0  
\end{array}
\right\}
\end{equation}
\end{enumerate}
\end{theorem}

For example, in dimension $n\leq 2$, the inequalities in \pref{eqC2G} correspond to the following diagram in~$\gT$.
\[\SCO{x_0}{x_1}{x_2}{y_0}{y_1}{y_2}
\]

A zero-dimensional \trdi is a \agB.

Items \textsl{2} and \textsl{3} give a \cov meaning to the statement $\Kdim(\gT)\leq n$ for $n\geq -1$. But in the general case, it seems impossible to understand \cot the statement $\Kdim(\gT)=n$ (for $n\geq 0$).

Regarding the \ddk of commutative rings, see \cite{Lom02}, \cite{CL2003} and \cite[Chapter XIII]{CACM}. 
This \dfn of the \ddk is close to (and probably more general than) the one given by \cite[page 584]{Lur2009}, which was motivated by the case of certain non-Noetherian rings.

In addition, it is impossible to prove \cot $\Kdim(\RR)\leq 0$. As $\RR$ is a reduced local ring, $\Kdim(\RR)\leq 0$ is equivalent to
the assertion \textsl{any~$x$ is null or invertible}, \cad \LPO.

\begin{lemma} \label{lemDimQuot}
If $\gT'$ is a quotient of $\gT$, then $\Kdim(\gT')\leq \Kdim(\gT).$ 
\end{lemma}
%
\begin{proof}
Use Item \emph{3} in \thref{th-dico-trdi-spec-dim1}.
\end{proof}
%

\begin{lemma} \label{lemDimProd}
Let $\gT_1$ and $\gT_2$ be \trdis and $\gT=\gT_1\times \gT_2$. Then, for $n\geq -1$ we have 
\[
\Kdim(\gT)\leq n\hbox{ \ssi }\;\Kdim(\gT_i)\leq n\hbox{ for }i=1,2.
\] 
In a shortened form: $\Kdim(\gT_1\times \gT_2)=\sup(\Kdim\gT_1,\Kdim\gT_2)$. 
\end{lemma}
%
\begin{proof}
First, each $\gT_i$ is a quotient of $\gT$. In the other direction, assume each $\gT_i$ is of dimension~\hbox{$\leq n$} and let $(\xzn)$ in $\gT$, with $x_i=(a_i,b_i)$. If $(\azn)$ has the \cop sequence $(\yzn)$ in $\gT_1$ and $(\bzn)$ has the \cop sequence $(\zzn)$ in $\gT_2$, then $(\uzn)$ (where $u_i=(y_i,z_i)$) is \cop of $(\xzn)$ in $\gT$. 
\end{proof}

\begin{theorem}[Dimension of \trdis, case of a closed cover] \label{lemrcfdimtrdi}~\\
Let~$\gT$ be a \trdi, $\fa$, $\fb$ two \ids such that $\fa\cap\fb=\so0$, $\gT_\fa=\gT/(\fa=0)$ and $\gT_\fb=\gT/(\fb=0)$.
Then $\Kdim\gT=\sup(\Kdim\gT_\fa,\Kdim\gT_\fb)$. 
\end{theorem}
%
\begin{proof}
We have a natural morphism $j_{\fa,\fb}:\gT\to \gT_\fa\times \gT_\fb, \,x\mapsto (\pi_\fa(x),\pi_\fb(x))$. Lemmas \ref{lemDimQuot} and \ref{lemDimProd} imply we have $\sup(\Kdim\gT_\fa,\Kdim\gT_\fb)\leq \Kdim\gT$.

\noindent Let us see the opposite inequality.

\noindent We begin by a \demo in \clama. 
Let $\fp$ be a \idep of $\gT$. If $\fp$ disappears in~$\gT_\fa$ we have $\fp\vu\fa=\gT$. In this case we have an $x\in\fa$ such that $x \notin \fp$, and for any $y\in\fb$, since $x\vi y=0$ and $\fp$ is prime, $y\in\fp$. So $\fb\subseteq \fp$ and $\fp$ remains a \idep in $\gT_\fb$.
Symmetrically, if $\fp$ disappears in $\gT_\fb$, it persists in~$\gT_\fa$.
Let us now consider a chain of \ideps in $\gT$ of maximal length.
The minimal \elt in the chain contains $\fa$ or $\fb$, so all the chain persists in $\gT_\fa$ or~$\gT_\fb$.

\noindent Finally let us give a \demo in \coma.\footnote{It should be interesting to check the \cov proof being a faithful translation of the classical one.} 

\noindent First we assume that $\fa$ and $\fb$ are \idps: $\fa=\dar a$ and $\fb=\dar b$ with $a\vi b=0$. 
We have a morphism of \textsl{unbounded \trdis}.\footnote{The image of this morphism is $\uar a$. We have $\pi_a(1)=1$, but $\pi_a(0)=a$.} $\gT_a\to \gT$, \hbox{$\pi_a(x)\mapsto a\vu x$} Let us consider the morphism $r_{a,b}:\gT_a\times \gT_b\to \gT,\,(x,y)\mapsto (a\vu x)\vi(b\vu y)$. We see that it is a morphism of \trdis because $r_{a,b}(0,0)=0$. We have $r_{a,b}\circ j_{a,b}=\Id_\gT$. 
Indeed, for $x\in \gT$, one has $(a\vu x)\vi(b\vu x)=(a\vi b)\vu x=x$. 
So $r_{a,b}$ is a surjective morphism and $\gT$ is a quotient of~$\gT_a\times \gT_b$. 

\noindent Let us see the \gnl case, where $\fa$ and $\fb$ are not \ncrt \tf. The quotient~$\gT_\fa$ is the filtered colimit of~$\gT_a$'s for $a\in\fa$ ($\fa$ is seen as a small filtered category). The lattice~$\gT_\fa\times \gT_\fb$ is the filtered colimit of~$\gT_a\times \gT_b$'s. 
Morphisms~$r_{a,b}$ are compatible,\footnote{For all $a\leq a'$ and $b\leq b'$, the triangle constituted by $r_{a',b'}$, $r_{a,b}$ and the morphism of passage to quotient $\gT_a\times \gT_b\to\gT_{a'}\times \gT_{b'}$ is commutative.} so they induce (by the \uvl \prt of a filtered colimit) a morphism $r_{\fa,\fb}:\gT_\fa\times \gT_\fb\to\gT$. For each $(a,b)$ we have the filtered colimit morphism $\iota_{a,b}:\gT_a\times \gT_b\to \gT_\fa\times \gT_\fb$ and we have the commutative diagram:

\centerline{
\xymatrix @C=1.2cm @R=1.2cm{
\gT\,\ar[dd]_{j_{\fa,\fb}}\ar[rr]^{j_{a,b}}\ar@/-1.5cm/[drrr]_{j_{a',b'}} && \,\gT_a\times \gT_b \ar[dd]^{r_{a,b}}\ar[ddll]_{\iota_{a,b}}\ar[dr]^{j_{(a,b),(a',b')}} 
\\
&&& \,\gT_{a'}\times \gT_{b'}\ar[dl]^{r_{a',b'}}\ar@/.5cm/[dlll]_{\iota_{a',b'}}
\\
\gT_\fa\times \gT_\fb\,\ar[rr]_{r_{\fa,\fb}} && \gT 
}
}

\noindent Hence
$r_{\fa,\fb}\circ j_{\fa,\fb}=\Id_\gT$.
\end{proof}

\begin{theorem}[Dimension of morphisms] \label{th-dico-trdi-spec-dim2} \emph{See \cite{CL2001-2018}, \cite[section XIII-7]{CACM}.} 
Let $\gT\subseteq \gT'$ and $f$ be the inclusion morphism. In \clama \propeq
\begin{enumerate}
\item The morphism $\Spec(f):\Spec(\gT')\to\Spec(\gT)$ has Krull dimension $\leq n$.
\item For any sequence $(x_0,\dots,x_n)$ in $\gT'$
there exists an integer $k\geq 0$ and \elts $a_1,\ldots,a_k\in \gT$ such that for each partition $(H,H')$ of $\{1,\ldots,k\}$, there exist $ y_0,\dots,y_n\in \gT'$ such that
\begin{equation} \label {eqdefDiTrRel}
\begin{array}{rclll}
\Vi_{j\in H'} a_j & \vda & y_n,\;x_n \\
y_n,\;x_n& \vda &y_{n-1},\;x_{n-1}  \\
\vdots\qquad & \vdots & \qquad \vdots \\
y_1,\;x_1& \vda & y_0,\;x_0 \\
y_0,\;x_0& \vda & \Vu_{j\in H} a_j \\
\end{array}
\end{equation}
\end{enumerate} 
\end{theorem}
For example, for the relative dimension $n\leq 2$, the inequalities in \pref{eqdefDiTrRel} correspond to the following diagram in~$\gT$.
with $u=\Vi_{j\in H'} a_j$ and $i=\Vu_{j\in H} a_j$.
\[\SCOR{x_0}{x_1}{x_2}{y_0}{y_1}{y_2}{u}{i}\]

Note that the dimension of the morphism $\gT\to\gT'$ is bounded by the dimension of $\gT'$: take the empty list ($k=0$) in Item \textsl{2} of \thref{th-dico-trdi-spec-dim2}. 

More \gnlt we have a \prco of the fundamental inequality given in \cite[A note on the dimension theory of rings]{Sei1953} for the case of commutative rings: $1+\dim \gT'\leq (1+\dim \gT)(1+\dim f)$.

\Subsubsection{Properties of spaces}

 A \trdi $\gT$ is said to be \textsl{normal} if each time one has $a \vu b = 1$ in $\gT$ there exist $x, y$ such that $a \vu x = b \vu y = 1$ and $x \vi y = 0$. See \cite{Weh2019,DST2019}.
Note that when replacing $x$ and $y$ with $x_1=x\vu(a\vi b)$ and $y_1=y\vu(a\vi b)$ we get $a \vu x_1 = b \vu y_1 = 1$ and $x_1 \vi y_1 = a\vi b$.

\begin{theorem} \label{th-dico-trdi-spec-esp1} \Propeq
\begin{enumerate}
\item The \sps $\Spec(\gT)$ is normal. 
\item The \trdi $\gT$ is normal. 
%
%
\end{enumerate}
\end{theorem}
\begin{theorem} \label{th-dico-trdi-spec-esp2} \Propeq
\begin{enumerate}
\item The \sps $\Spec(\gT)$ is completely normal. 
\item Each interval $[a,b]$ in $\gT$, seen as a \trdi, is normal. 
\item For all $a,b\in \gT$ there exist $x, y$ such that $a \vu b = a \vu y = x \vu b$ and $x \vi y = 0$.
\end{enumerate}
\end{theorem}
\begin{theorem} \label{th-dico-trdi-spec-esp3} \Propeq
\begin{enumerate}
\item Any \oqc in $\Spec(\gT)$ is a finite union of irreducible \oqcs. 
\item For all $a_1,\dots,a_n,b_1,\dots,b_m$ 
one has $a_1,\dots,a_n\vdash_\gT b_1,\dots,b_m$ \ssi there is a~$j$ such that
 $a_1,\dots,a_n\vdash_\gT b_j$. 
\item 
The \trdi $\gT$ is constructed from a \sad corresponding to a \talg. 
%
%
\end{enumerate}
\end{theorem}


\section[Dynamical theories and \sads]{Finitary \tdys and \sads}\label{secSadTrdiSpec}

References: \cite{CLR01,Lom98,Lom06}. 
A more detailed text is in preparation~\cite{Lom-tgac}.

\subsection{Finitary \tdys}

Finitary \tdys have been introduced in \cite{CLR01}. They are a version \gui{without logic, purely computational} of \textsl{coherent theories} (first order theories where all axioms are translations 
of \rdys).\footnote{A usual terminology in mathematical logic speaks about \gui{formal first order theories} when quantifiers are only applied to usual variables, corresponding to elements of the reference set considered as a model of the theory. Second order is used for formal theories using quantifiers on variables in the powerset
 of the reference set. General \tgm, with infinite $\Vou$, is not exactly first order, but absolutely not second order. As a consequence general \tgms are considered as first order theories in the corresponding literature. We use more or less this tradition in our use of \tgms by omitting to speak of \gui{first order \tgms.} We use instead \gui{finitary \tgms.}}

In this paper the authors introduce the notions of \gui{dynamical theory} and of \gui{dynamical proof.}
See also the paper \cite{BC2005} describing some advantages of this approach, and pioneering articles \cite[Sections 1.5 and 4.2]{pra1971}, \cite{Mat75} and \cite{Lif80}.

Dynamical theories use only dynamical rules, \cad deduction rules of the form
\begin{equation} \label{eqRgeom0}
\Gamma \vd \EXists{\und{y^1}} \Delta_1
\vou \cdots\vou \EXists{\und{y^m}}\Delta_m
\end{equation}
where $\Gamma$ and the $\Delta_i$'s are lists of atomic formulae in the language $\cL$ of the theory $\sa{T}=(\cL,\cA)$.

A shortened form is
\begin{equation} \label{eqRgeom}
\Gamma \vd \Exists{\und{y^1}} \Delta_1
\vou \cdots\vou \Exists{\und{y^m}}\Delta_m
\end{equation}

%

Axioms are dynamical rules, and theorems are valid dynamical rules (validity is described in a simple way and uses only a computational machinery) 

If \sa{T} is a \tco, the corresponding \textsl{(finitary) \tdy} uses limited methods of proof.
\begin{itemize}\itemsep=.05em
\item First, only atomic formulae are used: no use of new formulae involving connectors or quantifiers. We manipulate only lists of atomic formulae in the language~$\cL$.
\item Second, axioms are not seen as true formulae, but as \textsl{deduction rules}: an axiom such that \pref{eqRgeom} is used as a \rdy. 
Variables in the lists $\und{y^j}$ are \textsl{bound variables}.
\item Third, proofs are only proofs of \textsl{\rdys}.
\item Forth, the unique way for proving a \rdy is a computational tree
\gui{without logic.} At the root of the tree we find hypotheses of the \tho
we want to prove. The tree is developed by applying axioms
along a pure computational algebraic machinery in the structure described by \sa{T}.
Precise formal \dfns are given in \cite{CLR01}.\\
When using an axiom as \pref{eqRgeom}, we substitute arbitrary terms $t_i$ to free variables $x_i$ in the rule. If these hypotheses are valid at a leaf of the proof tree, this leave becomes a node from which start branches of computation. In each branch one introduces fresh variables corresponding to bound variables $\und{y^k}$ (these fresh variables must be distinct of free variables seen in the terms $t_i$) and each conclusion of the list~$\Delta_k$ is valid in its branch.
Finally, a conclusion is \textsl{valid} when it is proved at each leave of a proof tree.
\end{itemize}

\smallskip In a \tdy, each sort $S$ has an \egt predicate~\hbox{$\cdot=_S\cdot$}. Axioms allow us to substitute a term $t$ by a term $t'$, when the rule $\vd t=t'$ is valid, in any occurrence of an atomic formula in a valid \rdy.\footnote{Naturally, it is not allowed that $t$ or $t'$ contains a variable $x$ under the scope of an $\Exists \,x$.}

Simplest axioms for this scope are the following ones.
First, \egt has to be an \eqvc relation

\DeuxRegles{
\Lab{eq1} $\vd x=x$
\Lab{Eq3} $\,\, x=y\vet y=z\vd x=z$
}{ 
\Lab{eq2} $\,\, x=y\vd y=x$ 
}

\noindent Second, for each function symbol $f$ and each predicate $P$ in the signature (we take them with arity one for simplicity) we have suitable compatibility axioms.

\DeuxRegles{
\Lab{eq$_f$} $\,\, x=y\vd f(x)=f(y) $
}{
\Lab{Eq$_P$} $\,\, x=y\vet P(x)\vd P(y) $
}

\Subsubsection{Collapsus}\label{NOTABot}

A \rdy is called a \textsl{collapsus rule} when the right member is \gui{$\Faux$,} denoted by~$\Bot$. The symbol $\Bot$ is \ncrt in the language. It is an atomic formula, a predicate of arity $0$. On can also see $\Bot$ as the empty disjunction. When $\Bot$ is proved, the universe of the discourse collapses, and all atomic formulae become \gui{true,} or at least \gui{valid.} This is the meaning of \gui{ex falso quod libet,} which is the relevant meaning of $\Faux$ in \coma. So, in a \tdy, the rules 

\Regles {\lab {False$_{P}$} $\,\,\Bot\vd P$} 

\noindent are valid for all atomic formulae.

The language has also the logical constant $\Top$ meaning \gui{$\Vrai$,} with the following axiom.\label{NOTATop} 

\Regles {\lab{True}$\vd \Top$}

The symbol $\Top$ can also be seen as the empty conjunction.\footnote{When there is nothing to prove, prove nothing and all is OK. In a \tdy with at least one sort~$S$, $\Top$ is \eqv to $x=_Sx$.}
Constants $\Bot$ and $\Top$ are the unique logical symbols used in \tdys.

\Subsubsection{Classification of \rdys} 

A \tdy is termed \textsl{propositional} when it does not involve any sort. In this case, constants are $\Top$ and $\Bot$, plus possibly other constants of arity zero, seen as propositional constants.

A \rdy with neither $\,\Exists\,$ nor $\Bot$, nor $\vou $ on the right of $\vd$ is called \textsl{\ralg}. 
A \tdy is a \textsl{\talg} when all axioms are \ralgs.
A~\talg with a single, equality predicate, is called an \emph{algebraic theory}.\footnote{
In~\cite{CLR01}, \ralgs are called algebraic rules, \talgs are called algebraic theories and algebraic theories are called \peq.
theories.}

\label{regledirecte}
A \ralg is \textsl{direct} when its hypothesis (on the left of $\vd$) features a list of predicates over variables only, and when these variables are moreover pairwise distinct. \cite{CLR01} make use of direct rules for constructing formal \nsts and their variants. The latter are algebraic certificates of collapsus for certain \tdys.

A \ralg which is not direct is a \textsl{simplification rule}.

We use the following typographical conventions: names of direct rules are in lower case, names of simplification rules start with a capital letter and names of other dynamical rules are in capital letters.

\smallskip A \tdy is \textsl{\wdij} if in the axioms if its conclusion features no $\Exists$. 
A \rdy is \textsl{existential simple} if its conclusion is of the form $\Exists \ux\; \Delta$, where $\Delta$ is a finite list of atomic formulae.
A \ralg can be considered as a particular case of existential simple rule.
A \tdy is \textsl{regular} if its axioms are \ralgs or existential simple rules. The theory of \textsl{Bézout rings} (each \itf is principal) is clearly regular. 

\smallskip 
A theory is \textsl{existentially rigid} if its existential axioms are simple and correspond to provably unique existences. 
This is a slight \gnn of \twdijs. 
A regular theory which is existentially rigid is called \textsl{cartesian}. This is a slight \gnn of \talgs.

A theory is \textsl{rigid} (or disjunctive) when all its axioms are:
\begin{itemize}
\item Horn rules 
\item Disjunctive rules of type  $\Gamma \vd P\vou Q$ with the provable rule $\Gamma, P\vet Q\vd \Bot$ 
\item Existentially rigid existential rules
\end{itemize}

The theory of discrete fields can be stated as a rigid \tdy using the invertibility predicate. The theory of discrete real closed fields can also be stated as a rigid \tdy, as opposed to the theory of discrete algebraically closed fields.

\Subsubsection{A basic example}\label{TpeAc}

The \textsl{theory \sa{Cr} of commutative rings} is the paradigmatic example of a \tpe. The signature is 
$
\Sigma_{\sA{Cr}}=(\cdot=0\mathrel{;}\cdot+\cdot,\cdot\times \cdot,-\,\cdot,0,1)
$ 
with only three axioms (\reds):

\DeuxRegles{
\Lab{cr1} $\vd 0=0$
\Lab{cr3} $\,\, x=0\vet y=0\vd x+y=0$
}
{
\Lab{cr2} $\,\,x=0\vd x\times y=0 $
}

\smallskip 
The term \gui{$x-y$} is an abbreviation of \gui{$x+(-y)$} and the predicate \gui{$\cdot=\cdot$} is \textsl{defined}
by the convention: \gui{$x=y$} is an abbreviation for \gui{$x-y=0$.}

\smallskip \noindent \emph{Explanation.} The computational machinery of \pols with integer \coes is added to the computational machinery of dynamical proofs. This machinery, external to the \tdy, rewrites any term as a \pol with integer \coes in a normal form.
For example, the distributivity axiom $x(y+z)=xy+xz$ is replaced with the automatic computation which reduces to $0$ the term
 $x(y+z)-(xy+xz)$.
Similarly, transitivity of equality is obtained using axiom \tsbf{cr3}.

\subsection{Dynamic \agq structures} 

\SADs are explicit in~\cite{Lom98,Lom06} and implicit in \cite{CLR01}, where they are described through their presentations.
They are also implicit in \cite{Lom02} and, last but not least, in \cite[{D5}]{D5}, which was a main source: it is possible to compute inside the \alc of a \cdi, even if it is impossible to construct the structure. So it suffices to consider the \alc as a \sad à la D5 rather than a usual \salg: \textsl{lazy evaluation à la D5 gives a \cov semantic for the \alc of a \cdi}.

\smallskip If $\sa{T}=(\cL,\cA)$ is a \tdy, a \textsl{\sad of type~\sa{T}} is given by a set $G$ of generators and a set $R$ of \textsl{relations}. 

 By \dfn a relation is a closed atomic formula $P(\und{t})$ on the language $\cL\,\cup\, G$ with closed terms~$t_i$ in this language. Such a relation gives the axiom \gui{$\Vd P(\und{t})$} in the \sad.

From a \cov viewpoint, $G$ can be seen as a set à la Bishop. So if two objects $a,\,b$ give \elts of $G$ and if we have $a=_Gb$, the relation $a=b$ is automatically present in the set $R$ of relations that define the \sad $\big((G,R),\sa{T}\big)$.

\begin{notation}\label{notasadreglevalid}
We shall indicate that the rule \gui{$\,\Gamma\Vd \dots$} is valid in the \sad $\gS=\big((G,R),\sa{T}\big)$ in the following abridged form: \gui{$\, \Gamma\Vdi\gS \dots$}.
\end{notation}

\begin{definota} \label{notaTcl}
Let $\gS=\big((G,R),\sa{T}\big)$ be a \sad of type $\sa{T}=(\cL,\cA)$. The set of closed terms of $\gS$, \cad terms constructed on $\cL\cup G$, is denoted by $\Tcl(\gS)$. The set of closed atomic formulae is denoted by~$\Atcl(\gS)$. 
A \ralg~$\vd P$ with~$P\in\Atcl(\gS)$ is called \textsl{a fact of $\gS$}. The set of valid facts in $\gS$ is denoted by $\Atclv(\gS)$.
 \end{definota}
Intuitively, a \sad is an incompletely specified usual \salg. 

\begin{example} \label{exaSaCd}  
For instance, we obtain a \sad of discrete field

\snic{\gK=\big((G,R),\sa{Df}\big)}

\noindent by taking~\hbox{$G=\so{a,b}$} and
$R=\so{105=0,\,a^2+b^2-1=0}.$ This dynamical \cdi corresponds to an arbitrary field of characteristic
$3$ or $5$ or $7$ generated by two \elts $\alpha$ and
$\beta$ such that~\hbox{$\alpha^2+\beta^2=1$}. 

\noindent In addition to the \rdys valid in all discrete fields, we now also have all the ones obtained by extending the language with constants in $G$ and by adding to the axioms the relations in $R$.\eoe 
\end{example}

Note that a valid \rdy in a \sad uses for its \demo a computational tree, without logic, with a finite number of \gtrs, relations and axioms. 

\begin{notation} \label{notaT(A)}
When $\gA$ is a usual \salg on the language of $\sa{T}$ we note $\sa{T}(\gA)$ the \sad we get by taking the following \pn $(G,R)$: $G$ is the set of \elts of $\gA$ and $R$ is the set of valid facts in $\gA$. 
For example, if $\sa{T}$ is the theory of Bézout rings, and $\gA$ is an arbitrary commutative ring, we put in $R$ relations $a+b-c=0$, $a'b'-c'=0$ and $a+a''=0$ when $a,b\dots,a''$ are \elts of $\gA$ and when the relations hold in $\gA$. We say that we have added to the theory \sa{T} the \textsl{positive diagram of $\gA$}.

\end{notation}

\Subsubsection{Constructive models versus classical models}\label{subsecmodelescofsSAD}
Let us consider a \sad $\gA=\big((G,R),\sa T\big)$ of type \sa{T}. For simplifying notations we assume there is only one sort. A \textsl{model of~$\gA$} is a usual (static) \salg~$M$ on the language 
associated to~$\gA$ and satisfying axioms of $\gA$ (the ones of $\sa{T}$ and those given by the \pn of $\gA$).

When $\gA$ is defined by the empty \pn, we have \textsl{models of \sa T}.

\smallskip 
So, the notion of model is based on the intuitive notion of \textsl{\salg} à la Bourbaki. 
We can say that these \salgs are \gui{static} in contrast to \gnl \sads. Note that here the underlying set of the structure is a naive set (or several naive sets if there are several sorts) which is structured by giving predicates and functions (in the naive meaning)
subject to certain axioms.

From a \cov viewpoint, axioms in models must hold with the \cov meaning of \gui{or} and \gui{there exists}: in order to prove that a given \salg satisfies the axioms, we have to use the intuitionist logic. 
Let us note also that the set theory we use in our (external) reasoning about \sads is the informal set theory of Bishop.

\subsection{Conservative extensions}

A \tdy \sa{T} is defined by a pair of sets $(\cL,\cA)$ where $\cL$
is (the signature of) the formal language we use, and $\cA$ is the set of axioms, which are \rdys on the language~$\cL$. 
From our \cov viewpoint, these two sets are intuitive sets à la Bishop. 
These sets are in all cases very different of sets which are usually considered in categorical logic, where the external framework is given as purely formal, in \sa{ZFC} or in an extension of \sa{ZFC}.
In a similar way, it should be possible to use a formalization of \coma compatible with Bishop's set theory for studying \gnl \prts of \tdys.
In this paper we work within the informal set theory of Bishop.

A \tdy $(\cL',\cA')$ is a \textsl{simple extension} of $(\cL,\cA)$ if $\cL$ and $\cA$ are subsets of $\cL'$ and $\cA'$ (with the categorical meaning in Bishop's book). In this case, the \rdys formulated in the language $\cL$ and valid in $\sa T$ are valid in  $\sa T\;'$.
 
\begin{definition} \label{defitdyidentiques}
Two \tdys \textsl{on the same language} are said to be \textsl{identical} if they prove the same \rdys, \cad if axioms of each one are valid rules of the other one. In this case, models are the same in \coma as they are in \clama.\end{definition}

\begin{definition} \label{defiextcons}
We say that a \tdy $\sab{T}'$ is a \textsl{simple conservative extension of the theory \sa{T}}
if it is a simple extension of \sa{T} and if \rdys of \sa{T} which are valid in $\sab{T}'$
are valid in \sa{T}.
\end{definition}

\noindent {\bf Informal \dfn.} Other extensions, more \gnl than those of \Dfn \ref{defitdyidentiques}, have to be considered as \gui{changing nothing to the theory} on one hand because they are conservative and on the other hand because \cov models of any \sad are the same. In this case we say that $\sab{T}'$ is an \textsl{intuitively equivalent} extension of the theory \sab{T}.

\smallskip We are now looking at a number of these extensions. 

\Subsubsection{\Esid extensions}

\paragraph{Adding abbreviations in the language}~

\smallskip For example, we may introduce a function symbol \gui{$\Som(\cdot,\cdot,\cdot)$} as an abbreviation: $\Som(x,y,z)$ 
is an abridged notation for $(x+y)+z$.

Clearly this type of \dfns \gui{changes nothing} to the \tdy, it is an \inteq extension.

\paragraph{Adding predicates: conjunction, disjunction, existence}\label{secajout-et-ou-ex}~

\smallskip Let $\Delta=(A_1,\dots,A_n)$ be a list of predicates.

One may consider that the following introduction and elimination rules \und{define} the conjunction, as in natural deduction.

\regles{
 \lab{Intro-$\vii_{\Delta}$} $\,\,A_1\vet\dots \vet A_n\vd A_1\vii \dots \vii A_n$
 \lab{Elim-$\vii_{\Delta}$} $\,\, A_1\vii \dots \vii A_n \vd A_1\vet\dots \vet A_n$
}

\noindent Last rule is \eqv to the conjunction of rules $A_1\vii \dots \vii A_n\vd A_i$.

\smallskip Disjunction is more complicated, because in dynamical proofs there is no $\vou$ on the left of $\vd$. One may introduce the disjunction connector 
$\vuu$ in the language by using rules inspired by natural deduction.

\regles{
 \lab{Elim-$\vuu_{\Delta}$} $\,\,A_1\vuu \dots \vuu A_n \vd \,A_1\vou \dots \vou A_n$
 \lab{Intro-$\vuu_{\Delta,1}$} $\,\,A_1 \vd A_1\vuu \dots \vuu A_n$
 \lab{~} $\vdots$
 \lab{Intro-$\vuu_{\Delta,n}$} $\,\,A_n \vd A_1\vuu \dots \vuu A_n$
 }
 
\smallskip 
We introduce the existential quantifier $\exists$ by the following introduction and elimination rules (note that it is not a joke).

\regles{
 \lab{Intro-$\exists_{x,A}$} $\,\, A(x) \vd \exists x\,A(x)$
 \lab{Elim-$\exists_{x,A}$} $\,\,\exists x\,A(x) \vd \Exists x\,A(x)$
}

\paragraph{Adding a function symbol in case of unique existence}~

\sni We consider a \tdy $\sa{T}=(\cL,\cA)$.
Assume that $\cL$ has a predicate $P(u,x,y)$ of arity $k+1$ (we give the example with $k=2$) 
and that \sa{T} prove the following \rdys

\DeuxRegles{
\lab{Ex$_{P,u}$} $ \vd \Exists u\;P(u,x,y)$
}
{
\lab{Uniq$_{P,u}$} $ P(u,x,y),\,P(v,x,y)\vd u=v $
}

\noindent 
Let $\sab{T}'$ be the \tdy we get by adding to \sa{T} a new function symbol $f$ and the following axiom 
 
\UneRegle{df$_{P,u,f}$}{$\,\,P(u,x,y)\vd u=f(x,y)$}

\noindent We say that
\textsl{the rule \tsbf{df$_{P,u,f}$} defines the function symbol $f$}. 
\\
In this case the \tdy 
$\sab{T}'$ is a conservative extension 
of \sa{T}.

\begin{lemma} \label{lemdefiesid}
For a \tdy using previously defined extensions 
(abbreviations, conjunction predicate, disjunction predicate, existential quantifier, function symbols in case of unique existence) produces a conservative extension which does not change \cov models. 
\end{lemma}

\begin{definition} \label{defiesid}~
\begin{enumerate}
\item Such an extension is called \textsl{\esid simple}. 
\item Two \tdys $\sa{T}_1$ and $\sa{T}_2$ are said to be \textsl{\esid} if we have a \tdy which is, perhaps after renamings in the signatures, an \esid extension of $\sa{T}_1$ and $\sa{T}_2$. 
\end{enumerate}
\end{definition}

\Subsubsection{\Eseq extensions} \label{subsectdyeseq}

References: \cite{BH2017}, \cite{Tse2017}. We now examine situations obtained by extending a \tdy by adding well defined new sorts. This corresponds to allowable set constructions in Bishop's set theory.

\paragraph {Introducing a subsort}~

\sni Let us consider a \tdy \sa{T} and a sort $S$ in this theory. Let us consider a unary predicate $P(\cdot)$ on objects of type~$S$.
We define a new \tdy by adding \textsl{the subsort $U$ of $S$ defined by the predicate $P$} in the following way.
\begin{enumerate}
\item We add $U$ in the sorts.
\item We add a function symbol $j_P:U\to S$.
\item We add the following axioms.

\DeuxRegles
{
\lab{sso$_P$} {$\vd_{u:U} P(j_P(u))$}
}
{
\lab{SSO$_P$} {$\,\,P(a)\vd_{a:S} \Exists u \; j_P(u)=a$}
}

\item We define the equality predicate $x=_{U}y$ on $U$ 
as an abbreviation of $j_P(x)=j_P(y)$. 
\end{enumerate}

\paragraph {Introducing a sort for a finite product of sorts}~

\sni Let us consider a \tdy \sa{T} and sorts $S_1$, \dots, $S_n$ in this theory. 
We define a new \tdy by adding the \textsl{finite product $S$ of sorts $S_i$ } in the following way.
\begin{enumerate}
\item We add a sort $S$.
\item For each $i\in \lrbn$ we add a function symbol $\pi_i$ of type $S\to S_i$.
\item We define $x=_{S}y$ on $S$ 
as an abbreviation of $\Vii_{i=1}^n \pi_i(x)=\pi_i(y)$. 
\item We add a function symbol $\mathrm{Pr}$ of type $S_1\times \cdots\times S_n\to S$. 
\item We add axioms

\Regles{ \lab{fps$_i$} $ \vd_{x_1:S_1,\dots,x_n:S_n} \; \pi_i\big(\mathrm{Pr}(\xn)\big)=x_i$ 
}

\end{enumerate}

\paragraph {Introducing a quotient sort}~

\sni Let us consider a \tdy \sa{T} and a sort $S$ of this theory. Let us consider a binary predicate $E(x,y)$ on $S$. We assume that the rules saying that $E(x,y)$ is an \eqvc relation on $S$ are valid in~$\sa{T}$. 
We define a new \tdy by adding \textsl{the quotient sort $B$ of $S$ w.r.t.\ the \eqvc relation $E$} in the following way.
\begin{enumerate}
\item We add the sort $B$ and an equality predicate $x=_By$ on $B$.
\item We add a function symbol $\pi_{E}$ of type $S\to B$.
\item We add axioms

\DeuxRegles{
\lab{qs1$_E$} {$\,\,E(a,b)\vd_{a,b:S}\; \pi_{E}(a)=_B \pi_{E}(b)$}
\lab{QS$_E$} $\vd_{x:B}\; \Exists a \; \pi_{E}(a)=_Bx$
}{
\lab{Qs2$_E$} {$\,\,\pi_{E}(a)=_B\pi_{E}(b)\vd_{a,b:S}\; E(a,b) $}
}
\end{enumerate}

\paragraph {Introducing a sort for a finite disjoint sum of sorts}~

\sni Let us consider a \tdy \sa{T} and sorts $S_1$, \dots, $S_n$ in this theory. 
We define a new \tdy by adding \textsl{the disjoint sum of sorts $S_i$} in the following way.
\begin{enumerate}
\item We add the sort $S$ and an \egt predicate $x=_Sy$ on $S$.
\item For each $i\in \lrbn$ we add a function symbol $j_{S_i,S}$ of type $S_i\to S$.
\item We add the following axioms (we use $j_i$ as an abbreviation for $j_{S_i,S}$)

\Regles{
\lab{Fds$_i$} {$\,\,j_i(a)=_Sj_i(b)\vd_{a,b:S_i}\; a=_{S_i}b$ \hfill for $1\leq i\leq n$\hspace{4cm}~}
\lab{FDS$_{i,k}$} {$\,\,j_i(a)=_Sj_k(b)\vd_{a:S_i,b:S_k}\;\Bot $ \hfill  for $1\leq i< k\leq n$\hspace{4cm}~}
\lab{FDS} $\vd_{x:S}\; \Exists a_1 \; j_1(a_1)=_Sx \;\vou\;\cdots\;\vou\;\Exists a_n \; j_n(a_n)=_Sx$}

\end{enumerate}

\begin{lemma} \label{lemmaddingsortes}
Consider a \tdy obtained by using previously defined extensions, as in Lemma \ref{lemdefiesid} or by introducing a new sort as in the previous examples.
Then the new theory is \inteq to the first one. 
\end{lemma}

\begin{definition} \label{defieseq}~
\begin{enumerate}
\item Such an extension is called \textsl{\eseq simple}. 
\item Two \tdys $\sa{T}_1$ and $\sa{T}_2$ are said to be \textsl{\eseq} if we have a \tdy which is, perhaps after renamings in the signatures, an \eseq extension of $\sa{T}_1$ and $\sa{T}_2$. 
\item One says that \textsl{the \tdy $\sab{T}'$ is an extension
 of the \tdy \sa{T}} if it is a simple extension of a theory which is \eseq to \sa{T}.
\end{enumerate}
 
\end{definition}

\Eseq extensions do not change \cov models of \sads. 
Other conservative extensions may change \cov models. We are now looking at a number of these extensions. 

\Subsubsection{Other conservative extensions}

\paragraph{Adding classical logic}~

\smallskip For a \tdy, accepting classical logic is the same thing as allowing systematically beyond Definition \ref{defiesid}, 
the introduction of a predicate negating a previously defined predicate~$P$. 

When introducing the opposite predicate $Q$ (denoted as $\lnot P$ or $\ov P$) we add the axioms of Boolean logic.

\DeuxRegles{
\labu $\;\;P\vet Q\vd \Bot$
}
{
\labu $\vd P\vou Q$
}

From a \cov viewpoint, adding classical logic means in models of \sads that all predicates be decidable.

\begin{theorem}[Cut elimination] 
\label{thFond} 
 When computing in a \tdy the use of classical logic produces a conservative extension.
\end{theorem}

The use of classical logic is often seen by classical mathematicians as the possibility of transforming all \dems in \dems by contradiction.
For example the following rules are now \eqv.

\DeuxRegles{
\labu $\;\;A\vet B\vd C\vou D$
\labu $\;\;A\vet \ov C\vet\ov D \vd \ov B$}
{
\labu $\;\;A\vet B\vet \ov C\vet\ov D \vd \Bot$
\labu $\;\;B\vet \ov C\vet\ov D \vd \ov A$}

\paragraph{Skolemization}~ 

\sni In the following \tho, not only classical logic but also Skolemization
is authorized. This can be viewed as a local form of Choice.

\begin{theorem}[Skolemization, {\cite{BC2019}}] 
\label{thFondExists} 
~\\
Let us consider a \tdy \sa{T} and define $\sab{T}'$ to be the \gui{Skolemized} theory: all existential axioms are transformed 
by replacing $\Exists$ by the use of Skolem function symbols. 
Then~$\sab{T}'$ is a \cosv extension of \sa{T}.
\end{theorem}

\section[Distributive lattices associated to a \sad]{Distributive lattices and \spss associated to a \sad}
\label{subsectrdisad}

\subsection{Zariski spectrum and Zariski lattice of a commutative ring}

The Zariski lattice of a commutative ring can be defined using different extensions of the theory~\sa{Cr} of commutative rings.

We choose the theory of local rings since it plays a fundamental role in Grothendieck schemes.

We consider \prmt the \tdy \sa{Lr1} of \textsl{local rings with units}, based on the signature $(\,\cdot=0,\U(\cdot)\mathrel{;}\cdot+\cdot,\cdot\times \cdot,-\,\cdot,0,1\,)$.

The predicate~$\U(x)$ is defined as the invertibility predicate with the suitable axioms.
We add a collapsus axiom and Axiom \tsbf{LR} of local rings.

\DeuxRegles 
{\lab{CL$_{Lr1}$} $\,\, \U(0)\Vd \Bot$}
{\Lab{LR} $\,\, \U(x+y) \Vd \U(x) \vou \U(y)$}

Let $\gA$ be a commutative ring. We consider the \entrel $\vdash_{\gA,\mathrm{Zar} }$ on the underlying set of $\gA$ which is defined by the following \eqvc.
\vspace{-.8em}
\begin{equation} \label {eqZarclass}
\begin{aligned} 
 a_1,\dots,a_n &\,\vdash_{\gA,\mathrm{Zar}} c_1,\dots,c_m 
 \qquad\quad  \equidef  \\[.3em] 
\U(a_1)\vet \dots\vet \U(a_n) & \Vdi{\sA{Al1}(\gA)} \U(c_1)\vou \dots\vou \U(c_m) 
 \end{aligned}
\end{equation}

We define the \textsl{Zariski lattice of $\gA$}, denoted by $\ZarA$ or 
$\Zar(\gA)$,
as the one generated by the \entrel $\vdash_{\gA,\mathrm{Zar}}$.
 
The corresponding map $\DA:\gA\to\ZarA$ is called the \textsl{Zariski support of $\gA$}. When $\gA$ is fixed by the context we merely note $\rD$.

The usual \textsl{Zariski spectrum} in \clama is the dual \sps of $ZarA$.

Note that since $\rD(a_1)\vi\dots\vi\rD(a_n)=\rD(a_1\cdots a_n)$, \elts of $\ZarA$ can be written as $\rD(c_1,\dots,c_m):= \rD(c_1)\vu\dots\vu\rD(c_m)$.

\smallskip 
A more \elr theory is the theory \sa{Wzdr} of \textsl{\sdz} nontrivial rings. It is obtained by adding to \sa{Cr} a collapsus axiom and Axiom \tsbf{WZD}

\DeuxRegles {\lab{CL$_{R}$} $\,\, 1=0\Vd \Bot$}
{ \Lab{WZD} $\,\,xy=0 \Vd x=0 \vou y=0$.}

One proves the \eqvcs in the following \tho. Item $(4)$ is called a formal \nst. Hilbert's \nst is a more difficult topic. 
\begin{theorem}[\nst formel] \label{thNstFormel} ~\\
Let $\gA$ be a commutative ring, and $\an,c_1,\dots c_m\in\gA$. \Propeq
\vspace{-.5em}
\[ 
\begin{aligned} 
(1)\qquad\quad\;  \rD(a_1),\dots,\rD(a_n) &\,\vdash_{\ZarA} \rD(c_1),\dots,\rD(c_m) \\[.3em] 
(2)\qquad \U(a_1)\vet \dots\vet \U(a_n) &\Vdi{\sA{Al1}(\gA)} \U(c_1)\vou \dots\vou \U(c_m) \\ 
(3)\qquad c_1=0\vet \dots\vet c_m=0 &\Vdi{\sA{Asdz}(\gA)} a_1=0\vou \dots\vou a_n=0 \\ 
(4)\qquad \exists k>0\;\;(a_1 \cdots a_n)^k&\,\in\gen{c_1,\dots,c_m} 
 \end{aligned} \label {eqNstFormel}
\] 
\end{theorem}

As a consequence the \elt $\rD(c_1,\dots,c_m)$ of $\ZarA$ can be identified with the 
\id $\sqrt[\gA]{\gen{c_1,\dots,c_m}}$. Modulo this identification, the order relation is set inclusion.

\begin{corollary} \label{corZarA}
The lattice $\ZarA$ is generated by the least \entrel on (the underlying set of) $\gA$ such that (here $\bot=0_\gT$ and $\top=1_\gT$)

\DeuxRegles
{
\labu $\;\;0\vdash \bot$
\labu $\;\;ab\vdash a$
\labu $\;\;a+b\vdash a,b$
}
{
\labu $\;\; 1\vdash \top$
\labu $\;\;a,b\vdash ab$
} 

\noindent In other words, the map $\rD:\gA\to\ZarA$ satisfies the relations 
\[
\rD(0)=0,\;\rD(1)=1,\;\rD(ab)=\rD(a)\vi\rD(b),\;\rD(a+b)\leq \rD(a)\vu \rD(b),
\]
and any other map $\rD':\gA\to T$ satisfying these relations factorizes via $\ZarA$ with a unique lattice morphism $\ZarA\to T$.
\end{corollary}

\subsection[Real lattice of a commutative ring]{Real lattice and real spectrum of a commutative ring}
 
\noindent Let us consider a \textsl{\tdy of \codis}, based on the signature $(\,\cdot=0,\cdot\geq 0,\cdot>0\mathrel{;}\cdot+\cdot,-\,\cdot,0,1\,)$, for example the theory \sa{Dof} given in~\cite[Section 3]{CLR01}. 

The \textsl{real lattice}
 of a commutative ring $\gA$, denoted by $\Reel(\gA)$, is generated by the \entrel $\,\vdash_{\gA,\mathrm{Real}} $ on $\gA$ defined by the following \eqvc.
 
\vspace{-.8em}
\begin{equation} \label {eqreel}
\begin{aligned} 
 a_1,\dots,a_n &\,\vdash_{\gA,\mathrm{Real}} c_1,\dots,c_m 
 \qquad\quad  \equidef  \\[.3em] 
a_1>0\vet \dots\vet a_n>0 & \Vdi{\sA{Dof}(\gA)} c_1>0\vou \dots\vou c_m>0 
\end{aligned}
\end{equation}

We note $\rR:\gA\to \Reel(\gA)$ the corresponding map.
One proves the following \eqvc (\textsl{formal Positivstellensatz}) 
\[\rR(a_1),\dots,\rR(a_n) \,\vdash_{\Reel(\gA)} \rR(c_1),\dots,\rR(c_m) \Longleftrightarrow \exists k\in\NN\;\exists p\in C\;\;(a_1 \cdots a_n)^k+ p =0
\]
where $C$ is the positive cone generated by the $a_i$'s and $-c_j$'s.
If $n=0$ we replace $(a_1 \cdots a_n)^k$ with $1_\gA$.

The usual real spectrum $\Sper(\gA)$ is the dual \sps of $\Reel(\gA)$.
One can identify \elts of $\Sper(\gA)$ to prime cones of $\gA$. The \oqc corresponding to the \elt $\rR(a)\in\Reel\gA$ is given by $\sotq{\fc\in\Sper\gA}{-a\notin\fc}$.

One proves that the lattice 
$\Reel(\gA)$ is generated by the least \entrel on $\gA$ satisfying the following relations 

\DeuxRegles{
\labu $\,\,-x^2 \vdash $
\labu $\,\,x+y \vdash x , y $
\labu $\,\,xy \vdash x , -y $
}
{
\labu $\,\,\vdash 1 \phantom{x^2}$
\labu $\,\,x , y \vdash xy $
}

More \gnlt we can define $\Reel(\gA)$ and $\Sper(\gA)$ for any \sad $\gA$ of type \sa{Dof}.

For more details see \cite{CC00,Lom2020}.
 
\subsection{Other examples}

\paragraph{First example.} Let us consider a \sad $\gA=\big((G,R),\sa{T}\big)$ for a \tdy $\sa{T}=(\cL,\cA)$.
If $P(x,y)$ is a binary predicate in the signature, and if $Tcl =\Tcl(\gA)$ is the set of closed terms of~$\gA$, we get an \entrel $\vdash_{\gA,P}$ on $Tcl \times Tcl $ by letting
 
\vspace{-.8em}
\begin{equation} \label {eq1}
\begin{aligned} 
 (a_1,b_1),\dots,(a_n,b_n) &\,\vdash_{\gA,P} (c_1,d_1),\dots,(c_m,d_m) 
 \qquad\quad  \equidef  \\[.3em] 
 P(a_1,b_1)\vet \dots\vet P(a_n, b_n) & \Vdi{\gA} P(c_1, d_1)\vou \dots\vou P(c_m, d_m) 
 \end{aligned}
\end{equation}

\noindent Intuitively the \trdi generated by this \entrel is the lattice of \gui{truth values} of the predicate~$P$ in $\gA$.
 
\paragraph{More generally.} Let us consider a \sad $\gA=\big((G,R),\sa{T}\big)$ for a \tdy $\sa{T}=(\cL,\cA)$.
Let $S$ be a set of closed atomic formulae of $\gA$. We define the \entrel on $S$ associated to $\gA$ in the following way: 
 
\vspace{-.5em}
\begin{equation} \label {eq2}
\begin{aligned}
 A_1,\dots,A_n &\,\vdash_{\gA,S} B_1,\dots,B_m 
 \qquad\quad \equidef  \\[.2em] 
 A_1\vet \dots\vet A_n &\Vdi{\gA} B_1\vou\dots\vou B_m 
 \end{aligned}
\end{equation}
We can note $\Zar(\gA,S)$ the \trdi generated by this \entrel.



\paragraph{A conservative extension \sa{T1} of a \tdy \sa{T}} gives isomorphic Zariski lattices for \sads $\sa{T}(\gA)$ and $\sa{T1}(\gA)$, 
with a same set $S$ of closed terms. 
The lattice therefore gives a diminished image of the \sad. For example adding classical logic and Skolemizing a \tdy do not change associated lattices. Sometimes the study of the lattice is easier in the second theory.
But to recover the wealth of \tdys seen from a \cov viewpoint, it becomes necessary in \clama to use sheaves or toposes. 

\subsection{The absolute Zariski lattice of a \sad~$\gA$}

The (absolute) Zariski lattice of a \sad $\gA=\big((G,R),\sa{T}\big)$ is defined by taking~$S$ as the set $\Atcl(\gA)$ of all closed atomic formulae of~$\gA$. We denote it $\Zar(\gA,\sa T)$ or with a suitable name corresponding to the theory \sa{T}. For example~$\val(\gA)$
for the theory \Sa{val} (see \paref{subsubsecval}).
The \sps dual is called the \textsl{Zariski spectrum of the \sad $\gA$} (the valuative spectrum in the case of the theory \sa{val}). 

When the theory $\sa{T}$ is \wdij, $\Zar(\gA)$
is (up to a canonical isomorphism) the \trdi defined by the \entrel on
$\Atcl(\gA)$ generated by axioms in $R$ and instantiations of axioms of \sa{T} obtained by substituting variables by closed terms. 

\smallskip 
For example $(\ZarA)\oop$ can be seen as the absolute Zariski lattice of $\sa{Wzdr}(\gA)$.

\subsection{Spectrum and models in \clama}\label{specetmodeles} 

Here models are \gnlt seen from the viewpoint of \clama. 

One chooses for set $S$ of closed atomic formulae few predicates in the language, such that other predicates can be defined in \clama from those of $S$. 

Giving a point of the spectrum
$\Spec(\Zar(\gA,S))$, \cad a morphism \hbox{$\alpha:\Zar(\gA,S)\to\Deux$} means to give the truth value $\Vrai$ or $\Faux$ to closed atomic formulae in~$S$ ($\alpha(B)=1$ or~$0$). This implies giving also a truth value $\Vrai$ or $\Faux$ to all closed atomic formulae in~$\gA$. 
This allows us to construct a model of $\gA$ in the case of a \twdij.
We get in this way a \textsl{minimal} model with the meaning that all \elts are constructed from \gtrs in~$G$ by using function symbols in the signature.

\smallskip The choice of the set~$S$ is crucial for the topology
of the dual \sps. Two distinct choices for~$S$ can give the same points of the spectrum in \clama but may define two distinct \spss.


\section{Valuative lattice and spectrum of a commutative ring}\label{secdival}

\subsection{Valuative divisibility relation}\label{subsecdival1}

Reference: \cite{HK1994}.

\Subsubsection{Valuation domains, \rdv}\label{subsubsecdomval}

Remember that a subset $P$ of a $E$ is \textsl{detachable} when the \prt~\hbox{$x\in P$} is decidable for $x\in E$. For describing this situation in \tdys it is necessary to introduce a predicate $R(x)$ opposite to
the predicate $Q(x)$ meaning $x\in P$. So the following rules hold: {$\Vd  Q(x) \vou R(x)$}, {$  Q(x)\vet R(x)\vd \Bot$}. 

A ring is \textsl{integral} (or it is a \textsl{domain}) when each \elt is null or \ndz, and a ring is a \textsl{\cdi} when each \elt is null or invertible. These definitions do not exclude the trivial ring. 

A ring is \textsl{\sdz} when the nullity of a product $ab$ of two elements induces an explicit alternative, $a = 0$ or $b=0$. An integral domain is \sdz. In \clama the reciprocal is valid but not in \coma. 

We say that an ideal is \textsl{prime} if the quotient ring is \sdz. This definition does not exclude the ideal $\gen{1}$. 
The latter conventions are used in \cite{CACM}. They allow the authors to avoid 
negation and certain case by case reasonings, non-legitimate from a \cov viewpoint.

Nevertheless, for corresponding \tdys, we come back to the usual tradition for which local, \sdz, or integral rings, and \cdis, have to be nontrivial: this is \ncr for introducing collapsus axioms giving $\Bot$ as a possible conclusion in a \rdy. 

\smallskip A \textsl{valuation domain} $\gV$ is an integral domain whose \dve relation (on the multiplicative \mo $\gV/\gV\eti$) is a total order: $\forall x, y\,(x\di y\vu y\di x)$. 

If  $\gK$ is the  \cdf of $\gV$, $\gV$ is called a \textsl{valuation ring of $\gK$} and $(\gK,\gV)$ is called a \textsl{valued discrete field}. More generally, a subring $\gV$ of a discrete field $\gK$ is called a valuation ring of $\gK$ if for all $x \in \gK\eti$, $x$ or $x^{-1} \in \gV$.

From a \cov viewpoint, we define a \textsl{\cvd} by forcing the decidability of the relations $x\in \gV$ and $x\in \gV\eti$: this means that the divisibility in~$\gV$ has to be explicit. 

A valuation domain can also be characterized a local \dcd Bézout domain.\footnote{A local ring $\gA$ is said \textsl{\dcd} when the residual field $\gA/\Rad\gA$ is discrete. If the ring is nontrivial this means that units are detachable.}

In \cite[Section 4]{CLR01} a \tdy \sa{Vf} is introduced for valued discrete fields with this \cov viewpoint. The signature has three predicates $\Vr(x)$, $\U(x)$ and $\Rn(x)$, respectively, for $x\in \gV$, $x\in \gV\eti$ and $x\in\gV\setminus\gV\eti$. We come back to this theory in Section~\ref{secdival6}.

\smallskip 
In a \cvd $(\gK,\gV)$ we say that $x$ divides $y$ and we write $x\di y$ if there exists a $z\in\gV$ such that $xz=y$.
We note $\Gamma=\Gamma(\gV)$ the group $\gK\eti\!/\gV\eti$ (in additive notation), with the order relation $\leq$ induced by the relation $\di$ in $\gK\eti$. We note $\Gamma_\infty=\Gamma\cup\so\infty$ (where $\infty$ is a purely formal maximum \elt). Thus, the natural map $v:\gK\to\Gamma_\infty$ (with $v(0)=\infty$) is called the \textsl{valuation} of the \cvd. One has $v(xy)=v(x)+v(y)$ and $v(x+y)\geq \min(v(x),v(y))$ with \egt if $v(x)\neq v(y)$. We have also \fbox{$\gV=\sotq{x\in\gK}{v(x)\geq 0}$} and the unit group is characterized by \fbox{$\gV\eti=\sotq{x\in\gK}{v(x)= 0}$}.

\smallskip In \clama one defines a \textsl{valuative divisibility relation} $a\di b$ on a commutative ring $\gA$ as the reciprocal image of the divisibility relation on a \cvd $(\gK,\gV)$ by a ring morphism $\varphi:\gA\to\gK$. In other words, we have $a\di b$ in $\gA$ \ssi $\varphi(a)\di\varphi(b)$
\hbox{in $(\gK,\gV)$}.

\Subsubsection{Points of the valuative spectrum in \clama}\label{subsubsecspecval}

\begin{definition} \label{defispevclass}
In \clama, \elts of the \textsl{valuative spectrum $\,\Spev(\gA)$ of a commutative ring~$\gA$} are defined 
in the following way (see \cite{HK1994}): a point of $\SpevA$ is given by a pair~\hbox{$(\fp,\gV)$} where $\fp$ is a \idep of $\gA$ and $\gV$ a \arv of the fraction field $\gK=\Frac(\gA/\fp)$. 
 \end{definition}

This is analogous to the real spectrum: a point of $\SperA$ is given by a pair~\hbox{$(\fp,\gC)$} where $\fp$ is a \idep of $\gA$ and $\gC$ a positive cone of the fraction field $\gK=\Frac(\gA/\fp)$ (i.e., $\gC+\gC\subseteq \gC$, $\gC\cdot\gC\subseteq \gC$, $\gC\cup -\gC=\gK$, $\gC\cap -\gC=\so 0$).

\smallskip Distinct spectral topologies can be defined on the valuative spectrum, depending on the choice of basic \oqcs. 

This corresponds to suitable \trdis (as those described in Section~\ref{subsectrdisad}) which are defined w.r.t.\ \tdys that describe \prts of a \rdv.

We are now looking at a number of these theories. 

\subsection{Weakly disjunctive theories for a \rdv}

\Subsubsection{The theory \sa{val0}}\label{subsubsecval0}

\begin{definition} \label{defithval0}

One chooses the signature 
\[
\Sigma_\mathrm{val0}=(\,\cdot\di\cdot, \cdot=0 \mathrel{;} \cdot+\cdot, \cdot\times\cdot,-\,\cdot,0,1\,).
\] 
Axioms for $x=0$ and $a\di b$ 
are the following ones.

\DeuxRegles
{\Lab{vr0} $\,\,0\di x\Vd x=0$
\Lab{vr1} $\vd 1 \di -1$
\Lab{vr2} $\,\,a \di b \Vd ac \di bc$
\Lab{Vr1} $\,\,a \di b \vet b \di c \Vd a \di c$
}
{
\lab{CL$_{\val}$} $\,\, 0\di 1 \vd \Bot$ \quad (collapsus)
\Lab{Vr2} $\,\,a \di b \vet a \di c \Vd a \di b + c$
\Lab{VR1} $\vd a \di b \vou b\di a$
\Lab{VR2} $\,\,ax \di bx \Vd a \di b \vou 0 \di x$
}

We note \sa{val0} this \twdij.
 
\end{definition}

One proves easily $\vd -1 \di 1$, $\vd 1 \di 1$, $\vd 1 \di 0$, $\vd x \di x$, $\vd x \di 0\,$ and $\,x=0\Vd 0\di x$. Axioms \Tsbf{cr1}, \Tsbf{cr2} and \Tsbf{cr3} for commutative rings are satisfied.

Here is a precise statement saying that our theory describes correctly a \rdv.

\begin{theoremc}[Models of the theory \sa{val0} in \clama] \label{lemAxdiv} ~\\
In \clama a binary relation $a \di b$ on a ring $\gA$ satisfies axioms of \sa{val0} exactly in the following case. 
\begin{itemize}\itemsep=.1em
\item The set $\fp=\sotq{x\in \gA}{0\di x}$ is a \idep.  
We note $\ov a$ the \elt~$a$ of~$\gA$ seen in $\gA/\fp$, and $\gK$ the fraction field $\Frac(\gA/\fp)$. 
\item Fractions ${\ov a}/{\ov b}\in\gK$ such that $\ov b\neq 0$ and $b\di a$ make a \arv of $\gV$ of~$\gK$.  
\end{itemize}
In other words, axioms for $\cdot\di\cdot$ in \sa{val0} correspond to the \dfn of a \rdv (in \clama).
\end{theoremc} 
%
\begin{proof} First, one sees easily that the axioms hold if $\varphi:\gA\to\gK$ is a morphism from $\gA$ to a field $\gK$, if $\gV$ is a \arv of $\gK$ and if $a\di b$ means: $\exists x\in \gV\;x\varphi(a)=\varphi(b)$. 

\noindent Let us now prove that the axioms describe correctly this situation (in \clama).
\\
Axioms \tsbf{vr1} and \tsbf{vr2} give valid rules $\Vd a \di a$ and $\Vd a\di 0$. So, recalling axiom \tsbf{Vr1}, the relation \hbox{$a\di b$}
defines a preorder. And if $0 \di 1$ we get $0\di b$ and $a\di b$ for all $a,b$ without using Axiom \tsbf{CL$_{\val}$}.
\\
A particular case of \tsbf{vr2} is the following valid rule.

\penalty-2500
\Regles{ 
\lab{vr2'} $\,\,1\di a\vdg b \di ab$} 

\noindent Using~\tsbf{Vr2} and \tsbf{vr2'}, one sees that for all $a\in \gA$ such that $1\di a$, the set 
\[\uar a:=\sotq{b\in\gA}{a \di b}
\]
 is an ideal of $\gA$. 
In particular $\fp=\uar 0$ is an ideal.
\\
Let us prove that the relation $\cdot\di\cdot$ passes to the quotient by $\fp=\uar 0$.
Indeed, assume $0\di x$, it suffices to prove that $a\di a+x$ and $a+x\di a$. First, $a\di 0\di x$ and $a\di a$, hence $a\di a+x$ using axiom~\tsbf{Vr2}. Now, since we have $a+x\di a+x$ 
we get $a+x\di a+x-x$, i.e., \ $a+x\di a$.
\\
Let us prove that $\gA/\fp$ is \sdz: if $0\di yx$ then $0x\di yx$, and \tsbf{VR2} gives $\,\,0\di xy\Vd 0\di x \vou 0\di y$.
Using \tsbf{vr2}, \tsbf{Vr1} and \tsbf{Vr2} one sees that the fractions $\ov a/\ov b$ of $\gK$ 
make a subring~$\gV$ of~$\gK$. 
Finally, two inverse \elts in $\gK$ are written
$\ov a/\ov b$ and $\ov b/\ov a$. 
So, Axiom~\tsbf{VR1}
implies~$\gV$ is a \arv of $\gK$.
\end{proof}

Note that $\varphi(\gA)$ is not \ncrt a subring of $\gV$.

 \thref{lemAxdiv} legitimates the following \dfn in \coma.
\begin{definition} \label{defirdvcoma}
A binary relation $a\di b$ on a commutative ring is called a \textsl{\rdv} if axioms of \sa{val0} are satisfied.
\end{definition}

\begin{remark} \label{rem1val0} 
The three axioms of commutative rings are valid in \sa{val0}.

\DeuxRegles{
\lab{cr1} $\vd 0=0$
\lab{cr3} $\,\,\vd x=0 \vet y=0 \Vd x+y=0$
}
{
\lab{cr2} $\,\, x=0 \Vd xy=0$
}

 So, we adopt the convention given in \cite{CLR01} 
 that the purely computational part of commutative rings
is treated outside the formal theory (see \paref{TpeAc}). 

If we consider a commutative ring $\gA$ and the \sad~$\sa{val0}(\gA)$, all closed terms of $\sa{val0}(\gA)$ are equal to \elts of $\gA$ and there is no need of the three axioms for proving closed \rdys. \eoe
\end{remark}

\begin{remark} \label{remColl} 
Without using the collapsus axiom, when the relation $0\di1$ is valid, the ring becomes a singleton. Adding the axiom $\,\,0\di1\Vd\Bot$ amounts to throw the trivial ring in the empty universe.\footnote{Or, perhaps? In a black hole.} \eoe
\end{remark}

\begin{remark} \label{remA1A6coma} 
In \coma if a ring $\gA$ with a binary relation $\cdot\di\cdot$ is a model of~\sa{val0} we get a \idep $\fp$ of $\gA$ and a suitable subring~$\gV$ of the total fraction ring $\gK$ of $\gA/\fp$. But $\gK$ is not a priori a \cdi. So we prefer to think about $\gA$ as the \sad $\sa{val0}(\gA)$, in which we can \gui{do as if} 
we were in \clama, where all fields are discrete. \eoe
\end{remark}

\Subsubsection{The theory \sa{val} and some dynamical rules provable in it}\label{subsubsecval}

\begin{definition} \label{defithval}
The theory \SA{val} is obtained from the theory $\sa{val0}$ by removing the predicate~\hbox{$\cdot=0$}, suppressing Axiom \tsbf{vr0} and defining $x=0$ as an abbreviation of~$0\di x$. 
 \end{definition}

We have the same conclusion as in \thref{lemAxdiv}: in \clama the models of the theory \sa{val} are \advs.

Without referring to models in \clama, the \tdy \sa{val} 
proves some usual \prts of \cvds as valid \rdys or, sometimes, as admissible \rdys.

First the fact that a \adv is local, normal and \sdz correspond to the following valid rules. Proofs are without surprise directly translated from those of \thref{lemAxdiv}.

\Regles{ \lab{WZD} $\,\,0 = bx \Vd 0 = b \vou 0 = x$.}

\noindent Proof: consider Axiom \Tsbf{VR2} with $a=0=0x$.\qed

\smallskip In particular we have $\,\,0\di x^2\Vd0\di x$. 

\Regles{ \Lab{LOC} $\,\, 1\di x\vet 1\di y\vet x+y\di 1 \Vd x \di 1 \vou y \di 1$.}

\noindent Using \tsbf{VR1} we open two branches, the one where $x\di y$, the other where $y\di x$. In the first one, since $x\di x$, \tsbf{Vr2} gives $x\di x+y$. So, by \tsbf{Vr1}, $x\di 1$. Symmetrically in the second branch $y\di 1$.\qed

\smallskip The third rule is \agq.

\Regles{\Lab{Nor} $\,\, 1\di a_1\vet\dots 1\di a_n \vet 
y\big(\sum_{k=0}^{n}a_kx^ky^{n-k}\big)\di x^{n+1} \Vd y \di x$.}

\noindent We open branches $x\di y$ and $y\di x$. We have to examine the first one; e.g., with $n=2$. Since $x\di y$, $1\di a_0$, $1\di a_1$ and $1\di a_2$, we get 
\[x^2\di y^2\di a_0y^2,\quad 
x^2\di xy \di a_1xy\quad \hbox{ and }\quad x^2\di a_2x^2.
\]
So, $x^2\di a_2x^2+a_1xy+a_0y^2$ by~\tsbf{Vr2}, and $yx^2\di x^3$. From $yx^2\di xx^2$ we deduce $\Vd 0\di x^2\vou y\di x\,$ by~\tsbf{VR2}. And in the branch where $0\di x^2$, we have $y\di 0\di x$. \qed

\smallskip Now an admissible rule.

\Regles {\Lab{DIV} $\,\,x\di y\Vd \Exists a\,(1\di a\vet ax=y)$}

\begin{lemma} \label{lemDivadmissible}
If a \rdij is valid in the \sad $\gB=\big((G,R),\sa{val}\big)$ when using~\tsbf{DIV}, it is also valid in~$\gB$.
\end{lemma}

The proof of this lemma seems rather difficult. We can catch it after having proved a formal \vst: see Remark \ref{remthVstformelval}. 

\Subsubsection{\SADs of type \sa{val}}

\begin{definition} \label{defivalkK}~
\begin{enumerate}
\item Let $\gA$ be a commutative ring, we define the \sad $\sa{val}(\gA)$ as usual by taking the \pn given by the positive diagram of $\gA$.
\item Let $\gk\subseteq \gA$ be two rings,\footnote{We use $\gk$ for the first ring as an intuition given by the frequent context where $\gk$ is a \cdi.
} or more \gnlt let $\varphi:\gk\to\gA$ be an \alg. We note $\sa{val}(\gA,\gk)$ the \sad whose \pn is given by 
\begin{itemize}\itemsep=.1em
\item The positive diagram of~$\gA$ as commutative ring.
\item Axioms
$\vd 1 \di \varphi(x)$ for \elts $x$ of $\gk$.
%
\end{itemize}

%
%
\end{enumerate}
\end{definition}

The  two \sads $\sa{val}(\gA)$ and $\sa{val}(\gA,\gZ)$, where $\gZ$ is the least subring of~$\gA$ are canonically isomorphic.

\smallskip As a particular case of Rule \tsbf{WZD}, if $e$ is an \idm of~$\gA$, we get $\Vd e=0 \vou e=1$ in $\sa{val}(\gA)$. 
In particular, $\Vd 1 \di e$ and $\Vd 1\di f$ (where $f=1-e$). We deduce for any $a,b\in\gA$: $ae\di be\vet af\di bf\Vd a\di b$.
Indeed $a\di ae \di be$ and $a\di af \di bf$, so $a\di be+bf=b$. Let's summarize the situation.

\begin{lemma} \label{lemidmval}
If $e$ and $f\in\gA$ are two \cop \idms we have in $\sa{val}(\gA)$

\DeuxRegles
{\labu $\,\,a\di b\vd (ae\di be\vet af\di bf)$}
{\labu $\,\,ae\di be\vet af\di bf\vd a\di b $} 
\end{lemma}

\smallskip Now a consequence of the validity of \Tsbf{Nor} in \sa{val}. 

\begin{lemma} \label{lem-y-di-x}
Let $\gk\subseteq \gA$ be two rings and $x\in\gA$, $y\in\gk$. Then if $x$ is integral\footnote{The \elt $x$ is said to be \textsl{integral over the \id} $\fa$ of $\gk$ if it is a zero of a \pol $x^{n+1}+\sum_{i=0}^n a_i x^{n-i}$ with $a_i\in\fa^i$ for each $i$ \cite[\Dfn XII-2.1]{CACM}.} over the \id $\gen{y}$ of $\gk$, the \sad $\sa{val}(\gA,\gk)$ proves~\hbox{$y\di x$}.
\end{lemma}
We shall see in \thref{th-y-di-x} that this sufficient condition is also \ncr. 


\subsection{Valuative lattice and spectrum of a commutative ring}\label{subsecTrSpecval}

\Subsubsection{Several possible spectral topologies}

Let $\Gamma=\Gamma(\gV)$, then the point $(\fp,\gV)$ of $\Spev(\gA)$ is characterized by the associated valuation 
\[
w:\gA\to \Gamma_\infty, \;x \mapsto v(\pi_\fp(x))
\]
where $\pi_\fp:\gA\to\gA/\fp$ is the canonical surjection.

We use the notation $\SpevA$ for the spectrum with the topology which seems most natural to us, where \oqcs generating the topology are the following $\fO(a,b)$'s:
\[
\fO(a,b):=\sotq{w}{w(b)\leq w(a)},\quad a,b\in\gA
\]

In \cite{HK1994} the topology is generated by the following~$\fU(a,b)$'s:
\[
\fU(a,b):=\sotq{w}{w(b)\leq w(a),\,w(b)\neq \infty},\quad a,b\in\gA
\]
They note $\Spv(\gA)$ this \sps. 

\smallskip 
Thus, we get (with $\ov Y$ denoting the \cop set of $Y$) the following \egts. 
\begin{enumerate}
\item [$\bullet$] $\fU(b,b)=\sotq{w}{w(b)\neq \infty}=\ov{\fO(b,0)}$
\item [$\bullet$] $\fU(a,b)=\fO(a,b)\cap\,\ov{\fO(b,0)}$
\item [$\bullet$] $\ov{\fU(b,b)}=\sotq{w}{w(b)=\infty}=\fO(b,0)$ 
\item [$\bullet$] $\fO(a,b)=\fU(a,b)\cup\big(\ov{\fU(b,b)}\cap\,\ov{\fU(a,a)}\big)$ 
\end{enumerate}

\noindent Hence $\SpevA$ and
 $\SpvA$ define the same patch topology.

\smallskip The topology of $\SpvA$ considers as essential the \prt \gui{$b\di a\hbox{ and }b\neq 0$.}
This \dfn seems unnatural. In the same paper \cite{HK1994}, another spectral topology $\Spv'(\gA)$ is introduced with basic open sets $\fF(a,b):=\sotq{v}{v(a)<v(b)}=\ov{\fO(a,b)}$. This is the opposite \sps of $\SpevA$.

\smallskip 
We define now \trdis corresponding to these \spss.

\Subsubsection{The lattice $\val(\gA)$ and its spectrum $\Spev\gA$}

\begin{definition} \label{defivalA}
Let $\gA$ be a commutative ring. We consider the set (underlying)~$\gA\times \gA$ and we define on it the \entrel $\vdash_{\gA,\mathrm{val}}$ by the following \eqvc.

\vspace{-.8em}
\begin{equation} \label {eqvalA}
\begin{aligned} 
 (a_1,b_1),\dots,(a_n,b_n) &\,\vdash_{\gA,\mathrm{val}}\, (c_1,d_1),\dots,(c_m,d_m) 
 \quad \equidef  \\[.2em] 
 a_1\di b_1\vet\dots\vet a_n\di b_n &\Vdi{\sA{val}(\gA)} c_1\di d_1\vou \dots\vou c_m\di d_m 
 \end{aligned}
\end{equation}

\noindent The lattice $\val(\gA)$ is defined as generated by the \entrel $\vdash_{\gA,\mathrm{val}}$.

\noindent We note $\rDi_\gA:\gA\times \gA\to \val(\gA)$ (or simply $\rDi$) the corresponding map.
 \end{definition}

For example, with~\hbox{$a,b,c\in\gA$} 
\[\rDi(a,b)\vi\rDi(a,c)\leq \rDi(a,b+c),\quad \hbox{ in } \val(\gA).
\]
 
In fact, since $\cdot\di \cdot$ is the unique predicate of the \tdy \sa{val}, the lattice $\val(\gA)$ is the absolute Zariski lattice of the \sad $\sa{val}(\gA)$. 

Since the theory \sa{val} satisfies the rule \Tsbf{WZD}, the natural morphism $\val(\gA)\to\val(\Ared)$ is an isomorphism. 

\smallskip \thref{lemAxdiv} says that the \sad~\hbox{$\sa{val}(\gA)$} has in \clama minimal models given by the points of the valuative spectrum~\hbox{$\Spev(\gA)$}. 
After the explanations in paragraph \textsl{Spectrum and models in \clama} \paref{specetmodeles}, and examining the \dfn of basic open sets $\fO(a,b)$ generating the spectral topology, we get the following \tho in \clama.

\begin{theoremc} \label{thSpevSpecval}
The \sps $\Spev\gA$ is canonically homeomorphic to $\Spec(\val(\gA))$.
\end{theoremc}

In \coma, the \textsl{pointfree topology} given by the \trdi~$\val(\gA)$ is generated by formal opens: the \elts~$\rDi(a,b)$ of $\val(\gA)$. 

In \clama, where \spss have enough points, the topology
of $\Spec(\val(\gA))$ is a usual one, On the other hand we have defined (on the same underlying set) a topology on $\Spev \gA$
generated by open sets $\fO(a,b):=\sotq{w}{w(b)\leq w(a)}$. Clearly they correspond to formal $\rDi(b,a)$'s: this explains the homeomorphism between $\SpevA$
and $\Spec(\val(\gA))$.

\Subsubsection{The lattice $\vals(\gA)$ and the spectrum $\Spv\gA$}

Spaces $\Spev\gA$ and $\Spv\gA$ are different, but they have the same points given by suitable pairs~$(\fp,\gV)$.
 
We consider the \twdij \sa{val1} we get from $\sa{val}$ by adding the predicate~\hbox{$\cdot\neq 0$} opposite to $\cdot=0$. The theory \sa{val1} is a conservative extension of $\sa{val}$. Let us consider now the predicate \fbox{$J(a,b)\eqdef (a\di b \,\vii\, a\neq 0)$}.

We say that the \sps $\Spv\gA$ is homeomorphic to the spectrum of the following \trdi $\vals(\gA)$.

\begin{definition} \label{defivathe A}
The \trdi $\vals(\gA)$ is generated by the 
\entrel~$\vdash_{\gA,\mathrm{val}\sta}$ on $\gA\times \gA$ defined by the \eqvc
\vspace{-.8em}
\begin{equation} \label {eqvalp}
\begin{aligned} 
(a_1,b_1),\dots,(a_n,b_n) &\,\vdash_{\gA,\mathrm{val}\sta}\, (c_1,d_1),\dots,(c_m,d_m) 
 \quad \equidef  \\[.2em]
 J(a_1,b_1)\vet \dots\vet J(a_n,b_n) &\Vdi{\sA{val1}(\gA)} J(c_1,d_1)\vou \dots\vou J(c_m,d_m)
 \end{aligned}
\end{equation}

\noindent We note $\rJ_\gA:\gA\times \gA\to \vals(\gA)$ (or simply $\rJ$) the corresponding map.
 \end{definition}

Thus, the \elt $\rJ(b,a)$ of $\vals(\gA)$ corresponds to the open set $\fU(a,b)$ of 
$\Spv\gA$.

\subsection{Valuative lattice and spectrum of an \alg}\label{subsecTrSpecval2}

We consider in this section a \klg $\gA$, \cad a morphism $\varphi:\gk\to\gA$ of commutative rings. 

\Subsubsection{The lattice $\val(\gA,\gk)$ and the spectrum $\Spev(\gA,\gk)$}

We define the \trdi $\val(\gA,\gk)$
by using the \sad $\sa{val}(\gA,\gk)$ (see \dfn \ref{defivalkK}) in the same way as $\val(\gA)$
is defined by using the \sad $\sa{val}(\gA)$.

\begin{definition} \label{defivalkK0}
Let $\gk$ be a subring of a ring $\gA$, 
or more \gnlt let us consider a \klg $\varphi:\gk\to\gA$.
We define the \trdi $\val(\gA,\gk)$ as generated by the \entrel 
 $\vdash_{\gk,\gA,\mathrm{val}}$ on the set $\gA	\times \gA	$, which is defined by the following \eqvc. 
\vspace{-.8em}
\begin{equation} \label {eqvalkK}
\begin{aligned} 
 (a_1,b_1),\dots,(a_n,b_n) &\,\vdash_{\gk,\gA,\mathrm{val}}\, (c_1,d_1),\dots,(c_m,d_m) 
 \quad \equidef  \\[.2em] 
 a_1\di b_1\vet \dots\vet a_n\di b_n &\Vdi{\sA{val}(\gA,\gk)} c_1\di d_1\vou \dots\vou c_m\di d_m 
 \end{aligned}
\end{equation}

\noindent We note $\rDi_{\gk,\gA}:\gA	\times \gA \to \val(\gA,\gk)$ (or simply $\rDi$) the corresponding map.

\end{definition}

\begin{remark} \label{remvalKANOR} 
Since the theory \sa{val} satisfies the rule \Tsbf{Nor}, if 
$\ov\gk$ is the \cli of (the image of) $\gk$ in~$\gA\Red$, the \sad $\sa{val}(\gA,\gk)$ proves $1\di x$ for \ $x\in\ov\gk$. Thus, the natural morphism $\val(\gA,\gk)\to\val(\gA\Red,\ov\gk)$ is an isomorphism. \eoe 
\end{remark}

We note
 $\Spev(\gA,\gk):=\Spec(\val(\gA,\gk))$. Points of this spectrum correspond
to points $(\fp,\gV)$ of $\Spev(\gA)$ such that $\gV$ contains (the image of) $\gk$. Since the lattice $\val(\gA,\gk)$ is a quotient of $\val(\gA)$,
 $\Spev(\gA,\gk)$ is a sub\sps of $\Spev\gA$.

\smallskip
When $\gK$ is a \cdi transcendent over a subfield $\gk$, the \sps $\Spev(\gK,\gk)$ is often called the \textsl{Zariski-Riemann spectrum} of $(\gK,\gk)$, or also, the \textsl{abstract Riemann surface} of $(\gK,\gk)$. The points of this spectrum are the \arvs of $\gK$ containing $\gk$. 

In \coma we are mainly interested in the lattice $\val(\gK,\gk)$. The abstract Riemann surface is more than a \sps, a Grothendieck scheme. It has a good \cov description and the most important object is this scheme.

\smallskip Another special case is given when $\gA$ is the fraction field of a domain $\gk$, for example when~$\gk$ is the ring of all \agq integers.

\begin{lemma} \label{lem-valkk-valkK}
Let $\gk$ be a domain with fraction field $\gK$. The natural morphism $\val(\gk,\gk)\to\val(\gK,\gk)$ is an \iso. 
\end{lemma}
%
\begin{proof}
There are more atomic formulae in $\sa{val}(\gK,\gk)$ than in 
$\sa{val}(\gk,\gk)$. But if $x\di y$ is an atomic formula of 
$\sa{val}(\gK,\gk)$ where $x=\frac u v$ and $y=\frac s t$ with 
$u,v,s,t\in\gk$ and $v,t\neq 0$, formulae $x\di y$ and $ut\di sv$ are provably equivalent in $\sa{val}(\gK,\gk)$.

\noindent Thus it suffices to see that 
(for $a_i,b_i,c_j,d_j\in\gk$)
\[
a_1\di b_1\vet \dots\vet a_n\di b_n \Vdi{\sA{val}(\gK,\gk)} c_1\di d_1\vou \dots\vou c_m\di d_m
\]
\ssi 
\[
a_1\di b_1\vet \dots\vet a_n\di b_n \Vdi{\sA{val}(\gk,\gk)} c_1\di d_1\vou \dots\vou c_m\di d_m
\]
Indeed, one may always avoid fractions in a \demo by using the art of getting rid of denominators. 
\end{proof}
%

\Subsubsection{The center map (1)}\label{subseccomparvalZar0}

\begin{proposition}[The center map for \trdis $\val(\gA,\gA) $ and 
$\ZarA$] \label{propdefCentre} 
There exists a unique morphism of \trdi $\gamma:\ZarA\to\val(\gA,\gA)$
such that $\gamma(\DA(a))=\Di_\gA(a,1)$ for all $a\in\gA$. 
\end{proposition}
%
\begin{proof}
Using Corollary \ref{corZarA}, it is sufficient to see that the following rules are valid in $\sa{val}(\gA,\gA)$.

\DeuxRegles
{
\labu $\,\,0\di1\Vd \Bot$
\labu $\,\,ab\di1\Vd a\di1$
\labu $\,\,a+b\di1\Vd a\di1\vou b\di1$
}
{
\labu $\Vd 1\di1$
\labu $\,\,a\di1\vet\, b\di1\Vd ab\di1$
}

\noindent And for this it is sufficient to see that the following rules are
 valid in $\sa{val}(\gA)$. 

\DeuxRegles
{
\labu $\,\,0\di1\Vd \Bot$
\labu $\,\,1\di b\vet\, ab\di1\Vd a\di1$
\labu $\,\,1\di a\vet\, 1\di b\vet\, a+b\di1\Vd a\di1\vou b\di1$
}
{
\labu $\Vd 1\di1$
\labu $\,\, a\di1\vet\, b\di1\Vd ab\di1$
}

\noindent And this is easy.
\end{proof}

In \clama, the dual viewpoint is given by the \textsl{center spectral map} 
$\;\Spev(\gamma):\Spev(\gA,\gA)\to \Spec\gA\,$: the image of the point $(\fp,\gV)$ of $\Spev(\gA,\gA)$, 
is the \idep $\varphi^{-1}(\Rad(\gV))$, where $\varphi:\gA\to\Frac(\gA/\fp)$ is the canonical map (note that $\varphi(\gA)\subseteq \gV$ because we have $1\di a$ when $a\in\gA$).

\begin{remark} \label{rempropdefCentre} 
\thref{thVstformelval4} proves that the center map is injective.
A more direct \demo should be interesting. \eoe 
\end{remark}

\begin{theorem} \label{cor2thVstformelval4}
If $\gk$ is an integral arithmetic ring (a \ddp), the morphism $\gamma:\Zar\gk \to \val(\gk,\gk)$
 in \ref{propdefCentre} is an isomorphism of \trdis. 
\end{theorem}
%
\begin{proof} Let us note $\gK$ the fraction field of $\gk$. Using Lemma \ref{lem-valkk-valkK}, we replace $\val(\gk,\gk)$ with $\val(\gK,\gk)$ in the \demo.

\noindent We assume now that $\gamma$ is injective.

\noindent 
Let us prove that $\gamma$ is onto. 
By \dfn, in an \textsl{\anar}, for any couple $(a,b)$ we have $s,t,u,v$ satisfying 
\[
sa=ub, \;tb=va \,\hbox{ and }\,s+t=1.
\]
Let us see the consequences of the hypothesis $a\di b$ in $\sa{val}(\gK,\gk)$. \\
We have $ub=sa$ and $sa\di sb$ (because $a\di b$), thus $ub\di sb$; and by \Tsbf{VR2}: 
$a\di b\Vd u\di s \vou b=0$. 
\\
Since $s+t=1$, the rule \tsbf{LOC} gives $\Vd t\di 1 \vou s\di 1$.
\\
If $s\di1$, we have $\Vd u\di 1\vou b=0$. 
If $b\neq 0$, the branch $b=0$ dies because $b$ is invertible in $\gK$. We get in $\sa{val}(\gK,\gk)$: $\Vd u\di 1$.

\noindent Thus in the \sad $\sa{val}(\gK,\gk)$ we have the valid following rules.

\DeuxRegles
{
\labu $\,\,a\di b\vd t\di1\vou u\di1$ \quad if $b\neq 0$
\labu $\,\,t\di1 \vd a\di b$
}
{
\labu $\,\,u\di1\vd a\di b$
}

\smallskip \noindent So in the lattice $\val(\gK,\gk)$ we have 
\[\Di(a,b)=\Di(t,1)\vu \Di(u,1) \quad \hbox{ if } b\neq0.
\] 
As $\Di(a,0)=1$ this proves that $\gamma:\Zar\gk \to \val(\gK,\gk)$ is onto.
\end{proof}
%

\subsection{The theory \sa{Val} and the lattice $\Val(\gK,\gk)$}

The valuative lattice $\Val(\gK,\gk)$ is defined in the paper \cite[Space of Valuations]{Coq2009}. This lattice is very similar to $\val(\gK,\gk)$.
But its \dfn is implicitly based on a  \twdij \sa{Val}, which is distinct from
\sa{val}.

The theory \sa{Val} is an extension of the theory \sa{Cr}: one adds a predicate~$\Vr(x)$ and axioms that are satisfied in a \arv of a \cdi. The predicate $x\di y$ is not used, and axioms for $\Vr$ are minimalist.\footnote{Names \tsbf{vfi} and \tsbf{VF2} come from the theory of \cvds (see Section \ref{secdival6}).} 
The axioms we give are easily proved to be \eqv to those of \cite{CP2001,Coq2009}.

\DeuxRegles{
\lab{vf1} $\,\,x=0\vet \Vr(y) \Vd \Vr(x+y) $
\lab{vf3} $\,\,\Vr(x)\vet\Vr(y)\Vd \Vr(xy) $
\lab{VF2} $\,\,xy =1\Vd \Vr(x)\vou\Vr(y) $
}
{
\lab{vf2} $\vd \Vr(-1) $
\lab{vf4} $\,\, \Vr(x)\vet\Vr(y)\Vd \Vr(x+y) $
\lab{CL$_{\Val}$} $\,\, 0\di 1 \vd \Bot$ \quad (collapsus)

}

This theory is noted \sa{Val}. First an easy lemma.

\begin{lemma} \label{lemvalVal}
In the theory \sa{val}, when reading $\Vr(x)$ as an abbreviation of $1\di x$, axioms of \sa{Val} are valid rules. 
\end{lemma}
%
\begin{proof}
The only point to be proved is the validity \Tsbf{VF2}. We prove the following \rdij, which is a priori stronger.

\Regles {\labu $\,\,\Vr(xy)\Vd \Vr(x)\vou\Vr(y) $}

\noindent As a special case of \tsbf{VR1} we get

\Regles {\labu $\vd \Vr(x)\vou x\di 1 $} 

\noindent If $x\di 1$, then $x\di 1\di xy$, thus $1.x\di y.x$ and using \tsbf{VR3}, $1\di y \vou x\di 0$; in the last case $1\di xy\di 0$, collapsus!
\end{proof}

Notations in the article \cite{Coq2009} are slightly different of ours. 
 In \cite{Coq2009},~$\gk$ is always a domain contained in a field $\gK$. In this context our~$\val(\gK,\gk)$ is the same as $\Val(\gK,\gk)$ in \cite{Coq2009}. The notation $\Val(\gk)$ is an abbreviation of~
 $\Val(\Frac(\gk),\gk)$. This corresponds to our $\val(\Frac(\gk),\gk)$, isomorphic to~$\sa{val}(\gk,\gk)$. If we note $\gZ$ the minimal subring of $\gk$, our $\val(\gk)\simeq \val(\gk,\gZ)$ corresponds to $\Val(\Frac(\gk),\gZ)$, 

By comparison with the theory \sa{val}, the only difficulty with the theory \sa{Val} is it applies only for integral domains $\gk$. 

\smallskip We prove now that 
$\val(\gK,\gk)$ and $\Val(\gK,\gk)$ are isomorphic \trdis in the context of the paper \cite{Coq2009}. First we recall the \dfn of the lattice 
$\Val(\gK,\gk)$.

\begin{definition} \label{defiValkK} Let $\gK$ be a \cdi and $\gk$ a subring of $\gK$.
\begin{enumerate}
\item 
The \sad $\sa{Val}(\gK,\gk)$ is obtained by adding to the axioms of~\sa{Val} the positive diagram of $\gK$ as a commutative ring and the rules~\hbox{$\Vd \Vr(x)$} for $x\in\gk$.

\item 
The lattice $\Val(\gK,\gk)$ defined in \cite{Coq2009} is by \dfn the \trdi generated by the \entrel $\vdash_{\gK,\gk,\mathrm{Val}}$ on $\gK$ defined by the following \eqvc.

\vspace{-.8em}
\begin{equation} \label {eqZarclassVal}
\begin{aligned} 
 a_1,\dots,a_n &\,\vdash_{\gK,\gk,\mathrm{Val}} c_1,\dots,c_m 
 \qquad\quad  \equidef  \\[.2em] 
\Vr(a_1)\vet \dots\vet \Vr(a_n) & \Vdi{\sA{Val}(\gK,\gk)} \Vr(c_1)\vou \dots\vou \Vr(c_m) 
 \end{aligned}
\end{equation}

\noindent We note $\rV_{\gK,\gk}:\gK\to \Val(\gK,\gk)$ (or simply $\rV$) the corresponding map.
\end{enumerate}
 \end{definition}

In the paper \cite{Coq2009}, the \entrel is taken on $\gK\etl$ rather than on~$\gK$.
This does not change the generated \trdi since $\Vr(0)$ may be replaced with $\Vr(1)$ (or $\Top$). 

So the \dfn of $\Val(\gK,\gk)$ is similar to $\val(\gK,\gk)$
(\dfns \ref{defivalkK} and \ref{defivalkK0}), and Lemma~\ref{lemvalVal} gives a natural morphism
\begin{equation} \label {eqtheta}
\theta:\Val(\gK,\gk)\to \val(\gK,\gk)
\end{equation}

We have to prove that this is an \iso when $\gK$ is a \cdi. A preliminary little job is \ncr.

\Subsubsection{The theory $\sa{Val}^+$}

In order to prove that $\theta$ is an \iso, we first extend the theory \sa{Val} by introducing the predicate $y\di z$ with the following axioms (they prove that this predicate is \eqv to $\exists x\,(\Vr(x)\vii z=xy)$).

\DeuxRegles{
\Lab{Div} $\,\,\Vr(x)\vet z=xy \Vd y\di z $
}
{
\lab{DIV} $\,\,y\di z\Vd \Exists x\;(\Vr(x)\vet z=xy) $
}

One sees easily that $\Vr(x)$ is \eqv to $1\di x$, and that $y\di z$ is \eqv to $\exists x\,(\Vr(x)\vii z=xy)$.

We note $\sa{Val}^+$ this new \tdy. It is an \esid extension of \sa{Val}
(see Lemma \ref{lemdefiesid}). In particular, \dij valid rules do not change, and the lattice defined from $\sa{Val}^+(\gK,\gk)$ is the same as the one defined from $\sa{Val}(\gK,\gk)$.

Note now that if we have a valid rule 
\begin{equation} \label {eqlem49}
a_1\di b_1\vet \dots\vet a_n\di b_n \Vdi{\sA{val}(\gK,\gk)} c_1\di d_1\vou \dots\vou c_m\di d_m\qquad (a_i, b_i, c_j, d_j\in\gK),
\end{equation}

\noindent a dynamical proof of this rule can use only terms in $\gK$. Indeed, \sa{val} is a \twdij and it should be necessary to have an existential axiom in order that new \elts appear as fresh variables.

\begin{lemma} \label{lemvalVal2} Let $\gK$ be a \cdi and $\gk$ a subring.
\begin{enumerate}
\item For \sads $\sa{Val}^+(\gK,\gk)$ and $\sa{val}(\gK,\gk)$ axioms of \sa{val} that are used for the \dfn of the lattice $\val(\gK,\gk)$, \cad for the \demo of a rule \pref{eqlem49}, are valid in $\sa{Val}^+$.
\item Consequently if a rule \pref{eqlem49} is valid in $\sa{val}(\gK,\gk)$ it is also valid in $\sa{Val}^+(\gK,\gk)$.
\item Consequently the morphism $\theta$ given in \pref{eqtheta} is injective.
\end{enumerate}
 
\end{lemma}
%
\begin{proof} \textsl{1}.
For example let us see the rule \Tsbf{Vr2}. Assume $a\di b$, one introduces a fresh variable~$z$ such that $az=b$ and $\Vr(z)$. So, $acz=bc$, and using \tsbf{Div}, we get $a\di b$.
Other axioms have also direct \dems, except for disjunctive axioms \tsbf{VR1} and \tsbf{VR2}. In this case we shall use the fact that in a \demo of a rule \pref{eqlem49}, \elts are always in $\gK$.

\DeuxRegles
{
\lab{VR1} $\vd a \di b \vou b\di a$
}
{
\lab{VR2} $\,\,ax \di bx \Vd a \di b \vou 0 \di x$
}

\noindent Let us see \tsbf{VR1}. In a \demo, $a$ and $b$ are \elts of $\gK$. If $a=0$, then $b\di a$ in $\sa{Val}^+(\gK,\gk)$ because~\hbox{$a=0.b$} and $\Vr(0)$. In a similar way, if $b=0$, then $a\di b$ in $\sa{Val}^+(\gK,\gk)$.
When $a$ and~$b$ are nonzero, one has in $\gK$ an $x=ba^{-1}$ and a $y=ab^{-1}$, and $xy=1$. So, in $\sa{Val}^+(\gK,\gk)$, we have $\vd \Vr(x)\vou \Vr(y)$. If $\Vr(x)$ is valid, then $1\di x$ and $a\di ax=b$ by \tsbf{Vr2}. 
Similarly if $\Vr(y)$, then $b\di a$.
\\
Let us see \tsbf{VR2}. Here $a,b,x$ are \elts of $\gK$. 
If $x=0$ then $0\di x$. If $x\neq 0$, we have an inverse~$x^{-1}$ in $\gK$, and so $a=axx^{-1} \di bxx^{-1}=b$ (we use \tsbf{Vr2}). 
\end{proof}

\begin{remark} \label{remlemvalVal2} 
The \demo of Lemma \ref{lemvalVal2} uses the fact that $\gK$ is a \cdi. It seems impossible to obtain an \iso
$\val(\gK,\gk)\simeq \Val(\gK,\gk)$ without assuming $\gK$ to be \zedr.
Extending further $\sa{Val}^+$ so as to prove all the axioms of $\sa{val}$ would require adding  additional axioms, but the latter would look too ad hoc.\eoe
\end{remark}

\Subsubsection{Isomorphism of lattices $\val(\gK,\gk)$ and $\Val(\gK,\gk)$}

In order to prove that the morphism $\theta$ given in \pref{eqtheta} is an \iso, it is now sufficient to prove the following lemma.

\begin{lemma} \label{lemvalValfin} Let $\gK$ be a \cdi and $\gk$ a subring. In $\sa{Val}^+(\gK,\gk)$, any rule
\[
a_1\di b_1\vet \dots\vet a_n\di b_n \Vd c_1\di d_1\vou \dots\vou c_m\di d_m\qquad (a_i, b_i, c_j, d_j\in\gK)
\]
is always equivalent to a rule
\[
\Vr(x_1)\vet \dots\vet \Vr(x_k) \vd \Vr(y_1)\vou \dots\vou \Vr(y_\ell) \qquad (x_i,y_j\in\gK\etl)
\]
\end{lemma}
%
\begin{proof}
If $b_i=0$ we cancel $a_i\di b_i$ in the hypothesis (it is true).
If we have a $c_j=0$ with $d_j\neq 0$, we cancel $c_j\di d_j$ in the conclusion (this branch collapses). 
If there is an $a_i=0$ with $b_i\neq 0$, or if there is a\ $d_j=0$, the rule is valid.
It remains to be seen what happens when all elements are nonzero. In this case, we replace $a_i\di b_i$ and $c_j\di d_j$, respectively, with $\Vr(b_ia_i^{-1})$ and $\Vr(d_jc_j^{-1})$.
\end{proof}

We have proved the desired result.
\begin{theorem} \label{thValval}
Let $\gk$ a subring of a \cdi $\gK$. The natural morphism $\Val(\gK,\gk)\to \val(\gK,\gk)$
is an \iso.
\end{theorem}

The following corollary applies to \thref{thValval} the \lgb \elr machinery\ \num1 that will be explained in Section 
\ref{sec-A-qi}. 
\begin{corollary} \label{corthValval}
Let $\gk$ be a \qiri and  $\gK=\Frac(\gk)$. The natural morphism $\Val(\gK,\gk)\to \val(\gK,\gk)$
is an \iso.
\end{corollary}

Note that the previous result does not work for an arbitrary ring $\gk$.


\section{Valuative dimensions}\label{subsecdival}

In order to transfer results obtained in the case of integral domains to that of an  arbitrary ring~$\gA$, one could hope for defining an idealistic \gui{integral domain generated by $\gA$.} Unfortunately, such a generated integral domain does not exists in general as a usual object. A dynamical variant of this strategy would consists in defining a \tdy \sa{Ai} of integral domains, and in considering the \sad $\sa{Ai}(\gA)$ as a reasonable replacement for the desired ideal object.

In this section, devoted to the topic of valuative dimension, we propose instead to make use of the ring $\Amin$, the \gui{minimal pp-closure of $\gA$.} The latter was proposed by T.~Coquand in an unpublished note about a substitute for the GCD algorithm in $\AX$ for the case of a commutative ring $\gA$. Using $\Amin$ has indeed proved to be efficient for studying the valuative dimension in \cite{CACM}.

For this purpose, we give in section~\ref{sec-A-qi} a detailed constructive account of the classical theory of pp-rings, before constructing the ring $\Amin$ in section~\ref{sec-Amin} and listing a few of its properties. We think that the \zed ring $\Frac(\Amin)$ is the best possible replacement for the \cdf of an integral domain. In particular, the natural morphism $\gA\to \Amin$ is an \iso when 
$\gA$ is an integral domain.

In section~\ref{sec3vdim}, we use the ring  $\Amin$ in the comparison of three possible constructive variants for the valuative dimension of a ring $\gA$ and of an algebra \hbox{$\gk\to\gA$}.

\subsection{pp-rings}\label{sec-A-qi}
 
{\bf \Dfns and notations.} ~
\begin{itemize}
\item When the ring $\gA$ is clear form the context, we use the abbreviated form $a\epr:=\Ann_\gA(a)$.
We note also  $\fa\epr$ the annihilator of the \id~$\fa$.
\item  An \elt $a$ such that $a\epr=0$ is called a \textsl{nonzerodivisor} or is said to be \textsl{\ndz}.
\item  A ring morphism $\gA\to\gB$ is said to be \textsl{\ndz} if it sends any \ndz \elt on a \ndz \elt.
\item  We note $\Ared=\gA/\!\sqrt[\gA]{\gen{0}}$ the reduced ring generated by $\gA$.
\item  We note $\cP_n$ the set of finite subsets of $\so{1,\dots,n}$.
\end{itemize}

\begin{propdef} \label{lem-idm-associe-0} Let $\gB$ be a commutative ring and $a\in\gB$.
\begin{enumerate}
\item 
If there is an \elt $c$ such that $ca=a$ and $c\epr=a\epr$, then $c$ is an \idm. This \elt is \ncrt unique and we say it is an \emph{attached \idm to~$a$ in $\gB$}; we note this $c$ as $e_a$ or $a\eci$.
\item An \idm $e\in \gB$ is an \idm attached to $a$ \ssi $ea=a$
 and $a+(1-e)$ is \ndz.
\item If $\varphi:\gB\to\gC$ is \ndz and if $a\in\gB$ has an attached \idm $e_a$ in $\gB$, then $\varphi(e_a)$ is an attached \idm to $\varphi(a)$ in $\gC$. 
\item If $a$ is an \idm, it is an attached \idm to itself.
\item The \elt $a$ is \ndz \ssi $1$ is an attached \idm to $a$.
\item 
If $a\eci$ and $b\eci$ are attached \idms to $a$ and $b$, then $a\eci b\eci$ is an attached \idm to $ab$.
\end{enumerate}
If each \elt of a ring $\gA$ has an attached \idm we say that $\gA$ is a \emph{\qiri} (\idps are projective).
Otherwise said, the annihilator of any element is generated by an \idm.
\end{propdef}
%
\facile

In a \qiri, for $a\in\gA$, let $e_a$ be the unique attached \idm to $a$. We have $\gA\simeq \gA[1/e_a]\times \aqo{\gA}{e_a}$.
In the ring $\gA[1/e_a]$, $a$ is \ndz, and in $\aqo{\gA}{e_a}$, $a$ is null.
We then have $e_{ab}=e_a e_b$, $e_aa=a$ and $e_0=0$.

Conversely, suppose that a commutative ring is equipped with a unary law $a\mapsto \ci{a}$ which satisfies the following three axioms
\begin{equation}\label{eqaqis}
\ci{a}\,a=a,\quad
\ci{(ab)}=\ci{a}\,\ci{b},\quad
\ci{0}=0.
\end{equation}
Then, for all $a\in\gA$, $\ci{a}$ is an attached \idm to $a$, thus the ring is a \qiri.

\begin{lemma}[Splitting lemma for \qiris] \label{thScindageQi}
Let $x_1$, \dots, $x_n$ be $n$ \elts in a \qiri~$\gA$.
There exists a \sfio $(e_j)$ of cardinality~$2^n$ such that in each of the components~\hbox{$\gA[1/e_j]$}, each~$x_i$ is null or \ndz.
\end{lemma}

Knowing how to systematically split a \qiri into two components leads to the following \gnl method.
The essential difference with the previous splitting lemma is that we  do not know a priori the finite family of \elts which will provoke the splitting.

\rdb
\mni {\bf Local-global \elr machinery \num1.} \label{MethodeQI} \cite[Section IV-6]{CACM}
{\sl Most \algos that work with nontrivial integral \ris can be modified in order to work with pp-rings, by splitting the ring into two components each time that the \algo written for the integral \ris uses the \gui{is this \elt null or \ndz?} test. In the first component the \elt in question is null, in the second it is \ndz.}

\medskip We state now a lemma analogous to Lemma XI-4.21 in \cite{CACM}, where we replace \gui{quasi-inverse of $a$} with \gui{attached \idm to $a$.}

\begin{lemma}[The ring generated by an attached \idm]\label{lem-idm-associe-1}
Let $\gA$ be a reduced ring.
\begin{enumerate}
\item Let $a\in\gA\subseteq\gC$. Assume that $\gC$ is reduced and $a$ has an attached \idm~$a\eci$ in~$\gC$. We note $\gB=\gA[\ci{a}]\subseteq \gC$. 
As an \Amo, we get $\gB=a\eci\gB\oplus(1-a\eci)\gB$, and as a ring 
\[\gB\simeq \aqo{\gB}{1-a\eci}\times \aqo{\gB}{a\eci}=
\gA_1\times \gA_2
\]
with $\gA_1 \simeq a\eci\gB$ and $\gA_2 \simeq (1-a\eci)\gB$ as \Amos.
\begin{enumerate}
\item The natural \homo $\mu_1:\gA\to\gA_1$ 
(via $\gA\to\gB\to\gA_1$) is onto. Its kernel is~\hbox{$a\epr:=\Ann_\gA(a)$}.
\item
The natural \homo $\mu_2:\gA\to\gA_2$ (via $\gA\to\gB\to\gA_2$) is onto.
Its kernel is the intersection \hbox{$\fa=\gA\cap a\eci\gB$}
and satisfies the double inclusion
\[ 
\big(a\epr\big)\epr=\Ann_\gA\big(\Ann_\gA(a)\big)\supseteq \fa \supseteq \DA(a)\eqno(*)
\] 
\end{enumerate}
In short $\gA[a\eci]=\gB\simeq\gA\sur a\epr\times \gA\sur{\fa}$.
\item Conversely for $a\in A$, if an \id $\fa$ of $\gA$ satisfies inclusions $(*)$, the \elt $(1,0)$ is an attached \idm to
 (the image of) $a$ in the ring 
 \[\gB:=\gA\sur a\epr\times \gA\sur{\fa}=:\gA_1\times \gA_2\] and the canonical \homo 
from $\gA$ to $\gB$ is injective.
\end{enumerate}
\end{lemma}
\begin{proof} Let us note
$\pi_i:\gB\to\gA_i$ the canonical morphisms. Since $\pi_i(a\eci)= 0 \hbox{ or }1$, any \elt~$\pi_i(x)$ is equal to a $\mu_i(y)$ for an $y\in\gA$, so each $\mu_i$ is onto.

\smallskip \noindent \textsl{1a. } The kernel of $\mu_1$ is $\gA\cap\Ker \pi_1= \gA\cap\Ann_\gB(a)=\Ann_\gA(a)$.

\smallskip \noindent \textsl{1b.} The kernel of $\mu_2$ is $\fa:=\gA\cap\Ker \pi_2= \gA\cap a\eci\gB$. Clearly $a\in\fa$, thus $\DA(a)\subseteq \fa$. 
Finally we prove that
$\fa\,\Ann_\gA(a) = 0$, which implies
  $\fa\subseteq\Ann_\gA\big(\Ann_\gA(a)\big)$. Indeed if $x\in\fa$ and \hbox{$y\in \Ann_\gA(a)$}, we have $x=a\eci z$ for a $z\in\gB$ and $y\in \Ann_\gB(a)=\Ann_\gB(a\eci)$, thus $xy=za\eci y=0$.

\smallskip \noindent \textsl{2.} The image of $a$ in $\gB$ is $(a1_{\gA_1},0_{\gA_2})$, thus $(1_{\gA_1},0_{\gA_2})$
is an attached \idm to~$a$ in~$\gB$. Consider now an $x\in\gA$ whose image in $\gB$
is $0$.
On one hand $x=_{\gA_1}0$, so $ax=_\gA0$. On the other hand $x\,\Ann_\gA(a)=0$, so $x^2=_\gA0$, and $x=_\gA0$.
\end{proof}

On sees that the notation $\gA[a\eci]$ is 
ambiguous when $\DA(a)\neq\Ann_\gA\big(\Ann_\gA(a)\big)$.
This leads in particular to two natural notions of pp-closure of a reduced ring, depending on whether one decides to systematically favor the side $\DA(a)$ or the side $\Ann_\gA\big(\Ann_\gA(a)$.

Here is now a lemma which can simplify the construction of \qiris. It is analogous to Lemma XI-4.22 in \cite{CACM}, by replacing \gui{\zedr} with \gui{\qiri} and \gui{quasi-inverse} with \gui{attached \idm.}

\begin{lemma}\label{lem2qi} ~
\vspace{-.5em}
\begin{enumerate}\itemsep=1pt
\item When $a\in\gA\subseteq\gC$ where $\gC$ is a \qiri, we note $a\eci$ the attached \idm to $a$ in $\gC$. The smallest sub-\qiri of~$\gC$ containing~$\gA$ is equal to $\gA[(a\eci)_{a\in\gA}]$. 
\item More \gnlt, if $\gA\subseteq\gB$
where $\gB$ is reduced and if each \elt $a$ of $\gA$ has an attached \idm $a\eci$ in $\gB$,
then the subring $\gA[(a\eci)_{a\in\gA}]$ of~$\gB$ is a \qiri.
Moreover, each \elt of~$\gA[(a\eci)_{a\in\gA}]$ can be written \emph{in a standard form}, \cad as 
$ \som_j a_j e_j $ with $a_j$'s in $\gA$ and a list~$(e_j)_j$ of orthogonal \idms 
in the \agB generated by $a\eci$'s.\\
Finally, $\som_j{ a_j e_j}$ has an attached \idm in $\gB$ written in standard form: $\som_j 1 f_j=\som_j a_j\eci e_j$.
\end{enumerate}
\end{lemma}
\begin{proof}
We prove Item \textsl{2}, which implies clearly Item \textsl{1}.\\
Among \elts of $\gB$, let us consider the ones written as sums of products~$ab\eci$ \hbox{with $a$, $b\in\gA$}. Clearly they form a subring of $\gB$\footnote{Use Item \textsl{4} in Lemma \ref{lem-idm-associe-0}.}
which is equal to $\gA[(a\eci)_{a\in\gA}]$.
This ring contains the \agB generated by \idms $a\eci$. So if an \elt is written in standard form, it is \hbox{in $\gA[(a\eci)_{a\in\gA}]$}. \\
Let $x=\som_j{ a_j e_j}$ and $y= \som_k{ b_k f_k}$ be two \elts written in standard form. \\
Then $xy=\som_{j,k}(a_jb_k)e_jf_k $, which is standard. 
Let us see the sum $x+y$. One may assume $\som_je_j=1$ by adding the \idm $e=(1-\sum_je_j)$ to the list. Same thing for $\som_kf_k$. 
Then $x+y=\som_{j,k}(a_j+b_k)e_jf_k$, which is standard. \\
Finally, in a finite product of rings $\prod_\ell\gA_\ell$, an \elt $x=(x_\ell)_\ell$ has an attached \idm \ssi each $x_\ell$ has an attached \idm ${x_\ell}\eci$ in $\gA_\ell$, and in this case $x\eci=({x_\ell}\eci)_\ell$. This provides the last statement and proves that $\gA[(a\eci)_{a\in\gA}]$ is a \qiri.
\end{proof}

\subsection{The ring $\Amin$}\label{sec-Amin}

Référence: \cite[section~\hbox{XIII-7}]{CACM}.

\begin{lemma}\label{lem0Amin}
Let $\gA$ be a reduced ring
and $a\in\gA$.
We define
\[\gA_{\so{a}}\eqdefi\gA\sur{a\epr}\times \gA\sur{({a\epr})\epr}
\]
and we note $\psi_a:\gA\to\gA_{\so{a}}$ the canonical \homo. 
\begin{enumerate}
\item $\psi_a(a)=(\pi_1(a),0)$ and $(1,0)$
is an attached \idm to $\psi_a(a)$ in $\gA_{\so{a}}$.
\item $\psi_a$ is injective (we may consider $\gA$ as a subring of $\gA_{\so{a}}$).
\item $\gA_{\so{a}}$ is a reduced ring.
\item If $a$ has an attached \idm, $\psi_a$ is an \iso.
For example this is the case when~$\gA$ is a \qiri.
\item Let $\fb$ be an \id in $\gA_{\so{a}}$, then the \id $\psi_a^{-1}(\fb\epr)=\fb\epr\cap\gA$ is an annihilator \id in~$\gA$.
\item The morphism $\psi_a:\gA\to\gA_{\so{a}}$ is \ndz.
\end{enumerate}
\end{lemma}
For Item \textsl{6}, see Exercice XIII-19 in \cite{CACM}.

\begin{lemma}\label{lem1Amin}
Let $\gA$ be a reduced ring
and $a,b\in\gA$. Then, with the notations of Lemma~\ref{lem0Amin},
the rings $(\gA_{\so{a}})_{\so{b}}$ and $(\gA_{\so{b}})_{\so{a}}$ are canonically isomorphic.
\end{lemma}

\rem The case where $ab=0$ is typical: when we meet it, we would like to split the ring into components where things are \gui{clear.}
The previous construction then gives the three components
\[\gA\sur{(ab\epr)\epr}, \; \gA\sur{(a\epr b)\epr}\, \hbox{ and } \,\gA\sur{(a\epr b\epr)\epr}.\] 
In the first one, $a$ is \ndz and $b=0$, in the second one $b$ is \ndz and $a=0$, and in the third one $a=b=0$.
\eoe

\begin{thdef}[Minimal pp-closure]\label{thAmin} 
Let $\gA$ be a reduced ring.
We can define a ring $\Amin$ as a filtering colimit by iterating the basic construction which consists in replacing~$\gE$ (the \gui{current} ring, which contains $\gA$) by
\[\gE_{\so{a}}\eqdefi\gE\sur{a\epr}\times \gE\sur{({a\epr})\epr}=\gE\sur{\Ann_\gE (a) }\times \gE\sur{\Ann_\gE(\Ann_\gE (a) )},
\]
when $a$ ranges over $\gA$.
\begin{enumerate}\itemsep=1pt
\item This ring $\Amin$ is a \qiri, contains $\gA$ and is integral over $\gA$.
\item For all $x\in\Amin$,
$
x\epr\cap\gA$ is an annihilator \id in $\gA$.
\end{enumerate}
This ring $\Amin$ is called the \emph{minimal pp-closure of $\gA$}.
If $\gA$ is a \qiri, the natural morphism $\gA\to \Amin$ is an \iso.
 \\
In the general case, we take $\gA\qim\eqdefi (\Ared)\qim$ (when we do not know whether $\gA$ is reduced or not).

\end{thdef}

We give now a description of each ring
we get at a finite stage of the construction of~$\Amin$.

\begin{lemma}\label{lem4MorRc}
Let $\gA$ be a reduced ring and $(\ua) = (\an)$ a list of~$n$ \elts 
in~$\gA$. For $I\in\cP_n$, we note $\fa_I$ the \id
\[\fa_I = \big(\Prod_{i\in I} \gen{a_i}\epr \Prod_{j\notin I} a_j\big)\epr
= \big(\gen{a_i, i \in I}\epr \Prod_{j\notin I} a_j\big)\epr
.\]
Then $\Amin$ contains the following ring, product of $2^n$
quotient rings of~$\gA$ (some may be trivial):
\[\gA_{\so\ua} = \Prod_{I\in\cP_n} \gA\sur{\fa_I}.\]
 \end{lemma}

\subsection{Three \cov versions of valuative dimensions}\label{sec3vdim}

The valuative dimension of a commutative ring has many possible definitions.

In \clama the valuative dimension of a domain $\gk$, noted as $\Vdim(\gk)$, is equal to the maximum height of the valuation group $\Gamma(\gV)$ of a valuation ring $\gV$ satisfying $\gk\subseteq \gV \subseteq \Frac(\gk)$. The height of $\Gamma(\gV)$ is also the Krull dimension of $\gV$. Another definition of $\Vdim(\gk)$ is the  maximum length of a chain of \arvs of $\Frac(\gk)$ containing $\gk$.

Since the valuative dimension of a quotient domain of $\gk$ is bounded by
the valuative dimension of $\gk$, one extends the \dfn to an arbitrary ring in the following way: 
the valuative dimension of a ring $\gA$ is the l.u.b.\ of valuative dimensions of its quotients by prime ideals \cite{Cah90}.

\smallskip 
In the book \cite[section XIII-8]{CACM}, in the case of a domain $\gk$, authors use a classical \carn of the valuative dimension which is based on the Krull dimension of rings: the valuative dimension of a domain is the l.u.b.\ of Krull dimensions of overrings of~$\gk$ (the rings containing $\gk$ and contained in its fraction field). We note it $\vdim(\gk)$.
This \dfn of~$\vdim(\gk)$ is easily extended to \qiris. Finally, for an arbitrary ring, the authors use
the \dfn $\vdim(\gA)=\vdim(\Amin)$. They prove the following \eqvc  for their \dfn, \hbox{for $n\geq 0$}:
\fbox{$\vdim(\gA)\leq n\Leftrightarrow \Kdim(\AXn)\leq 2n$}. This \carn is known in \clama. So the \dfn in \cite{CACM} is an acceptable \cov version.

\smallskip On the other hand, for a domain $\gk$, it is clear in \clama that the valuative dimension of $\gk$ is equal to the dimension of the \trdi $\val(\Frac\gk,\gk)$. So it is also the Krull dimension of $\Val(\Frac\gk,\gk)\simeq \val(\Frac\gk,\gk)$ (see \thref{thValval}). So this is $\Vdim(\gk)$ introduced in \cite[Space of Valuations]{Coq2009}.

Since $\val(\Frac\gk,\gk)\simeq \val(\gk,\gk)$ (Lemma \ref{lem-valkk-valkK}), it is reasonable to propose as a \cov \dfn for the valuative dimension of an arbitrary ring $\gA$, the Krull dimension of the lattice $\val(\gA,\gA)$. We note this dimension $\Vdim(\gA)$.\footnote{For the moment we are not sure that it coincides to the classical definition in \clama.}

\smallskip Finally, in \clama, \cite{KY2020} give a new \cov \carn of the valuative dimension of an arbitrary commutative ring. 
 
\smallskip So we have three \cov \dfns: the one given in \cite[section XIII-8]{CACM}, the Krull dimension of $\val(\gA,\gA)$, and the \carn given by Kemper and Yengui. We now need to prove \cot that these three \cov \dfns are equivalent.

This proof will be given in an article by Yengui, Neuwirth and Lombardi.
The paper proves \cot that the first and the third \dfn are equivalent, and that for the case of domains they are equivalent to the second one (given in \cite{Coq2009}).

In fact, we now need to prove the following lemma.

\begin{lemma} \label{lemdimvalAAmin}
For an arbitrary commutative ring, the \trdis $\val(\gA,\gA)$ and $\val(\Amin,\Amin)$ have the same Krull dimension. 
\end{lemma}
%
\begin{proof}
The ring $\Amin$ is a filtered colimit of rings we construct, starting with $\gA$, as described in \thref{thAmin}. If $\gE$ is obtained at a preceding stage and if $x\in\gE$, the new ring is \hbox{$\gE/x\epr\!\times\gE/(x\epr)\epr$}. This is $\gE\mapsto \gE/\fa_1\times\gE/\fa_2 $ with $\fa_1\fa_2=0$.
If we show that this kind of construction does not change the dimension of the valuative lattice, we are done. In fact, since dimensions of~$\gA$ and~$\Ared$ are the equal, we may assume our ring is reduced.
\end{proof}
%

\begin{lemma} \label{lemdimvalAA1A2}
Let $\gA$ be a reduced ring $\fa_1$, $\fa_2$ two \ids s.t.\ $\fa_1\fa_2=0$, $\gA_1=\gA/\fa_1$ and $\gA_2=\gA/\fa_2$.
Let us note $\vald(\gA):=\val(\gA,\gA)$.
Then we have 
\[\Kdim(\vald(\gA))=
\sup(\Kdim(\vald(\gA_1)),\Kdim(\vald(\gA_2)))=\Kdim(\vald(\gA_1\!\times \gA_2)).
\] 
\end{lemma}
%
\begin{proof}
\textsl{First equality.} From $\gA_i=\gA/\fa_i$ we deduce that $\vald(\gA_i)=\vald(\gA)/(\fb_i=1)$ for a suitable filter~$\fb_i$ of $\vald(\gA)$. \Prmt, from the \eqvc \pref{eqvalkK} of \Dfn~\ref{defivalkK0}, we get for example for $\gA_1$
the following \eqvcs:\footnote{Note the \sad $\sa{val}(\gA_1,\gA_1)$
is obtained by adding relations $0\di x$ when $x\in\fa_1$ to the \sad $\sa{val}(\gA,\gA)$).}
\begin{equation*} \label {eqvalAi}
\begin{aligned} 
 (a_1,b_1),\dots,(a_n,b_n) &\,\vdash_{\val(\gA_1,\gA_1)}\, (c_1,d_1),\dots,(c_m,d_m) 
 \quad \equidef  
\\[.2em] 
 a_1\di b_1\vet \dots\vet a_n\di b_n &\Vdi{\sA{val}(\gA_1,\gA_1)} c_1\di d_1\vou \dots\vou c_m\di d_m \quad \equidef \\[.2em] 
0\di x_1\vet \dots\vet 0\di x_r\vet a_1\di b_1\vet \dots\vet a_n\di b_n &\Vdi{\sA{val}(\gA,\gA)} c_1\di d_1\vou \dots\vou c_m\di d_m \equidef  \\[.2em]
 (a_1,b_1),\dots,(a_n,b_n) &\,\vdash_{\val(\gA,\gA)/(\fb_1=1)}\, (c_1,d_1),\dots,(c_m,d_m)
 \end{aligned}
\end{equation*} 
with some $x_j$'s in $\fa_1$ and where $\fb_1$ is the filter of $\val(\gA,\gA)$
generated by elements $(0,x)$ \hbox{with $x\in\fa_1$}. 
Then let us prove that $\fb_1\cap\fb_2=\so 1$. Indeed, the filter $\fb_1\cap\fb_2$ is generated by \elts $(0,x)\vu(0,y)$ with $x\in\fa_1$
and $y\in\fa_2$, and since $0\di xy$ we get by \tsbf{WZD}, the valid rule \hbox{$\Vdi{\sA{val}(\gA,\gA)}0\di x \vou 0\di y$}.
We conclude with \thref{lemrcfdimtrdi}.\\
\emph{Second equality.} First $\vald(\gA_1\!\times \gA_2)\simeq \vald(\gA_1)\times \vald(\gA_2)$. Next, the dimension of the product of two lattices is the sup of the two dimensions.
\end{proof}

\noindent \textsl{Note.} Since $\Amin$ is a \qiri, $\val(\Amin,\Amin)\simeq \val(\Frac (\Amin),\Amin)$. So, $\Vdim(\gA)$ is equal to $\Vdim(\Amin)$ with the meaning in the article \cite{Coq2009}.

\begin{definition} \label{defivaluabs}
The dimension of the \trdi 
$\val(\gA)\simeq\val(\gA,\gZ)$ 
(where~$\gZ$ is the minimal subring of $\gA$) 
is called the \textsl{absolute valuative dimension} of the ring~$\gA$ and is noted as $\vdima(\gA)$. 
\end{definition}

The dimension $\vdima(\gA):=\Kdim(\val(\gA))$ is a priori greater than or equal to $\vdim(\gA):=\Kdim(\val(\gA,\gA))$ and it is often strictly greater, as for the ring~$\QQ$: $\vdima(\QQ)=\Kdim(\val(\QQ,\ZZ))=\Kdim(\val(\ZZ,\ZZ))=1$, and $\vdim(\QQ)=\Kdim(\val(\QQ,\QQ))= 0$. \eoe

\begin{remark} \label{remdim AAmingeneral} 
Same \demos work for the dimension of \trdis $\Zar(\gA)$, $\val(\gA)$, $\Reel(\gA)$ and $\Heit(\gA)$.\footnote{See \cite{CLQ2006}.} So we have the following analogous results.
\[
\begin{array} {cccccc}
\Kdim(\gA)&=&\Kdim(\Amin)&=&\Kdim(\gA_1\!\times \gA_2)\\[.3em]
\vdima(\gA)&=&\vdima(\Amin)&=&\vdima(\gA_1\!\times \gA_2)\\[.3em]
\Rdim(\gA)&=&\Rdim(\Amin)&=&\Rdim(\gA_1\!\times \gA_2)\\[.3em]
\Jdim(\gA)&=&\Jdim(\Amin)&=&\Jdim(\gA_1\!\times \gA_2)
\end{array} \eqno \hbox{\eoe}
\] 
\end{remark}

%


\section{Comparisons with a theory of \cvds}\label{secdival6}

\subsection{Introduction}

In this section we prove that the \twdij \sa{val} and the \tdy \sa{Vdf}  of \cvds, which is   introduced in \cite[Section 4]{CLR01}, are more or less identical.
This is obtained through  the formal \vst \ref{CLlval1} for \cvds.
This implies various formal \vsts for \sa{val}.

We compare also with the formal \vst for the theory \sa{Val} given in \cite{Coq2009}.

Note that \thref{thValval} is already an essential comparison result between $\val$ and $\Val$.

\subsection{The theory $\sa{Vdf}$ of \cvds}\label{subsubseccorval}

Section 4 in \cite{CLR01} gives a formal Positivstellensatz for the theory \sa{Vdf} of \cvds and a Positivstellensatz à la Hilbert for \ac \cvds. We prefer here the name of \textsl{\vst}. We are mainly interested in formal \vsts, they give algebraic certificates for the provability  in diverse \tdys for some \rdys.

\smallskip First we recall axioms given in \cite{CLR01} for the theory $\SA{Vdf}$ of \cvds. The language is the one commutative rings to which we add predicates \hbox{$\cdot \neq 0$}, $\Vr(\cdot)$, $\Rn(\cdot)$, $\U(\cdot)$.
The first one, $\cdot \neq 0$, is seen as the invertibility predicate, the predicate 
$\Vr(\cdot)$ interprets belonging to  the valuation ring, the predicate $\Rn(\cdot)$
interprets \elts residually null and $\U(\cdot)$  the units (invertible  \elts in the valuation ring).
So we have the following signature
\[
\Sigma_{\sA{Vdf}}=(\,\cdot=0,\cdot\neq 0,\Vr(\cdot),\Rn(\cdot),\U(\cdot)\mathrel{;}\cdot+\cdot, \cdot\times\cdot,-\,\cdot,0,1\,)
\]
Axioms of the theory \sa{Vdf} are the following \rdys. As consequence, the field is discrete and the \dve relation is explicit.

\smallskip \noindent $\bullet$ First we have direct rules. 

\DeuxRegles{
\lab{cr1} $\vd 0=0 $
\lab{cr3} $\,\, x=0 \vet y=0 \Vd x+y=0 $
\Lab{vf1} $\,\,x = 0 \vet \Vr(y) \vd \Vr(x+y) $
\Lab{vf3} $\,\, \Vr(x) \vet \Vr(y) \vd \Vr(xy)$
\Lab{vf5} $\,\,x = 0 \vet \Rn(y) \vd \Rn(x+y)$
\Lab{vf7} $\,\,\Rn(x) \vet \Vr(y) \vd \Rn(xy)$
\Lab{vf9} $\,\,\Rn(x) \vd \Vr(x)$
\Lab{vf10} $\,\,x = 0 \vet \U(y) \vd \U(x+y)$
\Lab{vf12} $\,\,\U(x) \vet \U(y) \vd \U(xy)$
\Lab{vf14} $\,\,\U(x) \vd \Vr(x)$
\Lab{vf15} $\,\,x = 0 \vet y \not= 0 \vd x+y \not= 0$
\Lab{vf17} $\,\,\U(x) \vd x\not= 0$}
{
\lab{cr2} $\,\, x=0 \Vd xy=0 $
\lab{~} $\phantom{\,\, x=0 \vet y=0 \Vd x+y=0 }$
\Lab{vf2} $\vd \Vr(-1)$
\Lab{vf4} $\,\,\Vr(x) \vet \Vr(y) \vd \Vr(x+y)$
\Lab{vf6} $\vd \Rn(0)$
\Lab{vf8} $\,\,\Rn(x) \vet \Rn(y) \vd \Rn(x+y)$
\lab{~} $\phantom{\,\,x = 0 \vet \U(y) \vd \U(x+y)}$
\Lab{vf11} $ \vd \U(1)$
\Lab{vf13} $\,\,\Rn(x) \vet \U(y) \vd \U(x+y)$
\lab{~} $\phantom{\,\,\U(x) \vd \Vr(x)}$
\Lab{vf16} $\,\,x \not=0 \vet y\not= 0 \vd x y \not= 0$
}

\noindent $\bullet$ The collapsus is the same as for domains.

\DeuxRegles{
\Lab{CL} $\,\,0 \not= 0 \vd \Bot$
}
{\lab{~} $~$
}

\noindent $\bullet$ Finally, we have simplification  rules  \tsbf{Vf1} and \tsbf{Vf2} followed by 4 \rdys.
 
\DeuxRegles{
\Lab{Vf1} $\,\,x y=1 \vd x \not= 0$
\Lab{VF1} $\,\,x \not= 0 \vd \Exists y \,\,xy=1$
\Lab{VF3} $\,\,xy =1 \vd \Vr(x) \vou \Vr(y)$
}
{
\Lab{Vf2} $\,\,\Vr(xy)\vet \U(x) \vd \Vr(y)$
\Lab{VF2} $\vd x = 0 \vou x \not= 0$
\Lab{VF4} $\,\,\Vr(x) \vd \U(x) \vou \Rn(x)$
}

If we want to force the valuation to be nontrivial, we introduce a constant~$c$ with the axiom 

\Regles{\Lab{vf0}$\,\,\Vr(c)\vd \Bot $}

There is only one existential axiom, \tsbf{VF1}, and it corresponds to a unique existence: the theory is very near to a \twdij. 

For the \tdy \SA{Vdfsc} of \textsl{\cvdscs} we add to the theory~\sa{Vdf} axioms of \spb algebraic closure.

\Regles{
\Lab{VF6$_n$} $\,\, \disc_Y (f)\neq 0\vd \Exists y \; f(y)=0,\qquad\hbox{ where } f(Y)=Y^n + a_{n-1} Y^{n-1}+ \cdots+ a_1 Y + a_0.$
}


\smallskip For the \tdy \SA{Vdfac} of \textsl{\cvdacs} we add to the theory \sa{Vdf} axioms of \agq closure.

\Regles{
\Lab{VF5$_n$} $\vd \Exists y \; y^n + a_{n-1} y^{n-1}+ \cdots+ a_1 y + a_0 = 0$
}

\begin{lemma} \label{lemModelesCv}
Let $\gA$ be a commutative ring and $\sa{Vdf}(\gA)$ the \sad of type \sa{Vdf}
constructed on (the positive diagram of) $\gA$. To give a minimal model of $\sa{Vdf}(\gA)$ is the same thing as giving a couple $(\fp,\gV)$ where $\fp$ is a detachable \idep 
 of $\gA$ and $\gV$ is a \arv of $\gK=\Frac(\gA/\fp)$. Moreover divisibility has to be  decidable. 
\end{lemma}
%
\begin{proof}
The \idep $\fp$ is the set of $x\in\gA$ such that $x=0$ is valid in the model. The quotient ring $\gB=\gA/\fp$ is integral, we note $\gK$ its fraction field, present  in the model thanks to Axiom~\tsbf{VF1}. The \arv  $\gV$
is then the set of fractions $z$ in $\gK$ such that $\Vr(z)$ is valid in the model. 
The divisibility relation is decidable because $\gV$ is a detachable subset of $\gK$: if $x\neq 0$, we have an inverse $y$ by \tsbf{VF1}. Then we have $\Vr(x)$ or $\Vr(y)$. In the second case, we have 
$\U(y)$ or $\Rn(y)$, and $\Vr(x)$ \ssi $\U(y)$. 
\end{proof}

It seems that \tsbf{VF1} is \ncr at the end of the proof. It implies that the model contains only fractions of \elts in $\gV$.

So,  \sads $\sa{Vdf}(\gA)$ and $\sa{val1}(\gA)$ have the same minimal models in \clama. But not \ncrt  in \coma since in $\sa{val1}(\gA)$ the \dve relation is not forced to be decidable.


We now define some  theories closely related to the theory \sa{Vdf}, which share a number of \thos with the latter. 
We make use of the terminology of \reds, simplification rules and \rdys, as explained in~\paref{regledirecte}.

\begin{definitions}[Geometric theories closely related to \sa{Vdf}] \label{defiApvandco} \label{defiCvd-}~
\begin{enumerate}
\item Let us note \sa{Apv}\footnote{Proto-valued rings in \cite{CLR01}.} the direct theory we get from the theory \sa{Vdf} when keeping only direct axioms, from~\tsbf{cr1} 
to~\tsbf{vf17}, and the collapsus.
\item The \talg \sa{Aqv}\footnote{Quasi-valued rings in \cite{CLR01}.} is intermediate between  \sa{Apv} and  \sa{Vdf}: we replace in \sa{Vdf}  Axioms \tsbf{VF1} to \tsbf{VF4} with the following simplification axioms, which are valid rules in \sa{Vdf}:
 
\DeuxRegles
{
 \Lab{Vf3} $\,\,\U(xy)\vet \Vr(x)\vet \Vr(y) \vd \U(y)$ 
 \Lab{Vf5} $\,\,\Rn(x^2) \vd \Rn(x)$ 
 \Lab{Vf7} $\,\,xy = 0\vet x \not= 0 \vd y= 0 $
}
{
 \Lab{Vf4} $\,\,\Rn(xy)\vet \U(x) \vd \Rn(y)$ 
 \Lab{Vf6} $\,\,xy \not= 0 \vd x\not=0$ 
 \Lab{Vf8} $\,\,x^2 = 0 \vd x = 0$ 
}

\vspace{-1.5em}
\Regles {\Lab{Vf9$_n$} $\,\,x^{n} =\sum_{k=0}^{n-1} a_k x^k \vet \Vr(a_{n-1})\vet \dots\vet \Vr(a_0) \vd \Vr(x)$}

The last axiom scheme says that $\gV$ is integrally closed.

\item The \twdij $\sa{Vdf}^-$ is obtained from $\sa{Aqv}$ by adding the three disjunctive axioms \tsbf{VF2}, \tsbf{VF3} and \tsbf{VF4}.

\item 
The theory $\sa{Vdf}^+$ is the extension of $\sa{Vdf}$ we get when we add
\begin{itemize}\itemsep=0em
\item The predicate $x \di y$ as abbreviation of \gui{$\exists z \,(\Vr(z) \vii xz=y$)} 
\item The predicate $x \ndi y$ as abbreviation of \gui{$y\neq 0\vii \exists z \,(\Rn(z)\vii yz=x)$} 
\end{itemize}
\end{enumerate}
\end{definitions}

We shall see that the theories \sa{Apv} and \sa{Vdf} collapse simultaneously (\ref{lemColsimApv}) and that the theories~\sa{Aqv} and~\sa{Vdf} prove the same \ralgs (\ref{thAqv}). Theories~$\sa{Vdf}^-$ and~\sa{Vdf} prove the same \rdijs, but we do not prove this result here.

The theory $\sa{Vdf}^+$  is by construction an \esid extension of $\sa{Vdf}$ (a fortiori a conservative extension).

\begin{definition} \label{defiCvdkA}
Let  $\gk\subseteq \gA$ be two rings, or more \gnlt let $\varphi:\gk\to\gA$ be an \alg, we note $\sa{Vdf}(\gA,\gk)$ the \sad whose \pn is given by 
\begin{itemize}\itemsep=0em
\item The positive diagram of~$\gA$ as commutative ring.
\item Axioms
$\vd \Vr(\varphi(x))$ for  $x$'s in $\gk$.
\end{itemize}
The \sad $\sa{Vdf}(\gA)$ is identical to $\sa{Vdf}(\gA,\gZ)$ where $\gZ$
is the minimal subring  of $\gA$.
 \end{definition}

We define in a similar way \sads $\sa{Apv}(\gA,\gk)$, $\sa{Aqv}(\gA,\gk)$, $\sa{Vdf}^-(\gA,\gk)$ and $\sa{Vdf}^+\!(\gA,\gk)$.

\subsection{Formal \vst for \sa{Vdf} and consequences}

\begin{valsatz}[Formal \vst for the theory \sa{Apv}] \label{CLlval1} ~\\
Let $GR=(G;\Rzero,\Rnz,\Rvr,\Rrn,\Ru)$ be a \pn over the signature $\Sigma_{\sA {Vdf}}$\footnote{More \prmt, $\Rzero$ is the subset of $\ZZ[G]$
whose \elts $p$ are assumed to be $=0$ in the \sad. Same thing for $\Rnz,\Rvr,\Rrn,\Ru$.} for a \sad. 
We note
\begin{itemize}\itemsep=0pt
\item $\Izero$ the \id of $\ZZ[G]$ generated by $\Rzero$ 
\item $\Mnz$ the \mo generated by $\Rnz$ 
\item $\Vvr$ the subring of $\ZZ[G]$ generated by $\Rvr \cup \Rrn \cup \Ru$ 
\item $\Irn$ the \id of $\Vvr$ generated by $\Rrn$ 
\item $\Mu$ the \mo generated by $\Ru$
\end{itemize}
\Propeq
\begin{enumerate}\itemsep=0pt
\item The \sad $(GR,\sa{Apv})$ collapses.
\item There is in $\ZZ[G]$ an \egt 
\[s (u+t) + z = 0\]
 with $s\in \Mnz$,
$u\in \Mu$, $t\in \Irn$ and $z\in \Izero$.
\end{enumerate}
\end{valsatz}

The following \tho is a crucial result in \cite{CLR01}.
It gives the condition of collapsus for \sads of type \sa{Vdfac}. 

\begin{theorem}[Simultaneous collapsus] \label{lemColsimApv} \emph{\cite[Theorem 4.3]{CLR01}} 
A \sad of type \sa{Apv} collapses \ssi it collapses as a \sad of type \sa{Vdf}, or of type \sa{Vdfac}, or of any other theory between \sa{Apv} and \sa{Vdfac}.
\end{theorem}

\begin{remark} \label{remlemColsimApv} 
If $(\gK,\gV)$ is a \cvd and if $\gL$ is a field extension of $\gK$
we deduce in \clama that the \adv $\gV$
can be extended in a \adv~$\gW$ of~$\gL$ such that $\gW\cap \gK=\gV$
\cite[Remark 4.6]{CLR01}.
\eoe 
\end{remark}

\Tho 4.18 in \cite{CLR01} is a \vst à la Hilbert for \ac  \cvd. It is obtained from \thref{lemColsimApv} by using the fact that the formal theory  which corresponds to \sa{Vdfac} is complete
(when the \cara of  $\gK$ and the one of the residual field are fixed). 

\smallskip The following \tho describes provable facts in  \sads of type \sa{Vdf}.

\begin{theorem} \label{prop-factvalf} \emph{\cite[Proposition 4.14]{CLR01}}
\\
Let $GR=(G;\Rzero,\Rnz,\Rvr,\Rrn,\Ru)$ be a \pn over the signature~$\Sigma_{\sA{Vdf}}$. Let $p$ be an \elt of $\Zg$.
We define $\Izero$, $\Mnz$, $\Vvr$, $\Irn$ and $\Mu$ as in \thref{CLlval1}. We consider proofs in the \sad $(GR,\sa{Vdf})$.
\begin{enumerate}\itemsep=0pt
\item [a)] A dynamical proof of $p = 0$ 
gives an \egt in $\Zg$ of type 
\[ p^n m(u+j) + i = 0
\]
with $m \in \Mnz$, $u \in \Mu$, $j \in \Irn$ and $i \in \Izero$.
\item [b)] A dynamical proof of $p \not= 0$
gives an \egt in $\Zg$ of type 
\[m (u+j) + i + bp = 0\]
with $m \in \Mnz$, $u \in \Mu$, $j \in \Irn$, $i \in \Izero$ and $b \in \Zg $.
\item [c)] A dynamical proof of $\Vr(p)$ gives an \egt in $\Zg$ of type  
\[m ( (u+j)p^{n+1} + a_n p^n +\cdots + a_1 p + a_0 ) + i = 0\] 
with $m \in \Mnz$, $u \in \Mu$,
$j \in \Irn$,  $a_k$’s $\in \Vvr$ and $i\in \Izero$.
\item [d)] A dynamical proof of $\Rn(p)$ gives an \egt in $\Zg$ of type 
\[ m ( (u+j)p^{n+1} + j_n p^n +\cdots + j_1 p + j_0 ) + i = 0\]
with $m \in \Mnz$,
$u \in\Mu$, $j$
and $j_k$’s $\in \Irn$ and $i \in \Izero$.
\item [e)] A dynamical proof of $\U(p)$ 
gives an \egt in $\Zg$ of type 
\[m ( (u+j) p^{n+1} + a_np^n +\cdots + a_1 p + (u' + j') ) + i = 0\]  
with $m \in \Mnz$,
$u,u'\in \Mu$, $j,j' \in \Irn$,  $a_k$’s $\in \Vvr$ and $i \in \Izero$.
\end{enumerate}
\end{theorem}

\smallskip 
\begin{corollary} \label{cor-prop-factvalf} 
Let $\gk\subseteq \gA$ be two rings, $x\in\gA$ and $\yn\in\gk$.
\begin{enumerate}
\item The rule $\vd x= 0$ is valid in $\sa{Vdf}(\gA,\gk)$ \ssi $x$ is nilpotent.
\item The rule $\vd x\neq 0$ is valid in $\sa{Vdf}(\gA,\gk)$ \ssi $x$
is invertible in~$\gA$.
\item The rule $\vd \Vr(x)$ is valid in $\sa{Vdf}(\gA,\gk)$ \ssi $x$
is integral over $\gk$.
\item The rule $\Rn(y_1)\vet \dots\vet \Rn(y_n) \vd \Rn(x)$ is valid in $\sa{Vdf}(\gA,\gk)$
\ssi~$x$ is weakly integral\footnote{This means that $x$ is a zero of a monic \pol whose nondominant \coes are in the \id.} over the \id $\gen{\yn}$ of $\gk$. 
\end{enumerate}
 
\end{corollary}

\begin{remark} \label{rem-cor-prop-factvalf} 
From Item \textsl{3} we deduce easily  in \clama that a domain~$\gk$ has for integral closure in its fraction field $\gK$ the intersection of \arvs of $\gK$ containing~$\gk$ \cite[Remark~4.15]{CLR01}.
\\
From Item \textsl{4}, we deduce  in \clama that if $(\gk,\fm)$ is a local domain, the \id~$\fm$ is the intersection of $\gk$ with the \idema of a \arv of the fraction field $\gK$ \cite[Corollary~4.7]{CLR01}. \eoe
 
\end{remark}

A remarkable consequence of \thref{prop-factvalf} is \thref{thAqv}.
\begin{theorem} \label{thAqv} \cite[Theorem 4.17]{CLR01}.
Theories \sa{Aqv} and \sa{Vdfac} prove the same \ralgs.
Same thing  for any intermediate theory.
\end{theorem}


\subsection{Formal \vsts for $\sa{val}$ and $\sa{Vdf}^+$}

In this section we prove a crucial result: a formal \vst for the theory~\sa{val},
a consequence of Theorems \ref{CLlval1} and \ref{lemColsimApv}.

We consider a ring $\gA$, we are searching for  a formal \vst for the \sad $\sa{val}(\gA)$.

Recall that the theory $\sa{Vdf}^+$ is defined at Item \textsl{4} of \Dfn\ref{defiApvandco} and that the theory \sa{val1} is the conservative extension of \sa{val} we get when we add the predicate $\cdot \neq 0$ opposite to~\hbox{$\cdot = 0$}.

\begin{lemma} \label{lemCv+val}~
\begin{enumerate}
\item The predicate $x\ndi y$ is the opposite predicate to $x\di y$ in the theory
$\sa{Vdf}^+$.
\item Axioms of \sa{val} are valid in $\sa{Vdf}^+$. 
\item The following \ralg is valid in \sa{val1}.
\vspace{-.2em}

\Regles{
\Lab{VR2.5} $\,\,x\neq 0\vet ax \di bx \Vd a \di b$
} 
\end{enumerate}
\end{lemma}
\facile

\Subsubsection{A formal \vst for $\sa{val}(\gA)$}

\begin{valsatz}[A formal \vst for $\sa{val}(\gA)$ 
 and $\sa{Vdf}^+\!(\gA)$\footnote{Note that if $\gA=\ZZ$ and if $a_i,b_i,c_j$ and $d_j$ are \idtrs we get a  \gnl formal \vst for the theory $\sa{Vdf}^+$ or for the theory \sa{val}.
}]\label{thVstformelval}~\\ 
Let $\gA$ be a commutative ring and $a_i$, $b_i$, $c_j$, $d_j\in\gA$. \Propeq 
\begin{enumerate}\itemsep=0pt
\item ~ 

\vspace{-2.4em}
\begin{equation} \label {eqthVstformelval1}
 (a_1,b_1),\dots,(a_n,b_n) \,\vdash_{\val(\gA)} \, (c_1,d_1),\dots,(c_m,d_m) 
\end{equation}
\item ~ 

\vspace{-2.4em} 
\begin{equation} \label {eqthVstformelval2}
a_1\di b_1\vet \dots\vet a_n\di b_n \Vdi{\sA{val1}(\gA)} c_1\di d_1\vou \dots\vou c_m\di d_m
\end{equation}
\item ~ 

\vspace{-2.4em} 
\begin{equation} \label {eqthVstformelval3}
a_1\di b_1\vet \dots\vet a_n\di b_n \Vdi{\sA{Vdf}^+\!(\gA)} c_1\di d_1\vou \dots\vou c_m\di d_m
\end{equation}
\item 
 Introducing \idtrs $X_i$ ($i\in\lrbn$) and $Y_j$ ($j\in\lrbm$)
we have in the ring $\gA[\uX,\uY]$ an \egt in the following form 
\begin{equation} \label {eqthVstformelval4}
d \,\big(1+\som_{j=1}^mY_jP_j(\uX,\uY)\big)\equiv 0 \mod \gen{(X_ia_i-b_i)_{i\in\lrbn},(Y_jd_j-c_j)_{j\in\lrbm}} 
\end{equation}
 where $d$ is in the \mo generated by  $d_j$'s, and  $P_j(\uX,\uY)$'s are in $\ZZ[\uX,\uY]$.
\item 
 Let us note $y_j=\frac{c_j}{d_j}$ vied in $\gB=\gA[\frac{1}{d_1\cdots d_m}]$. Let us introduce \idtrs $X_i$ ($i\in\lrbn$). One has in the ring $\gB[\uX]$ an \egt in the following form 
\begin{equation} \label {eqthVstformelval5}
1+\som_{j=1}^my_jP_j(\uX,\uy)\equiv 0 \mod \gen{(X_ia_i-b_i)_{i\in\lrbn}}
\end{equation}
 where $P_j(\uX,\uY)$'s are in $\ZZ[\uX,\uY]$.
\end{enumerate}
\end{valsatz}

%
\begin{proof} Item \textsl{5} is a simple rewriting of Item \textsl{4}.

\smallskip \noindent \textsl{1} $\Leftrightarrow$ \textsl{2}. By \dfn, Item \textsl{1} means that
the following rule is valid: 
\begin{equation} \label {eqthVstformelval6}
a_1\di b_1\vet \dots\vet a_n\di b_n \Vdi{\sA{val}(\gA)} c_1\di d_1\vou \dots\vou c_m\di d_m
\end{equation}
But \sa{val1} is a conservative extension of \sa{val}.

\smallskip \noindent \textsl{1} $\Rightarrow$ \textsl{3}. 
By Item \textsl{2} of Lemma \ref{lemCv+val}, the rule \pref{eqthVstformelval3} is valid.

\smallskip \noindent \textsl{3} $\Rightarrow$ \textsl{4}. 
The rule \pref{eqthVstformelval3} is equivalent to the following collapsus
\begin{equation} \label {eqthVstformelval3+}
a_1\di b_1\vet \dots\vet a_n\di b_n\vet c_1\ndi d_1\vet \dots\vet c_m\ndi d_m \Vdi{\sA{Vdf}^+\!(\gA)} \Bot 
\end{equation}
Inside the theory $\sa{Vdf}^+$, $a_i\di b_i$
is \eqv to  the existence of an $x_i$ such that $\Vr(x_i)$ and $x_ia_i=b_i$.
This~$x_i$ can be represented by a fresh variable $X_i$. Similarly $ c_j\ndi d_j$
is \eqv to the existence of an~$y_j$ such that $\Rn(y_j)$, $y_j d_j=c_j$ and $d_j\neq 0$. 
This~$y_j$ can be represented by a fresh variable $Y_j$.
The validity of collapsus \pref{eqthVstformelval3+} is then \eqv to the collapsus of the family of conditions 
\[
(\Vr(X_i),\,X_ia_i-b_i=0)_{i\in\lrbn},\; (\Rn(Y_j),\, Y_jd_j-c_j=0,\, d_j\neq 0)_{j\in\lrbm}
\]
in the \sad constructed on the positive diagram of $\gA$ by adding \gtrs $X_i$ and $Y_j$, \cad in the ring $\gA[\uX,\uY]$. We conclude with \vst~\ref{CLlval1} (with $s=u=1$ and $z=0$).

\smallskip \noindent \textsl{4} $\Rightarrow$ \textsl{2}. 
Let us assume an identity \pref{eqthVstformelval4}. We have to prove the validity of the rule~\pref{eqthVstformelval2}. We make a case by case reasoning using $\vd x=0\vou x\neq 0$ and $\vd a\di b\vou b\di a$.
\\
In a branch  where one $d_j$ is supposed null 
the rule \pref{eqthVstformelval2} is valid. So we can assume the $d_j$'s are nonzero, hence regular (rule \Tsbf{VR2.5}). 
\\
In a branch where $a_i=0$, the hypothesis of \pref{eqthVstformelval2}
implies $b_i=0$ and we can replace~\hbox{$a_i$, $b_i$} and~$X_i$ with $0$ 
in~\pref{eqthVstformelval4}. 
\\
Finally, we need to prove \pref{eqthVstformelval2} in the case where $a_i$'s and  $d_j$'s are nonzero, hence \ndz. Let us consider now $\gB$, the total fraction ring of $\gA$. We get in $\gB$ an \egt 
\begin{equation} \label {eqthVstformelval60}
0 =_\gB 1+\som_j \frac{c_j}{d_j} P_j\left(\frac{b_1}{a_1},\dots,\frac{b_n}{a_n},
\frac{c_1}{d_1},\dots,\frac{c_m}{d_m}\right)
\end{equation}
where \pols $P_j$'s have \coes in $\ZZ$. Let $\delta$ be a bound of degrees in  $Y_k$'s in  $P_j$'s 
and $\epsilon_i$ a bound of degrees in $X_i$'s. We let $d:=\prod_k d_k$, we multiply the preceding \egt by $u=d^\delta\prod_ia_i^{\epsilon_i}$ and we get
\begin{equation} \label {eqthVstformelval7}
u =_\gB \som_j \frac{c_j}{d_j} Q_j(b_1,a_1,\dots,b_n,a_n,
c_1,d_1,\dots,c_m,d_m)
\end{equation}
where each $Q_j$ is homogeneous of degree $\delta$ in each $(c_k,d_k)$ and of degree $\epsilon_i$ in $(a_i,b_i)$.
We let $e_j=\prod_{k:k\neq j} d_k$ and we multiply the preceding \egt by $d$. We get an \egt in $\gA$
\[
ud=\som_j c_j e_j q_j .
\] 
We now proceed by case analysis, and for the desired conclusion, we still need to deal with the case where we suppose $d_j\di c_j$ for each $j$. 
By case analysis again, we assume that one $c_je_j$, e.g.,\ $c_1e_1$, divides all the other ones. Then we get
\begin{itemize}\itemsep=.1em
\item $u\di q_j$ for each $q_j$ ($u$ divides each \coe of $q_j$'s) 
\item $uc_1 e_1\di c_j e_jq_j$ for each $q_j$  
\item so $c_1 e_1u\di\som_j c_j e_j q_j=ud=d_1e_1u$  
\end{itemize}

\noindent Finally, the rule \tsbf{VR2.5} allows us to simplify by $e_1u$: we get $c_1 \di d_1$.\\
Phew! 
\end{proof}
%

\begin{remark} \label{remvst1=>4} 
Perhaps a more direct proof of  the implication \textsl{1} $\Rightarrow$ \textsl{4} is possible,
without using the formal \vst \ref{CLlval1} for \cvds. On one hand one should prove that \prt \pref{eqthVstformelval4} defines an \entrel over $\gA\times \gA$ (the cut rule seems difficult), on the other hand that this \entrel satisfies
axioms of \sa{val}. 
For example the rule \Tsbf{Vr2}, $\,\,a \di b \vet a \di c \Vd a \di b + c$
corresponds to the following \egt 
\[
(b+c)(1-y(x_1+x_2))\equiv 0 \,\mod\gen{x_1a-b,\,x_2a-c,\,y(b+c)-a}.\eqno\hbox{\eoe}
\]
\end{remark}

\Subsubsection{Admissibility of the rule \tsbf{DIV} for the \twdij \sa{val}}

\begin{remark} \label{remthVstformelval} 
A corollary of \vst~\ref{thVstformelval} is  Lemma \ref{lemDivadmissible} concerning the admissibility of the existential rule   \Tsbf{DIV} in the theory \sa{val}. 
Indeed, the formal \vst is established for the theory $\sa{Vdf}^+$ (see Item  \textsl{4} in Definition \ref{defiApvandco}), where we introduce $\cdot\di\cdot$ with its \dfn. So, in this theory the rule~\tsbf{DIV} is valid. But the formal \vst works for the theory~\sa{val}. Thus,  \rdijs that are valid for the predicate $x\di y$ remain the same for the theory \sa{val} when we add the axiom \tsbf{DIV}. \eoe 
\end{remark}

\Subsubsection{A formal \vst for $\sa{val}(\gA,\gk)$}

A slight variant of \vst~\ref{thVstformelval}.

\begin{valsatz}[A formal \vst for $\sa{val}(\gA,\gk)$ and $\sa{Vdf}^+\!(\gA,\gk)$]\label{thVstformelvalbis} ~\\
Let $\gk\subseteq \gA$ be two commutative rings and $a_i$, $b_i$, $c_j$, $d_j\in\gA$. \Propeq 
\begin{enumerate}
\item One has 
\begin{equation} \label {eqthVstformelvalbis1}
 (a_1,b_1),\dots,(a_n,b_n) \vdash_{\val(\gA,\gk)} (c_1,d_1),\dots,(c_m,d_m) 
\end{equation}
%
\item One has 
\begin{equation} \label {eqthVstformelvalbis1.5}
a_1\di b_1\vet \dots\vet a_n\di b_n \Vdi{\sA{val}(\gA,\gk)} c_1\di d_1\vou \dots\vou c_m\di d_m
\end{equation}
\item One has 
\begin{equation} \label {eqthVstformelvalbis2}
a_1\di b_1\vet \dots\vet a_n\di b_n \Vdi{\sA{Vdf}^+\!(\gA,\gk)} c_1\di d_1\vou \dots\vou c_m\di d_m
\end{equation}
\item Let us note $y_j=\frac{c_j}{d_j}$ viewed in $\gB=\gA[\frac{1}{d_1\cdots d_m}]$. Consider \idtrs $X_i$ \hbox{($i\in\lrbn$)}. One has in the ring $\gB[\uX]$ an \egt in the following form 
\begin{equation} \label {eqthVstformelVabisl5}
1+\som_{j=1}^my_jP_j(\uX,\uy)\equiv 0 \mod \gen{(X_ia_i-b_i)_{i\in\lrbn}}
\end{equation}
 where  $P_j(\uX,\uY)$'s are in $\gk[\uX,\uY]$.
\end{enumerate}
\noindent 
\emph{Variant.} We are interested in understanding the \eqvc between Items 2, 3 and 4 when some $a_i$ and $c_j$ are null, \cad when we assume that some $b_i$ or $d_j$ are null. Let us note $e_i$ and $f_j$ the \elts of $\gA$ assumed to be null on the left and on the right of~$\vd$.\\
 \Propeq
\begin{enumerate}\setcounter{enumi}{4}
\item One has 

\vspace{-1.3em} 
\[ 
\begin{array}{c} 
a_1\di b_1\vet \dots\vet a_n\di b_n\vet e_{1}=0\vet \dots\vet e_k=0 \Vdi{\sA{val}(\gA,\gk)}~\hspace{10em} \\[.3em]
~\hspace{5em} c_1\di d_1\vou \dots\vou c_m\di d_m\vou f_1=0\vou \dots\vou f_\ell=0
 \end{array}
\]

\item Let us note $y_j=\frac{c_j}{d_j}$ viewed in $\gB=\gA[\frac{1}{d_1\cdots d_m\cdot f_1\cdots f_\ell}]$. Consider \idtrs $X_i$ ($i\in\lrbn$). One has in the ring $\gB[\uX]$ an \egt in the following form 
\begin{equation} \label {eqthVstformelVabis5}
1+\som_{j=1}^my_jP_j(\uX,\uy)\equiv 0 \mod \gen{(X_ia_i-b_i)_{i\in\lrbn}, (e_i)_{i\in\lrbk}}
\end{equation}
 where  $P_j(\uX,\uY)$'s are in $\gk[\uX,\uY]$.
\end{enumerate}
 
\end{valsatz}
\begin{proof} Equivalence of Items \textsl{1}, \textsl{3} and \textsl{4}
corresponds to the \eqvc of Items~\textsl{1},~\textsl{3} and~\textsl{5} in \vst \ref{thVstformelval}. Note that $\ZZ[\uX,\uY]$ (in Item \textsl{5}) is now replaced with $\gk[\uX,\uY]$ (in Item~\textsl{4}) since by hypothesis  \elts of $\gk$ are now integral.
\\
We now need to add a few words explaining the variant
(a similar variant work also for \vst~\ref{thVstformelval}):  Eq. \pref{eqthVstformelVabis5} is simply the analogous of  Eq. \pref{eqthVstformelVabisl5} for the ring $\aqo\gB{e_1,\dots,e_\ell}$.
\end{proof}

As corollary of \vst~\ref{thVstformelvalbis} we get the result which was announced after Lemma~\ref{lem-y-di-x}.

\begin{theorem} \label{th-y-di-x} 
Let $\gk\subseteq \gA$ be two rings, $x\in\gA$ and $y\in\gk$. 
\begin{enumerate}\itemsep=0pt
\item A \sad $\sa{val}(\gA,\gk)$ proves $y\di x$ \ssi $x$ is integral over the \id $\gen{y}$ of $\gk$. 
\item Same thing for $\sa{Vdf}^+\!(\gA,\gk)$.
\end{enumerate}

\end{theorem}
\begin{proof}
\textsl{1}. This is a consequence  of the \eqvc of Items \textsl{2} and \textsl{4} in \vst~\ref{thVstformelvalbis} when $n=0$ and $m=1$.

\smallskip \noindent \textsl{2}. Use Item \textsl{1} and the \eqvc of Items \textsl{2} and \textsl{3} in \vst~\ref{thVstformelvalbis}.
\end{proof}

\Subsubsection{The center map (2)}\label{subseccomparvalZar}

We are now fulfilling the promise made in  Remark \ref{rempropdefCentre}. In particular, this finishes the  \demo of \thref{cor2thVstformelval4}.

\begin{theorem}[The morphism $\gamma:\ZarA\to \val(\gA,\gA)$ is injective] \label{thVstformelval4}~\\
Let $\gA$ be a commutative ring and $a_i$, $c_j \in\gA$. \Propeq 
\begin{enumerate}
\item One has 
\[
 (a_1,1),\dots,(a_n,1) \,\,\vdash_{\val(\gA,\gA)} (c_1,1),\dots,(c_m,1)
\]

\item One has in the ring $\gA$ an \egt  
\[
 \big(\prod\nolimits_{i=1}^n a_i\big)^r+\som_{j=1}^mc_jp_j = 0
 \]  
\item One has in the Zariski lattice of $\gA$
\[\rD(a_1),\dots,\rD(a_n)\,\,\vdash_{\Zar(\gA)} \rD(c_1),\dots,\rD(c_m)
\]
\end{enumerate}
In particular the center map (\dfn \ref{propdefCentre}) $\gamma:\ZarA\to \val(\gA,\gA)$ is injective. 
\end{theorem}
%
\begin{proof}
Equivalence of Items \textsl{1} and \textsl{2} is a special case of the one given in \vst~\ref{thVstformelvalbis} (Items \textsl{1} and \textsl{4} \egt \pref{eqthVstformelVabisl5}): computing modulo $Xa-1$ in $\AX$ amounts to calculate in~$\gA[\frac 1 a]$).
Equivalence of Items \textsl{2} and~\textsl{3} is the formal \nst~\ref{thNstFormel}.
\end{proof}

\Subsubsection{Another formal \vst for $\sa{val}(\gA,\gk)$}

Another slight variant of \vst~\ref{thVstformelval}, which is particularly useful after Remark~\ref{remthVstformelvalter}.

\begin{valsatz}[Another formal \vst for $\sa{val}(\gA,\gk)$ and $\sa{Vdf}^+\!(\gA,\gk)$]\label{thVstformelvalter} 
Let $\gk\subseteq \gA$ be two commutative rings and $b_i, d_j\in\gA$. \Propeq 
\begin{enumerate}
\item One has 
\begin{equation} \label {eqthVstformelvalter1}
 (1,b_1),\dots,(1,b_n) \,\,\vdash_{\val(\gA,\gk)} (1,d_1),\dots,(1,d_m) 
\end{equation}
%
\item One has 
\begin{equation} \label {eqthVstformelvalter2}
 \Vr(b_1)\vet \dots\vet \Vr(b_n) \Vdi{\sA{Vdf}(\gA,\gk)} \Vr(d_1)\vou \dots\vou \Vr(d_m)
\end{equation}
%
\item One has 
\begin{equation} \label {eqthVstformelvalter2.5}
 \vr(b_1)\vet \dots\vet \vr(b_n) \Vdi{\sA{val2}(\gA,\gk)} \vr(d_1)\vou \dots\vou \vr(d_m)
\end{equation}
%
\item Let us note $\gB=\gA[\frac{1}{d_1\cdots d_m}]$,
we have in the ring $\gB$ an \egt 
\begin{equation} \label {eqthVstformelvalter3}
1=\som_{j=1}^md_j^{-1}P_j(b_1,\dots,b_n,d_1^{-1},\dots,d_m^{-1}) 
\end{equation}
 where $P_j(\Xn,\Ym)$'s are in $\gk[\uX,\uY]$.
\item (Case when $\gA=\gK$ is a \cdi) One has 
\begin{equation} \label {eqthVstformelvalter5} 
 \Vr(b_1)\vet \dots\vet \Vr(b_n) \Vdi{\sA{Val}(\gK,\gk)} \Vr(d_1)\vou \dots\vou \Vr(d_m)
\end{equation}
%
\item (Case when $\gA=\gK$ is a \cdi) One has 
\begin{equation} \label {eqthVstformelvalter6} 
 \rV(b_1),\dots,\rV(b_n) \,\,\vdash_{\Val(\gK,\gk)} \rV(d_1),\dots,\rV(d_m) 
\end{equation}
\end{enumerate}
\end{valsatz}
%
\begin{proof}
First items are a special case of \vst~\ref{thVstformelval}. The \eqvc with two last items (they are \eqv by \dfn)  is a consequence of the lattice \iso between  $\val(\gK,\gk)$ and $\Val(\gK,\gk)$ (\thref{thValval}). So we have $\gB=\gK$.
\end{proof}
%

\begin{remark} \label{remthVstformelvalter} 
As it was expected, we recover here the formal \vst for the predicate $\Vr$ given in the article \cite{Coq2009} 
for the case where $\gA$ is a \cdi and~$b_i$'s and~$d_j$'s are nonzero. 
See also \cite{CP2001} and \cite{Lom2000}. \eoe 
\end{remark}


\endgroup
\stopcontents[english]

\clearpage
\newpage
\thispagestyle{empty}
~
\clearpage
\newpage

\setcounter{page}{1}
\renewcommand\thepage{F\arabic{page}}\renewcommand\theHsection{F\arabic{section}}

\clearpage
\setcounter{section}{0}
\setcounter{theorem}{0}
\setcounter{equation}{0}

\selectlanguage{french}
\def\frenchproofname{\textsl{Démonstration}}

\newcommand \Vrai {\mathsf{Vrai}}
\newcommand \Faux {\mathsf{Faux}}
\renewcommand \det {\MA{\textrm{dét}}}
\renewcommand \Res {\textrm{Rés}}
\renewcommand \rdef {\textrm{déf}}
\renewcommand \Reel {\MA{\textsf{Réel}}}
\renewcommand\SA[1]{\rdb\sa{#1}\label{ftheorie#1}}
\renewcommand\Sa[1]{\hyperref[ftheorie#1]{\sa{#1}}}

\renewcommand\Lab[1]{\rdb\item[\tsbf{#1}]\label{fAx#1}}
\renewcommand\Tsbf[1]{\hyperref[fAx#1]{\tsbf{#1}}}

\FrenchFootnotes

\title{Treillis et spectres valuatifs}



\theoremstyle{plain}
\newtheorem{ftheorem}{Théorème}[subsection]
\newtheorem{fthdef}[ftheorem]{Théorème et définition}
\newtheorem{fpstf}[ftheorem]{Positivstellensatz formel}
\newtheorem{fpst}[ftheorem]{Positivstellensatz}
\newtheorem{flemma}[ftheorem]{Lemme}
\newtheorem{fcorollary}[ftheorem]{Corolaire}
\newtheorem{fconjecture}[ftheorem]{Conjecture}
\newtheorem{fproposition}[ftheorem]{Proposition}
\newtheorem{fpbu}[ftheorem]{Problème universel}
\newtheorem{fprpta}[ftheorem]{Propriétés attendues}
\newtheorem{fpropdef}[ftheorem]{Proposition et définition}
\newtheorem{ffact}[ftheorem]{Fait}
\newtheorem{fplcc}[ftheorem]{Principe local-global concret}

\newtheorem{ftheoremc}[ftheorem]{Th\'{e}or\`{e}me\etoz}
\newtheorem{flemmac}[ftheorem]{Lemme\etoz}
\newtheorem{fcorollaryc}[ftheorem]{Corolaire\etoz}
\newtheorem{fproprietec}[ftheorem]{Propri\'{e}t\'{e}\etoz}
\newtheorem{fpropositionc}[ftheorem]{Proposition\etoz}
\newtheorem{ffactc}[ftheorem]{Fait\etoz}
\newtheorem{fvalsatz}[ftheorem]{\vst}

\theoremstyle{definition}
\newtheorem{fconvention}[ftheorem]{Convention}
\newtheorem{fdefinition}[ftheorem]{Définition}
\newtheorem{fdfni}[ftheorem]{Définition informelle}
\newtheorem{fdefinitions}[ftheorem]{Définitions}
\newtheorem{fnotation}[ftheorem]{Notation}
\newtheorem{fproblem}[ftheorem]{Problème}
\newtheorem{fquestion}[ftheorem]{Question}
\newtheorem{fquestions}[ftheorem]{Questions}
\newtheorem{fcontext}[ftheorem]{Contexte}
\newtheorem{fdefinitionc}[ftheorem]{Définition\etoz}
\newtheorem{fdefinota}[ftheorem]{Définition et notation}

\theoremstyle{remark}
\newtheorem{fexample}[ftheorem]{Exemple}
\newtheorem{fexamples}[ftheorem]{Exemples}
\newtheorem{fnotes}[ftheorem]{Notes}
\newtheorem{fremark}[ftheorem]{Remarque}
\newtheorem{fremarks}[ftheorem]{Remarques}
\newtheorem{fcomment}[ftheorem]{Commentaire}



\newcommand{\vou}{\MA{\tsbf{ ou }}}
\newcommand{\Vou}{\MA{\tsbf{OU}}}
\newcommand \EXists[1] {\tsbf{Introduire }{#1}\tsbf{ tel que }\,}
\newcommand \vet {{\tsbf{,}}\,}
\newcommand \Atcl {\mathrm{Atcl}}
\newcommand \Tcl {\mathrm{Tcl}}
\newcommand \Atclv {\mathrm{Atclv}}

\newcommand{\adr}{anneau de Prestel\xspace}
\newcommand{\adrs}{anneaux de Prestel\xspace}
\newcommand{\rdnr}{relation de divisibilité nilradicale\xspace}
\newcommand{\rdnrs}{relations de divisibilité nilradicale\xspace}
\newcommand{\rdr}{relation de divisibilité radicale\xspace}
\newcommand{\rdrs}{relations de divisibilité radicale\xspace}



\newcounter{MF}
\newcommand\stMF{\stepcounter{MF}}

\newcommand{\lec}{\stMF\ifodd\value{MF}lecteur\xspace\else 
lectrice\xspace\fi}

\newcommand{\lecs}{\stMF\ifodd\value{MF}lecteurs\xspace\else 
lectrices\xspace\fi}

\newcommand{\alec}{\stMF\ifodd\value{MF}au lecteur\xspace\else%
à la lectrice\xspace\fi}

\newcommand{\dlec}{\stMF\ifodd\value{MF}du lecteur\xspace\else%
de la lectrice\xspace\fi}

\newcommand{\llec}{\stMF\ifodd\value{MF}le lecteur\xspace\else la lectrice\xspace\fi}

\newcommand{\Llec}{\stMF\ifodd\value{MF}Le lecteur\xspace\else La lectrice\xspace\fi}

\newcommand{\lui}{\ifodd\value{MF}lui\xspace\else
elle\xspace\fi}

\newcommand{\celui}{\ifodd\value{MF}celui\xspace\else
celle\xspace\fi}

\newcommand{\ceux}{\ifodd\value{MF}ceux\xspace\else
celles\xspace\fi}

\newcommand{\er}{\ifodd\value{MF}er\xspace\else
ère\xspace\fi}

\newcommand{\eux}{\ifodd\value{MF}eux\xspace\else
elles\xspace\fi}

\newcommand{\eUx}{\ifodd\value{MF}eux\xspace\else
euse\xspace\fi}

\newcommand{\eUX}{\ifodd\value{MF}eux\xspace\else
euses\xspace\fi}

\newcommand{\leux}{\ifodd\value{MF}leux\xspace\else
leuse\xspace\fi}

\newcommand{\il}{\ifodd\value{MF}il\xspace\else
elle\xspace\fi}

\newcommand{\ien}{\ifodd\value{MF}ien\xspace\else
ienne\xspace\fi}

\newcommand{\iens}{\ifodd\value{MF}iens\xspace\else
iennes\xspace\fi}

\newcommand{\e}{\ifodd\value{MF}\xspace \else e\xspace\fi}

\newcommand{\n}{\ifodd\value{MF}n\xspace\else nne\xspace\fi}

\makeatletter
\newcommand{\la}{\@ifstar{\ifodd\value{MF}le\else
la\fi}{\stMF\ifodd\value{MF}le\else la\fi}}
\makeatother

\newcommand \rem{\rdb
\noi{\sl Remarque. }}

\newcommand \REM[1]{\rdb
\noi{\sl Remarque#1. }}

\newcommand \rems{\rdb
\noi{\sl Remarques. }}

\newcommand \exl{\rdb
\noi{\bf Exemple. }}

\newcommand \EXL[1]{\rdb
\noi{\bf Exemple: #1. }}

\newcommand \exls{\rdb
\noi{\bf Exemples. }}

\newcommand \thref[1] {théorème~\ref{#1}}
\newcommand \paref[1] {page~\pageref{#1}}
\newcommand \pstfref[1] {Positivstellensatz formel~\ref{#1}}
\newcommand \pstref[1] {Positivstellensatz~\ref{#1}}

\newcommand\oge{\leavevmode\raise.3ex\hbox{$\scriptscriptstyle\langle\!\langle\,$}}
\newcommand\feg{\leavevmode\raise.3ex\hbox{$\scriptscriptstyle\,\rangle\!\rangle$}}

\newcommand\gui[1]{\oge{#1}\feg}

\newcommand \facile{\begin{proof}
La démonstration est laissée \alec.
\end{proof}
}

\newcommand \num {{n$^{\mathrm{ o}}$}}

\newcommand\comm{\rdb
\noi{\sl Commentaire. }}

\newcommand\COM[1]{\rdb
\noi{\sl Commentaire #1. }}

\newcommand\comms{\rdb
\noi{\sl Commentaires. }}

\newcommand\Pb{\rdb
\noi{\bf Problème. }}

\newcommand\eoq{\hbox{}\nobreak
\vrule width 1.4mm height 1.4mm depth 0mm}

\newcommand \Cad {C'est-à-dire\xspace}
\newcommand \recu {récurrence\xspace}
\newcommand \hdr {hypothèse de \recu}
\newcommand \cad {c'est-à-dire\xspace}
\newcommand \cade {c'est-à-dire encore\xspace}
\newcommand \ssi {si, et seulement si,\xspace}
\newcommand \cnes {condition nécessaire et suffisante\xspace}
\newcommand \spdg {sans perte de généralité\xspace}
\newcommand \Spdg {Sans perte de généralité\xspace}

\newcommand \Propeq {Les propriétés suivantes sont 
équivalentes.}
\newcommand \propeq {les propriétés suivantes sont 
équivalentes.}

\newcommand \Kev {$\gK$-\evc}
\newcommand \Kevs {$\gK$-\evcs}

\newcommand \Lev {$\gL$-\evc}
\newcommand \Levs {$\gL$-\evcs}

\newcommand \Qev {$\QQ$-\evc}
\newcommand \Qevs {$\QQ$-\evcs}

\newcommand \kev {$\gk$-\evc}
\newcommand \kevs {$\gk$-\evcs}

\newcommand \lev {$\gl$-\evc}
\newcommand \levs {$\gl$-\evcs}

\newcommand \Alg {$\gA$-\alg}
\newcommand \Algs {$\gA$-\algs}

\newcommand \Blg {$\gB$-\alg}
\newcommand \Blgs {$\gB$-\algs}

\newcommand \Clg {$\gC$-\alg}
\newcommand \Clgs {$\gC$-\algs}

\newcommand \klg {$\gk$-\alg}
\newcommand \klgs {$\gk$-\algs}

\newcommand \llg {$\gl$-\alg}
\newcommand \llgs {$\gl$-\algs}

\newcommand \Klg {$\gK$-\alg}
\newcommand \Klgs {$\gK$-\algs}

\newcommand \Llg {$\gL$-\alg}
\newcommand \Llgs {$\gL$-\algs}

\newcommand \QQlg {$\QQ$-\alg}
\newcommand \QQlgs {$\QQ$-\algs}

\newcommand \Rlg {$\gR$-\alg}
\newcommand \Rlgs {$\gR$-\algs}

\newcommand \RRlg {$\RR$-\alg}
\newcommand \RRlgs {$\RR$-\algs}

\newcommand \ZZlg {$\ZZ$-\alg}
\newcommand \ZZlgs {$\ZZ$-\algs}

\newcommand \Amo {$\gA$-module\xspace}
\newcommand \Amos {$\gA$-modules\xspace}

\newcommand \Bmo {$\gB$-module\xspace}
\newcommand \Bmos {$\gB$-modules\xspace}

\newcommand \Cmo {$\gC$-module\xspace}
\newcommand \Cmos {$\gC$-modules\xspace}

\newcommand \kmo {$\gk$-module\xspace}
\newcommand \kmos {$\gk$-modules\xspace}

\newcommand \Kmo {$\gK$-module\xspace}
\newcommand \Kmos {$\gK$-modules\xspace}

\newcommand \Lmo {$\gL$-module\xspace}
\newcommand \Lmos {$\gL$-modules\xspace}

\newcommand \Vmo {$\gV$-module\xspace}
\newcommand \Vmos {$\gV$-modules\xspace}

\newcommand \Ali {application $\gA$-\lin}
\newcommand \Alis {applications $\gA$-\lins}

\newcommand \Kli {application $\gK$-\lin}
\newcommand \Klis {applications $\gK$-\lins}

\newcommand \Bli {application $\gB$-\lin}
\newcommand \Blis {applications $\gB$-\lins}

\newcommand \Cli {application $\gC$-\lin}
\newcommand \Clis {applications $\gC$-\lins}

\newcommand \ac{algébriquement clos\xspace}  

\newcommand \acl {anneau \icl}
\newcommand \acls {anneaux \icl}

\newcommand \adp {anneau de Pr\"ufer\xspace}
\newcommand \adps {anneaux de Pr\"ufer\xspace}

\newcommand \adpc {\adp \coh}
\newcommand \adpcs {\adps \cohs}

\newcommand \adu {\alg de décomposition universelle\xspace}
\newcommand \adus {\algs de décomposition universelle\xspace}

\newcommand \adv {anneau de valuation\xspace}
\newcommand \advs {anneaux de valuation\xspace}

\newcommand \advl {anneau \dvla} 
\newcommand \advls {anneaux \dvlas} 

\newcommand \Afr {Anneau \frl}
\newcommand \Afrs {Anneaux \frls}
\newcommand \afr {anneau \frl}
\newcommand \afrs {anneaux \frls}

\newcommand \aFr {\hyperref[theorieAfr]{anneau \frl}}
\newcommand \aFrs {\hyperref[theorieAfr]{anneau \frls}}

\newcommand \afrr {\afr réduit\xspace}
\newcommand \afrrs {\afrs réduits\xspace}
\newcommand \Afrrs {\Afrs réduits\xspace}

\newcommand \afrvr {\afr avec \ravs}
\newcommand \aFrvr {\hyperref[theorieAfrrv]{\afrvr}}
\newcommand \afrvrs {\afrs avec \ravs}

\newcommand \aftr {anneau réticulé \ftm réel\xspace}
\newcommand \aftrs {anneaux réticulés \ftm réels\xspace}

\newcommand \aG {\alg galoisienne\xspace}
\newcommand \aGs {\algs galoisiennes\xspace}

\newcommand \agB {\alg de Boole\xspace}
\newcommand \agBs {\algs de Boole\xspace}

\newcommand \agH {\alg de Heyting\xspace}
\newcommand \agHs {\algs de Heyting\xspace}

\newcommand \agq{algébrique\xspace}
\newcommand \agqs{algébriques\xspace}

\newcommand \agqt{algébriquement\xspace}

\newcommand \aKr {anneau de Krull\xspace}
\newcommand \aKrs {anneaux de Krull\xspace}

\newcommand \alg {algèbre\xspace}
\newcommand \algs {algèbres\xspace}

\newcommand \algo{algorithme\xspace}
\newcommand \algos{algorithmes\xspace}

\newcommand \algq{algorithmique\xspace}
\newcommand \algqs{algorithmiques\xspace}

\newcommand \ali {application \lin}
\newcommand \alis {applications \lins}

\newcommand \alo {anneau local\xspace}
\newcommand \alos {anneaux locaux\xspace}

\newcommand \algb {anneau \lgb}
\newcommand \algbs {anneaux \lgbs}

\newcommand \alrd {\alo \dcd}
\newcommand \alrds {\alos \dcds}

\newcommand \anar {anneau \ari}
\newcommand \anars {anneaux \aris}

\newcommand \anor {anneau normal\xspace}
\newcommand \anors {anneaux normaux\xspace}

\newcommand \apf {\alg \pf}
\newcommand \apfs {\algs \pf}

\newcommand \apG {\alg prégaloisienne\xspace}
\newcommand \apGs {\algs prégaloisiennes\xspace}

\newcommand \arc {anneau réel clos\xspace}
\newcommand \aRc {\hyperref[theorieArc]{\arc}}
\newcommand \arcs {anneaux réels clos\xspace}

\newcommand \ari{arithmétique\xspace}  
\newcommand \aris{arithmétiques\xspace}  

\newcommand \Asr {Anneau \str}
\newcommand \Asrs {Anneaux \strs}
\newcommand \asr {anneau \str}
\newcommand \asrs {anneaux \strs}

\newcommand \asrvr {\asr avec \ravs}
\newcommand \asrvrs {\asrs avec \ravs}

\newcommand \atf {\alg \tf}
\newcommand \atfs {\algs \tf}

\newcommand \auto {automorphisme\xspace}
\newcommand \autos {automorphismes\xspace}


\newcommand \bdg {base de Gr\"obner\xspace}
\newcommand \bdgs {bases de Gr\"obner\xspace}

\newcommand \bdp {base de \dcn partielle\xspace}
\newcommand \bdps {bases de \dcn partielle\xspace}

\newcommand \bdf {base de \fap\xspace}

\newcommand \Bif {Borne inférieure\xspace} %
\newcommand \bif {borne inférieure\xspace} %
\newcommand \bifs {bornes inférieures\xspace} %

\newcommand \bsp {borne supérieure\xspace} %
\newcommand \bsps {borne supérieures\xspace} %


\newcommand \cac{corps \ac}  

\newcommand \calf{calcul formel\xspace}  

\newcommand \cara{caractéristique\xspace}  
\newcommand \caras{caractéristiques\xspace}  

\newcommand \carn{caractérisation\xspace}  
\newcommand \carns{caractérisations\xspace}  

\newcommand \carar{caractériser\xspace}

\newcommand \carf{de caractère fini\xspace}  

\newcommand \cdac{\cdi \ac}  
\newcommand \cdacs{\cdis \ac}  
\newcommand \cdi{corps discret\xspace}
\newcommand \cdis{corps discrets\xspace}

\newcommand \cdf{corps de fractions\xspace}

\newcommand \cdH{code de Hensel\xspace}
\newcommand \cdHs{codes de Hensel\xspace}

\newcommand \cdr{corps de racines\xspace}
  
\newcommand \cdv{changement de variables\xspace}  
\newcommand \cdvs{changements de variables\xspace}  

\newcommand \cla {clôture \agq}
\newcommand \clas {clôtures \agqs}

\newcommand \cli {clôture intégrale\xspace}
\newcommand \clis {clôtures intégrales\xspace}

\newcommand \codi {corps ordonné discret\xspace}
\newcommand \codis {corps ordonnés discrets\xspace}

\newcommand \coe {coefficient\xspace}
\newcommand \coes {coefficients\xspace}

\newcommand \coh {cohérent\xspace}
\newcommand \cohe {cohérente\xspace}
\newcommand \cohs {cohérents\xspace}

\newcommand \cohc {cohérence\xspace}

\newcommand \coli {combinaison \lin}
\newcommand \colis {combinaisons \lins}

\newcommand \com {comaximaux\xspace}
\newcommand \come {comaximales\xspace}

\newcommand \coo {coordonnée\xspace}
\newcommand \coos {coordonnées\xspace}

\newcommand \cop {complémentaire\xspace}
\newcommand \cops {complémentaires\xspace}

\newcommand \cosv {conservative\xspace}
\newcommand \cosvs {conservatives\xspace}

\newcommand \cOsv {\hyperref[defithconserv]{conservative}\xspace}
\newcommand \cOsvs {\hyperref[defithconserv]{conservatives}\xspace}

\newcommand \covr {corps ordonné avec \ravs}
\newcommand \covrs {corps ordonnés avec \ravs}

\newcommand \cpb {compatible\xspace} 
\newcommand \cpbs {compatibles\xspace} 

\newcommand \cpbt {compatibilité\xspace} 
\newcommand \cpbtz {compatibilité} 

\newcommand \crc {corps réel clos\xspace}
\newcommand \crcs {corps réels clos\xspace}

\newcommand \crcd {corps réel clos discret\xspace}
\newcommand \crcds {corps réels clos discrets\xspace}

\newcommand \cqi {clôture \qi} 
\newcommand \cqis {clôtures \qis} 

\newcommand \cvd{corps valué discret\xspace}
\newcommand \cvds{corps valués discrets\xspace}

\newcommand \cvdu{\cvd ultramétrique\xspace}
\newcommand \cvdus{\cvds ultramétriques\xspace}


\newcommand \dcd {résiduellement discret\xspace}
\newcommand \dcds {résiduellement discrets\xspace}

\newcommand \dcn {décomposition\xspace}
\newcommand \dcns {décompositions\xspace}

\newcommand \dcnb {\dcn bornée\xspace}

\newcommand \dcnc {\dcn complète\xspace}

\newcommand \dcnp {\dcn partielle\xspace}

\newcommand \dcp {décomposable\xspace}
\newcommand \dcps {décomposables\xspace}

\newcommand \ddk {dimension de~Krull\xspace}
\newcommand \ddi {de dimension inférieure ou égale à~}

\newcommand \ddp {domaine de Pr\"ufer\xspace}
\newcommand \ddps {domaines de Pr\"ufer\xspace}

\newcommand \Demo{Démonstration\xspace}     

\newcommand \demo{démonstration\xspace}     
\newcommand \demos{démonstrations\xspace}     

\newcommand \dems{démonstrations\xspace}

\newcommand \deno{dénominateur\xspace}     
\newcommand \denos{dénominateurs\xspace}   

\newcommand \deter {déterminant\xspace}  
\newcommand \deters {déterminants\xspace}  
  
\newcommand \Dfn{Définition\xspace}  
\newcommand \Dfns{Définitions\xspace}  
\newcommand \dfn{définition\xspace}  
\newcommand \dfns{définitions\xspace}  

\newcommand \dftr {droite réticulée \ftm réelle\xspace}
\newcommand \dftrs {droites réticulées \ftm réelles\xspace}

\newcommand \dij{disjonctive\xspace}  
\newcommand \dijs{disjonctives\xspace}  
 
\newcommand \dil{différentiel\xspace}  
\newcommand \dils{différentiels\xspace}  
\newcommand \dile{différentielle\xspace}  
\newcommand \diles{différentielles\xspace}  

\newcommand \dip{diviseur principal\xspace}
\newcommand \dips{diviseurs principaux\xspace}

\newcommand \discri{discriminant\xspace}  
\newcommand \discris{discriminants\xspace}  

\newcommand \divle {dimension divisorielle\xspace}

\newcommand \dit{distributivité\xspace}

\newcommand \dlg{d'élargissement\xspace}  

\newcommand \dok {domaine de Dedekind\xspace}
\newcommand \doks {domaines de Dedekind\xspace}

\newcommand \dvla {à diviseurs\xspace}
\newcommand \dvlas {à diviseurs\xspace}

\newcommand \dvld {\dvlt décomposé\xspace} %
\newcommand \dvlds {\dvlt décomposés\xspace} %

\newcommand \dvlg {divisoriel\xspace} 
\newcommand \dvlgs {divisoriels\xspace} 

\newcommand \dvli {\dvlt inversible\xspace} 
\newcommand \dvlis {\dvlt inversibles\xspace} 

\newcommand \dvlt {divisoriellement\xspace} %

\newcommand \dvz {diviseur de zéro\xspace}
\newcommand \dvzs {diviseurs de zéro\xspace}

\newcommand \dve {divisibilité\xspace}

\newcommand \dvee {à \dve explicite\xspace}

\newcommand \dvr {diviseur\xspace}
\newcommand \dvrs {diviseurs\xspace}


\newcommand \Eds {Extension des scalaires\xspace}
\newcommand \edss {extensions des scalaires\xspace}
\newcommand \eds {extension des scalaires\xspace}

\newcommand \eco {\elts \com}

\newcommand \egmt {également\xspace}

\newcommand \egt {égalité\xspace}
\newcommand \egts {égalités\xspace}

\newcommand \eli{élimination\xspace}  

\newcommand \elr{élémentaire\xspace}  
\newcommand \elrs{élémentaires\xspace}  

\newcommand \elrt{élémentairement\xspace}  

\newcommand \elt{élément\xspace}  
\newcommand \elts{éléments\xspace}  

\def \endo {endomorphisme\xspace}
\def \endos {endomorphismes\xspace}

\newcommand \entrel {relation implicative\xspace}
\newcommand \entrels {relations implicatives\xspace}

\newcommand \eqn {équation\xspace}  
\newcommand \eqns {équations\xspace}  

\newcommand \eqv {équivalent\xspace}  
\newcommand \eqve {équivalente\xspace}  
\newcommand \eqvs {équivalents\xspace}  
\newcommand \eqves {équivalentes\xspace}  

\newcommand \eqvc {équivalence\xspace}  
\newcommand \eqvcs {équivalences\xspace}  

\newcommand \esid {essentiellement identique\xspace}  
\newcommand \esids {essentiellement identiques\xspace}  

\newcommand \Esid {\hyperref[defitdyesidentiques]{\esid}}  
\newcommand \Esids {\hyperref[defitdyesidentiques]{\esids}}  

\newcommand \eseq {essentiellement \eqve}  
\newcommand \eseqs {essentiellement \eqves}  

\newcommand \Eseq {\hyperref[defitheseq]{\eseq}}  
\newcommand \Eseqs {\hyperref[defitheseq]{\eseqs}}  

\newcommand\evc{espace vectoriel\xspace} 
\newcommand\evcs{espaces vectoriels\xspace}

\newcommand \fab {\fcn bornée\xspace}
\newcommand \fabs {\fcns bornées\xspace}

\newcommand \fat {\fcn totale\xspace}
\newcommand \fats {\fcn totales\xspace}

\newcommand \fap {\fcn partielle\xspace}
\newcommand \faps {\fcns partielles\xspace}

\newcommand \fip {filtre premier\xspace}
\newcommand \fips {filtres premiers\xspace}

\newcommand \fipma {\fip maximal\xspace}
\newcommand \fipmas {\fips maximaux\xspace}

\newcommand \fcn {factorisation\xspace}
\newcommand \fcns {factorisations\xspace}

\newcommand \fdi {fortement discret\xspace}
\newcommand \fdis {fortement discrets\xspace}

\newcommand \fsa {fermé \sagq}
\newcommand \fsas {fermés \sagqs}

\newcommand \fsagc {fonction \sagc}
\newcommand \fsagcs {fonctions \sagcs}

\newcommand \fmt {formellement\xspace}

\newcommand \frl {fortement réticulé\xspace}
\newcommand \frle {fortement réticulée\xspace}
\newcommand \frls {fortement réticulés\xspace}

\newcommand \ftm {fortement\xspace}

\newcommand\gmt{géométrie\xspace}  
\newcommand\gmts{géométries\xspace}  

\newcommand\gaq{\gmt \agq}  

\newcommand\gmq{géométrique\xspace}  
\newcommand\gmqs{géométriques\xspace}  

\newcommand\gmqt{géométriquement\xspace}  

\newcommand\gne{généralisé\xspace}  
\newcommand\gnee{généralisée\xspace}  
\newcommand\gnes{généralisés\xspace}  
\newcommand\gnees{généralisées\xspace}  

\newcommand\gnl{général\xspace}  
\newcommand\gnle{générale\xspace}  
\newcommand\gnls{généraux\xspace}  
\newcommand\gnles{générales\xspace}  

\newcommand\gnlt{généralement\xspace}  

\newcommand\gnn{généralisation\xspace}  
\newcommand\gnns{généralisations\xspace}  

\newcommand\gnq {générique\xspace}  
\newcommand\gnqs {génériques\xspace}  

\newcommand\gnr{généraliser\xspace}  

\newcommand \gns{généralise\xspace}

\newcommand \gnt{généralité\xspace}
\newcommand \gnts{généralités\xspace}

\newcommand \grl{groupe \rtl}
\newcommand \grls{groupes \rtls}

\newcommand \gRl {\hyperref[theorieGrl]{\grl}}
\newcommand \gRls {\hyperref[theorieGrl]{\grls}}

\newcommand\gtr{générateur\xspace}  
\newcommand\gtrs{générateurs\xspace}  


\newcommand \homo {homomorphisme\xspace}
\newcommand \homos {homomorphismes\xspace}

\newcommand \hmg {homogène\xspace}
\newcommand \hmgs {homogènes\xspace}

\newcommand \icftr {intervalle compact réticulé \ftm réel\xspace}
\newcommand \icftrs {intervalles compacts réticulés \ftm réels\xspace}

\newcommand \icl {intégralement clos\xspace}
\newcommand \icle {intégralement close\xspace}
\newcommand \icles {intégralement closes\xspace}

\newcommand \icsr {intervalle compact \stm réticulé\xspace}
\newcommand \icsrs {intervalles compacts \stm réticulés\xspace}

\newcommand \icrc {intervalle compact réel clos\xspace}
\newcommand \icrcs {intervalles compact réels clos\xspace}

\newcommand \id {idéal\xspace}
\newcommand \ids {idéaux\xspace}

\newcommand \ida {\idt \agq}
\newcommand \idas {\idts \agqs}

\newcommand \idc  {\idt de Cramer\xspace}
\newcommand \idcs {\idts de Cramer\xspace}

\newcommand \idd {idéal déterminantiel\xspace}
\newcommand \idds {idéaux déterminantiels\xspace}

\newcommand \idema {idéal maximal\xspace}
\newcommand \idemas {idéaux maximaux\xspace}

\newcommand \idep {idéal premier\xspace}
\newcommand \ideps {idéaux premiers\xspace}

\newcommand \idemi {\idep minimal\xspace}
\newcommand \idemis {\ideps minimaux\xspace}

\newcommand \idf {idéal de Fitting\xspace}
\newcommand \idfs {idéaux de Fitting\xspace}

\newcommand \idif {idéal \dvlg fini\xspace}
\newcommand \idifs {idéaux \dvlgs finis\xspace}

\newcommand \idli {idéal \dvli\xspace} 
\newcommand \idlis {idéaux \dvlis\xspace} 

\newcommand \idm {idempotent\xspace}
\newcommand \idms {idempotents\xspace}
\newcommand \idme {idempotente\xspace}
\newcommand \idmes {idempotentes\xspace}

\newcommand \idp {idéal principal\xspace}
\newcommand \idps {idéaux principaux\xspace}

\newcommand \idt {identité\xspace}
\newcommand \idts {identités\xspace}

\newcommand \idtr {indéterminée\xspace}
\newcommand \idtrs {indéterminées\xspace}

\newcommand \ifr {idéal fractionnaire\xspace}
\newcommand \ifrs {idéaux fractionnaires\xspace}

\newcommand \imd {immédiat\xspace}
\newcommand \imde {immédiate\xspace}
\newcommand \imds {immédiats\xspace}
\newcommand \imdes {immédiates\xspace}

\newcommand \imdt {immédiatement\xspace}

\newcommand \indtr {inf-demi-treillis\xspace} 

\newcommand \inteq {intuitivement \eqve}
\newcommand \inteqs {intuitivement \eqves}

\newcommand \Inteq {\hyperref[defextintequiv]{\inteq}}
\newcommand \Inteqs {\hyperref[defextintequiv]{\inteqs}}

\newcommand \ing {inverse \gne}
\newcommand \ings {inverses \gnes}

\newcommand \iMP {inverse de Moore-Penrose\xspace}
\newcommand \iMPs {inverses de Moore-Penrose\xspace}

\newcommand \ipp {\idep potentiel\xspace}
\newcommand \ipps {\ideps potentiels\xspace}

\newcommand \ird {irréductible\xspace}
\newcommand \irds {irréductibles\xspace}

\newcommand \iso {isomorphisme\xspace}
\newcommand \isos {isomorphismes\xspace}

\newcommand \itf {idéal \tf}
\newcommand \itfs {idéaux \tf}

\newcommand \itid {intuitivement identique\xspace}
\newcommand \itids {intuitivement identiques\xspace}

\newcommand \iv {inversible\xspace}
\newcommand \ivs {inversibles\xspace}

\newcommand \ivdg {inverse divisoriel\xspace} 
\newcommand \ivdgs {inverses divisoriels\xspace} 

\newcommand \ivde {inverse divisorielle\xspace} 
\newcommand \ivdes {inverses divisorielles\xspace} 

\newcommand \ivda {inverse divisoriel\xspace} 
\newcommand \ivdas {inverses divisoriels\xspace} 


\newcommand \lgb {local-global\xspace}
\newcommand \lgbe {locale-globale\xspace}
\newcommand \lgbs {local-globals\xspace}

\newcommand \lin {linéaire\xspace}
\newcommand \lins {linéaires\xspace}

\newcommand \lint {linéairement\xspace}

\newcommand \lmo {\lot monogène\xspace}
\newcommand \lmos {\lot monogènes\xspace}

\newcommand \lnl {\lot \nl}
\newcommand \lnls {\lot \nls}

\newcommand \lot {localement\xspace}

\newcommand \lon {localisation\xspace}
\newcommand \lons {localisations\xspace}

\newcommand \lop {\lot principal\xspace}
\newcommand \lops {\lot principaux\xspace}

\newcommand \lsdz {\lot \sdz}

\newcommand \mdi {module des \diles}

\newcommand \mlm {module \lmo}
\newcommand \mlms {modules \lmos}

\newcommand \mlmo {matrice de localisation
monogène\xspace}
\newcommand \mlmos {matrices de localisation
monogène\xspace}

\newcommand \mlp {matrice de localisation
principale\xspace}
\newcommand \mlps {matrices de localisation
principale\xspace}

\newcommand \mo {mono\"{\i}de\xspace}
\newcommand \mos {mono\"{\i}des\xspace}

\newcommand \moco {\mos \com}

\newcommand \molo {morphisme de \lon\xspace}
\newcommand \molos {morphismes de \lon\xspace}

\newcommand \mom {monôme\xspace}
\newcommand \moms {monômes\xspace}

\newcommand \moquo {morphisme de passage au quotient\xspace}
\newcommand \moquos {morphismes de passage au quotient\xspace}

\newcommand \mpf {module \pf}
\newcommand \mpfs {modules \pf}

\newcommand \mpl {module plat\xspace}
\newcommand \mpls {modules plats\xspace}

\newcommand \mpn {matrice de \pn}
\newcommand \mpns {matrices de \pn}

\newcommand \mprn {matrice de \prn}
\newcommand \mprns {matrices de \prn}

\newcommand \mptf {module \ptf}
\newcommand \mptfs {modules \ptfs}

\newcommand \mrc {module \prc}
\newcommand \mrcs {modules \prcs}

\newcommand \mtf {module \tf}
\newcommand \mtfs {modules \tf}


\newcommand \ncr{nécessaire\xspace}       
\newcommand \ncrs{nécessaires\xspace}       

\newcommand \ncrt{nécessairement\xspace}       

\newcommand \ndz {régulier\xspace}
\newcommand \ndzs {réguliers\xspace}

\newcommand \nl {simple\xspace}
\newcommand \nls {simples\xspace}

\newcommand \noco {\noe \coh}
\newcommand \nocos {\noes \cohs}

\newcommand \Noe {Noether\xspace}

\newcommand \noe {noethérien\xspace}
\newcommand \noes {noethériens\xspace}
\newcommand \noee {noethérienne\xspace}
\newcommand \noees {noethériennes\xspace}

\newcommand \noet {noethérianité\xspace}

\newcommand \nst {Nullstellensatz\xspace}
\newcommand \nsts {Nullstellens\"atze\xspace}

\newcommand \op{opération\xspace}  
\newcommand \ops{opérations\xspace}
\newcommand \opari{\op\ari}  
\newcommand \oparis{\ops\aris}  
\newcommand \oparisv{\ops\arisv}  

\newcommand \oqc {ouvert \qc}
\newcommand \oqcs {ouverts \qcs}

\newcommand \ort{orthogonal\xspace}  
\newcommand \orte{orthogonale\xspace}  
\newcommand \orts{orthogonaux\xspace}  
\newcommand \ortes{orthogonales\xspace}  


\newcommand \pa {couple saturé\xspace}
\newcommand \pas {couples saturés\xspace}
 
\newcommand \paral{parallèle\xspace}  
\newcommand \parals{paallèles\xspace}  

\newcommand \paralm{parallèlement\xspace}

\newcommand \pb{problème\xspace}  
\newcommand \pbs{problèmes\xspace}  

\newcommand \peq {purement équationnelle\xspace}
\newcommand \peqs {purement équationnelles\xspace}

\newcommand \pf {de \pn finie\xspace}

\newcommand \plc {résiduellement \zed}
\newcommand \plcs {résiduellement \zeds}

\newcommand \Plg {Principe \lgb}
\newcommand \plg {principe \lgb}
\newcommand \plgs {principes \lgbs}

\newcommand \plga {\plg abstrait\xspace}
\newcommand \plgas {\plgs abstraits\xspace}

\newcommand \Plgc {\Plg concret\xspace}
\newcommand \plgc {\plg concret\xspace}
\newcommand \plgcs {\plgs concrets\xspace}

\newcommand \pn {présentation\xspace}
\newcommand \pns {présentations\xspace}

\newcommand \pog {\pol \hmg\xspace}
\newcommand \pogs {\pols \hmgs\xspace}

\newcommand \Pol {Polynôme\xspace}
\newcommand \Pols {Polynômes\xspace}

\newcommand \pol {polynôme\xspace}
\newcommand \pols {polynômes\xspace}

\newcommand \poll{polynomial\xspace}  
\newcommand \polls{polynomiaux\xspace}  
\newcommand \polle{polynomiale\xspace}  
\newcommand \polles{polynomiales\xspace}  

\newcommand \pollt{polynomialement\xspace}  

\newcommand \polfon {\pol fondamental\xspace}
\newcommand \polmu {\pol rang\xspace}
\newcommand \polmus {\pols rang\xspace}
\newcommand \polcar {\pol caractéristique\xspace}
\newcommand \polmin {\pol minimal\xspace}

\newcommand \prc {\pro de rang constant\xspace}
\newcommand \prcs {\pros de rang constant\xspace}

\newcommand \prcc {principe de \rcc}
\newcommand \prca {principe de \rca}
\newcommand \prce {principe de \rce}

\newcommand \prmt {précisément\xspace}
\newcommand \Prmt {Précisément\xspace}

\newcommand \prn {projection\xspace}
\newcommand \prns {projections\xspace}

\newcommand \pro {projectif\xspace}
\newcommand \pros {projectifs\xspace}

\newcommand \prr {projecteur\xspace}
\newcommand \prrs {projecteurs\xspace}

\newcommand \Prt {Propriété\xspace}
\newcommand \Prts {Propriétés\xspace}
\newcommand \prt {propriété\xspace}
\newcommand \prts {propriétés\xspace}

\newcommand \ptf {\pro \tf}
\newcommand \ptfs {\pros \tf}

\newcommand \qc {quasi-compact\xspace}
\newcommand \qcs {quasi-compacts\xspace}

\newcommand \qi {quasi intègre\xspace}
\newcommand \qis {quasi intègres\xspace}

\newcommand \qiv {quasi inverse\xspace}
\newcommand \qivs {quasi inverses\xspace}

\newcommand \qnl {quasi-\nl}
\newcommand \qnls {quasi-\nls}

\newcommand \ralg {règle \agq}
\newcommand \ralgs {règles \agqs}

\newcommand \rav {racine virtuelle\xspace}
\newcommand \ravs {racines virtuelles\xspace}

\newcommand \rcc {\rcm concret\xspace}
\newcommand \rca {\rcm abstrait\xspace}
\newcommand \rce {\rcc des égalités\xspace}

\newcommand \rcm {recollement\xspace}
\newcommand \rcms {recollements\xspace}

\newcommand \rcv {recouvrement\xspace} 
\newcommand \rcvs {recouvrements\xspace}

\newcommand \rde {relation de dépendance\xspace}
\newcommand \rdes {relations de dépendance\xspace}

\newcommand \rdi {\rde intégrale\xspace}
\newcommand \rdis {\rdes intégrales\xspace}

\newcommand \rdij {règle \dij}
\newcommand \rdijs {règles \dijs}

\newcommand \rdl {\rde \lin}
\newcommand \rdls {\rdes \lins}

\newcommand \rdt {résiduellement\xspace}

\newcommand \rdv {relation de \dve valuative\xspace}
\newcommand \rdvs {relations de \dve valuative\xspace}

\newcommand \rdy {règle dynamique\xspace}
\newcommand \rdys {règles dynamiques\xspace}

\newcommand \red {règle directe\xspace}
\newcommand \reds {règles directes\xspace}

\newcommand \rex {\hyperref[defexistsimple]{règle existentielle simple}\xspace}
\newcommand \rexs {\hyperref[defexistsimple]{règles existentielles simples}\xspace}

\newcommand \rexri {\hyperref[defitdyexrig]{règle existentielle rigide}\xspace}
\newcommand \rexris {\hyperref[defitdyexrig]{règles existentielles rigides}\xspace}

\newcommand \rsim {règle de simplification\xspace}
\newcommand \rsims {règles de simplification\xspace}

\newcommand \rtl {réticulé\xspace}
\newcommand \rtls {réticulés\xspace}

\newcommand \rmq {\rcm de quotients\xspace} 
\newcommand \rvq {\rcv par quotients\xspace} 
\newcommand \rmqs {\rcms de quotients\xspace} %
\newcommand \rvqs {\rcvs par quotients\xspace} %

\newcommand \rpf {réduite-de-présentation-finie\xspace}
\newcommand \rpfs {réduites-de-présentation-finie\xspace}


\newcommand \sad {\salg dynamique\xspace}
\newcommand \sads {\salgs dynamiques\xspace}

\newcommand \sagq {semi\agq}
\newcommand \sagqs {semi\agqs}

\newcommand \sagc {\sagq continue\xspace}
\newcommand \sagcs {\sagqs continues\xspace}

\newcommand \salg {structure \agq}
\newcommand \salgs {structures \agqs}

\newcommand \scentrel {relation semi-implicative\xspace}
\newcommand \scentrels {relations semi-implicatives\xspace}

\newcommand \sca {schéma affine\xspace}
\newcommand \scas {schémas affines\xspace}

\newcommand \scf {schéma spectral\xspace}
\newcommand \scfs {schémas spectraux\xspace}

\newcommand \scl {schéma \elr}
\newcommand \scls {schémas \elrs}

\newcommand \sdo {\sdr \orte}
\newcommand \sdos {\sdrs \ortes}

\newcommand \sdr {somme directe\xspace}
\newcommand \sdrs {sommes directes\xspace}

\newcommand \sdz {sans \dvz}

\newcommand \sfio {système fondamental d'\idms orthogonaux\xspace}
\newcommand \sfios {systèmes fondamentaux d'\idms orthogonaux\xspace}

\newcommand \sgr {\sys \gtr}
\newcommand \sgrs {\syss \gtrs}

\newcommand \slgb {strictement \lgb}
\newcommand \slgbs {strictement \lgbs}

\newcommand \sli {\sys \lin}
\newcommand \slis {\syss \lins}

\newcommand \smq {symétrique\xspace}
\newcommand \smqs {symétriques\xspace}

\newcommand \spb {séparable\xspace}  
\newcommand \spbs {séparables\xspace}

\newcommand \spe {spécification\xspace}
\newcommand \spes {spécifications\xspace}

\newcommand \spi {\spe incomplète\xspace}
\newcommand \spis {\spes incomplètes\xspace}

\newcommand \spl {séparable\xspace}  
\newcommand \spls {séparables\xspace}

\newcommand \spo {semipolynôme\xspace}
\newcommand \spos {semipolynômes\xspace}

\newcommand \spt{séparabilité\xspace}

\newcommand \srg {suite régulière\xspace}
\newcommand \srgs {suites régulières\xspace}

\newcommand \stf {strictement fini\xspace}
\newcommand \stfs {strictement finis\xspace}
\newcommand \stfe {strictement finie\xspace}
\newcommand \stfes {strictement finies\xspace}

\newcommand \stl {stablement libre\xspace}
\newcommand \stls {stablement libres\xspace}

\newcommand \stm {strictement\xspace}

\newcommand \str {\stm réticulé\xspace}
\newcommand \stre {\stm réticulée\xspace}
\newcommand \strs {\stm réticulés\xspace}
\newcommand \stres {\stm réticulées\xspace}

\newcommand \sul {supplémentaire\xspace}
\newcommand \suls {supplémentaires\xspace}

\newcommand \Sut {Support\xspace}
\newcommand \Suts {Supports\xspace}
\newcommand \sut {support\xspace}

\newcommand \syc {\sys de coordonnées\xspace}
\newcommand \sycs {\syss de coordonnées\xspace}

\newcommand \syp {\sys \poll}
\newcommand \syps {\syss \polls}

\newcommand \sys {système\xspace}
\newcommand \syss {systèmes\xspace}

\newcommand \talg {théorie \agq}
\newcommand \talgs {théories \agqs}

\newcommand \tco {théorie cohérente\xspace}
\newcommand \tcos {théories cohérentes\xspace}

\newcommand \tdij {théorie \dij}
\newcommand \tdijs {théories \dijs}

\newcommand \tdy {théorie dynamique\xspace}
\newcommand \tdys {théories dynamiques\xspace}

\newcommand \tel {théorie existentielle\xspace}
\newcommand \tels {théories existentielles\xspace}

\newcommand \telri {théorie existentiellement rigide\xspace}
\newcommand \telris {théories existentiellement rigides\xspace}

\newcommand \tf {de type fini\xspace}

\newcommand \tfo {théorie formelle\xspace}
\newcommand \tfos {théorie formelles\xspace}

\newcommand \tgm {théorie \gmq}
\newcommand \tgms {théories \gmqs}

\newcommand \Tho {Théorème\xspace}
\newcommand \Thos {Théorèmes\xspace}
\newcommand \tho {théorème\xspace}
\newcommand \thos {théorèmes\xspace}

\newcommand \thoc {théorème$\mathbf{^*}$~}

\newcommand \tpe {théorie \peq}
\newcommand \tpes {théories \peqs}

\newcommand \trdi {treillis distributif\xspace}
\newcommand \trdis {treillis distributifs\xspace}

\newcommand \trel {transformation \elr}
\newcommand \trels {transformations \elrs}

\newcommand \umd {unimodulaire\xspace}
\newcommand \umds {unimodulaires\xspace}

\newcommand \unt {unitaire\xspace}
\newcommand \unts {unitaires\xspace}

\newcommand \uvl {universel\xspace}
\newcommand \uvle {universelle\xspace}
\newcommand \uvls {universels\xspace}
\newcommand \uvles {universelles\xspace}


\newcommand \vala {valeur absolue\xspace}
\newcommand \valas {valeurs absolues\xspace}

\newcommand \valn {valuation\xspace}
\newcommand \valns {valuations\xspace}

\newcommand \valu {\vala \ultm}
\newcommand \valus {\valas \ultms}

\newcommand \vfn {vérification\xspace}
\newcommand \vfns {vérifications\xspace}

\newcommand \vmd {vecteur \umd}
\newcommand \vmds {vecteurs \umds}

\newcommand \vst {Valuativstellensatz\xspace}
\newcommand \vsts {Valuativstellensätze\xspace}

\newcommand \zed {zéro-dimensionnel\xspace}
\newcommand \zede {zéro-dimensionnelle\xspace}
\newcommand \zeds {zéro-dimensionnels\xspace}
\newcommand \zedes {zéro-dimensionnelles\xspace}

\newcommand \zedr {\zed réduit\xspace}
\newcommand \zedrs {\zeds réduits\xspace}

\newcommand \zmt {\tho de Zariski-Grothendieck\xspace}


\newcommand \cof {constructif\xspace}
\newcommand \cofs {constructifs\xspace}

\newcommand \cov {constructive\xspace}
\newcommand \covs {constructives\xspace}

\newcommand \coma {\maths\covs}
\newcommand \clama {\maths classiques\xspace}

\renewcommand \cot {constructivement\xspace}

\newcommand \matn {mathématicien\xspace}
\newcommand \matne {mathématicienne\xspace}
\newcommand \matns {mathématiciens\xspace}
\newcommand \matnes {mathématiciennes\xspace}

\newcommand \maths {mathématiques\xspace}
\newcommand \mathe {mathématique\xspace}

\newcommand \prco {démonstration \cov}
\newcommand \prcos {démonstrations \covs}

\thickmuskip = 7mu plus 2mu

\pagestyle{headings}
\patchcmd{\sectionmark}{\MakeUppercase}{}{}{}

\stMF
\startcontents[french]

\author{Henri Lombardi, Assia Mahboubi}
\maketitle

\rdb
\label{beginfrench}

\begin{abstract} 
La première partie de cet article est un survey concernant la méthode \cov dynamique mise au point en utilisant les \tdys et les \sads.
La méthode dynamique fournit un contenu calculatoire caché pour de nombreux objets abstraits des \clama qui semblent à priori inaccessibles \cot, e.g. la clôture \agq d'un corps (discret). Lorsqu'une \demo en \clama utilise ces objets abstraits et aboutit à un résultat concret, la méthode dynamique permet \gnlt de découvrir un \algo pour ce résultat concret.
La deuxième partie de l'article applique cette méthode à la théorie de la \dve. Nous comparons deux notions de spectre valuatif présents dans la littérature et nous introduisons une troisième notion, implicite dans un article qui traite la \tdy des \cvds \ac.
Les deux premières notions sont dues respectivement à Huber \& Knebusch et Thierry Coquand. Nous montrons que les treillis valuatifs correspondants sont essentiellement les mêmes. Nous établissons des \vsts formels correspondant à ces théories et nous comparons les différentes notions de dimension valuative qui en résultent.
\end{abstract}

\noindent {\bf Mots clés}. 
Théorie \cohe, \tdy, \sad, \coma, 
\trdi, espace spectral, 
treillis de Zariski, spectre de Zariski, treillis valuatif, spectre valuatif, dimensions valuatives. 

\smallskip \noindent {\bf MSC}. 03Fxx 13xx 18Fxx

\small

\setcounter{tocdepth}{4}
\markboth{Table des matières}{Table des matières}

\printcontents[french]{}{1}{}
\normalsize

\section{Introduction}
Cet article est écrit dans le style des \coma à la Bishop. Voir les ouvrages \cite{fBi67,fBB85,fBR1987,fACMC,fMRR,fYen2015}.

Dans tout ce qui suit, on utilise le vocabulaire et les notations des théories dynamiques. Voir \cite{fBC2005,fBC2019,fCoq2005,fCLR01,fLom98,fLom02,fLom06}.

En fait il s'agit ici de s'éclaircir les idées sur la théorie constructive des anneaux de valuation et des courbes \agqs telle qu'elle est déjà abordée dans \cite{fCoq2009,fCL2016,fCLQ2010,fCLR01}  et \cite[section~XIII-8]{fACMC}.  

Nous espérons que cette étude nous ouvrira la voie pour 
\begin{itemize}
\item comprendre (\cot) la théorie des diviseurs de Weil-Kronecker appliquée à un anneau géométrique intègre (donc prolonger l'étude de \cite{fCL2016});
\item comprendre (\cot) le \tho qui dit que la théorie de l'anneau des entiers algébriques est complète \cite{fvdD88,fvDM1990,fPre82,fPS1990};
\item  comprendre (\cot) le \tho de 
Riemann-Roch dans sa version algébrique générale (\cite{fFulton,fEdw1990}); 
\item  comprendre (\cot) la version Grothendieck du \tho de 
Riemann-Roch exposée dans~\cite{fBS1958}.

\end{itemize}

\medskip
Cet article commence dans les trois premières sections  par un survey concernant la méthode \cov dynamique mise au point en utilisant les \tdys et les \sads. 
La méthode dynamique fournit un contenu calculatoire caché pour de nombreux objets abstraits des \clama qui semblent à priori inaccessibles \cot, e.g. la clôture \agq d'un corps (discret). Lorsqu'une \demo en \clama utilise ces objets abstraits et aboutit à un résultat concret, la méthode dynamique permet \gnlt de découvrir un \algo pour ce résultat concret. 

La deuxième partie de l'article applique cette méthode à la théorie de la \dve.

Dans la théorie des anneaux de valuation, \cite{fHK1994} ont introduit le spectre valuatif d'un anneau commutatif arbitraire, à l'image du spectre de Zariski et du spectre réel. Par ailleurs \cite{fCoq2009} a introduit la version \cov du spectre valuatif d'un anneau intègre. Il s'agissait notamment pour nous de comparer ces deux notions en toute généralité.   

Dans les sections \ref{fsecdival0} à \ref{fsubsectrdisad} nous donnons des généralités sur les \trdis, les espaces spectraux, les \tgms, les \sads, et les relations entre ces différents objets.
La notion centrale de \sad correspond à l'idée intuitive d'une structure \agq incomplètement spécifiée. L'existence de modèles de ces structures est un outil omniprésent en \clama. Cette existence purement idéale a bien souvent comme seule contrepartie \cov le fait que la \sad en question ne s'effondre pas. Mais cela suffit en \gnl pour démontrer \cot les résultats concrets obtenus en \clama en utilisant l'existence purement idéale de ces modèles.   

La section \ref{fsecdival} est consacrée à plusieurs \tdys pour les \advs,
autour de celles nommées \sa{val} et \sa{Val}, lesquelles correspondent respectivement aux approches de \cite{fHK1994} et \cite{fCoq2009}.
Le \thref{fthValval} établit l'isomorphisme des \trdis (et donc celui des espaces spectraux) correspondant aux \sads définies dans les théories \sa{val} et \sa{Val} pour une \alg $\gk\to\gK$ lorsque~$\gk$ est un sous anneau d'un \cdi $\gK$.

La section \ref{fsubsecdival} aborde la dimension valuative des anneaux commutatifs et des \algs.
En particulier nous montrons que la dimension valuative d'un anneau commutatif $\gA$ est la \ddk du \trdi associé à la \sad $\sa{val}(\gA,\gA)$. Dans le cas d'un anneau intègre, la comparaison avec la dimension valuative définie \cot par \cite{fCoq2009} résulte de l'isomorphisme de \trdis établi au \thref{fthValval}. 
Nous établissons aussi \cot l'équivalence entre plusieurs approches \covs de différentes notions de dimension valuative d'un anneau ou d'une \alg.
Pour cela nous introduisons la clôture \qi minimale d'un anneau~$\gA$, qui sert de substitut à un trop hypothétique \gui{anneau intègre engendré par $\gA$}.  

La section \ref{fsecdival6} rappelle la \tdy \sa{Cvd} des \cvds donnée dans \cite{fCLR01}. L'objectif de cette section est semblable à celui de la section \ref{fsecdival}: on doit montrer que les \sads associées aux théories
 \sa{val},  \sa{Val} et  \sa{Cvd} sont essentiellement les mêmes.
Nous nous appuyons sur les \vsts formels établis pour \sa{Cvd} et développons des \vsts formels analogues pour  \sa{val} et \sa{Val}.
 Ceci donne une comparaison historique intéressante entre 
divers résultats de type \vst.   
En particulier le \vst formel~\ref{fthVstformelvalter} explique l'identité entre les certificats \agqs  issus des théories formelles développées (implicitement ou explicitement) dans \cite{fHK1994}, \cite{fCLR01} et \cite{fCoq2009}.

\smallskip \noindent {\bf Remerciements.} Nous remercions le rapporteur pour ses conseils fort pertinents.

\section{Treillis distributifs et espaces spectraux}\label{fsecdival0}
Références: \cite{fSto37,fCC00,fCL05} et \cite[Chapitres XI et XIII]{fACMC}.

\subsection{L'article fondateur de Stone}\label{fsubsecStoneAntiequiv}
L'article fondateur   \cite{fSto37} se traduit en langage moderne comme suit.

\textsl{La catégorie des \trdis est, en \clama, anti\eqve à la catégorie des espaces spectraux.}

Nous expliquons ceci un peu plus en détail.

\Subsubsection{Idéaux et filtres dans un \trdi}\label{fsubsecTrdiIdeFi}
 
Un \textsl{idéal} $\fb $ d'un \trdi $(\gT,\vi,\vu,0,1)$ est une
partie qui satisfait les contraintes:
\begin{equation}\label{feqIdeal}
\left.
\begin{array}{rcl}
  & & 0 \in \fb  \\
x,y\in \fb & \Longrightarrow  & x\vu y \in \fb  \\
x\in \fb ,\; z\in \gT& \Longrightarrow  & x\vi z \in \fb  \\
\end{array}
\right\}
\end{equation}
On note $\gT/(\fb=0)$ le treillis quotient obtenu en forçant les \elts de $\fb$ à être nuls. On peut aussi définir les \ids comme les noyaux de morphismes.

Un \textsl{\idp} est un \id engendré par un seul \elt $a$, il est égal à \hbox{$\dar a:=\sotq{x\in \gT}{x\leq a}$}. 
L'\id $\dar a$, muni des lois $\vi$ et $\vu$ de $\gT$ est un \trdi
dans lequel l'\elt maximum est $a$. L'injection canonique $\dar
a\rightarrow \gT$ n'est pas un morphisme de \trdis parce que
l'image de $a$ n'est pas égale à $1_\gT$. Par contre
l'application $\gT\rightarrow \dar a,\;x\mapsto x\vi a$
est un morphisme surjectif, qui définit donc $\dar a$ comme une
structure quotient $\gT/(a=1)$.

La notion de \textsl{filtre} est la notion opposée (\cad obtenue en renversant la relation d'ordre) à celle d'idéal.

Soient $\fa$ un \id et $\fv$ un filtre de $\gT$.
On dit que $(\fa,\fv)$ est un \textsl{couple saturé} dans~$\gT$ si
$$
(g\in \fv,\; x\vi g \in \fa) \Longrightarrow x\in \fa,
\hbox{ et }(a\in \fa,\; x\vu a \in \fv) \Longrightarrow x\in \fv.
$$

On peut aussi définir un couple saturé comme un couple $(\varphi^{-1}(0),\varphi^{-1}(1))$ pour un morphisme  $\varphi:\gT\to\gT'$ de \trdis.
Lorsque $(\fa,\fv)$ est un couple saturé, on a
les \eqvcs
$$
1\in \fa\; \;\Longleftrightarrow\;\; 0\in \fv
\;\; \Longleftrightarrow\;\; (\fa,\fv)=(\gT ,\gT )
$$

Si $A$ et $B$ sont deux parties de $\gT$ on note
\begin{equation}\label{feqvuvi}
A\vu B=\sotq{ a\vu b}  {a\in A,\,b\in B}  \; \hbox{ et } \; A\vi
B=\sotq{ a\vi b}  {a\in A,\,b\in B}.
\end{equation}
Alors l'\id engendré par deux \ids $\fa$ et $\fb$ est égal à
\begin{equation}\label{feqSupId}
\cI_\gT(\fa\cup \fb) = \fa\vu\fb.
\end{equation}

L'ensemble des \ids de $\gT$ forme lui-même un \trdi\footnote{En fait
il faut introduire une restriction pour obtenir vraiment un ensemble, de fa\c{c}on à ce que l'on ait un procédé bien défini de construction des \ids concernés.
Par exemple on peut considérer l'ensemble des \ids obtenus à partir des \idps par certaines opérations prédéfinies, comme les réunions et intersections dénombrables, les transporteurs etc.} pour
l'inclusion, avec pour borne inférieure de $\fa$ et $\fb$ l'\id 
\begin{equation}\label{feqInfId}
\fa\cap \fb=\fa\vi\fb.
\end{equation}
Ainsi les opérations $\vu$ et $\vi$ définies en (\ref{feqvuvi})
correspondent au sup et au inf dans le treillis des \ids.

Quand on considère le treillis des filtres, il faut
faire attention à ce que produit le renversement de la relation
d'ordre: $\fv_1\cap\fv_2=\fv_1\vu\fv_2$ est le inf des filtres~$\fv_1$ et $\fv_2$,
tandis que leur sup est le treillis engendré par $\fv_1\cup \fv_2$, égal à $\fv_1\vi\fv_2$.

\Subsubsection{Le spectre d'un \trdi} 

En \clama un {\em \id premier} $\fp$ d'un \trdi $\gT\neq \Un$ est un \id dont
le complémentaire $\fv$ est un filtre (qui est alors un {\em
filtre premier}).  On~a alors $\gT/(\fp=0,\fv=1)\simeq\Deux$.  Il
revient au même de se donner un \idep de $\gT$ ou un morphisme de
\trdis $\gT\rightarrow \Deux$.


On vérifie facilement que si $S$ est une partie génératrice du
\trdi $\gT$, un \idep~$\fp$ de $\gT$ est complètement
caractérisé par sa trace sur $S$ (cf.  \cite{fCC00}).

\begin{fdefinition} \label{fdefiSpecTrdi}
 Le \textsl{spectre d'un \trdi $\gT$} est l'ensemble $\Spec
\,\gT$ de ses \ideps, muni de la topologie suivante: une base
d'ouverts est donnée par les 
$$
\DT(a)\eqdefi\sotq{\fp\in\Spec
\,\gT}{a\notin\fp},\quad a\in \gT.
$$
 \end{fdefinition}
On vérifie en \clama que
\begin{equation} \label{feqDa}
\left.
\begin{array}{rclcrcl}
  \DT(a\vi b)   & =  & \DT(a)\cap \DT(b) ,&\quad & \DT(0)  & =  & 
\emptyset  ,\\
  \DT(a\vu b)   & =  & \DT(a)\cup \DT(b) ,&&  \DT(1) & =  &  
\Spec\,\gT.
  \end{array}
\right\}
\end{equation}

Le complémentaire de $\DT(a)$ est un fermé que l'on note $\VT(a)$.

On étend la notation $\VT(a)$ comme suit: si $I\subseteq\gT$, on
pose $\VT(I)\eqdefi\bigcap_{x\in I}\VT(x)$.  Si~$I$ engendre l'idéal $\fII$, on a $\VT(I)=\VT(\fII)$.  On dit parfois que $\VT(I)$ est \textsl{la
variété associée~à~$I$}.

\begin{fdefinition} \label{fdefiEspaceSpectral}
 Un espace topologique homéomorphe à un espace $\Spec(\gT)$
est appelé un \textsl{espace spectral}.  \end{fdefinition}

Les espaces spectraux proviennent de l'étude  \cite{fSto37}.

 \cite{fJoh1986}  appelle ces espaces des \textsl{espaces cohérents}.
\cite{fBW74} les appellent \textsl{Stone spaces}. Le nom \textsl{spectral space} est donné par \cite{fHoc1969}.

Avec la logique classique et l'axiome du choix, l'espace $\Spec (\gT)$
a \gui{suffisamment de points}: on peut retrouver le treillis $\gT$
à partir de son spectre.

\smallskip 
On dit qu'un point $\fp$ d'un espace spectral $X$ est le
\textsl{point générique du fermé $F$} \hbox{si $F=\ov{\so{\fp}}$}.  Ce point
(quand il existe) est nécessairement unique car les espaces spectraux
sont des espaces de Kolmogoroff.  En fait, les fermés $\ov{\so{\fp}}$
sont exactement tous les fermés irréductibles de $X$.  La relation
d'ordre $\fq\in\ov{\so{\fp}}$ sera notée $\fp\leq_X \fq$, et \hbox{l'on a} les équivalences
\begin{equation} \label {feqOrdreSpec}
\fp\leq_X \fq\;\Longleftrightarrow\; \ov{\so{\fq}}\subseteq \ov{\so{\fp}}\,.
\end{equation}
Les points fermés de $\Spec(\gT)$ sont les \idemas de $\gT$.
Lorsque $X=\Spec(\gT)$ la relation $\fp\leq_X \fq$ est simplement la
relation d'inclusion usuelle $\fp\subseteq \fq$ entre \ideps du \trdi $\gT$.

\Subsubsection{L'antiéquivalence de Stone}

Tout d'abord on a le \tho de Krull.


\medskip\noindent
{\bf \Tho de Krull }\label{fThKrull} (en \clama)\\ 
{\it Soient $\fa$  un \id et $\fv$ un filtre du \trdi $\gT$. Supposons que 
$\fa\cap\fv=\emptyset$.
Alors il existe un \idep $\fp$ tel \hbox{que $\fa\subseteq\fp$} et
$\fp\cap\fv=\emptyset$.
  }

\medskip
On en déduit que:
\begin{itemize}
\item L'application $a\in\gT\,\mapsto\,\DT(a)\in\cP(\Spec\,\gT)$
est injective: elle identifie $\gT$ à un treillis d'ensembles 
(\textsl{\tho de
représentation de Birkhoff}).
\item Si $\varphi : \gT\to\gT'$ est un \homo injectif l'application
$\varphi^\star:\Spec\,\gT'\to\Spec\,\gT$ obtenue par dualité est 
surjective.
\item Tout \id de $\gT$ est intersection des \ideps qui le 
contiennent.
\item L'application $\fII\mapsto \VT(\fII)$, des \ids de $\gT$ vers 
les fermés
de $\Spec\,\gT$ est un \iso d'ensembles ordonnés (pour l'inclusion 
et
l'inclusion renversée).
\end{itemize}

On montre aussi les résultats suivants.
\begin{itemize}
\item Les \oqcs de $\Spec \,\gT$ sont exactement les
$\DT(a)$. 
\item Les \oqcs de $\Spec \,\gT$
forment un \trdi de parties de~$\Spec \,\gT$, isomorphe à $\gT$ (d'après les \egts \eqref{feqDa}). 
\item Si  $X$ est un espace spectral, ses \oqcs forment un \trdi, noté
$\OQC(X)$. 
\item Pour tout \trdi $\gT$,
$\OQC(\Spec(\gT))$ est canoniquement isomorphe à $\gT$ et pour tout
espace spectral $X$, $\Spec(\OQC(X))$ est canoniquement 
homéomorphe~à~$X$. 
\end{itemize}
 
\begin{fdefinition} \label{fdefiApplispectrale}
Tout \homo $\varphi :\gT\rightarrow \gT'$ de \trdis fournit
\textsl{par dualité} une application continue 
\[\varphi^\star:\Spec\,\gT'\rightarrow \Spec \,\gT, \,\fp\mapsto \varphi^{-1}(\fp)
\]
qui est appelée une \textsl{application spectrale}. 
\end{fdefinition}

Pour qu'une application entre
espaces spectraux soit spectrale il faut et il suffit que l'image
réciproque de tout \oqc soit un \oqc. En particulier elle doit être continue.

\smallskip L'article fondateur de Stone donne la \carn suivante des espaces spectraux:
\begin{itemize}
\item L'espace est de Kolmogoroff (i.e., de type $\mathrm{T}_0$): 
étant donnés deux points il existe un voisinage de l'un des deux qui ne contient pas l'autre.
\item L'espace est \qc.
\item L'intersection de deux \oqcs est un \oqc.
\item Tout ouvert est réunion d'\oqcs.
\item Pour tout fermé $F$ et pour tout ensemble $S$ d'\oqcs tels 
que 
\[
\textstyle F\cap
\bigcap_{U\in S'} U\neq \emptyset\,\hbox{   pour toute partie finie  }\,S'
\,\hbox{  de  }\,S
\]
on a aussi
$F\cap \bigcap_{U\in S} U\neq \emptyset$.
\end{itemize}

\smallskip En présence des quatre premières propriétés,  \cite{fHoc1969} montre que la dernière  peut se
reformuler comme suit: 
\begin{itemize}
\item tout fermé irréductible admet un point
générique (i.e., est l'adhérence d'un point).
\end{itemize}

\smallskip L'article de Stone démontre pour l'essentiel que la catégorie des espaces spectraux et applications spectrales est antiéquivalente à
celle des \trdis (\cite[{II-3.3}, coherent locales]{fJoh1986},\cite{fSto37}).
La formulation moderne précise est la suivante.
\begin{ftheorem} \label{fthStoneAntiequiv}
Les foncteurs contravariants $\Spec$ et $\OQC$ entre la catégorie des \trdis et celle des espaces spectraux définissent une antiéquivalence. 
\end{ftheorem}

En \coma les espaces spectraux manquent souvent de points, et on essaie de traduire les discours des \clama sur les espaces spectraux (très fréquents en algèbre) en des discours \cofs sur les \trdis correspondants.

Un sous-espace topologique $Y$ d'un espace spectral $X$ est appelé un sous-espace spectral si le morphisme d'inclusion $Y\to X$ est une application spectrale.  La notion de sous-espace spectral est traduite par la notion de \trdi quotient (\thref{fpropSESP}). On~a \egmt des traductions satisfaisantes par exemple pour les notions de dimension de Krull, d'espace spectral normal, de morphismes lying over, going up et going down (voir la section~\ref{fsubsecAntiEquiv}). 

\smallskip 
Lorsque l'on remplace $\gT$ par le treillis opposé $\gT\oop$, obtenu en renversant la relation d'ordre, les points du spectre $X$ restent les mêmes en \clama et la topologie est remplacée par la \textsl{topologie opposée}, pour laquelle les \oqcs sont les fermés $\fV_\gT(a)$.

Il y a aussi la \textsl{topologie constructible} pour laquelle les \oqcs sont les combinaisons booléennes des $\fO_\gT(a)$ et $\fV_\gT(b)$. C'est le   dual de l'\agB $\Bo(\gT)$ engendrée par le \trdi $\gT$. Ce type d'espace spectral est aujourd'hui appelé un \textsl{espace de Stone}. 

\Subsubsection{Espaces spectraux finis}
 
 Dans le cas des \trdis \textsl{finis} on obtient les espaces spectraux
 finis, qui ne sont rien d'autre que les ensembles ordonnés finis,
 (car il suffit de connaître l'adhérence des points pour
 connaître la topologie) avec pour base d'ouverts les $\dar a$. 
 Les ouverts sont tous quasi-compacts, ce sont les parties initiales, et les
 fermés sont les parties finales.  Enfin, une application entre
 espaces spectraux finis est spectrale \ssi elle est croissante (pour
 les relations d'ordre associées).

La notion d'espace spectral apparaît ainsi comme une
généralisation pertinente de la notion d'ensemble ordonné fini
au cas infini. Voir \cite[Théorème XI-5.6, dualité entre ensembles ordonnés finis et \trdis finis]{fACMC}.

Dans le cas fini, si l'on identifie les ensembles sous-jacents à $\Spec\,\gT$ et $\Spec\,\gT\oop$ les deux spectres sont presque les mêmes: c'est le
même ensemble ordonné au renversement près de la relation
d'ordre.  En outre les ouverts et les fermés sont simplement échangés.

\subsection{Treillis distributifs et relations implicatives}

Une règle particulièrement importante
pour les \trdis, appelée \textsl{coupure}, est la
suivante
\begin{equation}\label{fcoupure1}
\hbox{si } \,  x\vi a \leq  b  \, \hbox{ et } \,  a \leq x\vu  b 
\,\hbox{ alors } \, (a \leq  b).
\end{equation}

Si $A\in\Pfe(\gT)$ (ensemble des parties finiment énumérées de~$\gT$)  on notera
\[\ndsp \Vu  A:=\Vu _{x\in A}x\qquad {\rm et}\qquad \Vi  A:=\Vi _{\!x\in A}x.
\]
On  note $A \vda B$ ou $A \vdash_\gT B$ la relation définie comme suit  sur l'ensemble $\Pfe(\gT)$:

\snic{A \vda B \; \; \equidef\; \; \Vi  A\;\leq \;
\Vu  B.}

Cette relation vérifie les axiomes suivants, dans lesquels on
écrit $x$ pour $\{x\}$ et $A, B$  \hbox{pour $ A\cup B$}.

\vspace{-.5em}
\[\arraycolsep3pt\begin{array}{rcrclll}
&    & x  &\vda& x    &\; &(R)     \\[1mm]
 \hbox{si } A \vda B &   \hbox{alors}  & A,A' &\vda& B,B'   &\; &(M)     \\[1mm]
\hbox{si } (A,x \vda B) \hbox{ et }  (A \vda B,x) 
& \hbox{alors} & A &\vda& B &\;
&(T).
\end{array}
\]
On dit que la relation est \textsl{réflexive}, \label{fremotr} \textsl{monotone} et
\textsl{transitive}.
La troisième règle (transitivité) peut être vue comme une
réécriture de la règle (\ref{fcoupure1}) et s'appelle \egmt la
règle de \textsl{coupure}.

\begin{fdefinition}
\label{fdefEntrel}
Pour un ensemble $S$ arbitraire, une relation binaire sur  $\Pfe(S)$  qui est
réflexive, monotone et transitive est
appelée une {\em \entrel} (en anglais, {\em entailment relation}).
\end{fdefinition}

Le \tho suivant est fondamental. Il dit que les
trois propriétés des \entrels sont exactement ce qu'il faut pour que
l'interprétation d'une relation implicative comme la trace de celle d'un \trdi soit adéquate.

\begin{ftheorem}[\Tho fondamental des \entrels]
\label{fthEntRel1} {\rm   \cite{fCC00}, \cite[\hbox{XI-5.3}]{fACMC}, \cite[Satz 7]{fLor1951}}.
Soit un ensemble~$S$  avec une \entrel
$\vdash_S$ sur $\Pfe(S)$. On considère le \trdi~$\gT$ défini par
\gtrs et relations comme suit: les \gtrs sont les
\elts de $S$ et les relations sont les

\snic {A\, \vdash_\gT \,  B}

\noindent chaque fois que $A\, \vdash_S \, B$.  Alors, pour tous $A$,
 $B$ dans $\Pfe(S)$,  on a

\snic {\hbox{si }A\, \vdash_\gT \,  B
\hbox{ alors } A\, \vdash_S \,  B.}

\noindent 
En particulier, deux \elts $x$ et $y$ de $S$ définissent le même \elt de $\gT$ \ssi
\hbox{on a} $x \vdash_S   y$ et $y \vdash_S   x$.
\end{ftheorem}

\smallskip 
\rem La relation $x\vdash_S y$ est à priori une relation de préordre, et non une relation d'ordre, sur~$S$. Notons $\ov x$ l'\elt $x$ vu dans l'ensemble ordonné~$\ov{S}$ associé à ce préordre, et pour une partie~$A$ de~$S$ notons $\ov A=\sotq{\ov x}{x\in A}$.
Dans l'énoncé du \tho on considère un \trdi $\gT$ qui donne sur $S$ la même \entrel que $\vdash_S$.
En toute rigueur, on aurait dû noter $ \ov A\, \vdash_\gT \, \ov B$
plutôt que $A\, \vdash_\gT \,  B$ pour tenir compte du fait que l'\egt dans $\gT$ est plus grossière que dans $S$. En particulier c'est $\ov{S}$, et non pas $S$, qui s'identifie à une partie de $\gT$.
\eoe

\subsection{Recollement de \trdis et d'espaces spectraux}

\Subsubsection{Quotients, recouvrements, recollements}

Un \textsl{\trdi quotient $\gT'$ de $\gT$} est donné par une relation 
binaire
$\preceq$ sur $\gT$ vérifiant les propriétés suivantes:
\begin{equation} \label{feqPreceq}
\left.
\begin{array}{rcl}
a\leq b&  \Longrightarrow  & a\preceq b   \\
a\preceq b,\,b\preceq c&  \Longrightarrow  & a\preceq c   \\
a\preceq b,\,a\preceq c&  \Longrightarrow  & a\preceq b\vi c   \\
b\preceq a,\,c\preceq a&  \Longrightarrow  & b\vu c\preceq a
\end{array}
\right\}
\end{equation}

\begin{fproposition}
\label{fpropIdealFiltre} Soient $\fa$ un \id et $\fv$ un filtre d'un \trdi $\gT$.
On considère le quotient $\gT'$ de $\gT$ défini par les
relations $x=0$ pour les $x\in \fa$ et $y=1$ pour les $y\in \fv$. Alors
  on a $x\leq_{\gT'}y$ \ssi
il existe  $a\in\fa$ et $v\in\fv$ tels que:
\begin{equation} \label{feqpropIdealFiltre}
x \vi v \; \leq_\gT\; y \vu a
\end{equation}
Nous noterons $\gT/(\fa=0,\fv=1)$ ce treillis quotient $\gT'.$
\end{fproposition}

En particulier, si le couple $(\fa,\fv)$ est saturé, le morphisme naturel 
$\varphi:\gT\to\gT'$
vérifie $\varphi^{-1}(0_{\gT'})=\fa$ et $\varphi^{-1}(1_{\gT'})=\fv$. 

Dans le cas d'un \idp  $\dar a$ on obtient $\gT/(\dar a=0)\simeq\uar a$ avec le morphisme $y\mapsto y\vu a$ de $\gT$ vers $\uar a$.

\smallskip 
En \alg commutative, si $\fa$ et $\fb$ sont deux \ids d'un anneau 
$\gA$
on a une suite exacte de \Amos 
\[
0\to\gA/(\fa\cap\fb)\vers{j}(\gA/\fa) \times 
(\gA/\fb)\vers{\delta}\gA/(\fa+\fb)\to 0 \quad (\delta(x,y)=x-y \mod \fa+\fb)
\]
qu'on peut lire en langage courant: le système de congruences  
$x\equiv
a \mod \fa$, $x\equiv b \mod \fb$ admet une solution \ssi $a\equiv
b\ \mod \fa+\fb$, et dans ce cas la solution est unique modulo 
$\fa\cap\fb$.
Il est remarquable que ce \gui{\tho des restes chinois} se 
généralise à un
système \textsl{quelconque} de congruences \ssi l'anneau est
\textsl{arithmétique} \cite[Theorem XII-1.6]{fACMC}, \cad si le treillis des \ids est distributif
(le \tho des restes chinois \gui{contemporain} concerne le cas 
particulier d'une
famille d'\ids deux à deux comaximaux, et il fonctionne sans hypothèse sur
l'anneau de base).

D'autres épimorphismes de la catégorie des anneaux commutatifs 
sont les localisations. Et un principe de recollement pour les localisations, extrêmement fécond, est le principe local-global): \textsl{si $(\xn)$ est un système d'\eco d'un anneau $\gA$
le morphisme $\gA\to\prod_{i\in\lrbn}\gA[1/x_i]$ identifie $\gA$ à un} sous-produit fini\footnote{Un sous-objet du produit fini en question, dans la catégorie des anneaux commutatifs.} \textsl{de ses localisés} (voir le principe \lgb concret XV-4.2 dans~\cite{fACMC} et aussi II-2.3, XV-2.1, XV-2.2, XV-2.3, XV-2.4 et XV-2.5). 
Ces résultats sont des versions \covs de \plgas qui disent que certaines \prts d'un \Amo ou d'une \Alg  sont satisfaites \ssi elles sont satisfaites après localisation en n'importe quel \idep (ou, parfois, au voisinage de n'importe quel \idep).

\smallskip De la même manière, on peut reconstruire un \trdi à partir
d'un nombre fini de ses quotients
si l'information que ces derniers contiennent est \gui{suffisante}. 
On peut voir ceci au choix comme une procédure de recollement 
(de passage du local au global), ou comme une version du \tho des restes chinois pour les \trdis. 

La notion de recouvrement  d'un \trdi par des treillis quotients est la notion duale de celle de recouvrement d'un espace spectral par des sous-espaces spectraux. Mais il ne suffit pas que des sous-espaces spectraux recouvrent l'espace total au sens d'une réunion ensembliste pour que le processus de reconstruction de la topologie spectrale de l'espace total à partir de celles de ses sous-espaces fonctionne correctement. De même la notion de recollement d'espaces spectraux au moyen d'isomorphismes de recollement compatibles entre eux trois par trois ne fonctionne pas si l'on utilise des isomorphismes de recollement le long de sous-espaces spectraux arbitraires.

C'est pourquoi dans les \thos \ref{ffactRecolTD} et \ref{fpropRecolTD} qui suivent nous nous limitons à des treillis quotients bien particuliers.

Ces résultats sont des analogues, pour la catégorie des \trdis, des résultats pour des notions similaires dans la catégorie des \grls (voir \cite[Principes de recouvrement XI-2.10 et XI-2.21]{fACMC}) et dans la catégorie des modules sur un anneau commutatif (voir \cite[Principes de recouvrement XI-4.19 et XIII-3.3, Principes de recollement XV-4.4, XV-4.4 bis et XV-4.6]{fACMC}).

\smallskip Voyons ces choses plus précisément.

\begin{fdefinition}
\label{fdefRecolTD}
Soit $\gT$ un \trdi, $(\fa_i)_{i=1,\ldots n}$ (resp. 
$(\fv_i)_{i=1,\ldots n}$)
une famille finie d'\ids (resp. de filtres)  de $\gT$.  On dit que 
les \ids
$\fa_i$ \textsl{recouvrent $\gT$} si $\bigcap_i\fa_i=\so{0}$. De 
même on dit
que les filtres $\fv_i$ \textsl{recouvrent $\gT$} si 
$\bigcap_i\fv_i=\so{1}$.
\end{fdefinition}

\smallskip Rappelons que pour $s\in\gT$ le quotient $\gT/(s=0)$ est isomorphe au filtre principal $\uar s$ que l'on voit comme un \trdi dont l'\elt minimum est $s$.

\begin{ftheorem}[Recouvrement d'un \trdi par des quotients convenables]
\label{ffactRecolTD}~\\
Pour un \id $\fb$ nous écrivons $x\equiv y\mod \fb$ comme 
abréviation pour  $x\equiv y$ \hbox{$\mod(\fb=0)$}.
Soit $\gT$ un \trdi, $(\fa_i)_{i=1,\ldots n}$ une famille finie 
d'\ids principaux ($\fa_i=\dar s_i$)  de~$\gT$ et $\fa=\bigcap_i\fa_i= \dar \Vi_i s_i $.
\begin{enumerate}
\item Si $(x_i)$ est une famille d'\elts de $\gT$ et si pour 
chaque $i,j$ on
a $x_i\equiv x_j\mod \fa_i\vu\fa_j$, alors il existe un unique $x$ 
modulo
$\fa$ vérifiant:  $x\equiv x_i\mod \fa_i\;(i=1,\ldots ,n)$.
\item Notons $\gT_i=\gT/(\fa_i=0)$, 
$\gT_{ij}=\gT_{ji}=\gT/(\fa_i\vu \fa_j=0)$,
$\pi_i:\gT\to\gT_i$ et $\pi_{ij}:\gT_i\to\gT_{ij}$ les projections 
canoniques.
Si les $\fa_i$ recouvrent $\gT$
, alors $(\gT,(\pi_i)_{i=1,\ldots n})$  est 
la limite
projective du diagramme $\big((\gT_i)_{1\leq i\leq n},(\gT_{ij})_{1\leq 
i<j\leq
n};(\pi_{ij})_{1\leq i\neq j\leq n}\big)$ (figure ci-dessous)
\end{enumerate}

{\hspace*{10em}{
\xymatrix @R=2em @C=7em{
          &  \gT \ar[rd]^{\pi _{k}}\ar[d]^{\pi _{j}}\ar[ld]_{\pi _{i}}\\
 \gT _i\ar[d]_{\pi _{ij}}\ar@/-0.75cm/[dr]^{\pi _{ik}} &
     \gT _j\ar@/-.75cm/[dl]_{\pi _{ji}}\ar@/-.75cm/[dr]^{\pi _{jk}} &
        \gT _k\ar@/-0.75cm/[dl]_{\pi _{ki}}\ar[d]^{\pi _{kj}} &
\\
 \gT _{ij}  & 
    \gT _{ik}   & 
      \gT _{jk}   
}
}}
\begin{enumerate}\setcounter{enumi}{2}
\item Soit maintenant $(\fv_i)_{i=1,\ldots n}$ une famille finie de 
filtres principaux,
notons $\gT_i=\gT/(\fv_i=1)$, 
$\gT_{ij}=\gT_{ji}=\gT/(\fv_i\vi\fv_j=1)$,
$\pi_i:\gT\to\gT_i$ et $\pi_{ij}:\gT_i\to\gT_{ij}$ les projections 
canoniques.
Si les $\fv_i$ recouvrent $\gT$, $(\gT,(\pi_i)_{i=1,\ldots n})$  est 
la limite
projective du diagramme $\big((\gT_i)_{1\leq i\leq n},(\gT_{ij})_{1\leq 
i<j\leq
n};(\pi_{ij})_{1\leq i\neq j\leq n}\big)$.
\end{enumerate}
\end{ftheorem}

%

\medskip Il y a aussi une procédure de recollement proprement dit dans certains cas particuliers pour les treillis quotients $\gT_i\to \gT_{ij}$ que l'on veut recoller.

\begin{fdefinition}[Morphisme de passage au quotient entre \trdis] \label{fdefimorquottreillis}~\\
Soit $\gT$ un \trdi et $u\in \gT$.
On peut identifier le filtre principal $\uar u$ au quotient $\gT/(\dar u=0)$ via le morphisme $p_u:x\mapsto x\vu u$. 
Plus \gnlt, un morphisme de \trdis $\alpha:\gT\to \gT'$ est appelé un \textsl{morphisme de passage au quotient par l'idéal $\dar u$} s'il existe un \iso (\ncrt unique) $\lambda:\gT'\to \gT/(u=0)$ tel que $\alpha=\lambda\circ p_u$.
\end{fdefinition}

\begin{flemma}[Dans un \trdi, les quotients principaux sont \gui{scindés}] \label{flemquoprinctrdi} ~ 
\\
Soit $\pi:\gT  \to \gT'$ un morphisme de \trdis et $s\in \gT$.
\Propeq
\begin{enumerate}
\item $\pi$ est un morphisme de passage au quotient de $\gT$ par l'idéal principal $\fa=\dar s$.
\item Il existe un morphisme $\varphi:\gT'\to\,\uar s$ tel que
$\pi\circ \varphi=\Id_{\gT'}$.
\end{enumerate}
Dans ce cas $\varphi$ est uniquement déterminé par $\pi$ et $s$.\\
Naturellement, l'énoncé analogue \gui{renversé} est valable pour un quotient par un filtre principal. 
\end{flemma}
%

\begin{ftheorem}[Recollement de \trdis]
\label{fpropRecolTD}
  Supposons donnés un ensemble fini totalement ordonné~$I$ et dans la catégorie des \trdis  un diagramme

\snic{\big((\gT_i)_{i\in I},(\gT_{ij})_{i<j\in I},(\gT_{ijk})_{i<j<k\in I};
(\pi_{ij})_{i\neq j},(\pi_{ijk})_{i< j, j\neq k\neq i}\big)}

\noindent 
comme dans la figure ci-après, 
ainsi qu'une famille d'\elts 

\snic
{(s_{ij})_{i\neq j\in I}\in \prod\nolimits_{i\neq j\in I}\gT_{i}}

\noindent satisfaisant les conditions suivantes:
\begin{itemize}
\item le diagrammme est commutatif ($\pi_{ijk}\circ \pi_{ij}=\pi_{ikj}\circ \pi_{ik}$ pour tous $i$, $j$, $k$ distincts), 
\item pour $i\neq j$, $\pi_{ij}$ est un morphisme de passage au quotient par l'\id $\dar s_{ij}$,
\item pour $i$, $j$, $k$ distincts, $\pi_{ij}(s_{ik})=\pi_{ji}(s_{jk})$ et  $\pi_{ijk}$ est un morphisme de passage au quotient par \hbox{l'\id $\dar\pi_{ij}(s_{ik})$}.   
\end{itemize}

\medskip  {\hspace*{10em}\small
\xymatrix @R=2em @C=7em{
 \gT_i\ar[d]_{\pi _{ij}}\ar@/-0.75cm/[dr]^{\pi _{ik}} &
     \gT_j\ar@/-.75cm/[dl]_{\pi _{ji}}\ar@/-.75cm/[dr]^{\pi _{jk}} &
        \gT_k\ar@/-0.75cm/[dl]_{\pi _{ki}}\ar[d]^{\pi _{kj}} &
\\
 ~\gT_{ij}~ \ar[rd]_{\pi _{ijk}} & 
    ~\gT_{ik}~  \ar[d]^{\pi _{ikj}} & 
      ~\gT_{jk}~  \ar[ld]^{\pi _{jki}} 
\\
   &  ~\gT_{ijk}~ 
\\
}
}

\smallskip \noindent Alors si $\big(\gT\,;\,(\pi_i)_{i\in I}\big)$ est la limite projective du diagramme, les~\hbox{$\pi_i:\gT\to \gT_i$} forment un recouvrement par quotients principaux de $\gT$, et le diagramme est isomorphe à celui obtenu
dans le \thref{ffactRecolTD}.
Plus précisément, il existe des $s_i\in\gT$ tels que chaque~$\pi_i$ est un morphisme de passage au quotient par l'\id $\dar s_i$ et $\pi_i(s_j)=s_{ij}$ pour tous $i\neq j$.

\noindent Le résultat analoque est valable pour les quotients par des filtres principaux.
\end{ftheorem}

\Subsubsection{Le point de vue dual}

\begin{fdefinition} \label{fdefisous-espace-spectral}
Une partie $X'$ d'un espace spectral $X$ est appelée un \textsl{sous-espace
spectral} lorsque la topologie induite par $X$  fait de $X'$ un espace spectral et que
 $\OQC(X')=\sotq{U\cap X'}{U\in\OQC(X)}$ (autrement dit l'injection canonique est une application spectrale). 
\end{fdefinition}

Le théorème suivant explique que la notion de \textsl{sous-espace spectral} est traduite par la notion de \textsl{\trdi quotient} dans l'antiéquivalence de catégories expliquée dans la section \ref{fsubsecStoneAntiequiv}.  

\begin{theoremc}[Caractérisations des sous-espaces spectraux]
\label{fpropSESP} ~   
\begin{enumerate}
\item  Soit $\gT'$ un treillis quotient de $\gT$ et $\pi:\gT\to\gT'$ la
projection canonique.  Notons \hbox{$X'=\Spec\,\gT'$}, $X=\Spec\,\gT$ et
$\pi^\star:X'\to X$ l'injection duale de $\pi$. Alors l'injection~$\pi^\star$ identifie~$X'$ à un sous-espace
spectral de $X$. 
\item Une partie $X'$ d'un espace spectral $X$ est un sous-espace 
spectral \ssi elle est fermée pour la topologie constructible.
\item Si $Z$ est une partie arbitraire d'un espace spectral 
$X=\Spec\,\gT$, son adhérence pour la topologie constructible est égale à 
$X'=\Spec\,\gT'$, où
$\gT'$ est le treillis quotient de~$\gT$ défini par la relation de 
préordre $\preceq$ suivante:
\begin{equation} \label{feqSSES}
a\preceq b\quad \Longleftrightarrow\quad (\fDT(a)\cap Z)\subseteq 
(\fDT(b)\cap Z)
\end{equation}
\end{enumerate}
\end{theoremc}
Un recollement de \trdis selon le  \thref{fpropRecolTD} correspond à un recollement d'espaces topologiques donnés avec des intersections 2 à 2 compatibles trois par trois. On peut recoller ainsi un nombre fini d'espaces spectraux donnés avec des intersections 2 à~2 qui sont des \oqcs\footnote{Variante: des fermés complémentaires d'\oqcs. C'est le même résultat, mais pour le treillis opposé et la topologie opposée sur le spectre.}  de chaque espace compatibles trois par trois\footnote{Dans le \thref{fpropRecolTD} la compatibilité trois par trois est donnée simplement par un \trdi $\gT_{ijk}$. Lorsqu'on recolle des espaces topologiques, on a à priori trois versions différentes de $U_1\cap U_2\cap U_3$ contenues respectivement dans $U_1$, $U_2$ et $U_3$. On doit d'abord identifier les deux versions de $U_i\cap U_j$ contenues dans $U_i$ et $U_j$ au moyen d'homéomorphismes. Et les trois identifications doivent être compatibles pour les trois versions de $U_1\cap U_2\cap U_3$.}. 

Ici, alors que le recollement des espaces spectraux est une évidence intuitive facile à démontrer en \clama, correspondant à la figure ci-dessous, la \demo du \thref{fpropRecolTD} est assez subtile (voir \cite{fCLQ2006}).

{
\smallskip 
\centerline{\includegraphics[width=11cm]{recollement.pdf}}
}

Dans la sous-section qui suit nous donnons d'autres exemples de comparaisons éclairantes entre \trdis et espaces spectraux duaux.

\subsection{Petit dictionnaire de l'antiéquivalence de catégories}
\label{fsubsecAntiEquiv}

Références: le théorème de Krull page~\pageref{fThKrull}, \cite[Theorem~IV-2.6]{fBW74}, \cite{fCC00,fCL2001-2018,fLom2020}.

Nous signalons ici sans démonstrations quelques résultats  déjà établis par ailleurs qui précisent l'antiéquivalence de catégories entre \trdis et espaces spectraux. Nous ajoutons les lemmes \ref{flemDimQuot}, \ref{flemDimProd} et le (nouveau) \thref{flemrcfdimtrdi} qui sera utilisé dans la section~\ref{fsubsecdival}.

\smallskip Le contexte est le suivant: soit $f:\gT\to\gT'$ un morphisme de treillis distributifs et $\Spec(f)$, noté $f\sta$, le morphisme dual, de $X'=\Spec(\gT')$ vers $X=\Spec(\gT)$, dans la catégorie des espaces spectraux.

Rappelons quelques \dfns usuelles en \clama.
\begin{itemize}
\item Le morphisme $f$ est dit \textsl{lying over} (en français, il possède la propriété de relèvement) lorsque~$f\sta$ est surjectif: tout \idep de $\gT$ est image réciproque d'un \idep de~$\gT'$.
\item  Le morphisme $f$ est dit \textsl{going up} (en français, il possède la propriété de montée pour les chaines d'\ideps) lorsque l'on a: \textsl{si $\fq\in X'$, $f\sta(\fq)=\fp$, et $\fp\subseteq\fp_2$ dans~$ X $, il existe  $\fq_2\in X'$ tel que
$\fq\subseteq\fq_2$ et $f\sta(\fq_2)=\fp_2$}.
\item  De même $f$ est dit \textsl{going down} (en français, il possède la propriété de descente pour les chaines d'\ideps) lorsque l'on a: \textsl{si $\fq\in X'$, $f\sta(\fq)=\fp$, et $\fp\supseteq\fp_2$ dans~$ X $, il existe  $\fq_2\in X'$ tel que $\fq\supseteq\fq_2$ et  $f\sta(\fq_2)=\fp_2$}.
\item  On dit que le morphisme $f$ \textsl{possède la propriété d'incomparabilité} lorsque ses  \gui{fibres} sont formées d'\ideps deux à deux incomparables: si $\fq_1\subseteq \fq_2\in X$ et $f\sta(\fq_1)=f\sta(\fq_2)$ dans~$X'$ alors $\fq_1= \fq_2$.
\item  L'espace spectral $ X $ est dit \textsl{normal} si tout point est majoré par un unique point fermé (tout \idep de $\gT$ est contenu dans un unique \idema).
\item L'espace spectral  $\SpecT$ est dit \textsl{complètement normal} si pour tous $x,y,z$ avec $x\in\ov{\so z}$ et $y\in\ov{\so z}$ on a $x\in\ov{\so y}$ ou 
$y\in\ov{\so x}$. 
\end{itemize}

\Subsubsection{Propriétés des morphismes}

\begin{ftheorem} \emph{\cite[Theorem~IV-2.6  et section V-8]{fBW74}}  \label{fth-dico-trdi-spec-mor1} 
 En \clama on a les équivalences suivantes.   
\begin{enumerate}
\item $f\sta$ est surjectif ($f$ est lying over) $\Longleftrightarrow$  $f$ est injectif $\Longleftrightarrow$ $f$ est un monomorphisme $\Longleftrightarrow$ $f\sta$ est un épimorphisme.
\item $f$ est un épimorphisme $\Longleftrightarrow$ $f\sta$ est un monomorphisme $\Longleftrightarrow$ $f\sta$ est injectif.
\item $f$ est surjectif\footnote{Autrement dit, puisque c'est une structure équationnelle, $f$  est un morphisme de passage au quotient.} $\Longleftrightarrow$ $f\sta$ est un isomorphisme sur son image, qui est un
sous-espace spectral de $X'$.
\end{enumerate} 
\end{ftheorem}

 Il y a des épimorphismes de \trdis non surjectifs.
Cela correspond à la possibilité d'un morphisme bijectif entre espaces spectraux qui ne soit pas un isomorphisme. Par exemple le morphisme spectral bijectif  $\Spec\Bo(\gT)\to\Spec\gT$ n'est pas (en général) un \iso et le morphisme de treillis correspondant $\gT\to\Bo(\gT)$ est un \gui{épimono} qui n'est pas (en général) surjectif.

\begin{ftheorem} \label{fth-dico-trdi-spec-mor2} 
 En \clama on a les équivalences suivantes.   
\begin{enumerate}
\item $f$ est going up $\Longleftrightarrow$  pour
tous $a,c\in\gT$ et $y\in\gT'$ on a
$$
f(a)\leq f(c)\vu y \;\Rightarrow\;\exists x\in\gT\; (a\leq c \vu x \hbox{ et } f(x)\leq y).
$$ 
\item $f$ est going down $\Longleftrightarrow$  pour
tous $a,c\in\gT$ et $y\in\gT'$ on a
$$
f(a)\geq f(c)\vi y \;\Rightarrow\;\exists x\in\gT\; (a\geq c \vi x \hbox{ et } f(x)\geq y).
$$
\item $f$ possède la propriété d'incomparabilité $\Longleftrightarrow$ $f$ est zéro-dimensionnel\footnote{Voir ci-dessous \ref{fth-dico-trdi-spec-dim2}.}.   
\end{enumerate} 
\end{ftheorem}
%
\begin{proof}
Voir \cite{fCL2001-2018}.
\end{proof}
%

\begin{ftheorem}[Open spectral map, \cite{fLom2020}] \label{fth-dico-trdi-spec-mor3} 
En \clama \propeq   
\begin{enumerate}
\item $\Spec(f)$ est une application ouverte.
\item  Il existe une application $\wi f:\gT'\to \gT$
satisfaisant les \prts suivantes.
\begin{enumerate}
\item \label{fi2a} Pour tous $c\in\gT$ et $b\in\gT'$, $b\leq f(c) \Leftrightarrow \wi f(b)\leq c$. \\
En particulier $b\leq f(\wi f(b))$ et $\wi f(b_1\vu b_2)=\wi f(b_1)\vu \wi f(b_2)$.
\item \label{fi2b} Pour tous $a,c\in\gT$ et $b\in\gT'$, $f(a)\vi b\leq f(c) \Leftrightarrow a\vi\wi f(b)\leq c $.
\item \label{fi2c} Pour tous $a\in\gT$ et $b\in\gT'$, $\wi f(f(a)\vi b)=a\vi \wi f(b)$.
\item \label{fi2d} Pour tout $a\in\gT$, $\wi f(f(a))=\wi f(1)\vi a$. 
\end{enumerate}
\item  Il existe une application $\wi f:\gT'\to \gT$
satisfaisant la \prt \ref{fi2b}.  
\item  Pour tout $b\in \gT$ la \bif \smash{$\Vi\limits_{b\leq f(c)} c$} existe, et si on la note $\wi f(b)$, 
 la \prt~\ref{fi2b} est satisfaite.  
\end{enumerate} 
\end{ftheorem}

\Subsubsection{Propriétés de dimension}

La dimension d'un espace spectral au sens des chaines de fermés irréductibles est aussi appelée sa \textsl{dimension de Krull}.  
La dimension de l'espace spectral vide est égale à $-1$.
Le \thref{fth-dico-trdi-spec-dim1} explique pourquoi on choisit en \coma de définir la dimension de Krull d'un \trdi $\gT$, notée $\Kdim\gT$, selon le point \emph{2}. L'initialisation se fait avec la dimension $-1$ pour le treillis trivial, réduit à un point.
Comme l'\eqvc entre les points \textsl{2} et \textsl{3} est \cov,  le point \textsl{3} donne aussi une \dfn \cov acceptable de la dimension de Krull d'un \trdi pour $n\geq 0$.

L'article \cite{fCL2003} explique le lien entre les points \textsl{1} et \textsl{3} dans le \tho  ainsi que le lien avec l'approche pionnière de \cite{fJoy76}. Le point \textsl{2} dans le \tho remonte à \cite{fCLR05}.

\begin{ftheorem}[Dimension des espaces] \emph{Voir \cite{fCL2003,fCLR05}.} \label{fth-dico-trdi-spec-dim1}   
En \clama, pour un \trdi non trivial et pour $n\geq 0$, \propeq
\begin{enumerate}
\item L'espace spectral $\Spec(\gT)$ est de dimension $\leq n$.  
\item Pour tout $x\in\gT$ le treillis quotient $\gT/(x=0,I_x=0)$, où $I_x=\sotq{y}{x\vi y=0}$, est de dimension $\leq n-1$. 
\item  
Pour toute suite $(x_0,\dots,x_n)$ dans $\gT$ il existe une suite \emph{complémentaire}  $(y_0,\dots,y_n)$ au sens suivant
\begin{equation}\label{feqC2G}
\left.\arraycolsep3pt
\begin{array}{rcl}
1& \vda  &   y_n, x_n\\
 y_n,  x_n & \vda  &  y_{n -1}, x_{n -1}  \\
\vdots~~& \vdots  &~~  \vdots \\
  y_1, x_1& \vda  &  y_0, x_0  \\
y_0, x_0& \vda  & 0     
\end{array}
\right\}
\end{equation}
\end{enumerate}
\end{ftheorem}

Par exemple, pour $n=2$ les inégalités dans le point \textsl{3} correspondent au dessin suivant dans le treillis~$\gT$.
$$\SCo{1.5pt}{1.2cm}{x_0}{x_1}{x_2}{y_0}{y_1}{y_2}$$

Un \trdi de dimension $\leq 0$ est une \agB. 

Les points \textsl{2} et \textsl{3} donnent un sens \cof à l'assertion $\Kdim(\gT)\leq n$ pour $n\geq -1$, mais l'assertion $\Kdim(\gT)=n$
ne semble pas pouvoir être  définie \cot dans le cas général (pour $n\geq 0$).

Concernant la \ddk des anneaux commutatifs voir \cite{fLom02,fCL2003}, \cite[chapitre XIII]{fACMC}. 
Cette \dfn de la \ddk  est proche de (et sans doute plus générale que) celle donnée par \cite[page 584]{fLur2009},
introduite pour couvrir certains cas d'anneaux non \noes.    

Par ailleurs il est impossible de démontrer \cot que $\Kdim(\RR)\leq 0$. En effet, comme~$\RR$ est un anneau local réduit, $\Kdim(\RR)\leq 0$ équivaut~à:  \textsl{tout~$x$ est nul ou inversible}, \cad \LPO.

\begin{flemma} \label{flemDimQuot}
Si $\gT'$ est un quotient de $\gT$, alors $\Kdim(\gT')\leq \Kdim(\gT).$ 
\end{flemma}
%
\begin{proof}
On applique le point \emph{3} du \thref{fth-dico-trdi-spec-dim1}.
\end{proof}
%

\begin{flemma} \label{flemDimProd}
Soient $\gT_1$ et $\gT_2$ deux \trdis et $\gT=\gT_1\times \gT_2$. Alors, pour tout entier $n\geq -1$ on a 
\[
\Kdim(\gT)\leq n\hbox{  \ssi  }\;\Kdim(\gT_i)\leq n\hbox{  pour  }i=1,2.
\]
On écrit cela sous forme abrégée comme suit: $\Kdim(\gT_1\times \gT_2)=\sup(\Kdim\gT_1,\Kdim\gT_2)$.   
\end{flemma}
%
\begin{proof}
D'une part, chaque $\gT_i$ est un quotient de $\gT$. D'autre part, supposons  que chaque~$\gT_i$ est de dimension $\leq n$ et soit $(\xzn)$ dans $\gT$, avec $x_i=(a_i,b_i)$. Si $(\azn)$ admet la suite \cop $(\yzn)$ dans $\gT_1$ et $(\bzn)$ admet la suite \cop $(\zzn)$ dans $\gT_2$, alors la suite $(\uzn)$ (où $u_i=(y_i,z_i)$) est \cop de $(\xzn)$ dans $\gT$. 
\end{proof}

\begin{ftheorem}[Dimension des \trdis, cas d'un recouvrement fermé] \label{flemrcfdimtrdi}~\\
Soient~$\gT$ un \trdi, $\fa$, $\fb$ deux \ids tels que $\fa\cap\fb=\so0$, $\gT_\fa=\gT/(\fa=0)$ et $\gT_\fb=\gT/(\fb=0)$.
Alors $\Kdim\gT=\sup(\Kdim\gT_\fa,\Kdim\gT_\fb)$. 
\end{ftheorem}
%
\begin{proof}
On~a un morphisme naturel $j_{\fa,\fb}:\gT\to \gT_\fa\times \gT_\fb, \,x\mapsto (\pi_\fa(x),\pi_\fb(x))$ de sorte que d'après les lemmes \ref{flemDimQuot} et \ref{flemDimProd} on~a $\sup(\Kdim\gT_\fa,\Kdim\gT_\fb)\leq \Kdim\gT$.

\noindent Voyons l'inégalité opposée.

\noindent Commençons par une \demo en \clama. 
Soit $\fp$ un \idep de~$\gT$. Si l'\id $\fp$ disparait dans $\gT_\fa$ on a $\fp\vu\fa=\gT$. Dans ce cas, on a un $x\in\fa$ tel que $x \notin \fp$, et pour un $y\in\fb$ arbitraire, comme $x\vi y=0$ et $\fp$ est premier, $y\in\fp$. Ainsi $\fb\subseteq \fp$ et $\fp$ reste un \idep dans $\gT_\fb$.
Symétriquement, si $\fp$ disparait dans $\gT_\fb$, il persiste dans~$\gT_\fa$.
Considérons maintenant une chaine d'\ideps dans $\gT$ de longueur maximale.
Le plus petit \elt de la chaine contient $\fa$ ou $\fb$, et donc toute la chaine persiste dans $\gT_\fa$ ou~$\gT_\fb$.

\noindent Enfin, on donne une \demo en \coma\footnote{Il serait intéressant de savoir si la \demo \cov résulte de la \demo classique par une procédure \gnle.}. 

\noindent Tout d'abord supposons $\fa$ et $\fb$ principaux: $\fa=\dar a$ et $\fb=\dar b$ avec $a\vi b=0$.  
On~a le morphisme injectif pour les structures de \textsl{\trdis non bornés\footnote{L'image de ce morphisme est $\uar a$. On~a bien $\pi_a(1)=1$, mais $\pi_a(0)=a$.}} $\gT_a\to \gT$, \hbox{$\pi_a(x)\mapsto a\vu x$}. On considère alors le morphisme $r_{a,b}:\gT_a\times \gT_b\to \gT,\,(x,y)\mapsto (a\vu x)\vi(b\vu y)$. On voit que c'est un morphisme de \trdis, car $r_{a,b}(0,0)=0$.  On~a $r_{a,b}\circ j_{a,b}=\Id_\gT$. 
En effet pour $x\in \gT$, $(a\vu x)\vi(b\vu x)=(a\vi b)\vu x=x$.  
Donc $r_{a,b}$ est un morphisme surjectif et $\gT$ est un quotient de $\gT_a\times \gT_b$. 

\noindent Voyons maintenant le cas \gnl où l'on ne suppose pas $\fa$ et $\fb$ \tf (donc principaux). Le quotient $\gT_\fa$ est colimite filtrée des $\gT_a$ pour les $a\in\fa$ (vu comme une petite catégorie filtrante).  Le treillis $\gT_\fa\times \gT_\fb$ est colimite filtrée des $\gT_a\times \gT_b$. 
Les morphismes~$r_{a,b}$ sont compatibles\footnote{Si $a\leq a'$ et $b\leq b'$, le triangle formé par $r_{a',b'}$,  $r_{a,b}$ et le morphisme de passage au quotient $\gT_a\times \gT_b\to\gT_{a'}\times \gT_{b'}$ est commutatif.}, donc induisent (par la \prt \uvle des colimites filtrées) un morphisme $r_{\fa,\fb}:\gT_\fa\times \gT_\fb\to\gT$. Pour chaque $(a,b)$ on a le morphisme  de colimite filtrée $\iota_{a,b}:\gT_a\times \gT_b\to \gT_\fa\times \gT_\fb$ et l'on obtient le diagramme commutatif:

\centerline{
\xymatrix @C=1.2cm @R=1.2cm{
\gT\,\ar[dd]_{j_{\fa,\fb}}\ar[rr]^{j_{a,b}}\ar@/-1.5cm/[drrr]_{j_{a',b'}}   && \,\gT_a\times \gT_b \ar[dd]^{r_{a,b}}\ar[ddll]_{\iota_{a,b}}\ar[dr]^{j_{(a,b),(a',b')}}   
\\
&&& \,\gT_{a'}\times \gT_{b'}\ar[dl]^{r_{a',b'}}\ar@/.5cm/[dlll]_{\iota_{a',b'}}
\\
\gT_\fa\times \gT_\fb\,\ar[rr]_{r_{\fa,\fb}}    && \gT  
}
}

\noindent On obtient donc
$r_{\fa,\fb}\circ j_{\fa,\fb}=\Id_\gT$.
\end{proof}

\begin{ftheorem}[Dimension des morphismes] \emph{Voir \cite{fCL2001-2018}, \cite[section XIII-7]{fACMC}.} \label{fth-dico-trdi-spec-dim2}   
Soit $\gT\subseteq \gT'$ et $f$ le morphisme d'inclusion. En \clama \propeq
\begin{enumerate}
\item Le morphisme $\Spec(f):\Spec(\gT')\to\Spec(\gT)$ est de dimension $\leq n$.
\item  Pour toute liste $(x_0,\dots,x_n)$ dans $\gT'$
il existe un entier $k\geq 0$ et des éléments $a_1,\ldots,a_k\in \gT$ tels que  pour tout couple de parties 
complémentaires $(H,H')$ de $\{1,\ldots,k\}$, il existe $ y_0,\dots,y_n\in \gT'$ tels 
que
\begin{equation} \label {feqdefDiTrRel}
\begin{array}{rclll}
\Vi_{j\in H'} a_j & \vda  &  y_n,\;x_n    \\
y_n,\;x_n& \vda  &y_{n-1},\;x_{n-1}      \\
\vdots\qquad & \vdots  & \qquad  \vdots    \\
y_1,\;x_1& \vda  &  y_0,\;x_0    \\
y_0,\;x_0& \vda  &  \Vu_{j\in H} a_j    \\
\end{array}
\end{equation}
\end{enumerate} 
\end{ftheorem}
Par exemple pour la dimension relative $\leq 2$ cela correspond au dessin suivant dans $\gT$ pour chaque $(H,H')$, avec $u=\Vi_{j\in H'} a_j$ et $i=\Vu_{j\in H} a_j$.
$$\SCOr{x_0}{x_1}{x_2}{y_0}{y_1}{y_2}{u}{i}$$

Notons que la dimension du morphisme $\gT\to\gT'$ est inférieure à la dimension de $\gT'$: prendre la liste vide ($k=0$) dans le point \textsl{2} du \tho~\ref{fth-dico-trdi-spec-dim2}. 

Plus \gnlt on a une \prco de l'inégalité fondamentale donnée dans \cite[A note on the dimension theory of rings]{fSei1953} dans le cas de la dimension des anneaux commutatifs: $1+\dim \gT'\leq (1+\dim \gT)(1+\dim f)$.

\Subsubsection{Propriétés des espaces}

 Un \trdi $\gT$ est dit \textsl{normal} lorsque chaque fois que $a \vu b = 1$ dans $\gT$  il existe $x, y$ tels que   $a \vu x = b \vu y = 1$ et  $x \vi y = 0$. Voir \cite{fWeh2019,fDST2019}.
Notons qu'en remplaçant $x$ et $y$ par $x_1=x\vu(a\vi b)$ et $y_1=y\vu(a\vi b)$
on obtient $a \vu x_1 = b \vu y_1 = 1$ et  $x_1 \vi y_1 = a\vi b$.

\begin{ftheorem} \label{fth-dico-trdi-spec-esp1}  \Propeq
\begin{enumerate}
\item L'espace spectral $\Spec(\gT)$ est normal. 
\item Le \trdi $\gT$ est normal. 
%
%
\end{enumerate}
\end{ftheorem}
\begin{ftheorem} \label{fth-dico-trdi-spec-esp2}  \Propeq
\begin{enumerate}
\item L'espace spectral $\Spec(\gT)$ est complètement normal. 
\item Tout intervalle $[a,b]$ de $\gT$, vu comme \trdi, est normal.  
\item Pour tous $a,b\in \gT$ il existe $x, y$ tels que   $a \vu b = a \vu y = x \vu b$ et  $x \vi y = 0$.
\end{enumerate}
\end{ftheorem}
\begin{ftheorem} \label{fth-dico-trdi-spec-esp3}  \Propeq
\begin{enumerate}
\item Dans $\Spec(\gT)$ tout \oqc est une réunion finie d'\oqcs irréductibles. 
\item 
Pour tous $a_1,\dots,a_n,b_1,\dots,b_m$ on a  $a_1,\dots,a_n\vdash_\gT b_1,\dots,b_m$ \ssi
 $a_1,\dots,a_n\vdash_\gT b_j$  pour un $j$. 
\item 
Le \trdi $\gT$ est construit à partir d'une  \sad pour une \talg. 
%
%
\end{enumerate}
\end{ftheorem}


\section[Théories dynamiques et \sads]{Théories dynamiques et \sads finitaires}\label{fsecSadTrdiSpec}

Références: \cite{fCLR01,fLom98,fLom06}. 
Une étude plus détaillée en cours de rédaction se trouve dans~\cite{fLom-tgac}.

\subsection{Théories dynamiques finitaires}

Les \tdys (finitaires) ont été  introduites dans \cite{fCLR01}. Elles sont une version purement calculatoire \gui{sans logique} des théories cohérentes (théories \gui{du premier ordre}
 dont tous les axiomes sont les traductions de \rdys)\footnote{Une terminologie courante en logique mathématique consiste à réserver \gui{théorie formelle du premier ordre} aux théories qui utilisent la quantification sur les seules variables usuelles, qui représentent des éléments courants d'un ensemble de référence, modèle de la théorie. Le second ordre est réservé aux théories formelles qui autorisent la quantification sur des variables prises dans l'ensemble des parties de l'ensemble de référence. Les théories \gmqs infinitaires, avec la présence de $\Vou$ infinis portant sur les \elts de l'ensemble de référence, ne sont pas vraiment du premier ordre, et pas du tout du second ordre. Il s'ensuit que dans la littérature concernant les \tgms, finitaires ou non, on les considère souvent comme des théories du premier ordre. Nous respectons à bas bruit cette tradition en omettant \gnlt de parler de \gui{théories du premier ordre}, au profit de \gui{théories finitaires}.}.

Dans l'article en question sont introduites les notions de \gui{dynamical theory} et de \gui{dynamical proof}.
Voir \egmt: l'article \cite{fBC2005} qui décrit un certain nombre d'avantages fournis par cette approche, et les articles précurseurs \cite[Sections 1.5 et 4.2]{fpra1971}, \cite{fMat75} et \cite{fLif80}.

Les \tdys utilisent uniquement des  \rdys, \cad des règles de la forme
\begin{equation} \label{feqRgeom0}
\Gamma  \vd   \EXists{\und{y^1}} \Delta_1
\vou \cdots\vou \EXists{\und{y^m}}\Delta_m
\end{equation}
où $\Gamma$ et les $\Delta_i$ sont des listes de formules atomiques du langage considéré. 

Sous forme succincte:\label{fNOTvalou} 
\begin{equation} \label{feqRgeom}
\Gamma  \vd   \Exists{\und{y^1}} \Delta_1
\vou \cdots\vou \Exists{\und{y^m}}\Delta_m
\end{equation}


%
%
Les axiomes sont des \textsl{\rdys} et les \textsl{\thos} sont les \rdys validées à partir des axiomes par un arbre de preuve selon un processus \elr bien défini.

Si \sa{T} est une \tco, la \textsl{\tdy finitaire} correspondante s'en différencie seulement 
par un usage extrêmement limité des méthodes de \demo:
\begin{itemize}
\item  Premièrement, on n'utilise jamais d'autres formules que les formules 
atomiques: on n'introduit jamais aucun nouveau prédicat utilisant des connecteurs 
logiques ou des quantificateurs. Seules sont manipulées des listes de formules 
atomiques du langage~$\cL$.
\item  Deuxièmement, les axiomes ne sont 
pas vus comme des formules vraies, mais comme des \textsl{règles de déduction}: un axiome tel que \pref{feqRgeom} est utilisé en tant que \rdy.  
Les variables dans les listes $\und{y^j}$ sont dites \textsl{muettes}.
\item  Troisièmement, on ne prouve que des \textsl{\rdys}.%
\item  Quatrièmement, la seule
manière de prouver une \rdy est un calcul arborescent
\gui{sans logique}. \`A la racine de l'arbre se trouvent 
des hypothèses
du \tho que l'on veut prouver. L'arbre se développe en appliquant les axiomes 
selon une pure machinerie de calcul algébrique dans la structure.
Les \dfns formelles précises sont données dans \cite{fCLR01}.\\
Lorsque l'on applique un axiome tel que \pref{feqRgeom}, on substitue aux variables libres $x_i$ présentes dans la règle des termes arbitraires $t_i$ du langage. Si les hypothèses, réécrites avec ces termes,
sont déjà prouvées à une feuille de l'arbre de preuve, alors cette feuille devient un noeud d'où partent des branches de calcul dans chacune desquelles on introduit des variables fraiches correspondant aux
variables muettes $\und{y^k}$ (ces variables fraiches doivent être distinctes des variables libres présentes dans les termes~$t_i$) et chaque conclusion de la liste~$\Delta_k$ est valide dans sa branche.
Enfin, on déclare \textsl{valide} une conclusion prouvée à chaque feuille d'un arbre de preuve ainsi construit.
\end{itemize}

\smallskip Dans une \tdy, chaque sorte est accompagnée d'un prédicat d'\egt~\hbox{$\cdot=\cdot$}. En outre, des axiomes  autorisent la substitution d'un terme $t$ par un terme $t'$, lorsque la règle $\vd t=t'$ est valide dans la théorie, en n'importe quelle occurrence d'une formule atomique présente dans une \rdy valide\footnote{Cela exclut le cas où $t$ ou $t'$ contient une variable  $x$ sous la dépendance d'un $\Exists \,x$.}.

Les axiomes les plus simples répondant à cette exigence sont les suivants.
Tout d'abord le fait que l'\egt est une relation d'\eqvc

\DeuxRegles{
\Lab{eq1} $\vd  x=x$
\Lab{Eq3} $\,\, x=y\vet y=z\vd  x=z$
}{ 
\Lab{eq2} $\,\, x=y\vd  y=x$ 
}

\noindent Ensuite pour chaque symbole de fonction $f$ et chaque symbole de prédicat $P$ dans la signature (nous les avons pris unaires pour simplifier).

\DeuxRegles{
\Lab{eq$_f$} $\,\, x=y\vd  f(x)=f(y) $
}{
\Lab{Eq$_P$} $\,\, x=y\vet P(x)\vd  P(y) $
}

\Subsubsection{Collapsus}\label{fNOTABot}

 Une \rdy s'appelle une \textsl{règle de collapsus} ou \textsl{d'effondrement} lorsque le second membre est \gui{le $\Faux$}, que l'on note $\Bot$. Le symbole $\Bot$ fait obligatoirement partie du langage. Il doit être rangé avec les formules atomiques, comme un prédicat d'arité nulle.  On peut aussi voir $\Bot$ comme désignant la disjonction vide. Lorsque l'on a prouvé  $\Bot$, l'univers du discours s'effondre, et toute formule atomique du langage est alors réputée \gui{vraie}, ou du moins \gui{valide}. C'est l'application de la règle  \gui{ex falso quod libet}, qui est la signification intuitive pertinente du $\Faux$ en \coma. Ainsi, dans une \tdy les règles 

\Regles {\lab {Faux$_{P}$} $\,\,\Bot\vd P$} 

\noindent sont valides pour toutes les formules atomiques.

On  donne aussi dans le langage la constante logique $\Top$ pour \gui{le $\Vrai$}, avec pour axiome la \ralg. \label{fNOTATop} 

\Regles {\lab{Vrai}$\vd \Top$}

On peut aussi voir $\Top$ comme désignant la conjonction vide\footnote{Quand il n'y a rien à démontrer, ne démontrons rien et tout sera OK. Par ailleurs, dans une \tdy comportant au moins une sorte $S$, $\Top$ est \eqv à $x=_Sx$.}.
Les constantes $\Bot$ et $\Top$ sont  les seuls vrais symboles logiques dans les \tdys.


\Subsubsection{Classification des règles dynamiques} 

Nous utilisons la terminologie donnée dans \cite{fCLR01} et ajoutons quelques précisions.

\smallskip Une \tdy est dite \textsl{propositionnelle} s'il n'y a aucune sorte. Les constantes sont alors $\Top$, $\Bot$ et en \gnl d'autres constantes d'arité nulle, vues comme des constantes de propositions.  

\smallskip Une  \rdy qui ne contient dans la conclusion (à droite du $\vd$) ni $\vou $, ni~$\,\Exists\,$, ni~$\Bot$ est appelée une \textsl{règle algébrique}. Une \tdy est dite \textsl{\agq} lorsqu'elle ne comporte comme axiomes que des \ralgs. Dans la littérature anglaise concernant la logique catégorique (étudiée dans le cadre des \clama), les \talgs ont pour nom \textsl{Horn theory}.

\label{fregledirecte}
Une \ralg est dite \textsl{directe} si, dans l'hypothèse (à gauche du $\vd$), il y a une liste de prédicats portant uniquement sur des variables, toutes les variables étant distinctes. 
\cite{fCLR01} utilisent les \reds pour construire les \nsts formels et leurs variantes, qui sont des certificats \agqs d'effondrement pour certaines \sads.  

Les \ralgs qui ne sont pas directes sont appelées des \textsl{règles de simplification}.

Typographiquement on opère les distinctions suivantes: le nom d'une \red est entièrement en minuscule, le nom d'une règle de simplification commence par une majuscule, le nom d'une \rdy qui n'est pas une \ralg est entièrement en majuscule.   
  
Un cas particulier de \talg est fourni par les \textsl{\tpes} qui sont les \talgs avec une seule sorte et pour seul prédicat le prédicat d'\egt. Dans la littérature anglaise, les \tpes ont pour nom \textsl{algebraic theory}.

\smallskip  Une \tdy est dite \textsl{\dij} si dans les axiomes, il n'y a pas de $\Exists$ dans la conclusion. On pourrait les appeler \textsl{weakly disjunctive} en anglais. 

\smallskip 
Une \rdy est dite \textsl{existentielle simple} si la conclusion est de la forme $\Exists \ux\; \Delta$ où $\Delta$ est une liste finie de formules atomiques.

Une \tdy est dite \textsl{existentielle} si ses axiomes sont tous des \ralgs ou existentielles simples (une \ralg peut d'ailleurs être considérée comme un cas particulier de \rex). Une théorie existentielle typique est la théorie des \textsl{anneaux de Bézout} (tout \itf est principal). Dans la littérature anglaise, une théorie existentielle est
appelée une \textsl{regular theory}.

\smallskip 
Les théories \textsl{existentiellement rigides} sont les \tdys dans lesquelles les axiomes existentiels sont simples et correspondent à des existences uniques. 
Cela généralise (de très peu) les théories \dijs. 

Une théorie existentielle existentiellement rigide
est dite \textsl{cartésienne}.
Cela généralise (de très peu) les \talgs. Dans la littérature anglaise on dit  \textsl{cartesian theory}.

\smallskip Une théorie est dite \textsl{rigide} si les axiomes sont de trois types 
\begin{itemize}
\item des \ralgs;
\item des règles \dijs du type $\Gamma\vd P\vou Q$ avec la règle prouvable $\Gamma\vet P\vet Q\vd \Bot$;
\item  des règles existentielles existentiellement rigides.
\end{itemize}
La théorie des \cdis peut être énoncée sous forme d'une \tdy rigide
en utilisant le prédicat d'inversibilité. La théorie des \crcds peut aussi être énoncée sous forme d'une \tdy rigide, contrairement à la théorie des \cdacs.  
Dans la littérature anglaise une théorie rigide est appelée une \textsl{disjunctive theory}.   

\Subsubsection{Un exemple de base}

La \textsl{théorie des anneaux commutatifs} \sa{Ac} est l'exemple paradigmatique d'une \tpe. Elle utilise la signature 
$
\Sigma_{\sA{Ac}}=(\cdot=0\mathrel{;}\cdot+\cdot,\cdot\times \cdot,-\,\cdot,0,1)
$ 
avec les seuls axiomes suivants (ce sont des \reds):\label{fTpeAc} 

\DeuxRegles{
\Lab{ac1} $\vd  0=0$
\Lab{ac3} $\,\, x=0\vet y=0\vd  x+y=0$
}
{
\Lab{ac2} $\,\,x=0\vd x\times y=0 $
}

\smallskip 
Le terme \gui{$x-y$} est
une abréviation de \gui{$x+(-y)$} et le prédicat binaire \gui{$\cdot=\cdot$} est \textsl{défini}
par la convention: \gui{$x=y$} est une abréviation pour \gui{$x-y=0$}.

\smallskip \noindent \emph{Explication.} Outre la machinerie calculatoire des preuves dynamiques, on utilise la machinerie calculatoire des \pols commutatifs à \coes entiers (\gui{extérieure} à la théorie), qui réécrit tout terme (formé sur les constantes et les variables) comme un \pol à \coes entiers sous une forme normale prédéfinie.

\noindent La règle de distributivité $x(y+z)=xy+xz$, par exemple, est alors confiée à un calcul automatique qui réduit à $0$ le terme
 $x(y+z)-(xy+xz)$. 
La transitivité de l'\egt est obtenue de la même manière en utilisant l'axiome \tsbf{ac3}.

\subsection{Structures \agqs dynamiques} 

Les \sads (finitaires) sont explicitement nommées dans~\cite{fLom98,fLom06}. Dans  \cite{fCLR01} elles sont implicites, mais décrites sous la forme de leurs présentations.
Elles sont \egmt implicites dans \cite{fLom02}, et, last but not least, dans \cite{fD5}, qui a été une source d'inspiration essentielle: on peut calculer de manière sûre dans la clôture algébrique d'un corps discret, même quand il n'est pas possible de construire cette clôture algébrique. 
Il suffit donc de considérer la clôture \agq comme une \sad \gui{à la D5} plutôt que comme une structure \agq usuelle: \textsl{l'évaluation paresseuse à la D5 fournit une sémantique \cov pour la clôture \agq d'un \cdi}.

\smallskip Si $\sa{T}=(\cL,\cA)$ est une \tdy, une \textsl{\sad de type~\sa{T}} est donnée par un 
ensemble $G$ de générateurs et un ensemble $R$ de \textsl{relations}. 

Une relation est par \dfn une formule atomique close $P(\und{t})$ construite sur le langage  $\cL\,\cup\, G$ avec des termes~$t_i$ clos dans ce langage. À une telle relation est associé l'axiome  \gui{$\Vd  P(\und{t})$} de la \sad.

D'un point de vue \cof l'ensemble $G$ peut être vu comme un ensemble à la Bishop. Cela implique que pour deux objets construits $a,b$ qui représentent des \elts de $G$, si l'on~a $a=_Gb$, la relation $a=b$ doit être présente dans l'ensemble $R$ des relations qui définissent la \sad $\big((G,R),\sa{T}\big)$.

\begin{fnotation}\label{fnotasadreglevalide}
Nous indiquerons que la règle  \gui{$\,\Gamma\Vd \dots$} est valide dans la \sad $\gS=\big((G,R),\sa{T}\big)$ sous la forme abrégée: \gui{$\, \Gamma\Vdi\gS \dots$}. 
\end{fnotation}

\begin{fdefinota} \label{fnotaTcl}
Soit $\gS=\big((G,R),\sa{T}\big)$ une \sad de type $\sa{T}=(\cL,\cA)$. L'ensemble des termes clos de $\gS$, \cad les termes construits sur $\cL\cup G$,  se note $\Tcl(\gS)$. L'ensemble des formules atomiques closes  se note~$\Atcl(\gS)$. 
Une \ralg~$\vd P$ pour~$P\in\Atcl(\gS)$ s'appelle \textsl{un fait de $\gS$}. L'ensemble des faits de $\gS$ valides dans $\gS$  se note $\Atclv(\gS)$.
  \end{fdefinota}

Intuitivement, une \sad est une structure algébrique usuelle incomplètement spécifiée. 

\begin{fexample} \label{fexaSaCd}  
 Par exemple on obtient une \sad de corps discret 
\[\ndsp\gK=\big((G,R),\sa{Cd}\big)\]
en prenant~\hbox{$G=\so{a,b}$} et 
$R=\so{105=0,\,a^2+b^2-1=0}.$ Ce \cdi dynamique  correspond à n'importe quel corps de caractéristique
$3$ ou $5$ ou $7$ engendré par deux \elts $\alpha$ et
$\beta$ vérifiant~\hbox{$\alpha^2+\beta^2=1$}. 

\noindent Outre les règles dynamiques valides dans tous les corps discrets,
il y a maintenant celles que l'on obtient en élargissant le
langage avec les constantes prises dans $G$ et en ajoutant aux  axiomes les 
relations prises dans $R$. \eoe
\end{fexample}

Notons qu'une \rdy valide dans une \sad utilise pour sa \demo un
pur calcul arborescent, sans logique, avec un nombre fini de \gtrs, de relations et d'axiomes. 

\begin{fnotation} \label{fnotaT(A)}
Lorsque $\gA$ est une \salg usuelle sur le langage de $\sa{T}$ (ou sur une partie de ce langage) on note $\sa{T}(\gA)$ la \sad obtenue en prenant la présentation $(G,R)$ suivante: $G$ est l'ensemble des \elts de $\gA$ et $R$ est l'ensemble des faits valides dans $\gA$. 
Par exemple si $\sa{T}$ est la théorie des anneaux de Bézout,  et $\gA$ est un anneau commutatif arbitraire, on doit mettre dans $R$ les relations $a+b-c=0$, $a'b'-c'=0$ et $a+a''=0$ lorsque $a,b,a',b',c',a''$ sont des \elts de $\gA$ et que les relations sont satifaites dans $\gA$. On dit que l'on a ajouté à la théorie \sa{T} le \textsl{diagramme positif de $\gA$}.

\end{fnotation}

\Subsubsection{Modèles constructifs versus modèles classiques}\label{fsubsecmodelescofsSAD}
On considère  une \sad  $\gA=\big((G,R),\sa T\big)$  de type \sa{T} avec une ou plusieurs sortes. Pour simplifier les notations nous supposons une seule sorte. Un \textsl{modèle  de~$\gA$}  est une \salg usuelle (statique) $M$ décrite dans le langage 
associé à~$\gA$ et vérifiant les axiomes de $\gA$ (ceux de $\sa{T}$ et ceux donnés par la \pn de $\gA$).

Lorsque $\gA$ est défini par la \pn vide, on parle de \textsl{modèles de \sa T}.

\smallskip 
La notion de modèle est donc basée à priori sur une notion intuitive de \textsl{structure \agq} à la Bourbaki. 
Nous pouvons qualifier ces \salgs de \gui{statiques} par contraste avec les \sads \gnles. Notons qu'ici l'ensemble \gui{sous-jacent} à la structure est un ensemble \gui{naïf} (ou plusieurs ensembles naïfs s'il y a plusieurs sortes) structuré par la donnée de prédicats et de fonctions (au sens naïf)
soumis à certains axiomes.

D'un point de vue \cof, les modèles doivent satisfaire les axiomes en respectant le sens intuitif du \gui{ou} et du \gui{il existe}: pour prouver qu'une \salg particulière satisfait les axiomes, on autorise uniquement la logique intuitionniste. 
Notons aussi que la théorie des ensembles à laquelle nous nous référons est à priori celle, informelle, de Bishop.

\subsection{Extensions conservatives}

Une \tdy \sa{T} est définie par un couple d'ensembles $(\cL,\cA)$ où $\cL$
est (la signature du) le langage formel utilisé, et $\cA$ l'ensemble des axiomes, qui sont des \rdys sur le langage $\cL$.  
Pour nous, d'un point de vue \cof, ces deux ensembles sont des ensembles intuitifs à la Bishop. Ils peuvent éventuellement être formalisés dans une théorie formelle \cov lorsqu'on étudie les \prts \gnles des \tdys. 
Ils sont en tout cas très différents des ensembles considérés usuellement en logique catégorique, où le cadre général est fourni comme un cadre purement formel, dans \sa{ZFC} ou dans l'une de ses extensions.

Une \tdy $(\cL',\cA')$ est une \textsl{extension simple} de $(\cL,\cA)$ si $\cL$ et $\cA$ sont des sous-ensembles de $\cL'$ et $\cA'$ (au sens catégorique défini dans Bishop).
Dans ce cas, les \rdys formulées dans $\cL$ et valides dans $\sa T$ sont valides dans $\sa T\;'$.    
 
\begin{fdefinition} \label{fdefitdyidentiques}
Deux \tdys \textsl{sur le même langage} sont dites \textsl{identiques} si elles prouvent les mêmes \rdys, \cad si les axiomes de chacune sont des règles valides dans l'autre. Dans ce cas les modèles sont les mêmes aussi bien en \coma qu'en \clama.
\end{fdefinition}

\begin{fdefinition} \label{fdefiextcons}
 On dit qu'une \tdy $\sab{T}'$ est une \textsl{extension conservative simple de la théorie \sa{T}}
si c'est un extension simple de \sa{T} et si les \rdys formulables dans \sa{T} et valides dans $\sab{T}'$
sont valides dans \sa{T}.
\end{fdefinition}

\noindent {\bf Définition informelle}. D'autres extensions, plus \gnles que celles de la \dfn \ref{fdefitdyidentiques}, doivent être considérées comme \gui{ne changeant rien à la théorie} d'une part parce que ce sont des extensions conservatives et d'autre part parce que les modèles \cofs de toute \sad ne changent pas non plus. On dit alors que $\sab{T}'$ est un extension \textsl{intuitivement équivalente} de la théorie \sa T.

\smallskip Nous allons en examiner un certain nombre. 

\Subsubsection{Extensions  \esids}

\paragraph{Ajout de  simples abréviations dans le langage}~

\smallskip Par exemple, on peut introduire un symbole de fonction \gui{$\Som(\cdot,\cdot,\cdot)$} pris comme une abréviation: $\Som(x,y,z)$ 
est un notation abrégée pour $(x+y)+z$.

 Il est  clair que ce type de \dfns \gui{ne change rien} à la \tdy: c'est une extension \inteq.

\paragraph{Ajout d'un prédicat de conjonction, de disjonction ou d'existence}\label{fsecajout-et-ou-ex}~

\smallskip Soit $\Delta=(A_1,\dots,A_n)$ une liste de prédicats.

On peut considérer que les règles d'introduction et d'élimination suivantes \und{définissent} la conjonction, comme en déduction naturelle.

\regles{
  \lab{Intro-$\vii_{\Delta}$} $\,\,A_1\vet\dots \vet A_n\vd A_1\vii  \dots \vii  A_n$
  \lab{Elim-$\vii_{\Delta}$} $\,\, A_1\vii \dots \vii  A_n \vd A_1\vet\dots \vet A_n$
}

\noindent La deuxième règle équivaut à la conjonction des règles $A_1\vii \dots \vii  A_n\vd A_i$.

\smallskip Le cas de la disjonction est un peu plus compliqué, car dans les preuves dynamiques, il n'y a pas de $\vou$ à gauche de $\vd$. On peut  introduire dans le langage le connecteur 
$\vuu$
pour créer de nouvelles formules en utilisant les règles inspirées de la déduction naturelle.

\regles{
  \lab{Elim-$\vuu_{\Delta}$} $\,\,A_1\vuu \cdots \vuu A_n \vd  \,A_1\vou \dots \vou A_n$
  \lab{Intro-$\vuu_{\Delta,1}$} $\,\,A_1 \vd  A_1\vuu \cdots \vuu A_n$
  \lab{~} $\vdots$
  \lab{Intro-$\vuu_{\Delta,n}$} $\,\,A_n \vd  A_1\vuu \cdots \vuu A_n$
 }
 
\smallskip 
L'utilisation du quantificateur existentiel dans le langage
d'une \tdy se fait au moyen des règles d'introduction et d'élimination suivantes (la deuxième n'est pas une plaisanterie, elle justife notre usage de $\Exists$ bien distingué de $\exists$).

\regles{
  \lab{Intro-$\exists_{x,A}$} $\,\, A(x) \vd \exists x\,A(x)$
  \lab{Elim-$\exists_{x,A}$} $\,\,\exists x\,A(x) \vd  \Exists x\,A(x)$
}

\paragraph{Ajout d'un symbole de fonction en cas d'existence unique}~

\smallskip On considère une \tdy  $\sa{T}=(\cL,\cA)$.
Supposons que $\cL$ contienne un prédicat  $P(u,x,y)$ d'arité $k+1$ (on donne l'exemple avec $k=2$) 
et que \sa{T} prouve les \rdys suivantes

\DeuxRegles{
\lab{Ex$_{P,u}$} $ \vd  \Exists u\;P(u,x,y)$
}
{
\lab{Uniq$_{P,u}$} $ P(u,x,y),\,P(v,x,y)\vd u=v $
}

\noindent 
 Soit alors $\sab{T}'$  la \tdy obtenue à partir de \sa{T} en ajoutant un nouveau symbole de fonction $f$ et l'axiome suivant 
 
\UneRegle{df$_{P,u,f}$}{$\,\,P(u,x,y)\vd u=f(x,y)$}

\noindent On dit que
\textsl{la règle \tsbf{df$_{P,u,f}$} définit le symbole de fonction $f$}. 
\\
Dans ce cas la \tdy 
$\sab{T}'$ est une extension conservative 
de la théorie \sa{T}.

\begin{flemma} \label{flemdefiesid}
Pour une \tdy  l'utilisation des facilités précédentes 
(abréviations, conjonction, disjonction, quantificateur existentiel, symboles de fonction en cas d'existence unique) produit une extension conservative qui ne change pas les modèles constructifs.  
\end{flemma}

\begin{fdefinition} \label{fdefiesid}~
\begin{enumerate}
\item Une telle extension est appelée une \textsl{\esid simple}. 
\item Deux \tdys sont dites \textsl{\esids} s'il existe une \tdy qui est, éventuellement après des renommages  dans les signatures, une extension \esid des deux premières. 
\end{enumerate}
\end{fdefinition}

\Subsubsection{Extensions essentiellement \eqves} \label{fsubsectdyeseq}

On examine maintenant des situations obtenues en étendant une \tdy par ajout de nouvelles sortes bien définies. Ceci correspond à des constructions légitimes d'ensembles chez Bishop.

Références: \cite{fBH2017}, \cite{fTse2017}.

\paragraph {Introduction d'une sous-sorte}~

\smallskip On considère une \tdy \sa{T} et une sorte $S$ de cette théorie. On considère un prédicat $P(\cdot)$ de la théorie portant sur les objets de type~$S$.
On définit une nouvelle \tdy  comportant \gui{la sous-sorte $U$ de la sorte $S$ définie par le prédicat $P$} de la manière suivante.
\begin{enumerate}
\item On ajoute $U$ dans les sortes.
\item On ajoute une fonction $j_P:U\to S$ dans les fonctions.
\item On ajoute les axiomes

\DeuxRegles
{
\lab{sso$_P$} {$\vd_{u:U}  P(j_P(u))$}
}
{
\lab{SSO$_P$} {$\,\,P(a)\vd_{a:S} \Exists u \; j_P(u)=a$}
}

\item On définit un prédicat $x=_{U}y$ d'\egt sur $U$ 
comme une abréviation de $j_P(x)=j_P(y)$.  
\end{enumerate}

\paragraph {Introduction d'une sorte produit fini de sortes}~

\smallskip On considère une \tdy \sa{T}  et des sortes $S_1$, \dots, $S_n$ de cette théorie. 
On définit une nouvelle \tdy  comportant \gui{la sorte $S$ produit fini des $S_i$} de la manière suivante.
\begin{enumerate}
\item On ajoute $S$ dans les sortes.
\item On ajoute pour chaque $i\in \lrbn$ un symbole de fonction $\pi_i$ de type $S\to S_i$.
\item On définit  $x=_{S}y$  sur $S$ 
comme une abréviation de $\Vii_{i=1}^n \pi_i(x)=\pi_i(y)$.  
\item On ajoute  un symbole de fonction $\mathrm{Pr}$ de type $S_1\times \cdots\times S_n\to S$ dans les fonctions. 
\item On ajoute les axiomes

\Regles{ \lab{spf$_i$} $ \vd_{x_1:S_1,\dots,x_n:S_n} \; \pi_i\big(\mathrm{Pr}(\xn)\big)=x_i$  
}

\end{enumerate}

\paragraph  {Introduction d'une sorte quotient}~

\smallskip On considère une \tdy \sa{T}  et une sorte $S$ de cette théorie. On considère un prédicat binaire $E(x,y)$ sur $S$. On suppose que les trois règles qui signifient que $E(x,y)$ est une relation d'\eqvc sur $S$ sont valides dans~$\sa{T}$. 
On définit une nouvelle \tdy  comportant \gui{la sorte $B$ quotient de $S$ par la relation d'\eqvc $E$} de la manière suivante.
\begin{enumerate}
\item On ajoute $B$ dans les sortes et un prédicat binaire $x=_By$ sur $B$.
\item On ajoute un symbole de fonction $\pi_{E}$ de type $S\to B$.
\item On ajoute les axiomes

\DeuxRegles{
\lab{sq1$_E$} {$\,\,E(a,b)\vd_{a,b:S}\; \pi_{E}(a)=_B \pi_{E}(b)$}
\lab{SQ$_E$} $\vd_{x:B}\; \Exists a \; \pi_{E}(a)=_Bx$
}{
\lab{Sq2$_E$} {$\,\,\pi_{E}(a)=_B\pi_{E}(b)\vd_{a,b:S}\; E(a,b) $}
}
\end{enumerate}

\paragraph {Introduction d'une sorte somme disjointe finie}~

\smallskip On considère une \tdy \sa{T}  et des sortes $S_1$, \dots, $S_n$ de cette théorie. 
On définit une nouvelle \tdy  comportant \gui{la sorte $S$ somme disjointe des $S_i$} de la manière suivante.
\begin{enumerate}
\item On ajoute $S$ dans les sortes et un prédicat binaire $x=_Sy$ sur $S$.
\item On ajoute pour chaque $i\in \lrbn$ un symbole de fonction $j_{S_i,S}$ de type $S_i\to S$.
\item On ajoute les axiomes (dans lesquels $j_i$ est une abréviation pour $j_{S_i,S}$)

\Regles{
\lab{Sdf$_i$} {$\,\,j_i(a)=_Sj_i(b)\vd_{a,b:S_i}\; a=_{S_i}b$ \hfill  pour $1\leq i\leq n$\hspace{4cm}~}
\lab{SDF$_{i,k}$} {$\,\,j_i(a)=_Sj_k(b)\vd_{a:S_i,b:S_k}\;\Bot  $ \hfill     pour $1\leq i< k\leq n$\hspace{4cm}~}
\lab{SDF} $\vd_{x:S}\; \Exists a_1 \; j_1(a_1)=_Sx \;\vou\;\cdots\;\vou\;\Exists a_n \; j_n(a_n)=_Ax$}

\end{enumerate}

\begin{flemma} \label{flemmasortesommefinie}
Lorsque l'on introduit une nouvelle sorte, sous-sorte, sorte produit fini, quotient ou somme disjointe finie de sortes présentes dans une \tdy (avec  les symboles de fonctions et  axiomes qui définissent la nouvelle sorte) on obtient une nouvelle \tdy \inteq à la première. 
\end{flemma}

\begin{fdefinition} \label{fdefitheseq}  On considère une \tdy  $\sa{T}$.
\begin{enumerate}
\item  Une extension simple  $\sab{T}'$ de \sa{T} est dite \textsl{extension essentiellement équivalente  simple de \sa{T}} si elle est identique à \sa{T} ou si elle est une extension \inteq obtenue  en application répétée d'ajouts 
autorisés dans les sous-sections précédentes:
\begin{itemize}
\item ajouts d'abréviations, de symboles de prédicats ou 
   de symboles de fonctions selon le lemme \ref{flemdefiesid},
\item ajouts   de nouvelles sortes bien définies selon le lemme  
\ref{flemmasortesommefinie}.
\end{itemize}
\item On dit que  \sa{T} et $\sab{T}'$ sont des théories \textsl{essentiellement équivalentes} si l'on peut trouver une \tdy $\sab{T}''$ qui est une extension essentiellement équivalente simple à la fois de \sa{T} et de $\sab{T}'$, à un renommage éventuel près de
certains symboles de sortes, de fonctions et de prédicats dans \sa T ou~$\sab{T}'$.  
\item On dit que \textsl{la \tdy $\sab{T}'$ est une extension
 de la \tdy \sa{T}} si elle est une extension simple d'une théorie essentiellement équivalente à \sa{T}.
\end{enumerate}
 
\end{fdefinition}

Les extensions \eseqs ne changent pas les modèles \cofs d'une \sad. 
Il en va autrement pour d'autres extensions conservatives qui peuvent rendre les modèles précédents non \cofs. C'est ce que nous allons examiner maintenant. 

\Subsubsection{Autres extensions conservatives}

\paragraph{Ajout de la logique classique}~

\smallskip Pour une \tdy,  l'utilisation de la logique classique revient à autoriser systématiquement, outre les facilités de la \dfn \ref{fdefitheseq}, 
l'introduction d'un prédicat opposé à tout prédicat déjà défini. 

Pour un prédicat $P$ déjà défini, on introduit alors  un prédicat opposé $Q$, noté $\lnot P$ ou $\ov P$, qui satisfait les deux axiomes de la logique booléenne.

\DeuxRegles{
\labu $\;\;P\vet Q\vd \Bot$
}
{
\labu $\vd P\vou Q$
}

D'un point de vue \cof, cela signifie que dans les modèles d'une \sad, on demande désormais que le prédicat soit décidable.

\begin{ftheorem}[Élimination des coupures] 
\label{fthFond} 
Pour une \tdy,  l'utilisation de la logique classique produit une extension conservative.
\end{ftheorem}

Un avantage de la logique classique est souvent intériorisé sous la forme de ramener toutes les \demos à des \demos par l'absurde.

Par exemple la règle

\Regles{
\labu $\;\;A\vet B\vd C\vou D$
}

\noindent devient \eqve à la règle

\Regles{
\labu $\;\;A\vet B\vet \ov C\vet\ov D \vd \Bot$
}

\paragraph{Skolémisation}~  

\smallskip Dans le \tho qui suit, ce n'est pas seulement la logique classique qui est utilisée, mais la skolémisation, laquelle peut être vue comme une forme locale de l'axiome du choix.

\begin{ftheorem}[Skolémisation, {\cite{fBC2019}}] 
\label{fthFondExists} 
~\\
On considère une \tdy \sa{T}. Notons $\sab{T}'$ la théorie \gui{skolémisée}, où l'on a skolémisé tous les axiomes existentiels en remplaçant les $\Exists$  par l'introduction de symboles de fonctions. 
Alors~$\sab{T}'$  est une extension \cosv de  \sa{T}.
\end{ftheorem}

\section[Treillis distributifs associés à une \sad]{Treillis distributifs et espaces spectraux associés à une \sad}
\label{fsubsectrdisad}

\subsection{Le treillis et le spectre de Zariski d'un anneau commutatif}

Le treillis de Zariski d'un anneau commutatif peut être obtenu à partir des règles valides dans différentes extensions de la théorie \sa{Ac} des anneaux commutatifs.

Nous choisissons la théorie des anneaux locaux en raison de leur rôle fondamental dans les schémas de Grothendieck.

On considère \prmt une \tdy des \textsl{anneaux locaux avec unités} \sa{Al1}, basée sur la signature $(\,\cdot=0,\U(\cdot)\mathrel{;}\cdot+\cdot,\cdot\times \cdot,-\,\cdot,0,1\,)$.

Cette théorie est une extension de la théorie des anneaux commutatifs.
Un prédicat~$\U(x)$ est défini comme le prédicat d'inversibilité au moyen des deux axiomes convenables.
On~ajoute l'axiome de collapsus  et l'axiome \tsbf{AL} des anneaux locaux proprement dit\footnote{Cela fait de la théorie \sa{Al1} une théorie \esid à une \tdij extension de la théorie des annaux commutatifs.}.

\DeuxRegles 
{\lab{CL$_{Al1}$} $\,\, \U(0)\Vd \Bot$}
{\Lab{AL} $\,\, \U(x+y) \Vd \U(x) \vou \U(y)$}

Soit $\gA$ un anneau commutatif. Considérons la \entrel $\vdash_{\gA,\mathrm{Zar} }$ sur l'ensemble sous-jacent à $\gA$ définie par l'\eqvc suivante 
\vspace{-.8em}
\begin{equation} \label {feqZarclass}
\begin{aligned} 
 a_1,\dots,a_n  &\,\vdash_{\gA,\mathrm{Zar}}    c_1,\dots,c_m   
  \qquad\quad     \equidef     \\[.3em] 
\U(a_1)\vet \dots\vet  \U(a_n) & \Vdi{\sA{Al1}(\gA)}   \U(c_1)\vou \dots\vou   \U(c_m) 
 \end{aligned}
\end{equation}

On définit le \textsl{treillis de Zariski de $\gA$}, noté $\ZarA$ ou 
$\Zar(\gA)$,
 comme celui engendré par la \entrel  $\vdash_{\gA,\mathrm{Zar}}$.
 
L'application correspondante $\DA:\gA\to\ZarA$ s'appelle le \textsl{support de Zariski de $\gA$}. Lorsque~$\gA$ est fixé par le contexte on note simplement $\rD$.

Le \textsl{spectre de Zariski} usuel est l'espace spectral dual de ce \trdi.

Notons que puisque $\rD(a_1)\vi\dots\vi\rD(a_n)=\rD(a_1\cdots a_n)$, les \elts de $\ZarA$ sont tous de la forme $\rD(c_1,\dots,c_m):= \rD(c_1)\vu\dots\vu\rD(c_m)$.

\smallskip 
Une théorie plus \elr est la théorie \sa{Asdz} des anneaux non triviaux \textsl{\sdz}, qui est obtenue à partir de \sa{Ac} en ajoutant l'axiome de collapsus et l'axiome \tsbf{SLZ}

\DeuxRegles {\lab{CL$_{Asdz}$} $\,\, 1=0\Vd \Bot$}
{ \Lab{SDZ} $\,\,xy=0  \Vd x=0 \vou y=0$.}

On démontre alors les \eqvcs suivantes (la première résulte de la définition
de~$\ZarA$). L'\eqvc avec le point \textsl{4} est l'objet du \nst formel.  Le \nst de Hilbert proprement dit est un résultat plus difficile. 
\begin{ftheorem}[\nst formel] \label{fthNstFormel} ~\\
 Soit $\gA$ un anneau commutatif, et $\an,c_1,\dots c_m\in\gA$. \Propeq
\vspace{-.5em}
\[ 
\begin{aligned} 
(1)\qquad\quad\;     \rD(a_1),\dots,\rD(a_n)  &\,\vdash_{\ZarA}    \rD(c_1),\dots,\rD(c_m)  \\[.3em] 
(2)\qquad    \U(a_1)\vet \dots\vet  \U(a_n)  &\Vdi{\sA{Al1}(\gA)}   \U(c_1)\vou \dots\vou   \U(c_m)  \\ 
(3)\qquad   c_1=0\vet \dots\vet  c_m=0  &\Vdi{\sA{Asdz}(\gA)}   a_1=0\vou \dots\vou   a_n=0  \\ 
(4)\qquad   \exists k>0\;\;(a_1 \cdots a_n)^k&\,\in\gen{c_1,\dots,c_m} 
 \end{aligned}  \label {feqNstFormel}
\] 
\end{ftheorem}

On peut donc identifier l'\elt $\rD(c_1,\dots,c_m)$ de $\ZarA$ 
à l'\id $\sqrt[\gA]{\gen{c_1,\dots,c_m}}$. Modulo cette identification, la relation d'ordre est la relation d'inclusion.

\begin{fcorollary} \label{fcorZarA}
Le treillis $\ZarA$ est engendré par la plus petite \entrel sur (l'ensemble sous-jacent~à)~$\gA$ satisfaisant les relations suivantes
(ici $\bot=0_\gT$ et $\top=1_\gT$).

\DeuxRegles
{
\labu $\;\;0\vdash \bot$
\labu $\;\;ab\vdash a$
\labu $\;\;a+b\vdash a,b$
}
{
\labu $\;\; 1\vdash \top$
\labu $\;\;a,b\vdash ab$
} 

\noindent En d'autres termes, l'application $\rD:\gA\to\ZarA$ satisfait les relations 
\[
\rD(0)=0,\;\rD(1)=1,\;\rD(ab)=\rD(a)\vi\rD(b),\;\rD(a+b)\leq \rD(a)\vu \rD(b),
\]
et toute autre application $\rD':\gA\to T$ qui satisfait ces relations se factorise via $\ZarA$ avec un unique morphisme de \trdis $\ZarA\to T$.
\end{fcorollary}

\subsection{Le treillis et le spectre réels d''un anneau commutatif}
  
\noindent On considère une \textsl{\tdy des corps ordonnés discrets}~\sa{Cod} basée sur la signature $(\,\cdot=0,\cdot\geq 0,\cdot>0\mathrel{;}\cdot+\cdot,-\,\cdot,0,1\,)$, par exemple celle décrite dans~\cite[Section~3]{fCLR01}. 

 Le \textsl{treillis réel}
 d'un anneau commutatif $\gA$, noté $\Reel(\gA)$, est engendré par la \entrel $\,\vdash_{\gA,\mathrm{Réel}} $ sur (l'ensemble sous-jacent à) $\gA$ définie par l'\eqvc suivante
 
\vspace{-.8em}
\begin{equation} \label {feqreel}
\begin{aligned} 
 a_1,\dots,a_n  &\,\vdash_{\gA,\mathrm{Réel}}    c_1,\dots,c_m   
  \qquad\quad     \equidef     \\[.3em] 
a_1>0\vet \dots\vet  a_n>0 & \Vdi{\sA{Cod}(\gA)}   c_1>0\vou \dots\vou   c_m>0 
\end{aligned}
\end{equation}

On note $\rR:\gA\to \Reel(\gA)$ l'application correspondante.
On a alors l'\eqvc suivante (que l'on appelle un \textsl{Positivstellensatz formel}) 
\[\rR(a_1),\dots,\rR(a_n) \,\vdash_{\Reel(\gA)}    \rR(c_1),\dots,\rR(c_m) \Longleftrightarrow \exists k\in\NN\;\exists p\in C\;\;(a_1 \cdots a_n)^k+ p =0
\]
où $C$ désigne le cône positif engendré par les $a_i$ et les $-c_j$.
Si $n=0$ on remplace $(a_1 \cdots a_n)^k$ par~$1_\gA$.

Le spectre réel usuel $\Sper(\gA)$ est l'espace spectral dual de ce \trdi $\Reel(\gA)$.
Les \elts de $\Sper(\gA)$ s'identifient en \clama aux cônes premiers de $\gA$. L'\oqc correspondant à l'\elt $\rR(a)\in\Reel\gA$ est donné par $\sotq{\fc\in\Sper\gA}{-a\notin\fc}$.

On peut démontrer que treillis 
$\Reel(\gA)$ est engendré par la plus petite \entrel sur $\gA$ satisfaisant les relations suivantes

\DeuxRegles{
\labu $\,\,-x^2 \vdash  $
\labu $\,\,x+y \vdash x , y $
\labu $\,\,xy \vdash x  , -y $
}
{
\labu $\,\,\vdash 1  \phantom{x^2}$
\labu $\,\,x , y \vdash xy $
}

De manière \gnle on peut définir $\Reel(\gA)$ et $\Sper(\gA)$ pour toute \sad $\gA$ du type \sa{Cod}.

Pour plus de détails voir \cite{fCC00,fLom2020}.
   
\subsection{Autres exemples}

\paragraph{Premier exemple.} Considérons une  \sad $\gA=\big((G,R),\sa{T}\big)$ pour une \tdy $\sa{T}=(\cL,\cA)$.
 Si $P(x,y)$ est un prédicat binaire dans la signature, et si  $Tcl=\Tcl(\gA)$  est l'ensemble des termes clos  de~$\gA$, on obtient une \entrel $\vdash_{\gA,P}$ sur  $Tcl \times Tcl$  en posant
 
\vspace{-.8em}
\begin{equation} \label {feq1}
\begin{aligned} 
 (a_1,b_1),\dots,(a_n,b_n)  &\,\vdash_{\gA,P}    (c_1,d_1),\dots,(c_m,d_m)   
  \qquad\quad     \equidef     \\[.3em] 
   P(a_1,b_1)\vet \dots\vet  P(a_n, b_n) & \Vdi{\gA}   P(c_1, d_1)\vou \dots\vou   P(c_m, d_m) 
 \end{aligned}
\end{equation}

Intuitivement le \trdi engendré par cette \entrel est le treillis des \gui{valeurs de vérité} du prédicat~$P$ dans la \sad $\gA$.
 
\paragraph{Plus généralement.} Considérons une  \sad $\gA=\big((G,R),\sa{T}\big)$ pour une \tdy $\sa{T}=(\cL,\cA)$.
Soit $S$  un ensemble de formules atomiques closes de $\gA$. On définit la \entrel sur $S$ associée à $\gA$  comme suit:  
 
\vspace{-.5em}
\begin{equation} \label {feq2}
\begin{aligned}
 A_1,\dots,A_n  &\,\vdash_{\gA,S} B_1,\dots,B_m   
   \qquad\quad \equidef     \\[.2em] 
    A_1\vet \dots\vet  A_n &\Vdi{\gA} B_1\vou\dots\vou B_m   
 \end{aligned}
\end{equation}
On pourra noter $\Zar(\gA,S)$ le \trdi engendré par cette \entrel.

%

\paragraph{Une extension conservative \sa{T1} d'une \tdy \sa{T}} donne lieu à des treillis de Zariski isomorphes pour les \sads $\sa{T}(\gA)$ et $\sa{T1}(\gA)$, 
pour un même ensemble $S$ de termes clos. 
Le treillis donne donc une image amoindrie de la \sad. Par exemple ajouter la logique classique et skolémiser une \tdy ne changent pas les treillis associés.
Mais le treillis peut être plus facile à étudier dans la deuxième théorie.
Pour retrouver la richesse des \tdys vues d'un point de vue \cof, il faut alors faire appel à la théorie des faisceaux ou des topos.

\subsection{Le treillis de Zariski (absolu) d'une \sad $\gA$}

Pour une \sad $\gA=\big((G,R),\sa{T}\big)$, le treillis de Zariski (absolu)  de~$\gA$ est défini en prenant pour~$S$ l'ensemble $\Atcl(\gA)$ de toutes les formules atomiques closes de~$\gA$. On le note $\Zar(\gA)$ ou $\Zar(\gA,\sa T)$ ou par un nom particulier correspondant à la théorie \sa{T}, par exemple~$\val\gA$
pour la théorie \Sa{val} (voir \paref{fsubsubsecval}).

L'espace spectral dual est appelé le \textsl{spectre de Zariski (absolu) de la \sad $\gA$}. On peut aussi lui attribuer un nom particulier. 

Lorsque la théorie $\sa{T}$ est \dij, le \trdi $\Zar(\gA)$
est (à un isomorphisme canonique près) celui défini par la \entrel sur 
$\Atcl(\gA)$ engendrée par les axiomes dans $R$ et les instantiations des axiomes de \sa{T} obtenues en y remplaçant les variables par des termes clos arbitraires. 

Pour un anneau commutatif $\gA$, on peut voir en particulier $(\ZarA)\oop$ comme le treillis de Zariski absolu de $\sa{Asdz}(\gA)$.

\subsection{Spectre et modèles en \clama}\label{fspecetmodeles} 

Les modèles dont on parle ici sont en \gnl vus du point de vue des \clama. Cela est justifié par le fait que l'utilisation de la logique classique ne change pas la validité des \rdys écrites dans le langage de départ. En particulier le \trdi $\Zar(\gA,S)$ reste le même lorsqu'on ne change pas $S$ tout en ajoutant la logique classique.

En général on choisit
pour ensemble $S$ les formules atomiques closes construites sur un seul ou sur un petit nombre de prédicats du langage, en s'arrangeant pour que les autres prédicats puissent être définis à partir de ceux de $S$ dans la théorie formelle correspondante, au moins en \clama. 

En effet, donner un point du spectre
$\Spec(\Zar(\gA,S))$, \cad un morphisme \hbox{$\alpha:\Zar(\gA,S)\to\Deux$} revient à attribuer la valeur $\Vrai$ ou  $\Faux$ aux formules atomiques closes de~$S$ (selon que~$\alpha(B)=1$ ou~$0$), donc aussi à toutes les formules atomiques closes de~$\gA$ lorsque~$S$ est bien choisi. 
Cela permet donc de construire un modèle de $\gA$ dans le cas d'une \tdij.
 Le modèle est \textsl{minimal} au sens où tous ses \elts sont construits  à partir des \gtrs~$G$ au moyen des symboles de fonction donnés dans la signature.

\smallskip  Le choix de l'ensemble~$S$ influe alors sur la topologie
de l'espace spectral associé. Deux choix différents de~$S$ peuvent donner les mêmes points du spectre en \clama mais pas la même topologie, et donc définir deux espaces spectraux différents.


%


\section{Treillis et spectre valuatifs d'un anneau commutatif}\label{fsecdival}

\subsection{Relation de divisibilité valuative}\label{fsubsecdival1}

Référence: \cite{fHK1994}.

\Subsubsection{Domaines de valuation, relation de divisibilité valuative}\label{fsubsubsecdomval}

Rappelons qu'une partie $P$ d'un ensemble $E$ est dite \textsl{détachable} lorsque la \prt~\hbox{$x\in P$} est décidable pour les $x\in E$. 
Pour décrire cette situation dans le cadre d'une \tdy on introduit les règles suivantes comme axiomes, pour un prédicat $Q(x)$ qui traduit l'appartenance $x\in P$, et pour le prédicat opposé $R(x)$ que l'on est obligé d'introduire dans la signature: \gui{$\Vd Q(x) \vou R(x)$}  et \gui{$Q(x) \vet R(x) \Vd \Bot $}

\smallskip Nous disons qu'un anneau est \textsl{intègre} (ou qu'il est un domaine d'intégrité), lorsque tout \elt est nul ou régulier, et qu'un anneau est un \textsl{corps discret} lorsque tout \elt est nul ou inversible. Ces \dfns n'excluent pas l'anneau trivial. 

Un anneau est dit \textsl{\sdz} lorsque, chaque fois que l'on a $ab=0$
dans l'anneau, on a l'alternative explicite: $a=0$ ou $b=0$.  
 Un anneau intègre est \sdz. La réciproque, valable en \clama, n'est pas assurée \cot.

Nous disons qu'un idéal est \textsl{premier} s'il donne au quotient un anneau \sdz. Cela n'exclut pas l'\id $\gen{1}$.
Les conventions précédentes (adoptées dans \cite{fACMC}) évitent d'utiliser la négation et permettent d'éviter certains raisonnements cas par cas  litigieux d'un point de vue \cof.

Mais pour les \tdys correspondantes, nous revenons à la tradition qui considère qu'un anneau local, \sdz, ou intègre, ainsi qu'un corps discret, doit être non trivial: cela permet d'avoir un axiome de collapsus qui introduit $\Bot$ comme conclusion possible d'une \rdy.  

\smallskip Un \textsl{domaine de valuation} $\gV$ est un anneau intègre dans lequel la relation de \dve (sur le \mo multiplicatif $\gV/\gV\eti$) est un ordre total: $\forall x, y\,(x\di y\vu y\di x)$. 

Si $\gK$ est le \cdf de $\gV$, on dit alors que $\gV$ est un \textsl{anneau de valuation du \cdi $\gK$} et que $(\gK,\gV)$ est un \textsl{corps discret valué}. Plus \gnlt, si $\gK$ est un \cdi, un sous-anneau $\gV$ est appelé un \adv de $\gK$ lorsque, pour tout $x\in\gK\eti$,
$x$ ou~$x^{-1} \in \gV$. 

%
%
%
%

D'un point de vue \cof, on définit un \textsl{\cvd} en imposant que le caractère discret concerne non seulement la relation d'égalité, mais aussi les relations $x\in \gV$ et $x\in \gV\eti$: elles doivent être décidables, ce qui veut dire que la relation de divisibilité dans~$\gV$ doit être explicite. Un domaine de valuation pour lequel la relation de \dve est explicite peut alors être caractérisé comme un domaine de Bézout local résiduellement discret\footnote{Un anneau local $\gA$ est dit \textsl{\dcd} lorsque le corps résiduel $\gA/\Rad\gA$ est discret. Si l'anneau est non trivial cela revient à dire que le groupe des unités $\Ati$ est une partie détachable.}.

En \clama tout corps valué est trivialement un \cvd en raison du principe du tiers exclu.

Dans la section 4 de l'article  \cite{fCLR01} est décrite une \tdy \sa{Cvd} pour les \cvds avec ce point de vue \cof. Sont utilisés pour cela trois prédicats $\Vr(x)$, $\U(x)$ et $\Rn(x)$ respectivement pour $x\in \gV$, $x\in \gV\eti$ et $x\in\gV\setminus\gV\eti$. Nous~y reviendrons dans la section~\ref{fsecdival6}.

\smallskip 
 Dans un \cvd $(\gK,\gV)$  on écrit $x\di y$ pour la relation de \dve étendue~à~$\gK$, \cad \prmt lorsqu'il existe \hbox{un $z\in\gV$} tel que $xz=y$.
On note $\Gamma=\Gamma(\gV)$ le groupe $\gK\eti\!/\gV\eti$ (noté additivement), avec la relation d'ordre $\leq$ induite par la relation~\hbox{$\cdot\di\cdot$} dans $\gK\eti$. On note $\Gamma_\infty=\Gamma\cup\so\infty$ (où $\infty$ est un \elt maximum ajouté de manière purement formelle). Dans ces conditions, l'application naturelle $v:\gK\to\Gamma_\infty$ (avec~\hbox{$v(0)=\infty$}) est appelée la \textsl{valuation} du \cvd. On~a alors $v(xy)=v(x)+v(y)$, $v(x+y)\geq \min(v(x),v(y))$ avec égalité si $v(x)\neq v(y)$. On~a aussi \fbox{$\gV=\sotq{x\in\gK}{v(x)\geq 0}$} et le groupe des unités est caractérisé par l'\egt \fbox{$\gV\eti=\sotq{x\in\gK}{v(x)= 0}$}.

\smallskip En \clama on définit une \textsl{relation de divisibilité valuative $a\di b$} sur un anneau commutatif $\gA$ comme l'image réciproque de la relation de divisibilité sur un \cvd $(\gK,\gV)$ par un homomorphisme d'anneau $\varphi:\gA\to\gK$. Autrement dit, on a $a\di b$ dans~$\gA$ \ssi $\varphi(a)\di\varphi(b)$
\hbox{dans $(\gK,\gV)$}.

\Subsubsection{Les points du spectre valuatif en \clama}\label{fsubsubsecspecval}

\begin{fdefinition} \label{fdefispevclass}
En \clama on définit les \elts du \textsl{spectre valuatif $\,\Spev(\gA)$ d'un anneau commutatif~$\gA$}
de la manière suivante (voir \cite{fHK1994}): un point de $\SpevA$ est donné par un couple~\hbox{$(\fp,\gV)$} où $\fp$ est un \idep de $\gA$ et $\gV$ un anneau de valuation du corps de fractions $\gK=\Frac(\gA/\fp)$. 
 \end{fdefinition}

On reconnait l'analogue du spectre réel: un point de $\SperA$ est donné par un couple~\hbox{$(\fp,\gC)$} où $\fp$ est un \idep de $\gA$ et $\gC$ un cône positif du corps de fractions $\gK=\Frac(\gA/\fp)$ (i.e. $\gC+\gC\subseteq \gC$, $\gC\cdot\gC\subseteq \gC$, $\gC\cup -\gC=\gK$, $\gC\cap -\gC=\so 0$).

Différentes topologies spectrales peuvent être définies sur le spectre valuatif au moyen d'ouverts convenables qui engendrent la topologie. 

Cela correspond à des \trdis convenables comme décrits dans la section~\ref{fsubsectrdisad} en rapport avec des \tdys qui décrivent 
de manière satisfaisante les \prts d'une
relation de divisibilité valuative, pour une signature donnée. Nous commençons par décrire certaines de ces théories.  

\subsection{Théories disjonctives pour la relation de \dve valuative}

\Subsubsection{La théorie \sa{val0}}\label{fsubsubsecval0}

\begin{fdefinition} \label{fdefithval0}

On choisit  la signature 
\[\Sigma_\mathrm{val0}=(\,\cdot\di\cdot, \cdot=0 \mathrel{;} \cdot+\cdot, \cdot\times\cdot,-\,\cdot,0,1\,).\]  
Les axiomes pour $x=0$ et $a\di b$ 
peuvent être pris comme suit pour décrire correctement une relation de divisibilité valuative.

\DeuxRegles
{\Lab{av0} $\,\,0\di  x\Vd x=0$
\Lab{av1}  $\vd 1 \di  -1$
\Lab{av2} $\,\,a \di  b \Vd ac \di  bc$
\Lab{Av1} $\,\,a \di  b \vet b \di  c \Vd a \di  c$
}
{
\lab{CL$_{\val}$} $\,\, 0\di 1 \vd \Bot$ \quad (collapsus)
\Lab{Av2} $\,\,a \di  b \vet a \di  c \Vd a \di  b + c$
\Lab{AV1} $\vd a \di  b \vou b\di a$
\Lab{AV2} $\,\,ax \di  bx  \Vd a \di  b \vou 0 \di x$
}

On note \sa{val0} la \tdij ainsi définie.
 
\end{fdefinition}

On démontre facilement $\vd -1 \di  1$, $\vd 1 \di  1$, $\vd 1 \di  0$,  $\vd x \di  x$, $\vd x \di  0$ et $x=0\Vd 0\di  x$. On vérifie aussi que les axiomes \Tsbf{ac1}, \Tsbf{ac2} et \Tsbf{ac3} des anneaux commutatifs sont satisfaits.

Donnons un énoncé précis pour l'affirmation selon laquelle la formalisation est correcte.

\begin{ftheoremc}[Modèles de la théorie \dij \sa{val0} en \clama] \label{flemAxdiv} ~\\
En \clama une relation binaire $a \di  b$ sur un anneau $\gA$ satisfait les axiomes ci-dessus exactement dans le cas suivant: 
\begin{itemize}
\item l'ensemble $\fp=\sotq{x\in \gA}{0\di x}$ est un \idep, 
on note $\ov a$ l'\elt~$a$ de~$\gA$ vu dans $\gA/\fp$, on note $\gK$ le corps de fractions $\Frac(\gA/\fp)$, 
\item les fractions  ${\ov a}/{\ov b}\in\gK$ telles que $\ov b\neq 0$ et $b\di a$ forment un anneau de valuation $\gV$ de~$\gK$, 
\end{itemize}
En d'autres termes les axiomes proposés pour $\cdot\di\cdot$ dans \sa{val0} correspondent à la \dfn d'une relation de \dve valuative (en \clama).
\end{ftheoremc}  
%
\begin{proof}  Tout d'abord, on voit facilement que les axiomes sont vérifiés si $\varphi:\gA\to\gK$ est un morphisme de $\gA$ dans un corps $\gK$, si $\gV$ est un \adv de $\gK$ et si $a\di b$ signifie: $\exists x\in \gV\;x\varphi(a)=\varphi(b)$. 

\noindent Montrons maintenant que les axiomes sont suffisants pour décrire cette situation (en \clama).
\\
Les axiomes \tsbf{av1} et \tsbf{av2} donnent les règles valides $\Vd a \di a$ et $\Vd a\di 0$. Ainsi, vu l'axiome \tsbf{Av1}, la relation \hbox{$a\di b$}
définit un préordre. Et si $0 \di 1$ alors on obtient $0\di b$ et $a\di b$ pour tous $a,b$ sans utiliser l'axiome \tsbf{CL$_{\val}$}.
\\
Un cas particulier de \tsbf{av2} est la règle valide suivante

\penalty-2500
\Regles{ 
\lab{av2'} $\,\,1\di a\vdg b \di ab$} 

\noindent En utilisant~\tsbf{Av2} et \tsbf{av2'}, on voit que pour tout $a\in \gA$ tel que $1\di a$, l'ensemble 
\[\uar a:=\sotq{b\in\gA}{a \di b}
\] 
est un idéal de $\gA$. 
En particulier $\fp=\uar 0$ est un idéal.
\\
Montrons que la relation $\cdot\di\cdot$ passe au quotient par $\fp=\uar 0$.
En effet, supposons $0\di x$. Il suffit de montrer que  $a\di a+x$ et $a+x\di a$. Tout d'abord $a\di 0\di x$ et $a\di a$, donc $a\di a+x$  par l'axiome~\tsbf{Av2}. Ensuite $a+x\di a+x$ donc par le point précédent  $a+x\di a+x-x$, i.e. $a+x\di a$.
\\
Montrons que $\gA/\fp$ est \sdz: si $0\di yx$ alors $0x\di yx$, et \tsbf{AV2} donne donc \hbox{$\,\,0\di xy\Vd 0\di x \vou 0\di y$}.
\\
En utilisant \tsbf{av2}, \tsbf{Av1} et \tsbf{Av2} on vérifie que les fractions $\ov a/\ov b$ de $\gK$ 
forment un sous-anneau~$\gV$ de~$\gK$. 
\\ 
Enfin, deux \elts inverses l'un de l'autre dans $\gK$ s'écrivent
$\ov a/\ov b$ et $\ov b/\ov a$.  
L'axiome \tsbf{AV1}
implique donc que $\gV$ est un anneau de valuation de $\gK$.
\end{proof}

On notera que $\varphi(\gA)$ n'est pas \ncrt un sous-anneau de $\gV$.

Vu le \thref{flemAxdiv}, on pose en \coma la \dfn suivante.
\begin{fdefinition} \label{fdefirdvcoma}
Une relation binaire $a\di b$ sur un anneau commutatif est appelée une \textsl{\rdv} si elle satisfait les axiomes de \sa{val0}.
\end{fdefinition}

\begin{fremark} \label{frem1val0}  
Dans la théorie \sa{val0} sont valides  les trois axiomes des
anneaux commutatifs.

\vspace{.3em}

\DeuxRegles{
\lab{ac1}  $\,\,\vd  0=0$
\lab{ac3}  $\,\,\vd  x=0 \vet y=0 \Vd  x+y=0$
}
{
\lab{ac2}  $\,\,  x=0 \Vd   xy=0$
}

Nous adoptons ici la convention donnée dans \cite{fCLR01} 
selon laquelle la partie purement calculatoire des anneaux commutatifs
est gérée \gui{à l'extérieur} de la théorie formelle  (voir \paref{fTpeAc}). 

Si maintenant on considère un anneau commutatif $\gA$ et la \sad~$\sa{val0}(\gA)$, tous les termes clos de $\sa{val0}(\gA)$ sont égaux à des \elts de $\gA$ et l'on n'a pas besoin de faire appel aux trois axiomes pour les \demos des \rdys \textsl{closes}. \eoe 
\end{fremark}

\begin{fremark} \label{fremColl}  
Sans l'axiome de collapsus, la relation $0\di1$, lorsqu'elle est valide, réduit~$\gK$ à l'anneau trivial. L'ajout de l'axiome $\,\,0\di1\Vd\Bot$ revient à faire disparaitre l'anneau trivial dans le néant\footnote{Ou, qui sait?, dans un trou noir.}.  \eoe
\end{fremark}

\begin{fremark} \label{fremA1A6coma} 
En \coma un anneau $\gA$ muni d'une relation $\cdot\di\cdot$ satisfaisant les axiomes de \sa{val0} donne un \idep $\fp$ de $\gA$ puis un sous-anneau \gui{convenable}~$\gV$ de l'anneau total des fractions $\gK$ de $\gA/\fp$. Mais $\gK$ n'est pas à priori un \cdi. Il est alors plus confortable de penser en termes de \sad, qui laisse la possibilité de \gui{faire comme si} 
on était en \clama, où tous les corps sont discrets. \eoe
\end{fremark}

\Subsubsection{La théorie \sa{val}, quelques \rdys prouvables}\label{fsubsubsecval}

\begin{fdefinition} \label{fdefithval}
La théorie  \SA{val}  est obtenue à partir de la théorie $\sa{val0}$ définie comme suit. La signature est obtenue en supprimant le prédicat $\cdot=0$. 
\[\Sigma_\mathrm{val}=(\,\cdot\di\cdot  \mathrel{;} \cdot+\cdot, \cdot\times\cdot,-\,\cdot,0,1\,)\] 
On supprime l'axiome \tsbf{av0}  de $\sa{val0}$ et on définit la relation $x=0$ comme une abréviation de~$0\di x$.  
 \end{fdefinition}

On aboutit à la même conclusion que dans le \thref{flemAxdiv}:  en \clama les modèles de la théorie \sa{val} sont les domaines de valuation.

Nous allons voir, sans référence aux modèles en \clama, que la \tdy \sa{val} 
retrouve certaines \prts des \cvds comme \rdys valides ou, parfois, seulement admissibles.

Tout d'abord le fait qu'un anneau de valuation est un anneau \sdz local normal correspond aux trois règles valides suivantes. Les démonstrations suivent sans surprise la démonstration du \thref{flemAxdiv}.

\vspace{.3em}

\Regles{ \lab{SDZ} $\,\,0 =  bx  \Vd 0 =  b \vou 0 = x$.}

\noindent Pour prouver \tsbf{SDZ} il suffit de prendre la règle \Tsbf{AV2} avec $a=0=0x$.\qed

\noindent  En particulier on a $\,\,0\di x^2\Vd0\di x$.  

\vspace{.3em}

\Regles{ \Lab{LOC} $\,\, 1\di  x\vet 1\di y\vet  x+y\di 1  \Vd x \di  1 \vou y \di 1$.}

\noindent D'après \tsbf{AV1} on ouvre deux branches, l'une où $x\di y$, l'autre où $y\di x$. Dans la première, comme $x\di x$, on a par \tsbf{Av2}, $x\di x+y$. Donc, par \tsbf{Av1}, $x\di 1$. Symétriquement dans l'autre branche $y\di 1$.\qed

\smallskip La troisième règle est \agq.

\vspace{.3em}

\Regles{ \Lab{Nor} $\,\, 1\di  a_1\vet\dots 1\di a_n \vet 
y\big(\sum_{k=0}^{n}a_kx^ky^{n-k}\big)\di x^{n+1} \Vd y \di  x$.}

\noindent On ouvre les branches $x\di y$ et $y\di x$. Il faut traiter la première. Prenons par exemple $n=2$. Puisque~$x\di y$, $1\di a_0$, $1\di a_1$ et $1\di a_2$,  on a 
\[x^2\di y^2\di a_0y^2,\quad 
x^2\di xy \di  a_1xy\quad \hbox{ et }\quad x^2\di a_2x^2.\] 
Donc $x^2\di a_2x^2+a_1xy+a_0y^2$ par~\tsbf{Av2}, et $yx^2\di x^3$. De $yx^2\di xx^2$ on déduit $\Vd 0\di x^2\vou y\di x\,$  par~\tsbf{AV2}. Et dans la branche où $0\di x^2$, \hbox{on a $y\di 0\di x$}. \qed

\smallskip Voici maintenant une règle admissible (en un sens précisé juste après) dans la théorie~\sa{val}.

\vspace{.3em}

\Regles {\Lab{DIV} $\,\,x\di y\Vd \Exists a\,(1\di a\vet ax=y)$}

\begin{flemma} \label{flemDivadmissible}
Si une \rdij 
est prouvable dans une \sad $\gB=\big((G,R),\sa{val}\big)$ après l'ajout de l'axiome~\tsbf{DIV}, elle est \egmt prouvable dans~$\gB$.
\end{flemma}

Il semble cependant que la preuve d'admissibilité n'est pas facile et
qu'elle doive attendre le \vst formel. Voir la remarque \ref{fremthVstformelval}. Nous n'utiliserons pas ce lemme avant de l'avoir démontré.

\Subsubsection{Structures \agqs dynamiques de type \sa{val}}

\begin{fdefinition} \label{fdefivalkK}~
\begin{enumerate}
\item Si $\gA$ est un anneau commutatif, la \sad $\sa{val}(\gA)$ est obtenue comme d'habitude en ajoutant le diagramme positif de l'anneau $\gA$.
\item Si $\gk\subseteq \gA$ sont deux anneaux\footnote{Nous utilisons $\gk$ comme petit anneau pour nous référer à l'intuition donnée dans la situation fréquente où~$\gk$ est un \cdi.
}, ou plus \gnlt si $\varphi:\gk\to\gA$ est une \alg, on note $\sa{val}(\gA,\gk)$ la \sad  dont la présentation est donnée par 
\begin{itemize}
\item le diagramme positif de~$\gA$ comme anneau commutatif;
\item les axiomes
$\vd 1 \di \varphi(x)$ pour les \elts $x$ de $\gk$.
%
\end{itemize}

%
%
\end{enumerate}
\end{fdefinition}

Les deux \sads $\sa{val}(\gA)$ et $\sa{val}(\gA,\gZ)$, où $\gZ$ est le plus petit sous-anneau de $\gA$, sont canoniquement isomorphes.

\smallskip Comme cas particulier de la règle \tsbf{SDZ}, si $e$ est un \idm d'un anneau~$\gA$, \hbox{on a $\Vd e=0 \vou e=1$} dans $\sa{val}(\gA)$. 
En particulier $\Vd 1 \di e$ et $\Vd 1\di f$ (où $f=1-e$). On en déduit alors pour tous $a,b\in\gA$: $ae\di be\vet af\di bf\Vd a\di b$.
En effet $a\di ae \di be$ et $a\di af \di bf$ donc $a\di be+bf=b$. Résumons la situation.

\begin{flemma} \label{flemidmval}
Si $e$ et $f\in\gA$ sont deux \idms complémentaires on a dans $\sa{val}(\gA)$

\vspace{.3em}

\DeuxRegles
{\labu $\,\,a\di b\vd (ae\di be\vet af\di bf)$}
{\labu $\,\,ae\di be\vet af\di bf\vd a\di b $}  
\end{flemma}

\smallskip Voici maintenant une conséquence de la validité de \Tsbf{Nor} dans \sa{val}. 

\begin{flemma} \label{flem-y-di-x}
Soient $\gk\subseteq \gA$ deux anneaux et $x\in\gA$, $y\in\gk$. Alors si $x$ est entier\footnote{L'\elt $x$ est dit \textsl{entier sur un \id} $\fa$ de $\gk$ s'il annule un \pol $x^{n+1}+\sum_{i=0}^n a_i x^{n-i}$ avec $a_i\in\fa^i$ pour chaque $i$ \cite[\dfn XII-2.1]{fACMC}. } sur l'\id~$\gen{y}$ de $\gk$, la \sad $\sa{val}(\gA,\gk)$ prouve~\hbox{$y\di x$}.
\end{flemma}
Nous verrons plus loin (\thref{fth-y-di-x}) que cette condition suffisante est \egmt \ncr. 


\subsection{Treillis et spectres valuatifs d'un anneau commutatif}\label{fsubsecTrSpecval}

\Subsubsection{Plusieurs topologies spectrales concurrentes}

Avec $\Gamma=\Gamma(\gV)$, le point $(\fp,\gV)$ de $\Spev(\gA)$ est caractérisé par la \gui{valuation} associée  
\[
w:\gA\to \Gamma_\infty, \;x \mapsto v(\pi_\fp(x))
\]
où $\pi_\fp:\gA\to\gA/\fp$ est la surjection canonique.

Nous réservons la notation $\SpevA$ pour le spectre  muni de la topologie qui nous semble la plus naturelle, celle où les ouverts qui engendrent la topologie sont les $\fO(a,b)$ définis comme suit:
\[
\fO(a,b):=\sotq{w}{w(b)\leq w(a)},\quad a,b\in\gA
\]

La topologie préférée dans \cite{fHK1994} est donnée par les ouverts générateurs de la topologie~$\fU(a,b)$ définis comme suit:
\[
\fU(a,b):=\sotq{w}{w(b)\leq w(a),\,w(b)\neq \infty},\quad a,b\in\gA
\]
L'espace spectral ainsi défini est noté $\Spv(\gA)$. 

\smallskip 
On~a alors (en notant $\ov Y$ la partie complémentaire de $Y$)  
\begin{enumerate}
\item [$\bullet$] $\fU(b,b)=\sotq{w}{w(b)\neq \infty}=\ov{\fO(b,0)}$
\item [$\bullet$] $\fU(a,b)=\fO(a,b)\cap\,\ov{\fO(b,0)}$
\item [$\bullet$] $\ov{\fU(b,b)}=\sotq{w}{w(b)=\infty}=\fO(b,0)$ 
\item [$\bullet$] $\fO(a,b)=\fU(a,b)\cup\big(\ov{\fU(b,b)}\cap\,\ov{\fU(a,a)}\big)$ 
\end{enumerate}

\noindent Ainsi   $\SpevA$ et
 $\SpvA$ définissent la même topologie constructible.

\smallskip Du point de vue d'une \tdy décrivant une relation de divisibilité valuative, la topologie de $\SpvA$ revient à considérer comme essentielle la \prt \gui{$b\di a\hbox{ et }b\neq 0$}.
Cette \dfn semble peu naturelle, d'autant qu'un peu plus loin dans l'article \cite{fHK1994} est introduite une topologie spectrale $\Spv'(\gA)$ avec pour ouverts de base les $\fF(a,b):=\sotq{v}{v(a)<v(b)}=\ov{\fO(a,b)}$.

\smallskip 
Nous décrivons maintenant les \trdis correspondant à ces espaces spectraux.

\Subsubsection{Le treillis $\val(\gA)$  et le spectre $\Spev\gA$}

\begin{fdefinition} \label{fdefivalA}
Soit $\gA$ un anneau commutatif. On considère l'ensemble (sous-jacent à)~$\gA\times \gA$ sur lequel on définit la \entrel $\vdash_{\gA,\mathrm{val}}$ par l'\eqvc suivante

\vspace{-.8em}
\begin{equation} \label {feqvalA}
\begin{aligned} 
 (a_1,b_1),\dots,(a_n,b_n)  &\,\vdash_{\gA,\mathrm{val}}\,   (c_1,d_1),\dots,(c_m,d_m)   
  \quad    \equidef     \\[.2em] 
  a_1\di b_1\vet\dots\vet a_n\di b_n &\Vdi{\sA{val}(\gA)}   c_1\di d_1\vou \dots\vou  c_m\di d_m 
 \end{aligned}
\end{equation}

\noindent Le treillis $\val(\gA)$ est défini comme celui engendré par la \entrel $\vdash_{\gA,\mathrm{val}}$.

\noindent On note $\rDi_\gA:\gA\times \gA\to \val(\gA)$ (ou simplement $\rDi$ si le contexte est clair) l'application correspondante.
 \end{fdefinition}

On~a par exemple, pour~\hbox{$a,b,c\in\gA$}   
\[\rDi(a,b)\vi\rDi(a,c)\leq \rDi(a,b+c),\quad \hbox{ dans } \val(\gA).\]
 
En fait, comme $\cdot\di \cdot$ est le seul prédicat de la \tdy \sa{val}, le treillis
$\val(\gA)$ est le treillis de Zariski absolu de la \sad  $\sa{val}(\gA)$. 

Comme la théorie \sa{val} satisfait la règle \Tsbf{SDZ}, le morphisme naturel  $\val(\gA)\to\val(\Ared)$ est un isomorphisme. 

\smallskip  D'après le \thref{flemAxdiv} la \sad~\hbox{$\sa{val}(\gA)$} admet en \clama pour modèles minimaux les points du spectre valuatif~\hbox{$\Spev(\gA)$} décrit au début de cette section. 
D'après ce qui est expliqué dans le paragraphe \textsl{Spectre et modèles en \clama} \paref{fspecetmodeles}, et vu la \dfn des ouverts $\fO(a,b)$ qui engendrent la topologie spectrale, on obtient le théorème suivant en \clama.

\begin{ftheoremc} \label{fthSpevSpecval}
L'espace spectral $\Spev\gA$ est canoniquement homéomorphe au spectre du \trdi $\val(\gA)$.
\end{ftheoremc}

En \coma, la \textsl{topologie sans point} donnée par le \trdi~$\val(\gA)$ est engendrée par les ouverts formels que sont les \elts~$\rDi(a,b)$ de $\val(\gA)$.  

En \clama, où il y a suffisamment de points dans 
$\Spec(\val(\gA))$, la topologie est une topologie usuelle. Par ailleurs on a défini l'espace~$\Spev \gA$,
dont les ouverts sont engendrés les $\fO(a,b):=\sotq{w}{w(b)\leq w(a)}$. 
Ceux-ci correspondent clairement aux $\rDi(b,a)$. Cela explique l'homéomorphisme entre $\Spev\gA$ et $\Spec(\val(\gA))$.

\smallskip \hum{Ne devrait-on pas permuter $a$ et $b$ dans la notation $\rDi(a,b)$?
Même si cela ne change pas les ouverts,  c'est un peu troublant
que $\fO(a,b)$ corresponde à $\rDi(b,a)$.}

\Subsubsection{Le treillis $\vals(\gA)$ et le spectre $\Spv\gA$}

 Comme indiqué précédemment, l'espace $\Spev\gA$ est à distinguer de l'espace $\Spv\gA$, bien que dans chacun de ces espaces les points puissent être représentés par les mêmes \hbox{couples $(\fp,\gV)$}.
 
 Considérons la \tdij \sa{val1} obtenue à partir $\sa{val}$ en ajoutant le prédicat~\hbox{$\cdot\neq 0$} opposé à $\cdot=0$. C'est une extension conservative de $\sa{val}$. Considérons maintenant le prédicat \fbox{$J(a,b)\eqdef (a\di b \,\vii\, a\neq  0)$}.

Alors, l'espace spectral $\Spv\gA$ est le spectre du \trdi $\vals(\gA)$ suivant.

\begin{fdefinition} \label{fdefival'A}
Le \trdi $\vals(\gA)$ est engendré par la 
\entrel~$\vdash_{\gA,\mathrm{val}\sta}$ sur (l'ensemble sous-jacent à)~$\gA\times \gA$ définie par

\vspace{-.8em}
\begin{equation} \label {feqvalp}
\begin{aligned} 
(a_1,b_1),\dots,(a_n,b_n)  &\,\vdash_{\gA,\mathrm{val}\sta}\,   (c_1,d_1),\dots,(c_m,d_m)   
  \quad    \equidef     \\[.2em]
 J(a_1,b_1)\vet \dots\vet J(a_n,b_n) &\Vdi{\sA{val1}(\gA)}   J(c_1,d_1)\vou \dots\vou  J(c_m,d_m)
 \end{aligned}
\end{equation}

\noindent On note $\rJ_\gA:\gA\times \gA\to \vals(\gA)$ l'application correspondante
(ou simplement $\rJ$ si le contexte est clair).
 \end{fdefinition}

Ainsi l'\elt $\rJ(b,a)$ de $\vals(\gA)$ correspond à l'ouvert $\fU(a,b)$ de 
$\Spv\gA$. Ces ouverts engendrent la topologie de $\Spv\gA$.

\subsection{Treillis et spectre valuatifs  d'une algèbre}\label{fsubsecTrSpecval2}

On considère dans cette section une \klg $\gA$, \cad un morphisme $\varphi:\gk\to\gA$ d'anneaux commutatifs. 

\Subsubsection{Le treillis $\val(\gA,\gk)$  et le spectre $\Spev(\gA,\gk)$}

Nous définissons le \trdi $\val(\gA,\gk)$
à partir de la \sad $\sa{val}(\gA,\gk)$ (voir \dfn \ref{fdefivalkK}) de la même manière que  $\val(\gA)$
est défini à partir de \sad $\sa{val}(\gA)$.

\begin{fdefinition} \label{fdefivalkK0}
Soit $\gk$  un sous-anneau d'un anneau $\gA$, 
ou plus \gnlt considérons une \klg $\varphi:\gk\to\gA$.
  On définit le \trdi $\val(\gA,\gk)$ comme celui engendré par la \entrel 
  $\vdash_{\gk,\gA,\mathrm{val}}$ sur l'ensemble $\gA	\times \gA	$ définie comme suit. 
\vspace{-.8em}
\begin{equation} \label {feqvalkK}
\begin{aligned} 
 (a_1,b_1),\dots,(a_n,b_n)  &\,\vdash_{\gk,\gA,\mathrm{val}}\,   (c_1,d_1),\dots,(c_m,d_m)   
  \quad    \equidef     \\[.2em] 
  a_1\di b_1\vet \dots\vet a_n\di b_n &\Vdi{\sA{val}(\gA,\gk)}   c_1\di d_1\vou \dots\vou  c_m\di d_m 
 \end{aligned}
\end{equation}

\noindent On note $\rDi_{\gk,\gA}:\gA	\times \gA \to \val(\gA,\gk)$ (ou simplement $\rDi$ si le contexte est clair) l'application correspondante.

\end{fdefinition}

\begin{fremark} \label{fremvalKANOR} 
Comme la théorie \sa{val} satisfait la règle \Tsbf{Nor}, si  
$\ov\gk$ est la \cli de (l'image de)~$\gk$ dans~$\gA\Red$, la \sad $\sa{val}(\gA,\gk)$ prouve $1\di x$ pour \hbox{les $x\in\ov\gk$}. Par suite  le morphisme naturel  $\val(\gA,\gk)\to\val(\gA\Red,\ov\gk)$ est un isomorphisme.   \eoe 
\end{fremark}

Nous noterons
 $\Spev(\gA,\gk)$ le spectre du treillis $\val(\gA,\gk)$. Les points de ce spectre correspondent
aux points $(\fp,\gV)$ de  $\Spev(\gA)$ tels que $\gV$ contienne (l'image de) $\gk$. Comme le treillis $\val(\gA,\gk)$ est un quotient de $\val(\gA)$,
 $\Spev(\gA,\gk)$ est un sous-espace spectral de $\Spev\gA$.

\smallskip
Lorsque $\gK$ est un \cdi transcendant sur un sous-corps $\gk$, l'espace spectral $\Spev(\gK,\gk)$ est souvent appelé le \textsl{spectre de Zariski-Riemann} de $(\gK,\gk)$, ou encore, la \textsl{surface de Riemann abstraite}  de $(\gK,\gk)$.  Ses points sont les anneaux de valuation de $\gK$ qui contiennent $\gk$. 

En \coma on étudie plutôt le treillis valuatif $\val(\gK,\gk)$, dont le spectre est $\Spev(\gK,\gk)$. En fait cette surface de Riemann abstraite n'est pas seulement un espace spectral mais un schéma de Grothendieck ayant pour base cet espace spectral, et c'est ce schéma qu'il est véritablement utile d'étudier, aussi bien du point de vue classique que du point de vue constructif.

\smallskip Un autre cas particulier est celui où $\gA$ est le corps des fractions d'un anneau intègre $\gk$, par exemple dans l'étude de la théorie de l'anneau $\gk$ de tous les entiers \agqs.

\begin{flemma} \label{flem-valkk-valkK}
Soit $\gk$ un anneau intègre de corps de fractions $\gK$. Le morphisme naturel $\val(\gk,\gk)\to\val(\gK,\gk)$ est un \iso. 
\end{flemma}
%
\begin{proof}
Il y a plus de formules atomiques dans $\sa{val}(\gK,\gk)$ que dans 
$\sa{val}(\gk,\gk)$. Mais si $x\di y$ est une formule atomique de 
$\sa{val}(\gK,\gk)$ où $x=\frac u v$ et $y=\frac s t$ avec 
$u,v,s,t\in\gk$ et $v,t\neq 0$, les formules $x\di y$ et $ut\di sv$ sont prouvablement \eqves dans $\sa{val}(\gK,\gk)$.

\noindent Il nous suffit donc de voir que l'on a
(pour des $a_i,b_i,c_j,d_j\in\gk$)
\[
a_1\di b_1\vet \dots\vet a_n\di b_n  \Vdi{\sA{val}(\gK,\gk)}   c_1\di d_1\vou \dots\vou  c_m\di d_m
\]
\ssi 
\[
a_1\di b_1\vet \dots\vet a_n\di b_n  \Vdi{\sA{val}(\gk,\gk)}   c_1\di d_1\vou \dots\vou  c_m\di d_m
\]
En effet, l'utilisation de fractions dans une \demo peut toujours être évitée par la technique usuelle qui consiste à chasser les dénominateurs. 
\end{proof}
%

\Subsubsection{Le morphisme centre (1)}\label{fsubseccomparvalZar0}

\begin{fproposition}[Le morphisme centre pour les \trdis $\val(\gA,\gA) $ et~$\ZarA$] \label{fpropdefCentre} 
Il existe un unique morphisme de \trdi  $\gamma:\ZarA\to\val(\gA,\gA)$
tel que \hbox{$\gamma(\DA(a))=\Di_\gA(a,1)$} pour tout $a\in\gA$. 
\end{fproposition}
%
\begin{proof}
D'après le corolaire \ref{fcorZarA}, il suffit de vérifier que les règles suivantes sont valides dans $\sa{val}(\gA,\gA)$.

\DeuxRegles
{
\labu $\,\,0\di1\Vd \Bot$
\labu $\,\,ab\di1\Vd a\di1$
\labu $\,\,a+b\di1\Vd a\di1\vou b\di1$
}
{
\labu $\Vd 1\di1$
\labu $\,\,a\di1\vet\, b\di1\Vd ab\di1$
}

\noindent Et il suffit pour cela de vérifier que les règles suivantes sont valides dans $\sa{val}(\gA)$. 

\DeuxRegles
{
\labu $\,\,0\di1\Vd \Bot$
\labu $\,\,1\di b\vet\, ab\di1\Vd a\di1$
\labu $\,\,1\di a\vet\, 1\di b\vet\, a+b\di1\Vd a\di1\vou b\di1$
}
{
\labu $\Vd 1\di1$
\labu $\,\, a\di1\vet\, b\di1\Vd ab\di1$
}

\noindent Ce que l'on démontre facilement.
\end{proof}

En \clama, le point de vue dual est donné par \textsl{l'application spectrale centre}  
$\;\Spev(\gamma):\Spev(\gA,\gA)\to \Spec\gA\,$ qui, au point $(\fp,\gV)$ de $\Spev(\gA,\gA)$, 
fait correspondre l'\idep $\varphi^{-1}(\Rad(\gV))$, où $\varphi:\gA\to\Frac(\gA/\fp)$ est l'application canonique (notez que $\varphi(\gA)\subseteq \gV$ parce que l'on force $1\di a$ pour les $a\in\gA$).

\begin{fremark} \label{frempropdefCentre} 
On démontrera plus loin (\thref{fthVstformelval4}) que le morphisme centre est injectif.
Sans doute une \demo plus directe est possible. \eoe   
\end{fremark}

\begin{ftheorem} \label{fcor2thVstformelval4}
Si $\gk$ est un anneau arithmétique intègre (un \ddp), le morphisme $\gamma:\Zar\gk \to \val(\gk,\gk)$
défini en \ref{fpropdefCentre}  est un isomorphisme de \trdis. 
\end{ftheorem}
%
\begin{proof} Notons $\gK$ le corps de fractions de $\gk$. D'après le lemme \ref{flem-valkk-valkK}, on peut remplacer $\val(\gk,\gk)$ par $\val(\gK,\gk)$ dans la \demo.

\noindent On suppose déjà démontré que $\gamma$ est injectif.

\noindent 
Démontrons que $\gamma$ est surjectif. 
Par \dfn, dans un \textsl{\anar}, pour tout couple $(a,b)$ on a $s,t,u,v$ vérifiant 
[sa=ub, \;tb=va \,\hbox{ et }\,s+t=1.]
Étudions les conséquences de l'hypothèse $a\di b$ dans $\sa{val}(\gK,\gk)$.   \\
On~a $ub=sa$ et $sa\di sb$ (car $a\di b$), donc $ub\di sb$ et par \Tsbf{AV2}: 
$a\di b\Vd u\di s \vou b=0$.  
\\
Comme $s+t=1$, la règle \tsbf{LOC}  donne $\Vd t\di 1 \vou s\di 1$.
\\
Dans le cas $s\di1$, on a $\Vd u\di 1\vou b=0$. 
Si $b\neq 0$, la branche $b=0$ meurt parce que $b$ est inversible dans $\gK$. Cela donne dans $\sa{val}(\gK,\gk)$:  $\Vd u\di 1$.

\noindent On obtient alors pour la \sad $\sa{val}(\gK,\gk)$ les règles valides suivantes.

\DeuxRegles
{
\labu $\,\,a\di b\vd t\di1\vou u\di1$\quad  si $b\neq 0$
\labu $\,\,t\di1 \vd a\di b$
}
{
\labu $\,\,u\di1\vd a\di b$
}

\smallskip \noindent Et donc dans le treillis $\val(\gK,\gk)$ on a l'\egt 
\[
\Di(a,b)=\Di(t,1)\vu \Di(u,1) \quad \hbox{ si } b\neq0.
\]
Comme  $\Di(a,0)=1$ cela montre que $\gamma:\Zar\gk \to \val(\gK,\gk)$ est surjectif.
\end{proof}
%

\subsection{La théorie \sa{Val} et le treillis $\Val(\gK,\gk)$}

Le treillis valuatif $\Val(\gK,\gk)$ défini dans l'article \cite[Space of Valuations]{fCoq2009} 
est un treillis tout à fait semblable au treillis $\val(\gK,\gk)$.
Mais sa \dfn s'appuie implicitement sur une \tdij \sa{Val} différente de 
\sa{val}.

Il s'agit de la théorie \sa{Ac} des anneaux commutatifs à laquelle on ajoute un prédicat~$\Vr(x)$ soumis à certains axiomes qui sont satisfaits par un \adv d'un \cdi. Le prédicat $x\di y$ n'est pas utilisé, et les axiomes pour $\Vr$ sont minimalistes\footnote{Les noms \tsbf{cvi} et \tsbf{CV2}
proviennent de la théorie des \cvds abordée dans la section \ref{fsecdival6}.}. On vérifie facilement que les axiomes que nous donnons sont \eqvs à ceux des articles \cite{fCP2001,fCoq2009}.

\DeuxRegles{
\lab{cv1} $\,\,x=0\vet \Vr(y) \Vd  \Vr(x+y)   $
\lab{cv3} $\,\,\Vr(x)\vet\Vr(y)\Vd \Vr(xy)   $
\lab{CV2} $\,\,xy =1\Vd  \Vr(x)\vou\Vr(y)  $
}
{
\lab{cv2} $\vd  \Vr(-1)   $
\lab{cv4} $\,\, \Vr(x)\vet\Vr(y)\Vd \Vr(x+y)   $
\lab{CL$_{\Val}$} $\,\, 0\di 1 \vd \Bot$ \quad (collapsus)
}

Nous noterons cette théorie \sa{Val}. Nous commençons par un lemme facile.

\begin{flemma} \label{flemvalVal}
Dans la théorie \sa{val}, si on lit $\Vr(x)$ comme une abréviation de $1\di x$, les axiomes de \sa{Val} sont des règles valides.  
\end{flemma}
%
\begin{proof}
Le seul point non évident est la validité de \Tsbf{CV2}. Démontrons la \rdij suivante, a priori plus forte.

\Regles {\labu $\,\,\Vr(xy)\Vd  \Vr(x)\vou\Vr(y)  $}

\noindent On~a comme cas particulier de \tsbf{AV1} la validité de la règle

\Regles {\labu $\vd  \Vr(x)\vou x\di 1  $} 

\noindent Si $x\di 1$, alors $x\di 1\di xy$, donc $1.x\di y.x$ et par \tsbf{AV3}, $1\di y \vou x\di 0$; dans le dernier cas  $1\di xy\di 0$, collapsus.
\end{proof}

Les notations de l'article \cite{fCoq2009} sont légèrement différentes des nôtres. 
 
Dans \cite{fCoq2009},~$\gk$ est toujours un anneau intègre contenu dans un corps $\gK$. Dans ce contexte, notre $\val(\gK,\gk)$ correspond bien  au treillis noté $\Val(\gK,\gk)$ dans \cite{fCoq2009}. Par contre $\Val(\gk)$ est une abréviation de $\Val(\Frac(\gk),\gk)$. Cela correspond à notre $\val(\Frac(\gk),\gk)$, isomorphe à $\sa{val}(\gk,\gk)$. Notons $\gZ$  le plus petit sous-anneau  de $\gk$. Notre $\val(\gk)\simeq \val(\gk,\gZ)$ correspond à $\Val(\Frac(\gk),\gZ)$,  

En comparaison de la théorie \sa{val}, le seul défaut que l'on puisse trouver à la théorie \sa{Val} est de n'être pas assez \gnle, puisqu'à priori elle réclame des anneaux intègres. 

\smallskip Nous allons maintenant nous convaincre que les deux treillis 
$\val(\gK,\gk)$ et $\Val(\gK,\gk)$ sont isomorphes dans le contexte de l'article \cite{fCoq2009}. Il nous faut d'abord rappeler la \dfn du treillis   
$\Val(\gK,\gk)$.

\begin{fdefinition} \label{fdefiValkK} Soit $\gK$ un \cdi et $\gk$ un sous-anneau de $\gK$.
\begin{enumerate}
\item 
La \sad $\sa{Val}(\gK,\gk)$ est obtenue en ajoutant aux axiomes de~\sa{Val} le diagramme positif de $\gK$ comme anneau commutatif et les règles~\hbox{$\Vd \Vr(x)$} pour \hbox{les $x\in\gk$}.

\item 
 Le treillis $\Val(\gK,\gk)$ de \cite{fCoq2009} est par \dfn le \trdi engendré par la \entrel $\vdash_{\gK,\gk,\mathrm{Val}}$ sur $\gK$ définie par l'\eqvc suivante.

\vspace{-.8em}
\begin{equation} \label {feqZarclassVal}
\begin{aligned} 
 a_1,\dots,a_n  &\,\vdash_{\gK,\gk,\mathrm{Val}}    c_1,\dots,c_m   
  \qquad\quad     \equidef     \\[.2em] 
\Vr(a_1)\vet \dots\vet  \Vr(a_n) & \Vdi{\sA{Val}(\gK,\gk)}   \Vr(c_1)\vou \dots\vou   \Vr(c_m) 
 \end{aligned}
\end{equation}

\noindent On note $\rV_{\gK,\gk}:\gK\to \Val(\gK,\gk)$ (ou simplement $\rV$ si le contexte est clair) l'application correspondante.
\end{enumerate}
 \end{fdefinition}

Dans l'article \cite{fCoq2009}, la relation implicative est prise sur $\gK\etl$ plutôt que sur $\gK$.
Mais cela ne change rien au \trdi engendré puisque $\Vr(0)$ peut être remplacé  par $\Vr(1)$ (ou par $\Top$). 

Ainsi $\Val(\gK,\gk)$ relève de la même \dfn que $\val(\gK,\gk)$
(\dfns \ref{fdefivalkK} et \ref{fdefivalkK0}) et le lemme \ref{flemvalVal} nous donne un morphisme naturel
\begin{equation} \label {feqtheta}
\theta:\Val(\gK,\gk)\to \val(\gK,\gk)
\end{equation}

Nous allons démontrer qu'il s'agit d'un \iso lorsque $\gK$ est un \cdi. Un petit travail préparatoire est \ncr.

\Subsubsection{La théorie $\sa{Val}^+$}

Pour démontrer que $\theta$ est un \iso, on commence par étendre la théorie \sa{Val} en introduisant le prédicat $y\di z$ au moyen des axiomes suivants qui le définissent comme équivalent à $\exists x\,(\Vr(x)\vii  z=xy)$.

\DeuxRegles{
\Lab{Div} $\,\,\Vr(x)\vet z=xy \Vd  y\di z   $
}
{
\lab{DIV} $\,\,y\di z\Vd  \Exists x\;(\Vr(x)\vet z=xy)  $
}

On voit facilement que $\Vr(x)$ équivaut à $1\di x$, et que $y\di z$ équivaut à $\exists x\,(\Vr(x)\vii  z=xy)$.

La nouvelle théorie, que nous notons $\sa{Val}^+$, est une extension \esid de \sa{Val}
puisqu'on s'est limité à introduire un nouveau prédicat correctement défini. 
En particulier, les \rdijs valides ne changent pas, et le \trdi défini à partir de  $\sa{Val}^+(\gK,\gk)$ est le même que celui défini à partir de $\sa{Val}(\gK,\gk)$. 

Maintenant nous notons que pour toute règle valide dans \sa{val} de la forme
\begin{equation} \label {feqlem49}
a_1\di b_1\vet \dots\vet a_n\di b_n  \Vdi{\sA{val}(\gK,\gk)}   c_1\di d_1\vou \dots\vou  c_m\di d_m\qquad (a_i, b_i, c_j, d_j\in\gK),
\end{equation}
\noindent   une démonstration dynamique de cette règle  peut se dérouler entièrement avec des termes pris dans $\gK$. En effet, \sa{val} est une \tdij et seul un axiome existentiel permettrait d'introduire par force de nouveaux \elts au titre de variables fraiches.

\begin{flemma} \label{flemvalVal2} Soient $\gK$ un \cdi et $\gk$ un sous-anneau.
\begin{enumerate}
\item Dans le cas des \sads $\sa{Val}^+(\gK,\gk)$ et $\sa{val}(\gK,\gk)$  les axiomes de \sa{val} utilisés pour la \dfn du treillis $\val(\gK,\gk)$, \cad pour la \demo  d'une règle  \pref{feqlem49}, sont valides pour $\sa{Val}^+$.
\item  En conséquence une règle \pref{feqlem49} valide pour $\sa{val}(\gK,\gk)$ est aussi valide dans $\sa{Val}^+(\gK,\gk)$.
\item En conséquence le morphisme $\theta$ donné en \pref{feqtheta} est injectif.
\end{enumerate}
  
\end{flemma}
%
\begin{proof} \textsl{1}.
Voyons par exemple la règle \Tsbf{Av2}. On suppose $a\di b$, on introduit donc une variable fraiche~$z$ telle que $az=b$ et $\Vr(z)$. Donc $acz=bc$, donc par \tsbf{Div}, $a\di b$.
Les autres axiomes ont \egmt des \demos directes, sauf pour les axiomes disjonctifs~\tsbf{AV1} et~\tsbf{AV2}. Nous utilisons alors le fait que dans une \demo d'une règle \pref{feqlem49}, les \elts sont toujours dans $\gK$.

\DeuxRegles
{
\lab{AV1} $\vd a \di  b \vou b\di a$
}
{
\lab{AV2} $\,\,ax \di  bx  \Vd a \di  b \vou 0 \di x$
}

\noindent Voyons \tsbf{AV1}. Dans une \demo, $a$ et $b$ sont des \elts de $\gK$. Si $a=0$, alors $b\di a$ dans $\sa{Val}^+(\gK,\gk)$ parce que $a=0.b$ et $\Vr(0)$. De même si $b=0$, alors $a\di b$ dans $\sa{Val}^+(\gK,\gk)$.
Si $a$ et~$b$ sont non nuls alors, on a dans $\gK$ un $x=ba^{-1}$ et un $y=ab^{-1}$, et $xy=1$. Donc dans  $\sa{Val}^+(\gK,\gk)$, on a $\vd \Vr(x)\vou \Vr(y)$. Si $\Vr(x)$, alors $1\di x$ donc $a\di ax=b$ en vertu de~\tsbf{Av2}. 
De même si $\Vr(y)$, alors $b\di a$.
\\
Voyons \tsbf{AV2}. Ici $a,b,x$ sont des \elts de $\gK$.  
Si $x=0$ on a $0\di x$. Si $x\neq 0$, il a un inverse~$x^{-1}$ dans $\gK$, et alors $a=axx^{-1}\di bxx^{-1}=b$ (on utilise \tsbf{Av2}). 
\end{proof}

\begin{fremark} \label{fremlemvalVal2} 
La démonstration du lemme \ref{flemvalVal2} utilise de manière décisive le fait que l'anneau~$\gK$ est un \cdi. Il ne semble pas qu'on puisse obtenir un \iso
$\val(\gK,\gk)\simeq \Val(\gK,\gk)$ lorsque $\gK$ n'est pas supposé \zedr. Il faudrait ajouter plusieurs axiomes à \sa{Val}  pour que $\sa{Val}^+$  prouve les axiomes de  $\sa{val}$, mais ces axiomes paraitraient beaucoup trop ad hoc.  \eoe
\end{fremark}

\Subsubsection{Isomorphisme des treillis $\val(\gK,\gk)$ et $\Val(\gK,\gk)$}

Pour démontrer que le morphisme $\theta$ donné en \pref{feqtheta} est un \iso, il nous suffit maintenant de démontrer le lemme suivant.

\begin{flemma} \label{flemvalValfin} Soient $\gK$ un \cdi et $\gk$ un sous-anneau. Dans $\sa{Val}^+(\gK,\gk)$, une règle
\[
a_1\di b_1\vet \dots\vet a_n\di b_n  \Vd   c_1\di d_1\vou \dots\vou  c_m\di d_m\qquad (a_i, b_i, c_j, d_j\in\gK)
\]
est toujours \eqve  à une règle
\[\Vr(x_1)\vet \dots\vet  \Vr(x_k)  \vd   \Vr(y_1)\vou \dots\vou   \Vr(y_\ell) \qquad (x_i,y_j\in\gK\etl)
\]

\end{flemma}
%
\begin{proof}
Pour les $b_i$ nuls, on supprime  $a_i\di b_i$ dans l'hypothèse, car elle est vraie.
Pour les $(c_j,d_j)$ avec $c_j=0$ et $d_j\neq 0$, on supprime  $c_j\di d_j$ dans la conclusion, car elle collapse. 
S'il~y~a un $a_i$ nul avec $b_i\neq 0$, ou s'il y a un $d_j$ nul, la règle est valide (elle équivaut à $\vd\Vr(1)$).
Il reste à voir le cas où tous les \elts sont non nuls. Dans ce cas, on peut remplacer  $a_i\di b_i$ et $c_j\di d_j$ respectivement par $\Vr(b_ia_i^{-1})$ et $\Vr(d_jc_j^{-1})$.
\end{proof}

On~a donc démontré le résultat souhaité.

\begin{ftheorem} \label{fthValval}
Soit $\gk$ un sous-anneau d'un \cdi $\gK$. Le morphisme naturel $\Val(\gK,\gk)\to \val(\gK,\gk)$
est un \iso.
\end{ftheorem}

Le corolaire suivant est obtenu à partir du théorème précédent en appliquant la machinerie locale-globale \elr\ \num1 expliquée dans la section 
\ref{fsec-A-qi}. 
\begin{fcorollary} \label{fcorthValval}
Soit $\gk$ un anneau \qi et soit  $\gK=\Frac(\gk)$. Le morphisme naturel $\Val(\gK,\gk)\to \val(\gK,\gk)$
est un \iso.
\end{fcorollary}

Notons que le résultat précédent ne s'étend pas au cas d'un anneau $\gk$ arbitraire.


\section{Dimension valuative d'un anneau, d'une algèbre}\label{fsubsecdival}

Lorsque l'on veut transférer des résultats obtenus pour les anneaux intègres à un anneau arbitraire~$\gA$, une solution pourrait être de définir un objet idéal \gui{l'anneau intègre générique engendré par~$\gA$}. Mais cet anneau n'existe pas en tant qu'objet usuel. Une variante dynamique de cette solution serait de définir une \tdy \sa{Ai} des anneaux intègres puis de considérer la \sad $\sa{Ai}(\gA)$ comme substitut raisonnable de l'objet idéal convoité.

Dans cette section, concernant le sujet de la dimension valuative, nous proposons plutôt d'utiliser l'anneau $\Amin$ \gui{clôture \qi minimale de $\gA$} qui a été défini par T.~Coquand dans une note non publiée concernant un substitut à l'\algo du pgcd dans $\AX$ lorsque~$\gA$ n'est pas un \cdi. L'utilisation de $\Amin$ a déjà fait la preuve de son efficacité pour la dimension valuative dans \cite{fACMC}.

Pour cela nous donnons dans la section \ref{fsec-A-qi} 
une version \cov détaillée de la théorie classique des anneaux \qis (pp-rings dans la littérature anglaise), puis nous construisons 
l'anneau $\Amin$ dans la section \ref{fsec-Amin} en indiquant quelques
\prts de cet objet. Nous pensons que l'anneau \zed $\Frac(\Amin)$ est
le meilleur substitut possible au \cdf d'un anneau intègre.
En particulier le morphisme naturel $\gA\to\Amin$ est un \iso lorsque~$\gA$ est un anneau intègre.

Dans la section \ref{fsec3vdim} nous utilisons l'anneau $\Amin$ dans la comparaison de trois versions \covs possibles pour les dimensions valuatives d'un anneau $\gA$ et d'une algèbre \hbox{$\gk\to\gA$}.

\subsection{Anneaux \qis}\label{fsec-A-qi}

\noindent 
{\bf Définitions et notations.} ~
\begin{itemize} 
\item
Si le contexte est clair on notera\, $a\epr\!:=\Ann_\gA(a)$.
 Nous utilisons aussi la notation~$\fa\epr$  pour l'annulateur d'un \id~$\fa$.
\item Un \elt $a$ tel que $a\epr=0$ est dit \textsl{\ndz} ou \textsl{non diviseur de zéro}.
\item  Un morphisme d'anneaux $\gA\to\gB$ est dit \textsl{\ndz} s'il transforme tout \elt \ndz en un \elt \ndz.
\item  On note  $\Ared=\gA/\!\sqrt[\gA]{\gen{0}}$ l'anneau réduit engendré par $\gA$.
\item  On note $\cP_n$ l'ensemble des parties finies de $\so{1,\dots,n}$. 
\end{itemize}

\begin{fpropdef} \label{flem-idm-associe-0} Soit $\gB$ un anneau commutatif et soit $a\in\gB$.
\begin{enumerate}
\item 
Si  on a un \elt $c$ tel que $ca=a$ et $c\epr=a\epr$, alors $c$  est un \idm. Cet \elt  est \ncrt unique et il est dit \emph{\idm attaché à~$a$ dans $\gB$}, on le notera $e_a$ ou~$a\eci$.
\item Un \idm $e\in \gB$ est \idm attaché à $a$ \ssi  $ea=a$
 et  $a+(1-e)$ est \ndz.
\item Si $\varphi:\gB\to\gC$ est \ndz et si $a\in\gB$ possède un \idm attaché $e_a$ dans $\gB$, alors $\varphi(e_a)$ est \idm attaché à $\varphi(a)$ dans $\gC$.  
\item Si $a$ est un \idm, il est \idm attaché à lui-même.
\item L'\elt $a$ est \ndz \ssi $1$  est \idm attaché à $a$.
\item 
Si $a\eci$ et $b\eci$ sont \idms attachés $a$ et $b$, alors $a\eci b\eci$ est \idm attaché à $ab$.
\end{enumerate}
Un anneau $\gA$ est dit \emph{\qi\footnote{En anglais, un \gui{pp-ring}: les \idps sont projectifs.}} si tout \elt $a$ possède un \idm attaché.
Autrement dit, l'annulateur de tout \elt est engendré par un \idm.
\end{fpropdef}
%
\facile

Comme $e_a$  est \idm, le morphisme naturel $\gA\to \gA[1/e_a]\times\aqo{\gA}{e_a} $ est un \iso.
Dans l'anneau $\gA[1/e_a]$, l'\elt $a$ est \ndz, et dans $\aqo{\gA}{e_a}$, $a$ est nul\footnote{En \clama, on a une partition du spectre de Zariski en deux ouverts de base, sur le premier $a=0$, sur le second $a$ est régulier.}.
On~a $e_{ab}=e_a e_b$, $e_aa=a$ et $e_0=0$.
Les anneaux \qis ont une \dfn purement équationnelle. Supposons en effet qu'un anneau commutatif soit muni d'une loi  unaire \hbox{$a\mapsto
\ci{a}$} qui vérifie les trois axiomes suivants:
\begin{equation}\label{feqaqis}
a\eci a=a,\quad
(ab)\eci=a\eci b\eci,\quad
0\eci=0
\end{equation}
Alors,   pour tout $a\in\gA$,  $a\epr=\gen{1-a\eci}$ et $a\eci$
est \idm attaché à $a$, de sorte que l'anneau  est \qi.

\begin{flemma}[Lemme de scindage \qi] \label{fthScindageQi}
Soient $n$ \elts $x_1$, \dots, $x_n$ dans un anneau \qi~$\gA$.
Il existe un \sfio $(e_j)$ de cardinal~$2^n$ tel que dans chacune des
composantes~\hbox{$\gA[1/e_j]$}, chaque~$x_i$ est nul ou \ndz. 
\end{flemma}

Le fait de pouvoir scinder systématiquement en deux composantes un anneau \qi
conduit à la méthode \gnle suivante.
La différence essentielle avec le lemme de scindage précédent
est que l'on  ne connait pas à priori la famille finie d'\elts qui
va provoquer le scindage.

\rdb
\mni {\bf Machinerie locale-globale \elr\ \num1.}\label{fMethodeQI}
{\sl La plupart des \algos qui fonctionnent avec les anneaux intègres non
triviaux peuvent être modifiés de manière à fonctionner avec les anneaux
\qis, en scindant l'anneau
en deux composantes chaque fois que l'\algo écrit pour les anneaux intègres
utilise le test
\gui{cet \elt est-il nul ou \ndz?}. Dans la première composante l'\elt en question
est nul, dans la seconde il est \ndz.}

\medskip Voici maintenant l'analogue du lemme XI-4.21 dans \cite{fACMC}, en remplaçant \gui{quasi inverse de $a$} par \gui{\idm attaché à $a$}.

\begin{flemma}[Anneau engendré par un \idm attaché]\label{flem-idm-associe-1}
Soit $\gA$ un anneau réduit.
\begin{enumerate}
\item Soit $a\in\gA\subseteq\gC$. On suppose que $\gC$ est réduit et que $a$ admet un \idm attaché~$a\eci$ dans $\gC$. On note $\gB=\gA[\ci{a}]\subseteq \gC$.\\
En tant que \Amo, $\gB=a\eci\gB\oplus(1-a\eci)\gB$, et en tant qu'anneau 
\[\gB\simeq \aqo{\gB}{1-a\eci}\times \aqo{\gB}{a\eci}=
\gA_1\times \gA_2
\]
avec $\gA_1 \simeq a\eci\gB$ 
 et  $\gA_2 \simeq (1-a\eci)\gB$ comme \Amos.
\begin{enumerate}
\item L'\homo naturel~\hbox{$\mu_1:\gA\to\gA_1$} 
(obtenu via $\gA\to\gB\to\gA_1$) est surjectif. Son noyau est~\hbox{$a\epr:=\Ann_\gA(a)$}.
\item
L'\homo naturel $\mu_2:\gA\to\gA_2$ (obtenu via $\gA\to\gB\to\gA_2$) est surjectif.
Son  noyau est l'intersection \hbox{$\fa=\gA\cap a\eci\gB$}
et  vérifie la double inclusion
$$ 
\big(a\epr\big)\epr=\Ann_\gA\big(\Ann_\gA(a)\big)\supseteq \fa \supseteq \DA(a)\eqno(*)  
$$
\end{enumerate}
En bref $\gA[a\eci]=\gB\simeq\gA\sur a\epr\times \gA\sur{\fa}$.
\item Inversement pour un $a\in A$, si un \id $\fa$ de $\gA$ vérifie les inclusions $(*)$, l'\elt $(1,0)$ est un \idm attaché à
 (l'image de) $a$ dans l'anneau 
 \[\gB:=\gA\sur a\epr\times \gA\sur{\fa}=:\gA_1\times \gA_2\] et l'\homo canonique
de $\gA$ dans $\gB$ est injectif.
\end{enumerate}
\end{flemma}
\begin{proof} Notons
$\pi_i:\gB\to\gA_i$ les \homos canoniques. Comme $\pi_i(a\eci)= 0 \hbox{ ou }1$, tout \elt $\pi_i(x)$ est égal à un $\mu_i(y)$ pour un $y\in\gA$, donc chaque $\mu_i$ est surjectif.

\smallskip \noindent \textsl{1a. } Le noyau de $\mu_1$  est $\gA\cap\Ker \pi_1= \gA\cap\Ann_\gB(a)=\Ann_\gA(a)$.

\smallskip \noindent \textsl{1b.}  Le noyau de $\mu_2$  est $\fa:=\gA\cap\Ker \pi_2= \gA\cap a\eci\gB$. On~a clairement $a\in\fa$ donc $\DA(a)\subseteq \fa$. 
Enfin, montrons que
$\fa\,\Ann_\gA(a) = 0$, d'où
     $\fa\subseteq\Ann_\gA\big(\Ann_\gA(a)\big)$. Soient en effet $x\in\fa$ et $y\in \Ann_\gA(a)$. On~a $x=a\eci z$ pour un $z\in\gB$ et $y\in \Ann_\gB(a)=\Ann_\gB(a\eci)$, donc $xy=za\eci y=0$.

\smallskip \noindent \textsl{2.} L'image de $a$ dans $\gB$ est $(a1_{\gA_1},0_{\gA_2})$, donc $(1_{\gA_1},0_{\gA_2})$
est bien un \idm attaché à~$a$ dans $\gB$. Soit maintenant $x\in\gA$ dont l'image dans $\gB$
est $0$.
D'une part $x=_{\gA_1}0$, donc $ax=_\gA0$. D'autre part  $x\,\Ann_\gA(a)=0$ donc $x^2=_\gA0$, et $x=_\gA0$.
\end{proof}

On voit
que la notation $\gA[a\eci]$ présente une 
ambigüité  lorsque $\DA(a)\neq\Ann_\gA\big(\Ann_\gA(a)\big)$.
Cela conduit en particulier à deux notions naturelles de clôture \qi  d'un anneau réduit, selon que l'on décide de privilégier systématiquement le coté $\DA(a)$ ou le coté $\Ann_\gA\big(\Ann_\gA(a)$.

Voici maintenant un lemme qui peut permettre de simplifier la construction d'anneaux \qis. C'est l'analogue du lemme XI-4.22 dans \cite{fACMC}, en remplaçant \gui{\zedr} par \gui{\qi} et \gui{quasi inverse} par \gui{\idm attaché}.

\begin{flemma}\label{flem2qi}~
\begin{enumerate}
\item Soient $\gA\subseteq\gC$ avec $\gC$ \qi. On note $a\eci$ l'\idm de $\gC$ attaché à $a\in\gA$. Le plus petit sous-anneau \qi de~$\gC$ contenant~$\gA$ est égal à $\gA[(a\eci)_{a\in\gA}]$. 
\item Plus \gnlt, si $\gA\subseteq\gB$
avec $\gB$ réduit et si tout \elt $a$ de $\gA$ possède un \idm attaché $a\eci$ dans $\gB$,
alors le sous-anneau $\gA[(a\eci)_{a\in\gA}]$ de~$\gB$ est \qi.
En outre, tout \elt de $\gA[(a\eci)_{a\in\gA}]$ peut être écrit \emph{sous forme standard}, \cad sous forme 
$ \som_j   a_j e_j $    avec les $a_j\in\gA$ et
une liste $(e_j)_j$ d'\idms 
deux à deux \orts dans l'\agB engendrée par les $a\eci$.\\
Enfin, $\som_j{  a_j e_j}$ possède un \idm associé dans $\gB$ écrit sous forme standard:  $\som_j  1  f_j=\som_j  a_j\eci  e_j$.
\end{enumerate}
\end{flemma}
\begin{proof}
Il suffit de démontrer le point \textsl{2}.\\
Parmi les \elts de $\gB$, ceux qui s'écrivent comme sommes de produits~$ab\eci$ avec $a$, $b\in\gA$ forment clairement un sous-anneau de $\gB$\footnote{Utiliser le point \textsl{4} du lemme \ref{flem-idm-associe-0}.},
qui est donc égal à $\gA[(a\eci)_{a\in\gA}]$.
Cet anneau contient l'\agB engendrée par les \idms $a\eci$. Donc tout \elt écrit sous forme standard est dans $\gA[(a\eci)_{a\in\gA}]$. \\
Soient $x=\som_j{  a_j e_j}$ et $y= \som_k{  b_k f_k}$ deux \elts écrits sous forme standard. \\
Alors $xy=\som_{j,k}(a_jb_k)e_jf_k $, qui est une écriture standard. 
Voyons la somme $x+y$. On peut supposer $\som_je_j=1$ en ajoutant   l'\idm $e=(1-\sum_je_j)$ à la liste. Même chose pour $\som_kf_k$. 
 Alors $x+y=\som_{j,k}(a_j+b_k)e_jf_k$. \\
Enfin, dans un produit fini d'anneaux $\prod_\ell\gA_\ell$, un \elt $x=(x_\ell)_\ell$ possède un \idm attaché \ssi chaque $x_\ell$  possède un \idm attaché ${x_\ell}\eci$ dans $\gA_\ell$, et dans ce cas $x\eci=({x_\ell}\eci)_\ell$. Ceci donne le dernier point de l'énoncé et achève de démontrer que $\gA[(a\eci)_{a\in\gA}]$ est \qi.
\end{proof}

\subsection{L'anneau $\Amin$}\label{fsec-Amin}

Référence: \cite[section~\hbox{XIII-7}]{fACMC}.

\begin{flemma}\label{flem0Amin}
Soit   $\gA$ un  anneau réduit
et $a\in\gA$.
On définit
\[\gA_{\so{a}}\eqdefi\gA\sur{a\epr}\times \gA\sur{({a\epr})\epr}
\]
et l'on note $\psi_a:\gA\to\gA_{\so{a}}$  l'\homo canonique. 
\begin{enumerate}
\item $\psi_a(a)=(\pi_1(a),0)$  et $(1,0)$
est \idm attaché à $\psi_a(a)$ dans $\gA_{\so{a}}$.
\item $\psi_a$ est injectif (on peut identifier $\gA$ à un sous-anneau de $\gA_{\so{a}}$).
\item L'anneau $\gA_{\so{a}}$ est réduit.
\item Si $a$ possède un \idm attaché, $\psi_a$ est un \iso.
C'est le cas par exemple si $\gA$ est \qi.
\item Soit $\fb$  un \id  dans  $\gA_{\so{a}}$, alors l'\id $\psi_a^{-1}(\fb\epr)=\fb\epr\cap\gA$ est un \id annulateur dans $\gA$.
\item L'injection  $\psi_a:\gA\to\gA_{\so{a}}$ est un morphisme \ndz.
\end{enumerate}
\end{flemma}
Pour le point \textsl{6}, voir l'exercice XIII-19 dans \cite{fACMC}.

\begin{flemma}\label{flem1Amin}
Soit   $\gA$ réduit
et $a,b\in\gA$. Alors avec les notations du lemme \ref{flem0Amin}
 les deux anneaux $(\gA_{\so{a}})_{\so{b}}$ et $(\gA_{\so{b}})_{\so{a}}$ sont canoniquement isomorphes.
\end{flemma}

\rem Le cas où $ab=0$ est typique: quand on le rencontre, on voudrait
bien scinder l'anneau en composantes où les choses sont \gui{claires}.
La construction précédente donne alors les trois composantes
\[
\gA\sur{(ab\epr)\epr}, \; \gA\sur{(a\epr b)\epr}\, \hbox{ et } \,\gA\sur{(a\epr b\epr)\epr}.
\]
Dans la première $a$ est \ndz et $b=0$, dans la seconde
$b$ est \ndz et $a=0$, et dans la troisième $a=b=0$.
\eoe

\begin{fthdef}[Clôture \qi minimale]\label{fthAmin} ~ 
\\
Soit $\gA$ un anneau réduit.
On définit l'anneau $\Amin$ comme colimite filtrée\footnote{En application du lemme \ref{flem1Amin}.} 
 en itérant la construction de base qui consiste à
 remplacer~$\gE$ (l'anneau \gui{en cours}, qui contient $\gA$) par
\[\gE_{\so{a}}\eqdefi\gE\sur{a\epr}\times \gE\sur{({a\epr})\epr}=\gE\sur{\Ann_\gE (a) }\times \gE\sur{\Ann_\gE(\Ann_\gE (a) )},\]
 lorsque $a$ parcourt $\gA$.
\begin{enumerate}
\item  Cet anneau $\Amin$ est \qi, il contient $\gA$ et il est entier sur $\gA$.
\item Pour tout $x\in\Amin$,
$x\epr\cap\gA$ est un \id annulateur dans $\gA$.
\item L'injection naturelle $\gA\to\Amin$ est un morphisme \ndz.
\item Si $\gA$ est \qi, cette injection est un \iso.
\end{enumerate}
Cet anneau $\Amin$ est appelé la \emph{clôture \qi minimale de $\gA$}.
\\
Dans le cas \gnl (on ignore si $\gA$ est réduit)
  on définit $\gA\qim\!\eqdefi (\Ared)\qim$.
\end{fthdef}

Voici une description de chaque anneau
obtenu à un étage fini de la construction de~$\Amin$.

\begin{flemma}\label{flem4MorRc}
Soit $\gA$ un anneau réduit et $(\ua) = (\an)$ une suite de~$n$ \elts 
de~$\gA$.  Pour $I\in\cP_n$, on note $\fa_I$ l'\id
\[\fa_I = \big(\Prod_{i\in I} \gen{a_i}\epr \Prod_{j\notin I} a_j\big)\epr
= \big(\gen{a_i, i \in I}\epr \Prod_{j\notin I} a_j\big)\epr
.\]
Alors $\Amin$ contient l'anneau suivant, produit de $2^n$
anneaux quotients de~$\gA$ (certains éventuellement nuls):
\[\gA_{\so\ua} = \Prod_{I\in\cP_n} \gA\sur{\fa_I}.\]
 \end{flemma}

\subsection{Trois versions \covs des dimensions valuatives}\label{fsec3vdim}

La dimension valuative d'un anneau commutatif peut être abordée de différentes manières.

En \clama la dimension valuative d'un anneau intègre $\gk$, notée $\Vdim(\gk)$, est égale à la hauteur maximum du groupe de valuation $\Gamma(\gV)$ d'un anneau de valuation $\gV$ tel que $\gk\subseteq \gV \subseteq \Frac(\gk)$. La hauteur de $\Gamma(\gV)$ est aussi la dimension de Krull de l'anneau $\gV$. Enfin $\Vdim(\gk)$ est encore la longueur maximum d'une chaine de sous-\advs de $\Frac\gk$ contenant $\gk$.

Comme la dimension valuative d'un quotient intègre de $\gk$ est inférieure où égale
à la dimension valuative de $\gk$, on peut étendre légitimement la \dfn à un anneau arbitraire comme suit:    
la dimension valuative d'un anneau $\gA$ est la borne supérieure des dimensions valuatives de ses quotients intègres  \cite{fCah90}.

\smallskip 
Dans \cite[section XIII-8]{fACMC}, pour le cas d'un anneau intègre $\gk$, les auteurs utilisent une \carn de la dimension valuative classique
basée uniquement sur la dimension de Krull des anneaux: la dimension valuative d'un anneau intègre est la borne supérieure des dimensions de Krull des sur-anneaux de $\gk$ (les anneaux squeezés entre $\gk$ et son corps de fractions). Nous la noterons $\vdim(\gk)$.
On étend sans difficulté cette \dfn au cas des anneaux \qis. Enfin, dans le cas \gnl les auteurs ont posé
par \dfn $\vdim(\gA)=\vdim(\Amin)$. Ils ont démontré pour cette \dfn l'\eqvc suivante dans le cas d'un anneau arbitraire,  pour $n\geq 0$:
\fbox{$\vdim(\gA)\leq n\Leftrightarrow \Kdim(\AXn)\leq 2n$}. Cette \carn est connue en \clama, la \dfn donnée dans \cite{fACMC} est donc bien une version \cov acceptable.

\smallskip Par ailleurs, pour un anneau intègre, il est clair en \clama que la dimension valuative de $\gk$ est égale à la dimension du \trdi 
$\val(\Frac\gk,\gk)$. Donc aussi à celle de $\Val(\Frac\gk,\gk)\simeq \val(\Frac\gk,\gk)$ (voir le \thref{fthValval}).
Comme $\val(\Frac\gk,\gk)\simeq \val(\gk,\gk)$ (lemme \ref{flem-valkk-valkK}), il est logique de proposer comme \dfn \cov alternative de la dimension valuative d'un anneau arbitraire $\gA$, la dimension du treillis $\val(\gA,\gA)$. Nous la noterons $\Vdim(\gA)$\footnote{Sans pour autant être assuré dès maintenant que cela coïncide en \clama avec la \dfn classique.}.

\smallskip Enfin, toujours en \clama, \cite{fKY2020} ont donné une nouvelle \carn \cov de la dimension valuative d'un anneau commutatif arbitraire. 
 
\smallskip Il reste à démontrer \cot que les trois \dfns \covs disponibles sont bien équivalentes. Celle développée dans \cite[section XIII-8]{fACMC}, la dimension du \trdi $\val(\gA,\gA)$ et la \carn donnée par Kemper et Yengui.

Ce sera l'objet d'un article par Yengui, Neuwirth et Lombardi, presque terminé.
L'article prouve \cot que la première et la troisième \dfn sont \eqves, et que dans le cas intègre elles sont \eqves à la seconde (qui remonte à \cite{fCoq2009}).
Il reste en fait à démontrer le lemme suivant.

\begin{flemma} \label{flemdimvalAAmin}
Pour un anneau commutatif arbitraire, les \trdis $\val(\gA,\gA)$ et $\val(\Amin,\Amin)$ ont même dimension de Krull.  
\end{flemma}
%
\begin{proof}
L'anneau $\Amin$ est une colimite filtrée d'anneaux obtenus, en partant 
de~$\gA$, au moyen de la construction décrite dans le \thref{fthAmin}. 
Si $\gE$ est l'anneau en cours et si $x\in\gE$, le nouvel anneau est $\gE/x\epr\!\times\gE/(x\epr)\epr$. Il s'agit donc d'une construction du type $\gE\mapsto \gE/\fa_1\times\gE/\fa_2 $ avec $\fa_1\fa_2=0$.
Si on démontre qu'une étape de cette construction ne modifie pas la dimension du treillis valuatif, on a terminé. En fait comme les dimensions ne changent pas en passant de $\gA$ à~$\Ared$, il suffit de démontrer le résultat pour un anneau réduit.
\end{proof}
%

\begin{flemma} \label{flemdimvalAA1A2}
Soit un anneau commutatif réduit $\gA$, $\fa_1$, $\fa_2$ deux \ids tels que $\fa_1\fa_2=0$, $\gA_1=\gA/\fa_1$ et $\gA_2=\gA/\fa_2$.
Notons $\vald(\gA):=\val(\gA,\gA)$.
Alors 
$$\Kdim(\vald(\gA))=
\sup(\Kdim(\vald(\gA_1)),\Kdim(\vald(\gA_2)))=\Kdim(\vald(\gA_1\!\times \gA_2)).
$$  
\end{flemma}
%
\begin{proof}
\textsl{Première égalité.} De $\gA_i=\gA/\fa_i$ on déduit $\vald(\gA_i)=\vald(\gA)/(\fb_i=1)$ pour un filtre~$\fb_i$ convenable de $\vald(\gA)$. \Prmt, en reprenant l'\eqvc  \pref{feqvalkK} de la \dfn~\ref{fdefivalkK0}, nous obtenons par exemple pour  $\gA_1$
les \eqvcs suivantes  en notant que la \sad $\sa{val}(\gA_1,\gA_1)$
est obtenue à partir de $\sa{val}(\gA,\gA)$ en ajoutant les relations $0\di x$ pour les $x\in\fa_1$:
\begin{equation*} \label {feqvalAi}
\begin{aligned} 
 (a_1,b_1),\dots,(a_n,b_n)  &\,\vdash_{\val(\gA_1,\gA_1)}\,   (c_1,d_1),\dots,(c_m,d_m)   
  \quad    \equidef     
\\[.2em] 
  a_1\di b_1\vet \dots\vet a_n\di b_n &\Vdi{\sA{val}(\gA_1,\gA_1)}   c_1\di d_1\vou \dots\vou  c_m\di d_m  \quad    \equidef  \\[.2em] 
0\di x_1\vet \dots\vet 0\di x_r\vet  a_1\di b_1\vet \dots\vet a_n\di b_n &\Vdi{\sA{val}(\gA,\gA)}   c_1\di d_1\vou \dots\vou  c_m\di d_m \equidef     \\[.2em]
 (a_1,b_1),\dots,(a_n,b_n)  &\,\vdash_{\val(\gA,\gA)/(\fb_1=1)}\,   (c_1,d_1),\dots,(c_m,d_m)
 \end{aligned}
\end{equation*} 
pour des $x_j\in\fa_1$ et où $\fb_1$ est le filtre de $\val(\gA,\gA)$
engendré par les $(0,x)$ pour $x\in\fa_1$.\\
On démontre ensuite que $\fb_1\cap\fb_2=\so 1$. En effet le filtre $\fb_1\cap\fb_2$ est engendré par les \elts de la forme $(0,x)\vu(0,y)$ avec $x\in\fa_1$
et $y\in\fa_2$, et comme $0\di xy$ on obtient par \tsbf{SDZ}, la règle valide $\Vdi{\sA{val}(\gA,\gA)}0\di x \vou 0\di y$.
\\
On conclut avec le  \thref{flemrcfdimtrdi}.\\
\emph{Deuxième égalité.} On démontre  $\vald(\gA_1\!\times \gA_2)\simeq \vald(\gA_1)\times \vald(\gA_2)$. Ensuite: la dimension d'un produit de deux treillis est le sup des deux dimensions.
\end{proof}

\noindent \textsl{Note.} Comme $\Amin$ est \qi, $\val(\Amin,\Amin)\simeq \val(\Frac(\Amin),\Amin)$, et $\Vdim(\gA)$ défini ci-dessus est égal à $\Vdim(\Amin)$ au sens de l'article \cite{fCoq2009}.

\begin{fdefinition} \label{fdefivaluabs}
La dimension du \trdi 
$\valA\simeq\val(\gA,\gZ)$ 
(où $\gZ$ est le plus petit sous-anneau de $\gA$) 
est appelée la \textsl{dimension valuative absolue} de l'anneau $\gA$ et notée $\vdima(\gA)$.  
\end{fdefinition}

La dimension $\vdima(\gA):=\Kdim(\valA)$ est à priori supérieure ou égale à $\vdim(\gA):=\Kdim(\val(\gA,\gA))$ et elle peut être strictement supérieure, comme  pour l'anneau~$\QQ$: $\vdima(\QQ)=\Kdim(\val(\QQ,\ZZ))=\Kdim(\val(\ZZ,\ZZ))=1$, et $\vdim(\QQ)=\Kdim(\val(\QQ,\QQ))= 0$. \eoe

\begin{fremark} \label{fremdim AAmingeneral} 
Les mêmes \demos fonctionnent pour la dimension des \trdis $\Zar(\gA)$, $\val(\gA)$, $\Reel(\gA)$ et $\Heit(\gA)$\footnote{Voir \cite{fCLQ2006}.} et fournissent donc les résultats analogues suivants
(sous les mêmes hypothèses).
\vspace{-1em}
\[
\begin{array} {cccccr}
\Kdim(\gA)&=&\Kdim(\Amin)&=&\Kdim(\gA_1\!\times \gA_2)\\[.3em]
\vdima(\gA)&=&\vdima(\Amin)&=&\vdima(\gA_1\!\times \gA_2)\\[.3em]
\Rdim(\gA)&=&\Rdim(\Amin)&=&\Rdim(\gA_1\!\times \gA_2)\\[.3em]
\Jdim(\gA)&=&\Jdim(\Amin)&=&\Jdim(\gA_1\!\times \gA_2) 
\end{array} \eqno\hbox{\eoe}
\]   
\end{fremark}


\section{Comparaisons avec une théorie des \cvds}\label{fsecdival6}

\subsection{Introduction}

Nous démontrons dans cette section que la \tdij \sa{val} et la \tdy des \cvds \sa{Cvd} introduite dans \cite[Section 4]{fCLR01} sont presque identiques.
On utilise pour cela le \vst formel \ref{fCollval1} des \cvds.
On en déduit des \vsts formels pour  \sa{val}.

Cela permet aussi de comparer avec le \vst formel pour la théorie \sa{Val} démontré dans \cite{fCoq2009}.

Notons que nous avons déjà démontré un résultat de comparaison essentiel entre $\val$  et $\Val$ avec le \thref{fthValval}.

\subsection{La théorie $\sa{Cvd}$ des \cvds}\label{fsubsubseccorval}

Dans l'article \cite{fCLR01} la section 4 étudie le Positivstellensatz pour les \cvds algébriquement clos. Nous l'appellerons ici \textsl{Valuativstellensatz}. Dans l'étude présente, nous sommes avant tout intéressés par les \vsts formels, lesquels donnent des certificats \agqs de prouvabilité dans diverses \tdys pour certaines \rdys.

\smallskip Nous rappelons tout d'abord les axiomes qui sont donnés dans \cite{fCLR01} pour la théorie $\SA{Cvd}$ des \cvds. Le langage est celui des anneaux commutatifs auquel on ajoute quatre prédicats \hbox{$\cdot \neq 0$}, $\Vr(\cdot)$, $\Rn(\cdot)$, $\U(\cdot)$.
Le premier,  $\cdot \neq 0$, est vu comme prédicat d'inversibilité, le prédicat 
$\Vr(\cdot)$ interprète l'appartenance à l'anneau de valuation, le prédicat $\Rn(\cdot)$
interprète les \elts résiduellement nuls et $\U(\cdot)$ les unités (les \elts inversibles dans l'anneau de valuation).
La signature est donc la suivante
\[
\Sigma_{\sA{Cvd}}=(\,\cdot=0,\cdot\neq 0,\Vr(\cdot),\Rn(\cdot),\U(\cdot)\mathrel{;}\cdot+\cdot, \cdot\times\cdot,-\,\cdot,0,1\,)
\]
Les axiomes de la théorie \sa{Cvd} sont les \rdys suivantes. Ils impliquent que le corps est discret et que la relation de \dve est explicite.

\smallskip \noindent $\bullet$ Tout d'abord les règles directes. 

\DeuxRegles{
\lab{ac1} $\vd  0=0   $
\lab{ac3} $\,\,  x=0 \vet y=0 \Vd  x+y=0   $
\Lab{cv1} $\,\,x = 0 \vet  \Vr(y)   \vd   \Vr(x+y) $
\Lab{cv3} $\,\, \Vr(x) \vet  \Vr(y) \vd   \Vr(xy)$
\Lab{cv5} $\,\,x = 0 \vet  \Rn(y)  \vd   \Rn(x+y)$
\Lab{cv7} $\,\,\Rn(x) \vet   \Vr(y)  \vd   \Rn(xy)$
\Lab{cv9} $\,\,\Rn(x)   \vd  \Vr(x)$
\Lab{cv10} $\,\,x = 0 \vet  \U(y)   \vd   \U(x+y)$
\Lab{cv12} $\,\,\U(x) \vet  \U(y)  \vd   \U(xy)$
\Lab{cv14} $\,\,\U(x)   \vd  \Vr(x)$
\Lab{cv15} $\,\,x = 0 \vet  y \not= 0   \vd   x+y \not= 0$
\Lab{cv17} $\,\,\U(x)  \vd   x\not= 0$}
{
\lab{ac2} $\,\,  x=0 \Vd   xy=0   $
\lab{~} $\phantom{\,\,  x=0 \vet y=0 \Vd  x+y=0   }$
\Lab{cv2} $\vd  \Vr(-1)$
\Lab{cv4} $\,\,\Vr(x) \vet  \Vr(y)   \vd  \Vr(x+y)$
\Lab{cv6} $\vd   \Rn(0)$
\Lab{cv8} $\,\,\Rn(x) \vet  \Rn(y)   \vd  \Rn(x+y)$
\lab{~} $\phantom{\,\,x = 0 \vet  \U(y)   \vd   \U(x+y)}$
\Lab{cv11}  $ \vd   \U(1)$
\Lab{cv13} $\,\,\Rn(x) \vet  \U(y)   \vd   \U(x+y)$
\lab{~} $\phantom{\,\,\U(x)   \vd  \Vr(x)}$
\Lab{cv16} $\,\,x \not=0 \vet y\not= 0   \vd   x y \not= 0$
}

\noindent $\bullet$ L'axiome d'effondrement est celui des anneaux intègres.

\DeuxRegles{
\Lab{CL} $\,\,0 \not= 0  \vd   \Bot$
}
{\lab{~} $~$
}

\noindent $\bullet$ Enfin, on a les règles de simplification \tsbf{Cv1} et \tsbf{Cv2} suivies de 4 \rdys.
 
\DeuxRegles{
\Lab{Cv1} $\,\,x y=1   \vd   x \not= 0$
\Lab{CV1} $\,\,x \not= 0   \vd   \Exists y \,\,xy=1$
\Lab{CV3} $\,\,xy =1   \vd   \Vr(x)  \vou   \Vr(y)$
}
{
\Lab{Cv2} $\,\,\Vr(xy)\vet  \U(x)  \vd   \Vr(y)$
\Lab{CV2} $\vd   x = 0  \vou   x \not= 0$
\Lab{CV4} $\,\,\Vr(x)   \vd   \U(x)  \vou   \Rn(x)$
}

Si l'on veut imposer que la valuation soit non triviale, il faut 
introduire une constante~$c$ avec l'axiome 

\Regles{\Lab{cv0}$\,\,\Vr(c)\vd \Bot $}

Notons que le seul axiome existentiel, à savoir \tsbf{CV1}, correspond à une existence unique, donc la théorie est existentiellement rigide, ce qui nous éloigne peu du cadre des \tdijs. 

Pour la \tdy \SA{Cvdsc} des \cvds \textsl{séparablement clos} on ajoute à la théorie \sa{Cvd} les axiomes de clôture \spb.

\Regles{
\Lab{CV6$_n$} $\,\, \disc_Y (f)\neq 0\vd   \Exists y \; f(y)=0,\qquad\hbox{ où } f(Y)=Y^n + a_{n-1} Y^{n-1}+ \cdots+ a_1 Y + a_0.$
}


\smallskip Pour la \tdy \SA{Cvdac} des \cvds \textsl{\ac} on ajoute à la théorie \sa{Cvd} les axiomes de clôture \agq.

\Regles{
\Lab{CV5$_n$} $\vd   \Exists y \; y^n + a_{n-1} y^{n-1}+ \cdots+ a_1 y + a_0 = 0$
}

\begin{flemma} \label{flemModelesCv}
Soit $\gA$ un anneau commutatif et $\sa{Cvd}(\gA)$ la \sad de type \sa{Cvd}
construite sur (le diagramme positif de) $\gA$. Donner un modèle minimal de $\sa{Cvd}(\gA)$ revient à donner un couple $(\fp,\gV)$ où $\fp$ est un \idep 
détachable de $\gA$ et $\gV$ un anneau de valuation du corps $\gK=\Frac(\gA/\fp)$. En outre la relation de divisibilité doit être décidable.   
\end{flemma}
%
\begin{proof}
L'\idep $\fp$ est l'ensemble des $x\in\gA$ tels que $x=0$ est valide dans le modèle. L'anneau quotient $\gB=\gA/\fp$ est intègre, on note $\gK$ son corps de fractions, qui existe dans le modèle en vertu de l'axiome \tsbf{CV1}. L'anneau de valuation $\gV$
est alors l'ensemble des fractions $z$ dans $\gK$ telles que $\Vr(z)$ est valide dans le modèle. 
\\
La relation de divisibilité est décidable parce que $\gV$ est une partie détachable de $\gK$: si $x\neq 0$ on a un inverse $y$ en vertu de \tsbf{CV1}. On~a alors $\Vr(x)$ ou $\Vr(y)$. Dans le deuxième cas, on a 
$\U(y)$ ou $\Rn(y)$, et $\Vr(x)$ \ssi $\U(y)$.  
\end{proof}

Il semble que \tsbf{CV1} est \ncr dans le dernier argument. Il sert à forcer le modèle à ne contenir que des fractions d'\elts de $\gV$.

Ainsi, les \sads $\sa{Cvd}(\gA)$ et $\sa{val1}(\gA)$ ont les mêmes modèles minimaux en \clama. Mais pas en \coma car pour $\sa{val1}(\gA)$ on ne suppose pas que la relation de divisibilité soit décidable.


\medskip Nous définissons maintenant des théories étroitement reliées à la théorie \sa{Cvd} et qui partagent certains \thos avec elle. Nous utilisons la terminologie des \reds, règles de simplification et autres règles dynamiques que nous avons expliquée \paref{fregledirecte}.

\begin{fdefinitions}[Théories \gmqs étroitement reliées à \sa{Cvd}] \label{fdefiApvandco}  \label{fdefiCvd-}~
\begin{enumerate}
\item On note \sa{Apv}\footnote{Proto-valued rings dans \cite{fCLR01}.}  la théorie directe obtenue à partir de la théorie \sa{Cvd} en ne gardant que les axiomes directs, de~\tsbf{ac1} 
à~\tsbf{cv17}, et l'axiome d'effondrement.
\item La théorie \agq \sa{Aqv}\footnote{Quasi-valued rings dans \cite{fCLR01}.}  est intermédiaire entre la théorie directe \sa{Apv} et la \tdy \sa{Cvd}: on remplace dans  \sa{Cvd} les axiomes \tsbf{CV1} à \tsbf{CV4} par les axiomes de simplification suivants, qui sont des règles valides dans \sa{Cvd}:
 
\DeuxRegles
{
 \Lab{Cv3} $\,\,\U(xy)\vet \Vr(x)\vet \Vr(y) \vd  \U(y)$ 
 \Lab{Cv5} $\,\,\Rn(x^2)  \vd  \Rn(x)$ 
 \Lab{Cv7} $\,\,xy = 0\vet  x \not= 0   \vd  y= 0  $
}
{
 \Lab{Cv4} $\,\,\Rn(xy)\vet  \U(x)  \vd  \Rn(y)$ 
 \Lab{Cv6} $\,\,xy \not= 0  \vd  x\not=0$ 
 \Lab{Cv8} $\,\,x^2 = 0  \vd  x = 0$ 
}

\vspace{-.6em}
\Regles {\Lab{Cv9$_n$} $\,\,x^{n} =\sum_{k=0}^{n-1} a_k x^k \vet \Vr(a_{n-1})\vet \dots\vet \Vr(a_0)  \vd  \Vr(x)$}

Le dernier schéma d'axiomes dit que l'anneau intègre $\gV$ est intégralement clos.

\item La \tdij $\sa{Cvd}^-$ est la théorie $\sa{Aqv}$ à laquelle on ajoute les
trois axiomes disjonctifs \tsbf{CV2}, \tsbf{CV3} et \tsbf{CV4}.

\item 
La théorie $\sa{Cvd}^+$ est l'extension de  $\sa{Cvd}$  obtenue en ajoutant
\begin{itemize}
\item le prédicat $x \di y$ comme  abréviation de \gui{$\exists z \,(\Vr(z) \vii xz=y$)}, 
\item  le prédicat $x \ndi y$ comme abréviation de \gui{$y\neq 0\vii \exists z \,(\Rn(z)\vii yz=x)$}. 
\end{itemize}
\end{enumerate}
\end{fdefinitions}

On verra que les théories \sa{Apv} et \sa{Cvd} s'effondrent simultanément (\ref{flemColsimApv}), que les théories~\sa{Aqv} et~\sa{Cvd} prouvent les mêmes \ralgs (\ref{fthAqv}). Les théories~$\sa{Cvd}^-$ et~\sa{Cvd} prouvent les mêmes \rdijs, mais ce résultat ne sera pas démontré ici.

Quant à la théorie $\sa{Cvd}^+$, elle est par construction une extension \esid de  $\sa{Cvd}$, à fortiori conservative.

\begin{fdefinition} \label{fdefiCvdkA}
Si $\gk\subseteq \gA$ sont deux anneaux, ou plus \gnlt si $\varphi:\gk\to\gA$ est une \alg, on note $\sa{Cvd}(\gA,\gk)$ la \sad  dont la présentation est donnée par 
\begin{itemize}
\item le diagramme positif de~$\gA$ comme anneau commutatif;
\item les axiomes
$\vd \Vr(\varphi(x))$ pour les \elts $x$ de $\gk$.
\end{itemize}
La \sad $\sa{Cvd}(\gA)$ est identique à $\sa{Cvd}(\gA,\gZ)$ où $\gZ$
est le sous-anneau minimal de $\gA$.
 \end{fdefinition}

On définit de manière analogue les \sads  $\sa{Apv}(\gA,\gk)$, $\sa{Aqv}(\gA,\gk)$, $\sa{Cvd}^-(\gA,\gk)$ et $\sa{Cvd}^+\!(\gA,\gk)$.

\subsection{\vst formel pour \sa{Cvd} et conséquences}

\begin{fvalsatz}[Valuativstellensatz formel pour la théorie \sa{Apv}] \label{fCollval1} ~\\
Soit $GR=(G;\Rzero,\Rnz,\Rvr,\Rrn,\Ru)$ une \pn sur la signature $\Sigma_{\sA {Cvd}}$\footnote{Plus \prmt, $\Rzero$ est la partie de $\ZZ[G]$ formée par les \elts $p$ pour lesquels on suppose $p=0$ dans la \sad. Et la même chose pour $\Rnz,\Rvr,\Rrn,\Ru$.} pour une \sad. 
\\
On note
\begin{itemize}
\item $\Izero$ l'\id de $\ZZ[G]$ engendré par $\Rzero$,
\item  $\Mnz$ le \mo engendré par $\Rnz$,
\item  $\Vvr$ le sous-anneau de $\ZZ[G]$ engendré par $\Rvr \cup \Rrn \cup \Ru$,
\item $\Irn$ l'idéal de $\Vvr$ engendré par $\Rrn$ et
\item  $\Mu$ le \mo engendré par $\Ru$.
\end{itemize}
\Propeq
\begin{enumerate}
\item La \sad $(GR,\sa{Apv})$ s'effondre (i.e., la règle $0\neq 0$ est valide).
\item Il y a  dans $\ZZ[G]$ une \egt  
\[
s (u+t) + z = 0
\] 
avec $s\in \Mnz$,
$u\in \Mu$, $t\in \Irn$ et $z\in \Izero$.
\end{enumerate}
\end{fvalsatz}

Le \tho suivant est un résultat essentiel de \cite{fCLR01}.
Il donne la condition d'effondrement pour les \sads de type \sa{Cvdac}. 

\begin{ftheorem}[Effondrement simultané] \label{flemColsimApv} \emph{\cite[Theorem 4.3]{fCLR01}}\\
Une \sad de type \sa{Apv} s'effondre \ssi elle s'effondre en tant que \sad de type  \sa{Cvd}, ou de type   \sa{Cvdac}, ou de toute autre théorie intermédiaire entre \sa{Apv} et \sa{Cvdac}.
\end{ftheorem}

\begin{fremark} \label{fremlemColsimApv} 
Si $(\gK,\gV)$ est un \cvd et si $\gL$ est un corps extension de $\gK$
on déduit en \clama du \vst précédent que l'anneau de valuation~$\gV$
peut être étendu en un \adv $\gW$ de $\gL$ tel que $\gW\cap \gK=\gV$
\cite[Remark~4.6]{fCLR01}.
\eoe    
\end{fremark}

Le \tho 4.18 dans \cite{fCLR01} est un \vst à la Hilbert pour les \cvds \ac, obtenu à partir du \thref{flemColsimApv} en s'appuyant sur le fait que la théorie formelle correspondant à \sa{Cdvac} est complète
(lorsque la \cara du \cvd $\gK$ et celle du corps résiduel sont fixées). 

Le \tho suivant décrit les faits prouvables dans  les \sads de type \sa{Cdv}.

\begin{ftheorem} \label{fprop-factvalf} \emph{\cite[Proposition 4.14]{fCLR01}}
\\
Soit ${GR}=(G;\Rzero,\Rnz,\Rvr,\Rrn,\Ru)$ une présentation  
sur la signature~$\Sigma_{\sA{Cdv}}$. Soit $p$ un \elt de $\Zg$.
On définit $\Izero$, $\Mnz$, $\Vvr$, $\Irn$ et $\Mu$ comme dans le  \thref{fCollval1}. On se place dans la \sad $(GR,\sa{Cvd})$.
\begin{enumerate}
\item [a)] Une preuve dynamique du fait $p = 0$ 
produit une \egt dans $\Zg$ du type 
\[ p^n m(u+j) + i = 0
\]
avec $m \in \Mnz$, $u \in \Mu$, $j \in \Irn$ et $i \in \Izero$.
\item [b)] Une preuve dynamique du fait $p \not= 0$
produit une \egt dans $\Zg$ du type  
\[m (u+j) + i + bp = 0
\]
avec $m \in \Mnz$, $u \in \Mu$, $j \in \Irn$, $i \in \Izero$ and $b \in \Zg $.
\item [c)] Une preuve dynamique du fait $\Vr(p)$ produit une \egt dans $\Zg$ du type 
\[ m ( (u+j)p^{n+1} + a_n
p^n +\cdots + a_1 p + a_0 ) + i = 0
\]
avec $m \in \Mnz$, $u \in \Mu$,
$j \in \Irn$, les $a_k \in \Vvr$ et $i\in \Izero$.
\item [d)] Une preuve dynamique du fait $\Rn(p)$ produit une \egt dans $\Zg$ du type 
\[ m ( (u+j)p^{n+1} + j_n
p^n +\cdots + j_1 p + j_0 ) + i  = 0
\]
avec $m \in \Mnz$,
$u \in\Mu$, $j$
et les $j_k\in \Irn$ et $i \in \Izero$.
\item [e)]  Une preuve dynamique du fait $\U(p)$ 
produit une \egt dans $\Zg$ du type 
\[ m ( (u+j) p^{n+1} + a_n
p^n +\cdots + a_1 p + (u' + j') ) + i = 0 
\] 
avec $m \in \Mnz$,
$u,u'\in \Mu$, $j,j' \in \Irn$, les $a_k$ dans $\Vvr$ et $i \in \Izero$.
\end{enumerate}
\end{ftheorem}

\begin{fcorollary} \label{fcor-prop-factvalf} 
Soient $\gk\subseteq \gA$ deux anneaux, $x\in\gA$ et $\yn\in\gk$.
\begin{enumerate}
\item La règle $\vd x= 0$ est valide dans $\sa{Cvd}(\gA,\gk)$ \ssi $x$ est nilpotent.
\item La règle $\vd x\neq 0$ est valide dans $\sa{Cvd}(\gA,\gk)$ \ssi $x$
est inversible dans~$\gA$.
\item  La règle $\vd \Vr(x)$ est valide dans $\sa{Cvd}(\gA,\gk)$ \ssi $x$
est entier sur $\gk$.
\item  La règle $\Rn(y_1)\vet \dots\vet \Rn(y_n) \vd \Rn(x)$ est valide dans $\sa{Cvd}(\gA,\gk)$
\ssi~$x$ est faiblement entier\footnote{On entend par là que $x$ annule un \pol unitaire dont tous les \coes non dominants sont dans l'\id en question.} sur l'\id $\gen{\yn}$ de $\gk$. 
\end{enumerate}
 
\end{fcorollary}

\begin{fremark} \label{frem-cor-prop-factvalf} 
Du point \textsl{3} on déduit facilement en \clama le fait qu'un anneau
intègre~$\gk$ a pour clôture intégrale dans son corps de fractions $\gK$ l'intersection des \advs de $\gK$ qui contiennent $\gk$ \cite[Remark~4.15]{fCLR01}.
\\
Avec le point \textsl{4}, on déduit en \clama que si $(\gk,\fm)$ est un anneau local intègre, l'\id $\fm$ est l'intersection de $\gk$ avec l'\idema d'un \adv du corps de fractions $\gK$ \cite[Corollary~4.7]{fCLR01}. \eoe
    
\end{fremark}

Une conséquence du \thref{fprop-factvalf} est le résultat remarquable suivant \cite[Theorem 4.17]{fCLR01}.
\begin{ftheorem} \label{fthAqv}
Les théories \sa{Aqv} et \sa{Cvdac} prouvent les mêmes règles algébriques.
Il en va de même pour toute théorie intermédiaire.
\end{ftheorem}


\subsection{\vsts formels pour   $\sa{val}$ et $\sa{Cvd}^+$}

Dans cette section on démontre un résultat crucial: un \vst formel  pour la théorie~\sa{val},
comme conséquence des  \thos \ref{fCollval1} et \ref{flemColsimApv}.

On considère un anneau $\gA$, on cherche un \vst formel pour la \sad  $\sa{val}(\gA)$.

Rappelons que la théorie $\sa{Cvd}^+$ est définie au point \textsl{4} de la \dfn \ref{fdefiApvandco} et que la théorie \sa{val1} est l'extension conservative de \sa{val} obtenue en ajoutant le prédicat $\cdot \neq 0$ opposé au prédicat~\hbox{$=  0$}.

\begin{flemma} \label{flemCv+val}~
\begin{enumerate}
\item Le prédicat $x\ndi y$ est le prédicat opposé à $x\di y$ dans la théorie
$\sa{Cvd}^+$.
\item Les axiomes de \sa{val} sont valides dans $\sa{Cvd}^+$. 
\item La \ralg suivante est valide dans \sa{val1}.

\Regles{
\Lab{AV2.5} $\,\,x\neq 0\vet ax \di bx \Vd a \di b$
} 
\end{enumerate}
\end{flemma}
\facile

\Subsubsection{Un \vst formel pour  $\sa{val}(\gA)$}

\begin{fvalsatz}[Un \vst formel pour  $\sa{val}(\gA)$ 
 et $\sa{Cvd}^+\!(\gA)$\footnote{Notez qu'en prenant $\gA=\ZZ$ et pour $a_i,b_i,c_j$ et $d_j$ des \idtrs on trouve un \vst formel \gnl pour la théorie $\sa{Cvd}^+$ ou pour la théorie \sa{val}.
}]\label{fthVstformelval}~\\ 
On considère un anneau commutatif $\gA$ et des \elts $a_i$, $b_i$, $c_j$, $d_j\in\gA$.  \Propeq 
\begin{enumerate}
\item  ~ 

\vspace{-2.4em}
\begin{equation} \label {feqthVstformelval1}
 (a_1,b_1),\dots,(a_n,b_n)  \,\vdash_{\val(\gA)} \,   (c_1,d_1),\dots,(c_m,d_m) 
\end{equation}
\item  ~ 

\vspace{-2.4em} 
\begin{equation} \label {feqthVstformelval2}
a_1\di b_1\vet \dots\vet a_n\di b_n \Vdi{\sA{val1}(\gA)}   c_1\di  d_1\vou \dots\vou  c_m\di d_m
\end{equation}
\item  ~ 

\vspace{-2.4em} 
\begin{equation} \label {feqthVstformelval3}
a_1\di b_1\vet \dots\vet a_n\di b_n \Vdi{\sA{Cvd}^+\!(\gA)}   c_1\di  d_1\vou \dots\vou  c_m\di d_m
\end{equation}
\item 
 En introduisant des \idtrs $X_i$ ($i\in\lrbn$) et $Y_j$ ($j\in\lrbm$)
on a dans l'anneau  $\gA[\uX,\uY]$ une \egt que l'on peut résumer sous la forme suivante 
\begin{equation} \label {feqthVstformelval4}
d \,\big(1+\som_{j=1}^mY_jP_j(\uX,\uY)\big)\equiv 0 \mod  \gen{(X_ia_i-b_i)_{i\in\lrbn},(Y_jd_j-c_j)_{j\in\lrbm}} 
\end{equation}
 où $d$ est dans le \mo engendré par les $d_j$, et les $P_j(\uX,\uY)$ sont dans $\ZZ[\uX,\uY]$.
\item 
 On note  $y_j=\frac{c_j}{d_j}$ vu dans    $\gB=\gA[\frac{1}{d_1\cdots d_m}]$. Introduisons des \idtrs $X_i$ ($i\in\lrbn$). On~a dans l'anneau  $\gB[\uX]$ une \egt que l'on peut résumer sous la forme suivante 
\begin{equation} \label {feqthVstformelval5}
1+\som_{j=1}^my_jP_j(\uX,\uy)\equiv 0 \mod  \gen{(X_ia_i-b_i)_{i\in\lrbn}}
\end{equation}
 où  les $P_j(\uX,\uY)$ sont dans $\ZZ[\uX,\uY]$.
\end{enumerate}
\end{fvalsatz}

%
\begin{proof} Le point \textsl{5} est une simple réécriture du point \textsl{4}.

\smallskip \noindent \textsl{1} $\Leftrightarrow$ \textsl{2}. Le point \textsl{1} signifie par \dfn que
la règle suivante est valide:  
\begin{equation} \label {feqthVstformelval6}
a_1\di b_1\vet \dots\vet a_n\di b_n \Vdi{\sA{val}(\gA)}   c_1\di d_1\vou \dots\vou  c_m\di d_m
\end{equation}
Or \sa{val1} est une extension conservative de \sa{val}.

\smallskip \noindent \textsl{1} $\Rightarrow$ \textsl{3}.  
D'après le point \textsl{2} du lemme \ref{flemCv+val}, la règle \pref{feqthVstformelval3} est valide.

\smallskip \noindent \textsl{3} $\Rightarrow$ \textsl{4}.  
La règle \pref{feqthVstformelval3} est \eqve au collapsus suivant
\begin{equation} \label {feqthVstformelval3+}
a_1\di b_1\vet \dots\vet a_n\di b_n\vet c_1\ndi d_1\vet \dots\vet  c_m\ndi d_m \Vdi{\sA{Cvd}^+\!(\gA)}  \Bot  
\end{equation}
Dans la théorie  $\sa{Cvd}^+$,  $a_i\di b_i$
équivaut à l'existence d'un  $x_i$ tel que $\Vr(x_i)$ et $x_ia_i=b_i$.
Cet~$x_i$ peut être représenté par une variable fraiche $X_i$. De même  $ c_j\ndi  d_j$
équivaut à l'existence d'un~$y_j$ tel que $\Rn(y_j)$, $y_j d_j=c_j$ et $d_j\neq 0$. 
Cet~$y_j$ peut être représenté par une variable fraiche $Y_j$.
La validité du collapsus \pref{feqthVstformelval3+} équivaut donc au collapsus de la famille de conditions 
\[
(\Vr(X_i),\,X_ia_i-b_i=0)_{i\in\lrbn},\; (\Rn(Y_j),\, Y_jd_j-c_j=0,\, d_j\neq 0)_{j\in\lrbm}
\]
dans la \sad construite à partir du diagramme positif de $\gA$ en ajoutant des \gtrs $X_i$ et $Y_j$, \cad dans l'anneau $\gA[\uX,\uY]$. On conclut en appliquant le \vst~\ref{fCollval1} (avec $s=u=1$ et $z=0$).

\smallskip \noindent \textsl{4} $\Rightarrow$ \textsl{2}. 
On suppose avoir une identité \pref{feqthVstformelval4} et l'on veut démontrer la validité de la règle~\pref{feqthVstformelval2}. On  raisonne cas par cas en utilisant $\vd x=0\vou x\neq 0$ et  $\vd a\di b\vou b\di a$.
\\
Dans une branche où l'un des $d_j$ est supposé nul 
la règle \pref{feqthVstformelval2} est valide\footnote{Pour se rassurer, on peut vérifier que l'on a bien une identité \pref{feqthVstformelval4}, mais ce n'est pas le propos ici.}.
On suppose donc tous les $d_j$ non nuls, donc simplifiables (règle \Tsbf{AV2.5}). 
\\
Dans une branche où $a_i=0$, l'hypothèse de \pref{feqthVstformelval2}
implique $b_i=0$ et l'on peut remplacer~\hbox{$a_i$, $b_i$} et~$X_i$ par $0$ 
dans~\pref{feqthVstformelval4}. 
\\
Finalement il reste à montrer \pref{feqthVstformelval2} dans le cas où les $a_i$ et les $d_j$ sont supposés
non nuls donc réguliers. On se  place alors dans $\gB$, anneau total de fractions de $\gA$, dans lequel on obtient une \egt 
\begin{equation} \label {feqthVstformelval60}
0 =_\gB 1+\som_j \frac{c_j}{d_j} P_j\left(\frac{b_1}{a_1},\dots,\frac{b_n}{a_n},
\frac{c_1}{d_1},\dots,\frac{c_m}{d_m}\right)
\end{equation}
pour des \pols $P_j$ à \coes dans $\ZZ$. Soit $\delta$ un majorant des degrés en les $Y_k$ dans les $P_j$ 
et $\epsilon_i$ un majorant des degrés en $X_i$. On pose $d=\prod_k d_k$, on multiplie l'\egt précédente par $u=d^\delta\prod_ia_i^{\epsilon_i}$ et l'on obtient
\begin{equation} \label {feqthVstformelval7}
u  =_\gB \som_j \frac{c_j}{d_j}\, Q_j(b_1,a_1,\dots,b_n,a_n,
c_1,d_1,\dots,c_m,d_m)
\end{equation}
où chaque $Q_j$ est homogène de degré $\delta$ en chaque $(c_k,d_k)$ et de degré $\epsilon_i$ en $(a_i,b_i)$.
En posant $e_j=\prod_{k:k\neq j} d_k$ et en mutlipliant l'\egt précédente par $d$ on obtient une \egt dans $\gA$
\[
ud=\som_j c_j e_j q_j .
\] 
 Nous raisonnons maintenant cas par cas, et vu la conclusion que nous avons en vue, il reste~à traiter le cas où l'on suppose $d_j\di c_j$ pour chaque $j$. 
En raisonnant encore cas par cas on peut supposer que l'un des $c_je_j$, disons $c_1e_1$ par exemple, divise tous les autres. On  obtient alors
\begin{itemize}
\item $u\di q_j$ pour chaque $q_j$ ($u$ divise chacun de leurs \coes),
\item $uc_1 e_1\di c_j e_jq_j$ pour chaque $q_j$,%
\item donc $c_1 e_1u\di\som_j c_j e_j q_j=ud=d_1e_1u$, 
\end{itemize}

\noindent Enfin, la règle \tsbf{AV2.5} nous permet de simplifier par $e_1u$, on obtient $c_1 \di  d_1$.\\
Ouf!
\end{proof}
%

\begin{fremark} \label{fremvst1=>4} 
On pourrait sans doute démontrer plus directement l'implication \textsl{1} $\Rightarrow$ \textsl{4}
sans passer par le \vst formel \ref{fCollval1} pour les \cvds. On démontrerait d'une part que la \prt \pref{feqthVstformelval4} définit une \entrel sur $\gA\times \gA$ (la règle de coupure semble difficile), d'autre part qu'elle \gui{satisfait les
axiomes de \sa{val}}. 
Par exemple la règle \Tsbf{Av2}, $\,\,a \di  b \vet a \di  c \Vd a \di  b + c$
correspond à l'\egt suivante 
$$
(b+c)(1-y(x_1+x_2))\equiv 0 \,\mod\gen{x_1a-b,\,x_2a-c,\,y(b+c)-a}.\eqno\hbox{\eoe}
$$
    \end{fremark}

\Subsubsection{Admissibilité de la règle \tsbf{DIV} pour la \tdij \sa{val}}

\begin{fremark} \label{fremthVstformelval} 
Un corolaire du \vst~\ref{fthVstformelval} est le lemme \ref{flemDivadmissible} concernant l'admissibilité de la règle existentielle \Tsbf{DIV} dans la théorie \sa{val}. 
En effet le \vst formel est démontré pour la théorie $\sa{Cvd}^+$ (voir le point \textsl{4} de la \dfn \ref{fdefiApvandco}) dans laquelle on a introduit le prédicat $\cdot\di\cdot$ avec sa \dfn, et donc pour laquelle la règle~\tsbf{DIV} est valide. Mais le \vst formel fonctionne pour la théorie~\sa{val} toute seule. Donc les \rdijs valides pour le prédicat $x\di y$ restent les mêmes pour la théorie \sa{val} lorsqu'on lui ajoute l'axiome~\tsbf{DIV}. \eoe    
\end{fremark}

\Subsubsection{Un \vst formel  pour  $\sa{val}(\gA,\gk)$}

Une légère variante du  \vst~\ref{fthVstformelval}.

\begin{fvalsatz}[Un Valuativstellensatz formel pour  $\sa{val}(\gA,\gk)$ et $\sa{Cvd}^+\!(\gA,\gk)$]\label{fthVstformelvalbis} 
On considère deux anneaux commutatifs $\gk\subseteq \gA$ et des \elts $a_i$, $b_i$, $c_j$, $d_j\in\gA$. \Propeq 
\begin{enumerate}
\item On~a 
\begin{equation} \label {feqthVstformelvalbis1}
 (a_1,b_1),\dots,(a_n,b_n)  \vdash_{\val(\gA,\gk)}   (c_1,d_1),\dots,(c_m,d_m) 
\end{equation}
%
\item On~a 
\begin{equation} \label {feqthVstformelvalbis1.5}
a_1\di b_1\vet \dots\vet a_n\di b_n \Vdi{\sA{val}(\gA,\gk)}   c_1\di  d_1\vou \dots\vou  c_m\di d_m
\end{equation}
\item On~a 
\begin{equation} \label {feqthVstformelvalbis2}
a_1\di b_1\vet \dots\vet a_n\di b_n \Vdi{\sA{Cvd}^+\!(\gA,\gk)}   c_1\di  d_1\vou \dots\vou  c_m\di d_m
\end{equation}
\item On note  $y_j=\frac{c_j}{d_j}$ vu dans    $\gB=\gA[\frac{1}{d_1\cdots d_m}]$. Introduisons des \idtrs $X_i$ \hbox{($i\in\lrbn$)}. On~a dans l'anneau  $\gB[\uX]$ une \egt que l'on peut résumer sous la forme suivante 
\begin{equation} \label {feqthVstformelVabisl5}
1+\som_{j=1}^my_jP_j(\uX,\uy)\equiv 0 \mod \gen{(X_ia_i-b_i)_{i\in\lrbn}}
\end{equation}
 où  les $P_j(\uX,\uY)$ sont dans $\gk[\uX,\uY]$.
\end{enumerate}
\noindent \emph{Variante.} On sera intéressé plus loin pour expliciter ce que devient l'\eqvc entre les points 2, 3 et 4 lorsque certains des $a_i$ et des $c_j$ sont nuls, \cad lorsqu'on suppose nuls certains $b_i$ ou $d_j$. On note $e_i$ et $f_j$ les \elts de $\gA$ supposés nuls à gauche et à droite de~$\vd$.\\
 \Propeq
\begin{enumerate}\setcounter{enumi}{4}
\item On~a

\vspace{-1.3em} 
\[ 
\begin{array}{c} 
a_1\di b_1\vet \dots\vet a_n\di b_n\vet e_{1}=0\vet \dots\vet e_k=0 \Vdi{\sA{val}(\gA,\gk)}~\hspace{10em} \\[.3em]
~\hspace{5em}   c_1\di  d_1\vou \dots\vou  c_m\di d_m\vou f_1=0\vou \dots\vou f_\ell=0
 \end{array}
\]

\item On note  $y_j=\frac{c_j}{d_j}$ vu dans    $\gB=\gA[\frac{1}{d_1\cdots d_m\cdot f_1\cdots f_\ell}]$. Introduisons des \idtrs $X_i$ ($i\in\lrbn$). On~a dans l'anneau  $\gB[\uX]$ une \egt que l'on peut résumer sous la forme suivante 
\begin{equation} \label {feqthVstformelVabis5}
1+\som_{j=1}^my_jP_j(\uX,\uy)\equiv 0 \mod  \gen{(X_ia_i-b_i)_{i\in\lrbn}, (e_i)_{i\in\lrbk}}
\end{equation}
 où  les $P_j(\uX,\uY)$ sont dans $\gk[\uX,\uY]$.
\end{enumerate}
  
\end{fvalsatz}
\begin{proof} L'équivalence des points \textsl{1}, \textsl{3} et \textsl{4}
correspond à l'\eqvc des points~\textsl{1},~\textsl{3} et~\textsl{5} dans
le \vst \ref{fthVstformelval}. Notez que $\ZZ[\uX,\uY]$ (dans le point \textsl{5}) est maintenant remplacé par $\gk[\uX,\uY]$ (dans le point \textsl{4}) pour tenir compte de l'hypothèse que les \elts de $\gk$ sont déclarés entiers dans les axiomes.
\\
Il reste  à ajouter deux mots pour expliquer la variante,
qui aurait pu être donnée aussi pour le \vst~\ref{fthVstformelval}: l'\eqn \pref{feqthVstformelVabis5} est juste l'analogue de l'\eqn \pref{feqthVstformelVabisl5} pour l'anneau $\aqo\gB{e_1,\dots,e_\ell}$.
\end{proof}

Nous obtenons comme corolaire du \vst~\ref{fthVstformelvalbis} le résultat qui tient la promesse donnée après le lemme
\ref{flem-y-di-x}.

\begin{ftheorem} \label{fth-y-di-x}~\\
Soient $\gk\subseteq \gA$ deux anneaux, $x\in\gA$ et $y\in\gk$. 
\begin{enumerate}
\item La \sad  $\sa{val}(\gA,\gk)$ prouve $y\di x$ \ssi $x$ est entier sur l'\id $\gen{y}$ de $\gk$. 
\item Même chose pour $\sa{Cvd}^+\!(\gA,\gk)$.
\end{enumerate}

\end{ftheorem}
\begin{proof}
\textsl{1}. Cela se déduit de l'\eqvc des points \textsl{2} et \textsl{4} dans le \vst~\ref{fthVstformelvalbis} lorsque  $n=0$ et $m=1$.

\smallskip \noindent \textsl{2}. Utiliser le point \textsl{1} et l'\eqvc des points \textsl{2} et \textsl{3} dans le \vst~\ref{fthVstformelvalbis}.
\end{proof}

\Subsubsection{Le morphisme centre (2)}\label{fsubseccomparvalZar}

Nous réalisons maintenant la promesse faite dans la remarque \ref{frempropdefCentre}. En particulier, cela complète la \demo du \thref{fcor2thVstformelval4}.

\begin{ftheorem}[Le morphisme $\gamma:\ZarA\to \val(\gA,\gA)$ est injectif] \label{fthVstformelval4}~\\
Soient $\gA$ un anneau commutatif et des \elts $a_i$,  $c_j \in\gA$. \Propeq 
\begin{enumerate}
\item On~a 
$$
 (a_1,1),\dots,(a_n,1)  \,\,\vdash_{\val(\gA,\gA)}    (c_1,1),\dots,(c_m,1) 
$$

\item On~a dans l'anneau  $\gA$ une \egt de la forme 
$$
  \big(\prod\nolimits_{i=1}^n a_i\big)^r+\som_{j=1}^mc_jp_j = 0 
$$ 
\item On~a dans le treillis de Zariski de $\gA$
$$\rD(a_1),\dots,\rD(a_n)\,\,\vdash_{\Zar(\gA)}    \rD(c_1),\dots,\rD(c_m)$$
\end{enumerate}
En particulier  le morphisme centre (\dfn \ref{fpropdefCentre})  $\gamma:\ZarA\to \val(\gA,\gA)$ est injectif. 
\end{ftheorem}
%
\begin{proof}
L'\eqvc des points \textsl{1} et \textsl{2} est un cas particulier de celle donnée par le \vst~\ref{fthVstformelvalbis} (points \textsl{1} et \textsl{4} 
: calculer modulo $Xa-1$ dans $\AX$ revient à calculer dans $\gA[\frac 1 a]$).
L'\eqvc des points \textsl{2} et~\textsl{3} est le \nst formel~\ref{fthNstFormel}.
\end{proof}

\Subsubsection{Un autre \vst formel pour $\sa{val}(\gA,\gk)$}

Voici une autre légère variante du  \vst~\ref{fthVstformelval}, particulièrement utile au vu de la remarque \ref{fremthVstformelvalter}.

\begin{fvalsatz}[Un autre Valuativstellensatz formel pour $\sa{val}(\gA,\gk)$ et $\sa{Cvd}^+\!(\gA,\gk)$]\label{fthVstformelvalter}  
On considère deux anneaux commutatifs $\gk\subseteq \gA$  et des \elts  $b_i$ et \hbox{$d_j\in\gA$}. \Propeq 
\begin{enumerate}
\item On~a 
\begin{equation} \label {feqthVstformelvalter1}
 (1,b_1),\dots,(1,b_n)  \,\,\vdash_{\val(\gA,\gk)}   (1,d_1),\dots,(1,d_m) 
\end{equation}
%
\item On~a 
\begin{equation} \label {feqthVstformelvalter2}
 \Vr(b_1)\vet \dots\vet  \Vr(b_n) \Vdi{\sA{Cvd}(\gA,\gk)}    \Vr(d_1)\vou \dots\vou   \Vr(d_m)
\end{equation}
%
\item On~a 
\begin{equation} \label {feqthVstformelvalter2.5}
 \vr(b_1)\vet \dots\vet  \vr(b_n) \Vdi{\sA{val2}(\gA,\gk)}    \vr(d_1)\vou \dots\vou   \vr(d_m)
\end{equation}
%
\item En notant $\gB=\gA[\frac{1}{d_1\cdots d_m}]$
on a dans l'anneau  $\gB$ une \egt  
\begin{equation} \label {feqthVstformelvalter3}
1=\som_{j=1}^md_j^{-1}P_j(b_1,\dots,b_n,d_1^{-1},\dots,d_m^{-1})  
\end{equation}
 où  les $P_j(\Xn,\Ym)$ sont dans $\gk[\uX,\uY]$.
\item (Cas où  $\gA=\gK$ est un \cdi) On~a
\begin{equation} \label {feqthVstformelvalter5} 
 \Vr(b_1)\vet \dots\vet  \Vr(b_n) \Vdi{\sA{Val}(\gK,\gk)}    \Vr(d_1)\vou \dots\vou   \Vr(d_m)
\end{equation}
%
\item (Cas où  $\gA=\gK$ est un \cdi) On~a 
\begin{equation} \label {feqthVstformelvalter6} 
 \rV(b_1),\dots,\rV(b_n)  \,\,\vdash_{\Val(\gK,\gk)}   \rV(d_1),\dots,\rV(d_m) 
\end{equation}
\end{enumerate}
\end{fvalsatz}
%
\begin{proof}
Les premiers points sont un cas particulier du \vst~\ref{fthVstformelval}. Pour l'\eqvc avec les deux derniers (qui sont \eqvs par \dfn du treillis), cela résulte en fait de l'\iso des treillis $\val(\gK,\gk)$ et $\Val(\gK,\gk)$ (\thref{fthValval}). On~a alors $\gB=\gK$.
\end{proof}
%

\begin{fremark} \label{fremthVstformelvalter} 
Comme on le souhaitait, on retrouve bien ici le \vst formel pour le prédicat $\Vr$ donné dans l'article \cite{fCoq2009} 
dans le cas où $\gA$ est un \cdi et les~$b_i$ et~$d_j$ sont non nuls. 
Voir aussi \cite{fCP2001} et \cite{fLom2000}. \eoe    
\end{fremark}



\end{document}